%% file: MDF_Driver.tex
\pdfoutput=1
\documentclass[11pt,reqno]{amsart}
\usepackage{cite}
\usepackage{color}
\usepackage{fullpage}
\usepackage{graphicx}
\usepackage{enumerate}
\usepackage{subfigure}

\usepackage{multirow}
\usepackage{array, supertabular}
\usepackage{tabularx}
\usepackage{longtable}
\usepackage{booktabs}

\usepackage{algpseudocode}
\usepackage{algorithm}

\newtheorem{theorem}{Theorem}

\newtheorem{remark}[theorem]{Remark}

\newtheorem{proposition}[theorem]{Proposition}
\newtheorem{corollary}[theorem]{Corollary}

\title{Modification to Darcy model for high pressure and 
high velocity applications and associated mixed finite 
element formulations}
\author{$\mathrm{J.~Chang}^{*}$} \thanks{*{\tiny Graduate student}}
\author{$\mathrm{K.~B.~Nakshatrala}^{\#}$ \\ 
{\tiny Department of Civil \& Environmental Engineering, 
University of Houston, Houston 77204-4003.}}
\thanks{{\tiny ${}^{\#}$Correspondence to: Dr.~Kalyana Babu Nakshatrala,  
\textbf{\emph{e-mail:}} knakshatrala@uh.edu, \textbf{\emph{Phone:}}+1-713-743-4418.}}

\date{\today}

\begin{document}

\maketitle

\vspace{-0.3in} 

\begin{abstract}
The Darcy model is based on a plethora of assumptions. One 
of the most important assumptions is that the Darcy model 
assumes the drag coefficient to be constant. However, 
there is irrefutable experimental evidence that viscosities 
of organic liquids and carbon-dioxide depend on the 
pressure. Experiments have also shown that the drag 
varies nonlinearly with respect to the velocity at 
high flow rates. In important technological applications 
like enhanced oil recovery and geological carbon-dioxide 
sequestration, one encounters both high pressures 
and high flow rates. It should be emphasized that 
flow characteristics and pressure variation under 
varying drag are both quantitatively and qualitatively 
different from that of constant drag. 
Motivated by experimental evidence, we consider the 
drag coefficient to depend on both the pressure and 
velocity. We consider two major modifications to the
Darcy model based on the Barus formula and Forchheimer 
approximation. The proposed modifications to the Darcy 
model result in nonlinear partial differential equations, 
which are not amenable to analytical solutions. To this 
end, we present mixed finite element formulations based 
on least-squares (LS) formalism and variational multi-scale 
(VMS) formalism for the resulting governing equations. The 
proposed modifications to the Darcy model and its associated 
finite element formulations are used to solve realistic 
problems with relevance to enhanced oil recovery. 
We also study the competition between the nonlinear 
dependence of drag on the velocity and the dependence 
of viscosity on the pressure. To the best of the authors' 
knowledge such a systematic study has not been 
performed. 
\end{abstract}


\input{Sections/Ch1_Introduction}

\input{Sections/Ch2_Governing_equations}

\input{Sections/Ch3_Mixed_formulations}

\input{Sections/Ch4_Numerical_benchmark}

\input{Sections/Ch5_Enhanced_oil_recovery}
 
\input{Sections/Ch6_Conclusions}

\section*{ACKNOWLEDGMENTS}
The authors acknowledges the financial support from the Department of Energy. 
The opinions expressed in this paper are those of the authors and do not necessarily reflect that of the sponsors.

\bibliographystyle{unsrt}
\bibliography{bibliography}
 
\input{Sections/App_Figures}
\end{document}

%% file: Sections/Ch1_Introduction.tex
\section{INTRODUCTION AND MOTIVATION}
\label{Ch:Intro}

Understanding the flow of fluids through porous 
media plays a crucial role in various technological 
applications (e.g., designing filters, enhanced oil 
recovery, geological carbon-dioxide sequestration) 
and for mathematical modeling in various branches 
of engineering (e.g., civil engineering, petroleum 
engineering, polymer engineering). 
Arguably, the most popular model in the studies 
on flow through porous media is the Darcy model, 
which is named after the French hydraulics engineer 
Henry Darcy who first proposed the equation in 1856 
\cite{Darcy_1856}. 
Darcy originally developed the model empirically 
based on the experiments on the flow of water 
in sand beds. 
However, the Darcy model can be given firm mathematical 
basis at least in two different ways. One approach is by 
applying the volume averaging theory on the Navier-Stokes 
equations \cite{Derivation_of_Darcys_Law}. The other 
approach is using the theory of interacting continua 
(also known as mixture theory). For example, see reference 
\cite[Introduction]{Nakshatrala_Rajagopal_IJNMF_2011_v67_p342}). 
In this paper, the latter approach will be employed. 
  
\subsection{Limitations of Darcy model, and its generalizations}
Darcy equations model the flow of an incompressible 
fluid in rigid porous media by stating that the 
(Darcy) velocity is linearly proportional to the 
gradient of the pressure. 
It is important to note that the Darcy equation is 
simply an approximation of the balance of linear 
momentum in the context of theory of interacting 
continua. It merely predicts the flux but cannot 
predict stresses in solids. That is, this model 
cannot be used with modification when there is 
deformation / damage of the porous solid (e.g., 
in the case of hydraulic fracture). For completeness 
and future reference, let us enumerate the key 
assumptions behind the Darcy model 
\cite{Nakshatrala_Rajagopal_IJNMF_2011_v67_p342}. 
$\bullet$ There is no mass production of individual 
constituents (i.e., there are no chemical reactions). 
$\bullet$ The porous solid is assumed to be rigid. 
Thus, the balance laws for the solid are 
trivially satisfied. In particular, the 
stresses in the solid are what they need 
to be to ensure that the balance of linear 
momentum is met. 
$\bullet$ The fluid is assumed to be homogeneous 
and incompressible.
$\bullet$ The velocity and its gradient are 
assumed to be small so that the inertial 
effects can be ignored.
$\bullet$ The partial stress in the fluid is 
that of an Euler fluid. That is, there is no 
dissipation of energy between fluid layers.
$\bullet$ The only interaction force is at 
the fluid and pore boundaries.
$\bullet$ Darcy model assumes that the drag 
coefficient is independent of the pressure 
and the velocity. 
Experimental studies have shown that the 
viscosity of organic liquids (and hence the 
drag coefficient) depend on the pressure 
\cite{Bridgman,CarlBarus} and the drag 
coefficient depends on the velocity at 
high velocities \cite{Forchheimer_1901_v45_p1782,JacobBear}. 
Therefore, Darcy model as it is not appropriate 
for applications involving high pressure and 
high flow velocities. Application of Darcy 
model in such situations can result in 
erroneous predictions of discharge fluxes 
and inaccurate pressure contours. 
While several generalizations of the standard Darcy 
model have been proposed in the literature, none of 
these studies addressed the study undertaken in this 
paper. \emph{In particular, the prior studies did not 
address the combined effect of pressure-dependent 
viscosity and the dependence of drag coefficient 
on the velocity}. 

\subsubsection{Enhanced oil recovery}
\begin{figure}[h]
  \centering
  \includegraphics[scale=0.55]{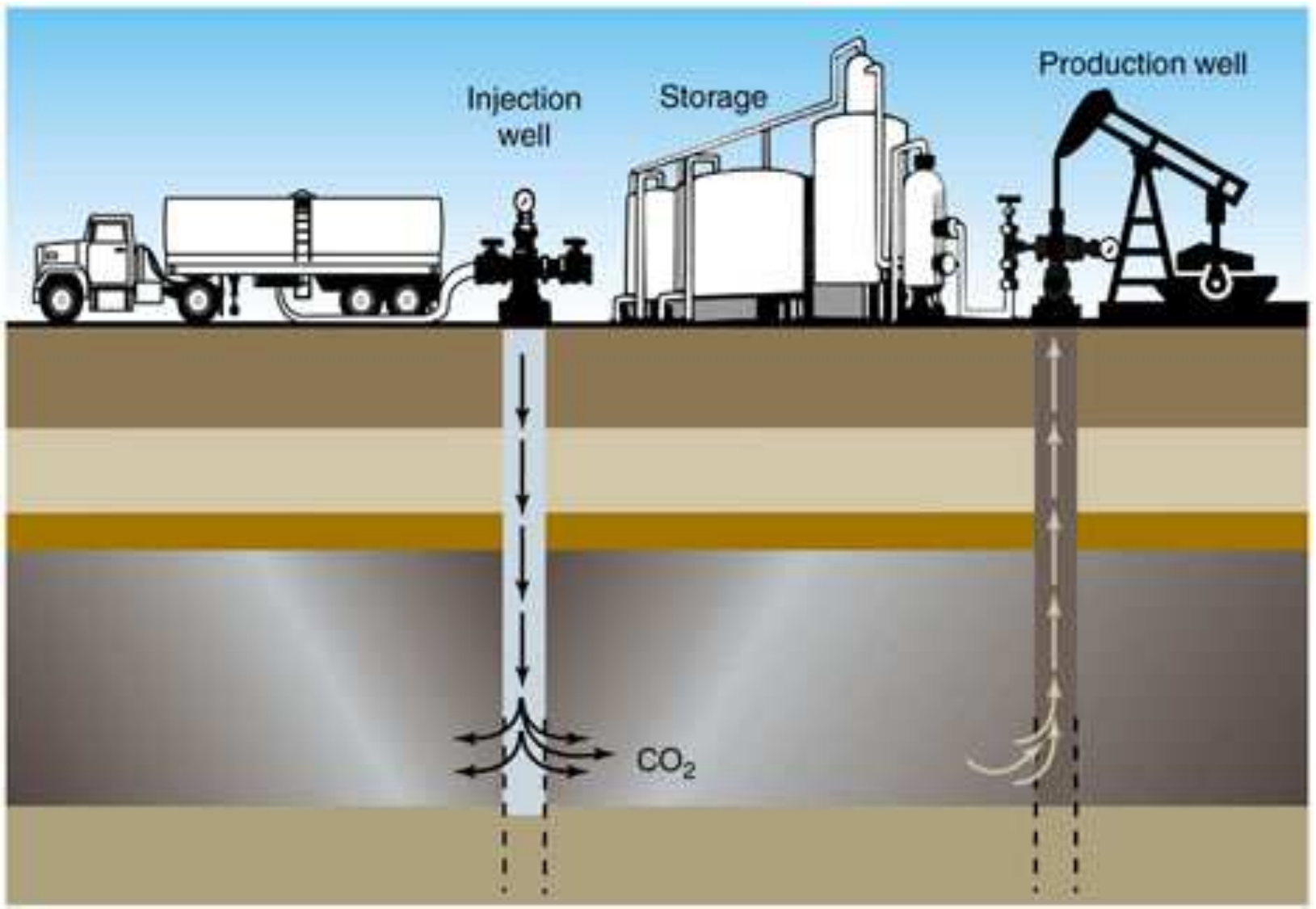}
  \caption{A pictorial description of enhanced oil recovery. 
  [Source:~https://www.llnl.gov/str/November01/Kirkendall.html]}
  \label{fig:Intro_eor}
\end{figure}
Over the years people have used Darcy model beyond its 
range of applicability. One example of misuse is in the 
modeling enhanced oil recovery (EOR) applications. As 
illustrated in Figure \ref{fig:Intro_eor}, steam / 
carbon-dioxide gas is injected into the ground through 
injection wells. The 
gas create a pressure build up in the ground (i.e., the 
porous media) and pushes the fluid (i.e., raw oil) out 
through the production wells. High pressures ranging from 
$10 - 100 \; \mathrm{MPa}$ are employed, and such high 
pressures can lead to inaccurate flow estimates or 
pressure contours if the original Darcy model is used. 
Oil reservoir simulations are tricky by nature because 
of the possibility of having varying permeability within 
layers, impervious zones, non-rigid rock and soil 
formations, and pockets of natural gases. Seismic 
imaging and field experimentation may not always 
return the most accurate data so one must be extremely 
cautious when providing parameters to run numerical 
models. Using the right Darcy modification(s) allows 
one to predict more accurate production rates, help 
industries determine where to allocate their resources, 
and prevent environmental damage from unintended 
cracking in the subsurface.

\subsection{Mixed formulations}
For the proposed model, we present mixed finite element 
formulations based on two different approaches: least-squares 
(LS) finite element method \cite{Jiang_LSFEM} and variational 
multi-scale (VMS) formalism \cite{Hughes_CMAME_1995_v127_p387}. 
It is well-known that care should be taken when working 
with mixed formulations. In order to get stable results, 
a mixed formulation should either satisfy or circumvent 
the Ladyzhenskaya-Babu\v ska-Brezzi (LBB) stability 
condition \cite{Brezzi_Fortin}. Both the proposed 
mixed formulations proposed in this paper circumvent 
the LBB condition. 

The least-squares finite element method (LSFEM) is 
based on the minimization of the residuals in a 
least-squares sense. One can always obtain a 
symmetric positive definite system of algebraic 
equations, even for non-self-adjoint problems. 
The LSFEM provides greater accuracy for the 
derivatives of primal variables when compared 
to single-field formulations, boundary conditions 
are easy to manage, the conformity of finite 
element spaces is sufficient to guarantee 
stability, and all variables can use the same 
finite element space. For further discussion on the
LSFEM, see references \cite{Bochev_Gunzburger_FEMLS,
Jiang_LSFEM}. 

The VMS formalism adds stabilization 
terms to the classical mixed formulation. The stabilization 
terms and the stabilization parameter can be derived in a 
consistent manner (e.g., see references 
\cite{Hughes_CMAME_1995_v127_p387,
Nakshatrala_Turner_Hjelmstad_Masud_CMAME_2006_v195_p4036}). 
A mixed formulation based on the VMS formalism falls under 
the category of stabilized methods, as in some sense the 
formulation is obtained by stabilizing the classical mixed 
formulation \cite{Donea_Huerta}. 
Several studies as shown in references 
\cite{KA_Mardal_Robust,Masud_IJNMF_2007_v54_p665,
VariationalHighPorous,MR_Correa_Velocity, Urquiza}
have proposed various stabilized formulations that 
provide accurate solutions of Darcy model through 
porous media, but none of these studies considered 
the proposed model. 
Some notable mixed formulation based on the VMS formalism 
for Darcy-type equations are references 
\cite{Masud_Hughes_CMAME_2002_v191_p4341,
Nakshatrala_Turner_Hjelmstad_Masud_CMAME_2006_v195_p4036,
Nakshatrala_Rajagopal_IJNMF_2011_v67_p342,
Nakshatrala_Turner_2013_arXiv}. We shall demonstrate in 
a subsequent section that the proposed mixed 
formulation encompasses these mixed formulations. 

\subsubsection{Local mass balance}
A common drawback of finite element formulations is 
that they need not possess local / element-wise mass 
balance property. In particular, the mixed formulations 
from both the LS and VMS formalisms do not possess the 
local mass balance property. It should be noted that while 
it is possible to achieve local mass conservation for the 
LS formalism, one would no longer be able to obtain continuous 
nodal quantities (e.g., see reference 
\cite{Bochev_Gunzburger_LocalConservative}). 
Several independent studies \cite{Petrov_Galerking_Darcy,
Local_conservative_Shuyu,Local_conservative_comparative} 
have successfully developed conservative finite element 
formulations for flow through porous media problems, but
none of them have been extended to modifications of 
Darcy's model. Another relevant work is reported in 
reference \cite{Turner_Nakshatrala_Subsurface} in 
which the effect of error in local mass balance on 
the transport of chemical species is studied. 
This study considered coupled flow and transport 
problems, and the comparison is made between a 
VMS-based formulation and the locally mass 
conservative Raviart-Thomas formulation. 
Herein, we shall perform a systematic study on the 
performance of the proposed mixed formulations for 
various flow models with respect to the local 
mass balance property. 
This study is intended to serve two purposes. First, 
it will guide users of the finite element method 
(FEM) on the extent of the violation of local mass balance 
under mixed formulations. Second, it will encourage 
researchers to improve the performance of finite 
element formulations with respect to local balance 
under arbitrary interpolation functions for the 
velocity and pressure. 

\subsection{Main contributions of this paper}
Several contributions have been made in this paper 
with respect to modeling of flow through porous 
media, associated mixed formulations, and numerical 
solutions of representative problems. Some of the 
main ones are as follows: 
\begin{enumerate}[(i)]
\item To propose a generalization of the Darcy model 
  by taking into account both the dependence of viscosity 
  on the pressure and the dependence of drag coefficient 
  on the velocity. The classical Darcy-Forchheimer and 
  the modified Darcy model that is considered in reference 
  \cite{Nakshatrala_Rajagopal_IJNMF_2011_v67_p342} will 
  be special cases of the generalized model considered 
  in this paper. The generalization is referred to as 
  the modified Darcy-Forchheimer model. 
\item To present some theoretical results pertaining to 
  the modified Darcy-Forchheimer model, and demonstrate 
  their utility in validating numerical implementations.
\item To develop a mixed formulation based on 
  LS formalism for the modified 
  Darcy-Forchheimer model and study the 
  effect of weighting on the convergence 
  and accuracy of the solutions. 
\item To construct a mixed formulation based on 
  the VMS formalism for the modified Darcy-Forchheimer 
  model, which encompasses as special cases some of the 
  existing mixed formulations proposed for simpler models. 
\item To compare the numerical performances of 
  VMS and LS based mixed formulations. 
\item To document the local mass balance error 
  under both these formalisms for the standard 
  Darcy model and for its generalizations.
\item It has been claimed in references 
  \cite{Nakshatrala_Turner_Hjelmstad_Masud_CMAME_2006_v195_p4036,
  Hughes_Masud_Wan_CMAME_2006_v195_p3347} 
  that VMS formulation is the only known mixed 
  formulation that satisfies three-dimensional 
  constant patch test under non-constant Jacobian 
  finite elements for the standard Darcy model. 
  Herein, we show that the proposed LS based 
  formulation (which is different from the 
  variational formulation proposed in these 
  references) also satisfies three-dimensional 
  constant flow patch tests.
\item To discuss the implications and applicability of 
  these modified models in numerical simulations of 
  enhanced oil recovery. It will also be show that 
  the pressure profiles of Darcy-Forchheimer are 
  qualitatively and quantitatively different from 
  that of a modification of Darcy model that takes 
  into account the dependence of viscosity on the 
  pressure. 
\item To illustrate an important competing effect 
  due to the dependence of viscosity on the pressure 
  and the dependence of drag coefficient on the 
  velocity. Specifically, to show the dependence 
  of drag coefficient on the velocity will create 
  steep gradients in the pressure near the projection 
  well. On the other hand, the dependence of viscosity 
  on the pressure creates steep pressure gradients 
  near the injection wells. However, both these 
  effects give rise to ceiling flux. 
\end{enumerate}

\subsection{Organization of the paper}
The remainder of the paper is organized as follows. 
In Section \ref{Ch:Governing} modifications to Darcy 
model using Barus formula and Forchheimer terms are 
presented. In Section \ref{Ch:Mixed}, mixed finite 
element formulations based on LS and VMS formalisms 
are proposed.
In Section \ref{Ch:benchmark}, several representative 
test problems are solved to show the performance and 
convergence of the proposed mixed finite element 
formulations, and  to illustrate the predictive 
capabilities of the modified Darcy-Forchheimer 
model. In Section \ref{Ch:EOR}, several representative 
problems with relevance to enhanced oil recovery 
are simulated, and the numerical solutions from 
the various models and formalisms are compared. 
Conclusions are drawn in Section \ref{Ch:Conclusions}.

%% file: Sections/Ch2_Governing_equations.tex
\section{GOVERNING EQUATIONS:~DARCY MODEL AND ITS GENERALIZATION}
\label{Ch:Governing}

Let $\Omega \subset \mathbb{R}^{nd}$ be an open and bounded 
set, where ``$nd$'' denotes the number of spatial dimensions. 
Let $\partial \Omega := \mathrm{cl}(\Omega) - \Omega$ be the 
boundary (where $\mathrm{cl}(\Omega)$ is the set closure of 
$\Omega$), which is assumed to be piecewise smooth. A spatial 
point in $\mathrm{cl}(\Omega)$ is denoted by $\mathbf{x}$. 
The gradient and divergence operators with respect to 
$\mathbf{x}$ are, respectively, denoted by $\mathrm{grad}
[\cdot]$ and $\mathrm{div}[\cdot]$. Let $\mathbf{v}:\Omega 
\rightarrow \mathbb{R}^{nd}$ denote the ``Darcy" velocity 
vector field (which is a homogenized velocity), and let 
$p:\Omega \rightarrow \mathbb{R}$ denote the pressure 
field. 
The boundary is divided into two parts, denoted by $\Gamma^{v}$ 
and $\Gamma^{p}$, such that $\Gamma^{v} \cap \Gamma^{p} = \emptyset$ 
and $\Gamma^{v} \cup \Gamma^{p} = \partial \Omega$. $\Gamma^{v}$ 
is the part of the boundary on which the normal component of the 
velocity is prescribed, and $\Gamma^{p}$ is part of the boundary 
on which the pressure is prescribed. 

We now consider the flow of an incompressible fluid 
through rigid porous media based on modifications to 
the standard Darcy model. The governing equations take 
the following form: 
\begin{subequations}
  \label{Eqn:GE}
  \begin{align}
    \label{Eqn:GE_darcy}
    &\alpha(\mathbf{v},p, \mathbf{x}) \mathbf{v}(\mathbf{x}) + 
    \mathrm{grad}[p(\mathbf{x})] = \rho \mathbf{b}(\mathbf{x}) 
    \quad \mathrm{in} \; \Omega \\
    \label{Eqn:GE_Continuity}
    &\mathrm{div}[\mathbf{v}(\mathbf{x})] = 0 \quad \mathrm{in} 
    \; \Omega \\
    \label{Eqn:GE_Vn}
    &\mathbf{v}(\mathbf{x}) \cdot \mathbf{\hat{n}}(\mathbf{x}) = 
    v_{n}(\mathbf{x}) \quad \mathrm{on} \; \Gamma^{v} \\
    \label{Eqn:GE_p0}
    &p(\mathbf{x}) = p_0(\mathbf{x}) \quad \mathrm{on} 
    \; \Gamma^{p}
  \end{align}
\end{subequations}
where $\alpha$ is the drag coefficient (which can depend 
on the velocity and pressure, and can spatially vary), 
$v_n(\mathbf{x})$ is the prescribed normal component of 
the velocity, $p_0(\mathbf{x})$ is the prescribed pressure, 
$\rho$ is the density of the fluid, $\mathbf{b}(\mathbf{x})$ 
is the specific body force, and $\mathbf{\hat{n}}(\mathbf{x})$ 
is the unit outward normal vector to the boundary. 
It can be shown that equation \eqref{Eqn:GE_darcy} is an 
approximation to the balance of linear momentum under the 
mathematical framework offered by the theory of interacting 
continua (e.g., see reference \cite[Introduction]
{Nakshatrala_Rajagopal_IJNMF_2011_v67_p342}). 
A more thorough discussion on the theory of interacting 
continua can be found in the several appendices of 
reference \cite{Truesdell_rational_thermodynamics}, 
Atkin and Craine \cite{Atkin_Craine_QJMAM_1976_v29_p209}, 
and Bowen \cite{Bowen}. 

\subsection{Boundary conditions and well-posedness}
We now briefly discuss the well-posedness of the 
aforementioned boundary value problem given by 
equations \eqref{Eqn:GE_darcy}--\eqref{Eqn:GE_p0} 
in the sense of Hadamard \cite{McOwen}. If $\Gamma^{v} 
= \partial \Omega$ (i.e., the normal component of the 
velocity is prescribed on the entire boundary), one 
has to meet the following compatibility condition 
for well-posedness:
\begin{align}
  \int_{\Gamma^{v} = \partial \Omega} v_n(\mathbf{x}) 
  \; \mathrm{d} \Gamma = 0
  \end{align}
which is a direct consequence of the divergence theorem. 
To wit, 
\begin{align}
  0 = \int_{\Omega} \mathrm{div}[\mathbf{v}(\mathbf{x})] 
  \; \mathrm{d} \Omega = \int_{\partial \Omega} \mathbf{v}
  (\mathbf{x}) \cdot \hat{\mathbf{n}}(\mathbf{x}) \; 
  \mathrm{d} \Gamma = \int_{\partial \Omega} v_n(\mathrm{x}) 
  \; \mathrm{d} \Gamma
\end{align}
Moreover, if $\Gamma^{p} = \emptyset$ (i.e., $\partial 
\Omega = \Gamma^{v}$), one needs to augment the above 
equations \eqref{Eqn:GE_darcy}--\eqref{Eqn:GE_p0} with 
an additional condition for uniqueness of the solution. 
Otherwise, one cannot find the pressure uniquely. In 
the Mathematics literature, the uniqueness is typically 
achieved by meeting the condition
\begin{align}
  \int_{\Omega} p(\mathbf{x}) \; \mathrm{d} \Omega = 0
\end{align}
which basically fixes the datum for the pressure. However, 
this approach is seldom used in a computational setting as 
it is difficult to enforce the above condition numerically. 
An alternative is to fix the datum for the pressure by 
prescribing the pressure at a point, which is commonly 
employed in various computational settings and is also 
employed in this paper. 

It should also be noted that the no-slip boundary 
condition is not compatible with the Darcy model 
and the generalization that is considered in this 
paper. A simple mathematical explanation can be 
provided by noting that the inclusion of no-slip 
boundary condition (in addition to the no-penetration 
boundary condition) will make the boundary value 
problem over-determined. Also, it is noteworthy 
that the governing equations based on Darcy model 
are first-order (in terms of number of derivatives) 
with respect to the field variables $\mathbf{v}
(\mathbf{x})$ and $p(\mathbf{x})$. 

\subsection{Darcy model, experimental evidence, and its generalization}
The Darcy model assumes that the drag coefficient is 
independent of the pressure and velocity. In addition, 
the Darcy model assumes the drag coefficient to be of 
the form
\begin{align}
  \alpha = \frac{\mu}{k}
\end{align}
where $\mu$ is the coefficient of viscosity of the fluid, 
and $k$ is the coefficient of permeability. From the above 
discussion it is evident that Darcy model cannot be employed 
for situations in which the viscosity depends on the pressure, 
permeability depends on the (pore) pressure, or drag does not 
depend linearly on the velocity of the fluid (i.e., the drag 
coefficient depends on the velocity). Several experiments 
have shown unequivocally that these three situations occur 
in nature, which will now be discussed. 

\subsubsection{Pressure-dependent viscosity}
Bridgman \cite{Bridgman} has shown that the 
viscosity of several organic liquids depend on the 
pressure, and in fact, the dependence is exponential. 
Notable scientists such as Andrade \cite{Andrade} 
and Barus \cite{CarlBarus} have performed laboratory 
experiments on liquids to determine the relationship 
between pressure and viscosity. In recent years, 
research such as that in \cite{vanLeeuwen} has 
been able to obtain empirical evidence to delineate 
and confirm the dependency of viscosity on pressure. 
Furthermore, numerical studies have been performed 
in references \cite{M_Franta_Rajagopal,A_Him_FE_approximate} 
to record the differences these pressure dependent viscosity 
equations make for several fluid problems like the 
Navier-Stokes equation.

There are several ways one can generalized the standard 
Darcy model. For example, one can model the friction 
between the layers of the fluid, which the standard 
Darcy model neglects. This is approach taken by 
Brinkman (see references \cite{Brinkman_ASR_1947_vA1_p27,
Shriram_Nakshatrala}). The research conducted in this paper 
focuses on generalizing the standard Darcy model by modifying 
the drag to depend on the velocity and the pressure.  

To account for the dependence of the viscosity 
(and hence the drag) on the pressure, Barus' 
formula \cite{Szeri}will be used. The drag 
coefficient based on Barus' formula can be 
rewritten as
\begin{align}
  \label{Eqn:GE_Barus}
  &\alpha(p,\mathbf{x}) = \frac{\mu(p)}{k(\mathbf{x})} = 
  \frac{\mu_0}{k(\mathbf{x})} \exp[\beta_{\mathrm{B}} p]
\end{align}
where $\mu_0$ is the fixed viscosity of the fluid and 
$\beta_{\mathrm{B}}$ is the Barus coefficient that is 
obtained experimentally. This proposed modification 
states that the viscosity varies exponentially with 
pressure, and one can determine the Barus coefficient 
$\beta_{\mathrm{B}}$ using laboratory experiments, and 
its value for common organic liquids (e.g., Naphthenic 
mineral oil) has been documented in the literature. 
For example, see references 
\cite{Bridgman,Abramson_PRE_2009_v80_021201,
Vesovic_Wakeham_Olchowy_Sengers_Watson_Millat_JPhysChem_1990_v19_p763,
Hoglund_Wear_1999_v232_p176}). 

\subsubsection{High velocity flows and inertial effects}
It has been experimentally observed that for high velocity 
flows in porous media, the flux (and hence the flow rate) 
is not linearly proportional to the gradient of the pressure. 
This can be explained by noting that inertial effects can 
play a dominant role for high velocity flows. The standard 
Darcy model completely ignores inertial effects. To address 
the nonlinear dependence of the flux on the gradient of the 
pressure for high velocity flows, Philipp Forchheimer, an 
Austrian scientist (1852--1933), proposed that the drag 
coefficient to depend on the velocity of the fluid 
\cite{Forchheimer_1901_v45_p1782}. Herein, the model 
that is obtained after incorporating Forchheimer's 
modification will be referred to as the 
\emph{Darcy-Forchheimer model}. 

It is noteworthy that the Darcy-Forchheimer model can 
be obtained from the Navier-Stokes equations using the 
volume averaging method \cite{ISI:A1996VK96700002}. In 
typical geotechnical and civil engineering applications, 
one encounters low velocities so the inertial effects 
can be disregarded, and the standard Darcy model is 
adequate. However, in high pressure applications like 
enhanced oil recovery one may often encounter high 
velocities so inertial effects must be accounted for. 
The Darcy-Forchheimer model is written as 
\begin{align}
  \label{Eqn:GE_Forchheimer}
  &\alpha(\mathbf{v},\mathbf{x}) = 
  \frac{\mu_0}{k(\mathbf{x})}    
  + \beta_{\mathrm{F}} \|\mathbf{v}\| 
\end{align}
where $\beta_{\mathrm{F}}$ is the Forchheimer 
or inertial coefficient, and $\|\cdot\|$ 
is the 2-norm. That is, 
\begin{align}
  \|\mathbf{v}\| = \sqrt{\mathbf{v}(\mathbf{x}) 
    \cdot \mathbf{v}(\mathbf{x})}
\end{align}
Several people have proposed their own experimental, 
theoretical, or computational formulations for the 
Forchheimer coefficient (see reference \cite{W_sobieski}). 
For instance, one way to express $\beta_{\mathrm{F}}$ is
\begin{align}
  \label{Eqn:GE_inertial}
  \beta_{\mathrm{F}} = \frac{c_{\mathrm{F}}\rho}{\sqrt{k_{I}}}
\end{align}
where $c_{\mathrm{F}}$ is a dimensionless form-drag constant 
and $k_{I}$ is the inertial permeability, both of which can 
be obtained experimentally. Successful mixed finite element 
formulations have been performed on the Darcy-Forchheimer 
model in references \cite{EJ_Park,Hao_Pan_Forchheimer}, 
but they all use different variants of the Forchheimer 
coefficient. For the purpose of this paper, $\beta_{\mathrm{F}}$ 
shall remain as a user-defined parameter.

\begin{remark}
  The laws of (Newtonian) mechanics are Galilean 
  invariant. Therefore, one need to construct the 
  constitutive relations to be Galilean invariant 
  so as to be consistent with the laws of mechanics. 
  At first glance, it may look like the model 
  \eqref{Eqn:GE_Forchheimer} and equation 
  \eqref{Eqn:GE_darcy} are not Galilean invariant, 
  as the velocity and the 2-norm of the velocity 
  are not invariant under Galilean transformations. 
  However, it should be note that the velocity 
  $\mathbf{v}(\mathbf{x})$ in these cases is 
  the relative velocity between the velocity 
  of the fluid and the velocity of the porous 
  solid. In Darcy-type models, it is tacitly 
  assumed that the porous solid is rigid and 
  does not undergo any motion, which is also 
  the case in this paper. Therefore, the 
  velocity $\mathbf{v}(\mathbf{x})$ is 
  equal to the velocity of the fluid. 
  Noting that the velocity $\mathbf{v}(\mathbf{x})$ 
  is the relative velocity is important, as the 
  relative velocity and the norm of the relative 
  velocity are Galilean invariant. Hence, the model 
  \eqref{Eqn:GE_Forchheimer} is invariant under 
  Galilean transformations. This will be the case 
  even with the other models considered in this 
  paper.
\end{remark}

\begin{remark}
  Some porous solids exhibit strong correlation between 
  permeability and porosity, and studies presented in 
  reference \cite{PhysRevA.46.7680} show that the 
  porosity is affected by the (pore) pressure. For 
  these porous solids, one can conclude that the 
  pressure affects the permeability, which in turn 
  will give rise to the dependence of drag coefficient 
  on the pressure. In this paper we do not solve any 
  problem that involves permeability depending on the 
  pressure. However, the proposed mixed formulations 
  can be easily extended to handle such problems. 
\end{remark}

\subsubsection{Proposed model: Modified Darcy-Forchheimer model}
A major focus of this research will study the effects 
of incorporating pressure-dependent viscosity into 
the Darcy-Forchheimer model. The drag coefficient 
can then be rewritten as
\begin{align}
  \label{Eqn:GE_ModifiedForchheimer}
  &\alpha(\mathbf{v},p,\mathbf{x}) =  \frac{\mu(p)}{k(\mathbf{x})} 
  + \beta_{\mathrm{F}}\|\mathbf{v}\| = \frac{\mu_0}{k(\mathbf{x})} 
  \exp[\beta_{\mathrm{B}} p] + \beta_{\mathrm{F}} \|\mathbf{v}\| 
\end{align}
%
%
The proposed model is suitable for applications like 
enhanced oil recovery, geological carbon-dioxide 
sequestration, and filtration process. The terms $\mu(p)$ 
and $\beta_{F} \|\mathbf{v}\|$ (which are both nonlinear) 
can have competitive effects, and neglecting either of 
these can give erroneous results for these applications. 

It will now be shown that the modified Darcy-Forchheimer 
model is dissipative. That is, the proposed constitutive 
model satisfies the second law of thermodynamics. Within 
the context of theory of interacting continua for bodies 
undergoing isothermal processes \cite{Bowen}, the total 
rate of dissipation density at a spatial point $\mathbf{x} 
\in \Omega$, $\xi_{\mathrm{total}}(\mathbf{x})$, is written 
as
\begin{align}
  \xi_{\mathrm{total}}(\mathbf{x}) = \xi_{\mathrm{solid}}
  (\mathbf{x}) + \xi_{\mathrm{fluid}}(\mathbf{x}) + 
  \xi_{\mathrm{interaction}}(\mathbf{x})
\end{align}
where $\xi_{\mathrm{solid}}(\mathbf{x})$ and $\xi_{\mathrm{fluid}}
(\mathbf{x})$ are the bulk rate of dissipation densities within the 
solid and the fluid, and $\xi_{\mathrm{interaction}}(\mathbf{x})$ 
is the bulk rate of dissipation density due to interaction of the 
solid and the fluid at their corresponding interfaces. 
Since the solid is assumed to be rigid,  
\begin{align}
  \xi_{\mathrm{solid}}(\mathbf{x}) = 0
\end{align}
The fluid is assumed to be perfect (i.e., an Euler 
fluid), so there is no (internal) dissipation within 
the fluid. Thus we have 
\begin{align}
  \xi_{\mathrm{fluid}}(\mathbf{x}) = 0
\end{align}
However, it should be emphasized that there is 
dissipation at the interface between the solid 
and fluid, which is due to the drag. Hence the 
total rate of dissipation density at a spatial 
point $\mathbf{x}$ is given by 
\begin{align} 
  \xi_{\mathrm{total}}(\mathbf{x}) = \xi_{\mathrm{interaction}} 
  (\mathbf{x}) = \alpha(\mathbf{v},p,\mathbf{x}) \, \| 
  \mathbf{v}(\mathbf{x}) \| ^ 2
\end{align}
where $\| \cdot \|$ is the 2-norm norm and $\mathbf{v}
(\mathbf{x})$ is the relative velocity of the fluid 
with respect to the solid. By ensuring that $\alpha
(\mathbf{v},p,\mathbf{x}) > 0$ one can satisfy the 
second law of thermodynamics \emph{a priori}. For the 
modified Darcy-Forchheimer model given by equation 
\eqref{Eqn:GE_ModifiedForchheimer} $\alpha(\mathbf{v},p,\mathbf{x}) 
> 0$, as $\mu_0 > 0, k(\mathbf{x}) > 0$, $\beta_{\mathrm{F}} 
\geq 0$, $\|\mathbf{v}\| \geq 0$ and $\exp[\cdot] > 0$.
The total dissipation due to drag in the entire domain 
takes the following form:
\begin{align}
  \Phi := \int_{\Omega} \xi_{\mathrm{interaction}}
  (\mathbf{x}) \; \mathrm{d} \Omega = \int_{\Omega} 
  \alpha(\mathbf{v},p,\mathbf{x}) \mathbf{v}(\mathbf{x}) 
  \cdot \mathbf{v}(\mathbf{x}) \; \mathrm{d} \Omega
\end{align}
which is clearly non-negative. 

\begin{remark}
  A remark is warranted on the interpretation(s) 
  of the quantity $p(\mathbf{x})$, which was 
  referred to as the pressure earlier. 
  Within the theory of constraints 
  \cite{Carlson_Fried_Tortorelli_JE_2003_v70_p101,
    OReilly_Srinivasa_PRSLSA_2001_v457_p1307}, 
  the quantity $p(\mathbf{x})$ is the undetermined multiplier 
  that arises due to the incompressibility constraint, which 
  is given by equation \eqref{Eqn:GE_Continuity}. Note that 
  $p(\mathbf{x})$ is not referred to as a Lagrange multiplier 
  as there are no Lagrange multipliers under the mathematical 
  framework for constraints that is outlined in references 
  \cite{Carlson_Fried_Tortorelli_JE_2003_v70_p101,
    OReilly_Srinivasa_PRSLSA_2001_v457_p1307}. 
  Under the theory of interacting continua, the partial 
  (Cauchy) stress in the fluid for Darcy model takes the 
  form 
  \begin{align}
    \mathbf{T}^{(f)} = -p(\mathbf{x}) \mathbf{I},
  \end{align}
  where $\mathbf{I}$ is the second-order identity tensor. 
  Therefore, under the theory of interacting continua 
  framework, $p(\mathbf{x})$ can be considered as the 
  mechanical pressure in the fluid. Note that the 
  mechanical pressure is defined as the negative of 
  the mean normal stress (see Batchelor \cite{Batchelor}). 
  Therefore, for the modified Darcy-Forchheimer model, 
  $p(\mathbf{x})$ is both the mechanical pressure in the 
  fluid, and the undetermined multiplier to enforce the 
  incompressibility constraint. The above discussion 
  on the precise identity and role of $p(\mathbf{x})$ 
  will be extremely important if one wants to make 
  further generalizations / modifications to the 
  proposed model. In particular, to extend the 
  proposed model to incorporate degradation and 
  fracture of the porous solid, which will be 
  part of our future work. 
\end{remark}

\subsection{Some theoretical results for the modified 
  Darcy-Forchheimer model}
For the entire discussion in this subsection, assume that 
$\Gamma^{v} = \partial \Omega$. We shall also assume that 
the body force is a conservative vector field. That is, 
there exists a scalar potential field $\phi(\mathbf{x})$ 
such that $\rho \mathbf{b}(\mathbf{x}) = \mathrm{grad}
[\phi]$. 
We shall refer to a vector field $\widetilde{\mathbf{v}}
(\mathbf{x}):\Omega \rightarrow \mathbb{R}^{nd}$ as 
\emph{kinematically admissible} if it satisfies the 
following conditions: 
\begin{enumerate}[(i)]
\item $\widetilde{\mathbf{v}}(\mathbf{x})$ is solenoidal  
  (i.e., $\mathrm{div}[\widetilde{\mathbf{v}}(\mathbf{x})] 
  = 0 \; \mathrm{in} \; \Omega$)
\item $\widetilde{\mathbf{v}}(\mathbf{x})$ satisfies 
  the boundary conditions (i.e., $\widetilde{\mathbf{v}}
  (\mathbf{x}) \cdot \hat{\mathbf{n}}(\mathbf{x}) = v_n
  (\mathbf{x}) \; \mathrm{on} \; \partial \Omega$)
\end{enumerate}
Note that $\widetilde{\mathbf{v}}(\mathbf{x})$ 
need not satisfy the balance of linear momentum 
given by equation \eqref{Eqn:GE_darcy}. We now 
present an important property that the solutions 
of modified Darcy-Forchheimer equations 
\eqref{Eqn:GE_darcy}--\eqref{Eqn:GE_p0} 
satisfy:~the minimum dissipation inequality. 

\begin{proposition}{[Minimum dissipation inequality]}
  Let $\left\{\mathbf{v}(\mathbf{x}), p(\mathbf{x})
  \right\}$ be the solution of equations 
  \eqref{Eqn:GE_darcy}--\eqref{Eqn:GE_p0}. 
  Any kinematically admissible vector field 
  $\widetilde{\mathbf{v}}(\mathbf{x})$ has 
  to satisfy the following inequality:
  \begin{align}
    \int_{\Omega} \alpha(\mathbf{v}(\mathbf{x}),
    p(\mathbf{x}),\mathbf{x}) \mathbf{v}(\mathbf{x}) 
    \cdot \mathbf{v}(\mathbf{x}) \; \mathrm{d} \Omega \leq 
    \int_{\Omega} \alpha(\mathbf{v}(\mathbf{x}),
    p(\mathbf{x}),\mathbf{x}) \widetilde{\mathbf{v}}
    (\mathbf{x}) \cdot \widetilde{\mathbf{v}}(\mathbf{x}) 
    \; \mathrm{d} \Omega 
  \end{align}
\end{proposition}
\begin{proof}
  Let 
  \begin{align}
    \delta \mathbf{v}(\mathbf{x}) := 
    \widetilde{\mathbf{v}}(\mathbf{x}) - 
    \mathbf{v}(\mathbf{x})
  \end{align}
  From the hypothesis, $\delta \mathbf{v}(\mathbf{x})$ 
  satisfies the following equations:
  \begin{subequations}
    \begin{align}
      &\delta \mathbf{v}(\mathbf{x}) \cdot \hat{\mathbf{n}} 
      (\mathbf{x}) = 0 \quad \forall \mathbf{x} \in \partial 
      \Omega \\
      &\mathrm{div}[\delta \mathbf{v}] = 0 \quad 
      \forall \mathbf{x} \in \Omega 
    \end{align}
  \end{subequations}
  Let us simplify the expression for the difference 
  in the total dissipation due to drag:
  \begin{align}
    \delta \Phi &:= \int_{\Omega} \alpha(\mathbf{v}
    (\mathbf{x}),p(\mathbf{x}),\mathbf{x}) 
    \widetilde{\mathbf{v}} 
    (\mathbf{x}) \cdot \widetilde{\mathbf{v}}(\mathbf{x}) \; 
    \mathrm{d} \Omega - 
    \int_{\Omega} \alpha(\mathbf{v}(\mathbf{x}),
    p(\mathbf{x}),\mathbf{x}) \mathbf{v}(\mathbf{x}) 
    \cdot \mathbf{v}(\mathbf{x}) \; \mathrm{d} 
    \Omega \\
    &= \int_{\Omega} \alpha(\mathbf{v}(\mathbf{x}),
    p(\mathbf{x}),\mathbf{x}) 
    \delta \mathbf{v}(\mathbf{x}) 
    \cdot \left(\delta \mathbf{v}(\mathbf{x}) + 2 
    \mathbf{v}(\mathbf{x}) \right) \; \mathrm{d} \Omega \\
    &= \int_{\Omega} \alpha(\mathbf{v}(\mathbf{x}),
    p(\mathbf{x}),\mathbf{x}) \delta \mathbf{v}
    (\mathbf{x}) \cdot \delta \mathbf{v}(\mathbf{x}) 
    \; \mathrm{d} \Omega 
    + 2 \int_{\Omega} \alpha(\mathbf{v}(\mathbf{x}),
    p(\mathbf{x}),\mathbf{x}) \delta \mathbf{v}
    (\mathbf{x}) \cdot \mathbf{v}(\mathbf{x}) \; 
    \mathrm{d} \Omega \\
    &\geq 2 \int_{\Omega} \alpha(\mathbf{v}(\mathbf{x}),
    p(\mathbf{x}),\mathbf{x}) \delta \mathbf{v}
    (\mathbf{x}) \cdot \mathbf{v}(\mathbf{x}) 
    \; \mathrm{d} \Omega = 2 \int_{\Omega} \delta \mathbf{v}
    (\mathbf{x}) \cdot \mathrm{grad}[\phi(\mathbf{x}) - 
      p(\mathbf{x})] \; \mathrm{d} \Omega
  \end{align}
  Using Green's identity, we obtain the following inequality:
  \begin{align}
    \delta \Phi \geq  2 \int_{\partial \Omega} \delta \mathbf{v}
    (\mathbf{x}) \cdot \hat{\mathbf{n}}(\mathbf{x}) \left(\phi
    (\mathbf{x}) - p(\mathbf{x})\right) \; \mathrm{d} \Gamma
    - 2 \int_{\Omega} \mathrm{div}[\delta \mathbf{v}
      (\mathbf{x})] \left(\phi(\mathbf{x}) - 
      p(\mathbf{x})\right) \; \mathrm{d} \Omega = 0
  \end{align}
  This completes the proof.
\end{proof}
It is easy to obtain the following corollary for 
the case of constant drag coefficient.
\begin{corollary}
  Let $\mathbf{v}_1(\mathbf{x})$ and $\mathbf{v}_2(\mathbf{x})$ 
  be two Darcy velocities (i.e., they satisfy equations 
  \eqref{Eqn:GE_darcy}--\eqref{Eqn:GE_p0}) corresponding 
  to two different \emph{constant} drag coefficients but 
  for the same conservative body force and velocity 
  boundary conditions, and for a given domain. Then 
  the velocities satisfy 
  \begin{align}
    \int_{\Omega} \mathbf{v}_1(\mathbf{x}) \cdot 
    \mathbf{v}_1(\mathbf{x}) \; \mathrm{d} \Omega = 
    \int_{\Omega} \mathbf{v}_2(\mathbf{x}) \cdot 
    \mathbf{v}_2(\mathbf{x}) \; \mathrm{d} \Omega
  \end{align}
\end{corollary}
%
\begin{remark}
  Let the drag coefficient be independent of 
  the velocity and the pressure. Let $\left\{
  \mathbf{v}_1(\mathbf{x}), p_1(\mathbf{x})
  \right\}$ and $\left\{\mathbf{v}_2(\mathbf{x}), p_2
  (\mathbf{x})\right\}$ be the solutions of equations 
  \eqref{Eqn:GE_darcy}--\eqref{Eqn:GE_p0} for the 
  prescribed data $\left\{\mathbf{b}_1(\mathbf{x}), 
  {v_n}_1(\mathbf{x})\right\}$ and $\left\{\mathbf{b}_2
  (\mathbf{x}), {v_n}_2(\mathbf{x})\right\}$, respectively. 
  These fields satisfy the following relation:
  \begin{align}
    \label{Eqn:MDF_reciprocal}
    \int_{\Omega} \rho \mathbf{b}_1(\mathbf{x}) \cdot 
    \mathbf{v}_2(\mathbf{x}) \; \mathrm{d} \Omega
    - \int_{\partial \Omega} p_1(\mathbf{x}) {v_n}_2(\mathbf{x}) 
    \; \mathrm{d} \Gamma
    = \int_{\Omega} \rho \mathbf{b}_2(\mathbf{x}) \cdot 
    \mathbf{v}_1(\mathbf{x}) \; \mathrm{d} \Omega
    - \int_{\partial \Omega} p_2(\mathbf{x}) {v_n}_1(\mathbf{x}) 
    \; \mathrm{d} \Gamma
  \end{align}
  The solutions to Darcy equations satisfy an identity 
  similar to the Betti reciprocal relations in the 
  theory of elasticity \cite{Sadd,Truesdell_Noll} 
  and creeping flows \cite{Guazzelli_Morris}.
  However, equation \eqref{Eqn:MDF_reciprocal} is 
  not valid for modified Darcy-Forchheimer model. 
\end{remark}
The above results not only have theoretical significance 
but can also be invaluable in testing a numerical 
implementation.

%% file: Sections/Ch3_Mixed_formulations.tex
\section{MIXED TWO-FIELD WEAK FORMULATIONS}
\label{Ch:Mixed}

It is, in general, not possible to obtain analytical 
solutions for the mathematical models presented in 
the previous section. Hence, one may have to resort 
to numerical solutions. One of the main goals of 
this paper is to present mixed finite element 
formulations based on the least-squares (LS) and the 
variational multi-scale (VMS) formalisms for solving 
the boundary value problem arising from the 
modified Darcy-Forchheimer model. Note that 
the standard Darcy, Forchheimer and modified 
Darcy models are special cases of the proposed 
modified Darcy-Forchheimer model. Therefore, 
the proposed mixed formulations can be used 
to solve these models, and encompass some 
of the prior mixed formulations that have 
developed for these simpler models.

The following function spaces will be used 
in the remainder of this paper: 
\begin{subequations}
  \begin{align}
    \label{Eqn:DF_function_p_strong}
    &\mathcal{P} := \left\{p(\mathbf{x}) \in H^{1}(\Omega) 
    \; \vert \; p(\mathbf{x}) = p_0(\mathbf{x}) \; \mathrm{on} 
    \; \Gamma^{p} \right\} \\
    \label{Eqn:DF_function_q_strong}
    &\mathcal{Q} := \left\{q(\mathbf{x}) \in H^{1}(\Omega) 
    \; \vert \; q(\mathbf{x}) = 0 \; \mathrm{on} \; 
    \Gamma^{p} \right\} \\
     \label{Eqn:DF_function_q_weak}
    &\widetilde{\mathcal{Q}} := \left\{q(\mathbf{x}) 
     \in H^{1}(\Omega) \right\}, \\
    &\mathcal{V} := \left\{\mathbf{v}(\mathbf{x}) 
    \in \left(L_2(\Omega)\right)^{nd} \; \vert \; 
    \mathrm{div}[\mathbf{v}] \in L_2(\Omega), \; 
    \mathbf{v}(\mathbf{x}) \cdot \hat{\mathbf{n}}
    (\mathbf{x}) = v_n(\mathbf{x}) \; \mathrm{on} 
    \; \Gamma^{v} \right\} \\
    &\mathcal{W} := \left\{\mathbf{w}(\mathbf{x}) 
    \in \left(L_2(\Omega)\right)^{nd} \; \vert \; 
    \mathrm{div}[\mathbf{w}] \in L_2(\Omega), \; 
    \mathbf{w}(\mathbf{x}) \cdot \hat{\mathbf{n}} 
    (\mathbf{x}) = 0 \; \mathrm{on} \; \Gamma^{v} 
    \right\}
  \end{align}
\end{subequations}
where $L_2(\Omega)$ and $H^{1}(\Omega)$ are standard 
Sobolev spaces \cite{Brezzi_Fortin}. Note that two 
different function spaces are defined for the pressure 
trial function. If the pressure is prescribed strongly 
on $\Gamma^{p}$ then the function space given in equation 
\eqref{Eqn:DF_function_p_strong} will be used for the 
pressure trial function, and the function space given 
in equation \eqref{Eqn:DF_function_q_strong} will be 
used for the pressure test function. If the pressure 
is prescribed weakly on $\Gamma^{p}$ then the function 
space given in equation \eqref{Eqn:DF_function_q_weak} 
will be used for both trial and test functions of the 
pressure.
It should be emphasized that both $L_2(\Omega)$ and 
$H^1(\Omega)$ are Hilbert spaces under the standard 
$L_2$ inner-product \cite{Evans_PDE}. The standard 
$L_2$ inner-product over a set $K$ will be denoted 
as $(\cdot;\cdot)_{K}$, and is defined as
\begin{align}
  (a;b)_K := \int_{K} a \cdot b \; \mathrm{d} K
\end{align}
For simplicity, the subscript $K$ will be dropped 
if $K = \Omega$. Note that for volume integrals 
$K \subseteq \Omega$ and for surface integrals 
$K \subseteq \partial \Omega$. In a subsequent 
section on numerical results, the error will be 
measured in $L_2$ norm and $H^1$ seminorm. To 
this end, the $L_2$ norm on $\Omega$ is defined 
as 
\begin{align}
  \|a\|_{L_2(\Omega)} := \sqrt{\int_{\Omega} 
    a \cdot a \; \mathrm{d} \Omega}
\end{align}
The $H^1$ seminorm on $\Omega$ is defined as
\begin{align}
  |a|_{H^1(\Omega)} := \sqrt{\int_{\Omega} 
    \mathrm{grad}[a] \cdot \mathrm{grad}[a] 
    \; \mathrm{d} \Omega}
\end{align}
The $H^1$ norm on $\Omega$ can then be defined as 
\begin{align}
  \|a\|_{H^1(\Omega)} := \sqrt{\|a\|_{L_2(\Omega)}^2 
    + |a|^2_{H^1(\Omega)}}
\end{align}
For further details on inner-product spaces and normed 
spaces, see references \cite{Reddy_Functional_Analysis,
Oden_Demkowicz}.

The aforementioned modifications to the standard 
Darcy model result in nonlinear partial differential 
equations, as the drag coefficient depends on the 
pressure and/or the velocity. To solve the resulting 
nonlinear equations, linearization is first performed, 
and then the LS and VMS formalisms shall be utilized 
to construct mixed two-field weak formulations. 
To this end, let us define the following linearization 
functionals:
\begin{align}
  \label{Eqn:Linearization_next}
  &\mathcal{D}^{(i+1)} := \vartheta 
  \left(\frac{\partial \alpha}{\partial p}
  \mathbf{v}^{(i)}\right)p^{(i+1)} + \vartheta 
  \left(\frac{\partial \alpha}{\partial\mathbf{v}}
  \otimes\mathbf{v}^{(i)}\right)\mathbf{v}^{(i+1)} \\
  \label{Eqn:Linearization_current}
  &\mathcal{D}^{(i)} := \vartheta 
  \left(\frac{\partial \alpha}{\partial p} 
  \mathbf{v}^{(i)}\right) p^{(i)} + \vartheta 
  \left(\frac{\partial \alpha}{\partial\mathbf{v}}
  \otimes\mathbf{v}^{(i)}\right)\mathbf{v}^{(i)} \\
  &\mathcal{G} := \vartheta \left(
  \frac{\partial \alpha}{\partial p}
  \mathbf{v}^{(i)}\right) q + \vartheta 
  \left(\frac{\partial \alpha}{\partial\mathbf{v}}
  \otimes\mathbf{v}^{(i)}\right)\mathbf{w}
\end{align}
where superscripts $(i)$ and $(i+1)$ represent 
solutions for the current and next iteration 
respectively, $\otimes$ denotes the standard 
tensor product \cite{Chadwick}, and $\vartheta 
\in [0 ,1]$ is a user-defined parameter to 
choose the type of linearization. One can 
achieve Picard's linearization by choosing 
$\vartheta = 0$ and consistent linearization 
by choosing $\vartheta = 1$.
\begin{remark}
  It should be noted that $\mathbf{v}$, $p$, $\mathbf{w}$, 
  $q$, $\mu$, $\mathbf{b}$, $k$, and $\hat{\mathbf{n}}$ 
  are all functions of $\mathbf{x}$. The drag coefficient 
  and its derivatives will be functions of $p^{(i)}$, 
  $\mathbf{v}^{(i)}$ and $\mathbf{x}$. For notational 
  simplicity, these dependencies will not be explicitly 
  indicated. 
\end{remark}

\subsection{A mixed formulation based on least-squares formalism}
Consider an abstract mathematical problem 
defined by a set of partial differential 
equations in the form:
\begin{align}
  & \mathbf{L}\mathbf{u} = \mathbf{f} 
  \quad \mathrm{in} \quad \Omega \\
  & \mathbf{B}\mathbf{u} = \mathbf{0} 
  \quad \mathrm{in} \quad \Gamma
\end{align}
where $\mathbf{L}$ is the differential operator, 
$\mathbf{B}$ is the boundary operator, $\mathbf{u}$ 
is the unknown vector, and $\mathbf{f}$ is the 
forcing vector. A corresponding least-squares 
functional can be constructed as follows:
\begin{align}
  \Pi[\mathbf{u}] = \frac{1}{2}\int_{\mathrm{\Omega}}
  \left\|\mathbf{L}\mathbf{u} - \mathbf{f}\right\|^{2}
  \mathrm{d\Omega} + \frac{1}{2}\int_{\mathrm{\Gamma}}\left
  \|\mathbf{B}\mathbf{u} - \mathbf{0}\right\|^{2}\mathrm{d\Gamma}
\end{align}
A weak form based on least-squares formalism can 
be obtained by requiring the G\^{a}teaux variation 
to vanish along any $\mathbf{w}$ that satisfies 
the essential boundary conditions. That is,
\begin{align}
  \delta \Pi[\mathbf{w},\mathbf{u}] = 0 
  \quad \forall \mathbf{w}
\end{align}
where 
\begin{align}
  \delta\Pi[\mathbf{w},\mathbf{u}] := 
  \mathop{\mathrm{lim}}_{\epsilon\rightarrow 0}
  \frac{\Pi[\mathbf{w}+\epsilon\mathbf{u}] - 
    \Pi[\mathbf{w}]}{\epsilon} \equiv 
  \left[\frac{d}{d\epsilon}\Pi[\mathbf{w} 
      + \epsilon\mathbf{u}]\right]_{\epsilon=0}
\end{align}
provided the limit exists. For further details 
on the G\^{a}teaux variation see references 
\cite{Spivak,Holzapfel,Glowinski}.

Studies in reference \cite{Linearization_Payette} have 
shown that minimizing the problem after linearization 
produces more accurate results. Also, minimizing a 
least-squares-based functional before linearization will 
create additional terms in the resulting weak 
formulation and significantly increase the difficulty of 
implementation, so we shall employ the former approach. Inserting 
equations \eqref{Eqn:Linearization_next} and 
\eqref{Eqn:Linearization_current} into equation 
\eqref{Eqn:GE_darcy} results in the following governing 
equations:
\begin{subequations}
  \label{Eqn:governing_equations_linearized}
  \begin{align}
    \label{Eqn:GEL_darcy}
    &\alpha\mathbf{v}^{(i+1)} + \mathcal{D}^{(i+1)}- \mathcal{D}^{(i)} +\mathrm{grad}[p^{(i+1)}] = \rho\mathbf{b}
    \quad \mathrm{in} \; \Omega \\
    \label{Eqn:GEL_Continuity}
    &\mathrm{div}[\mathbf{v}^{(i+1)}] = 0 \quad \mathrm{in} \; \Omega \\
    &\mathbf{v}^{(i+1)} \cdot \mathbf{\hat{n}} = v_n \quad \mathrm{on} \; \Gamma^{v} \\
    &p^{(i+1)} = p_0 \quad \mathrm{on} \; \Gamma^{p}
  \end{align}
\end{subequations}
In reference \cite{LS_issues}, it has been shown that for the Navier-Stokes equation, 
an introduction of a mesh dependent variable in the LS formulation 
greatly improves the accuracy of the solution. Thus for the Darcy modifications, 
two variants of the LS formulation will be 
considered by employing the following weights:
\begin{align}
  \mathbf{A} = 
  \left\{\begin{array}{cl}
    \mathbf{I} & \mbox{weight 1} \\
    \alpha\mathbf{I}& \mbox{weight 2}
  \end{array}
  \right.
  \label{Eqn:GEL_weights}
\end{align}
For all the models considered in this paper, the 
second-order tensor $\mathbf{A}$ is symmetric and 
positive definite. This implies that the tensor is 
invertible. In addition, the square root theorem 
ensures that its square root exists \cite{Gurtin}. 
Employing the minimization approach on equations 
\eqref{Eqn:GEL_darcy} and \eqref{Eqn:GEL_Continuity} 
results in the functional
\begin{align}
  \label{Eqn:GEL_minimized}
  &\Pi_{\mathrm{LS}}[\mathbf{v}^{(i+1)},p^{(i+1)}] := \frac{1}{2}\int_{\mathrm{\Omega}}\left\|\mathbf{A}^{-1/2}(\alpha\mathbf{v}^{(i+1)}+\mathcal{D}^{(i+1)}- \mathcal{D}^{(i)}
	+ \mathrm{grad}[p^{(i+1)}]-\rho\mathbf{b})\right\|^2\mathrm{d\Omega} \nonumber \\
	& \qquad\qquad\qquad\quad + \frac{1}{2}\int_{\mathrm{\Omega}}
	\left\|\mathrm{div}[\mathbf{v}^{(i+1)}]\right\|^2\mathrm{d\Omega}
\end{align}
Let $\mathbf{v}^{(i+1)}\rightarrow\mathbf{v}^{(i+1)} + \epsilon\mathbf{w}$ and $p^{(i+1)}\rightarrow p^{(i+1)} + \epsilon q$ 
where $\mathbf{v}^{(i+1)}$ and $\mathbf{w}$ are the velocity trial and test functions respectively and $p^{(i+1)}$ and $q$ are the pressure trial and test functions respectively. 
Applying the 
G$\hat{\mathrm{a}}$teaux variation on equation \eqref{Eqn:GEL_minimized} results in the functional
\begin{align}
  &\delta\Pi_{\mathrm{LS}}[\mathbf{v}^{(i+1)},p^{(i+1)};\mathbf{w},q] = \left[\frac{\mathrm{d}}{\mathrm{d}\epsilon}\Pi_{\mathrm{LS}}[\mathbf{v}^{(i+1)}+\epsilon\mathbf{w},p^{(i+1)}+
  		\epsilon q]\right]_{\epsilon = 0} \nonumber\\
  	&\qquad = \int_{\Omega}\left(\alpha\mathbf{w} + \mathcal{G} + \mathrm{grad}[q]\right)\cdot\mathbf{A}^{-1}
  \left(\mathrm{\alpha}\mathbf{v}^{(i+1)} + \mathcal{D}^{(i+1)}- \mathcal{D}^{(i)} + \mathrm{grad}[p^{(i+1)}] - \rho \mathbf{b}\right) \nonumber\\
  	&\qquad+ \mathrm{div}[\mathbf{w}]\cdot\mathrm{div}[\mathbf{v}^{(i+1)}] \mathrm{d\Omega}
\end{align}
and setting it equal to zero gives the weak form. 
The final statement for the modified Darcy-Forchheimer 
model can be rearranged and written as follows: Given 
$\mathbf{v}^{(i)}$ and$p^{(i)}$ find $\mathbf{v}^{(i+1)}
\in \mathcal{V}$ and $p^{(i+1)} \in \mathcal{P}$ such 
that we have
\begin{align}
\label{Eqn:weak_least_squares}
  	&\left(\alpha\mathbf{w};\mathbf{A}^{-1}\alpha\mathbf{v}^{(i+1)}\right) + \left(\alpha\mathbf{w};\mathbf{A}^{-1}\mathcal{D}^{(i+1)}\right)
	+ \left(\alpha\mathbf{w};\mathbf{A}^{-1}\mathrm{grad}[p^{(i+1)}]\right) \nonumber \\
	&\quad+\left(\mathcal{G};\mathbf{A}^{-1}\alpha\mathbf{v}^{(i+1)}\right) + \left(\mathcal{G};\mathbf{A}^{-1}\mathcal{D}^{(i+1)}\right)
	+ \left(\mathcal{G};\mathbf{A}^{-1}\mathrm{grad}[p^{(i+1)}]\right) \nonumber \\
	&\quad+\left(\mathrm{grad}[q];\mathbf{A}^{-1}\alpha\mathbf{v}^{(i+1)}\right) + \left(\mathrm{grad}[q];\mathbf{A}^{-1}\mathcal{D}^{(i+1)}\right) + 
	\left(\mathrm{grad}[q];\mathbf{A}^{-1}\mathrm{grad}[p^{(i+1)}]\right) \nonumber \\
	&\quad+ \left(\mathrm{div}[\mathbf{w}];\mathrm{div}[\mathbf{v}^{(i+1)}]\right) = \left(\alpha\mathbf{w};\mathbf{A}^{-1}\rho\mathbf{b}\right) + 
	\left(\mathcal{G};\mathbf{A}^{-1}\rho\mathbf{b}\right) + \left(\mathrm{grad}[q];\mathbf{A}^{-1}\rho\mathbf{b}\right) \nonumber \\
	&\quad+ \left(\alpha\mathbf{w};\mathbf{A}^{-1}\mathcal{D}^{(i)}\right) + \left(\mathcal{G};\mathbf{A}^{-1}\mathcal{D}^{(i)}\right)
	+ \left(\mathrm{grad}[q];\mathbf{A}^{-1}\mathcal{D}^{(i)}\right) \nonumber \\
	&\qquad\quad\forall \mathbf{w} \in \mathcal{W}, \; \forall q \in \mathcal{Q}
\end{align}

\subsection{A mixed formulation based on the variational multi-scale formalism}
Following the derivation given in reference 
\cite{Nakshatrala_Turner_Hjelmstad_Masud_CMAME_2006_v195_p4036}, 
one can derive a mixed formulation based on VMS formalism. It 
should be noted that in the previous derivations, the governing 
equations were not linearized and were solved using a Newton-Raphson 
approach. It should also be noted that the pressure boundary 
condition is weakly prescribed (i.e., a Neumann boundary condition) 
and acts normal to the boundary so the function space in equation 
\eqref{Eqn:DF_function_q_weak} is utilized.
 
After incorporating linearization terms into the governing 
equations, the resulting weak form based on the VMS formalism 
can be written as follows: Given $\mathbf{v}^{(i)}$ and $p^{(i)}$ 
find $\mathbf{v}^{(i+1)} \in \mathcal{V}$ and $p^{(i+1)} \in 
\mathcal{P}$ such that we have 
\begin{align}
 & \left(\mathbf{w}; \alpha\mathbf{v}^{(i+1)}\right)
  +\left(\mathbf{w}; \mathcal{D}^{(i+1)}\right)
  -\left(\mathrm{div}[\mathbf{w}]; p^{(i+1)}\right)
  +\left(\mathbf{w} \cdot \mathbf{\hat{n}}; p_0\right)_{\Gamma^{p}} 
  -\left(q; \mathrm{div}[\mathbf{v}^{(i+1)}]\right) \nonumber \\
  &\qquad \underbrace{-\frac{1}{2} \left(\alpha \mathbf{w} + \mathrm{grad}[q]; 
  \alpha^{-1}\left(\alpha\mathbf{v}^{(i+1)} + \mathcal{D}^{(i+1)} + \mathrm{grad}[p^{(i+1)}]\right)\right)}_{\mbox{stabilization term}} \nonumber \\
  &\qquad = \left(\mathbf{w}; \rho\mathbf{b} +  \mathcal{D}^{(i)}\right) \underbrace{-\frac{1}{2} \left(\alpha \mathbf{w} + \mathrm{grad}[q]; 
  \alpha^{-1}\left(\rho\mathbf{b} + \mathcal{D}^{(i)}\right)\right)}_{\mbox{stabilization term}} \nonumber \\
  &\qquad\quad \forall \mathbf{w} \in \mathcal{W}, \; 
  \forall q \in \widetilde{\mathcal{Q}}
\end{align}
The proposed VMS formulation encompasses some of the 
existing mixed formulations proposed for simpler 
models. For the standard Darcy model, this weak 
formulation reduces to the one presented in reference 
\cite{Nakshatrala_Turner_Hjelmstad_Masud_CMAME_2006_v195_p4036}.
For the modified Darcy model, this formulation 
with $\vartheta = 0$ (i.e., Picard linearization) 
reduces to the one presented in reference 
\cite{Nakshatrala_Turner_2013_arXiv}. 
Reference \cite{Nakshatrala_Rajagopal_IJNMF_2011_v67_p342} 
considered the modified Darcy model. However, the mixed 
formulation proposed in reference 
\cite{Nakshatrala_Rajagopal_IJNMF_2011_v67_p342} 
took a different approach by first constructing 
a weak formulation based on VMS formalism before 
linearization. The resulting nonlinear equations 
are then solved using the Newton-Raphson method. 
Algorithm \ref{Algo:MDF_Pseudocode} outlines 
the steps in implementing the proposed mixed 
finite element formulations. 

\begin{algorithm}[h]
  \begin{algorithmic}
    \State Set $(i) = 1$;
    \State Initialize data $\mathbf{v}^{(i)} = \mathbf{1}$ and $p^{(i)} = 1$
    \While {true} \Comment{nonlinear solver}
	    \If {$(i) > $ maximum number of iterations}
	    	\State break
		\Comment{solution did not converge}
	    \EndIf
	    \State Get $\alpha$ using $\mathbf{v}^{(i)}$ and $p^{(i)}$
	    \State Assemble stiffness matrices and forcing vectors
	    \State Solve and obtain $\mathbf{v}^{(i+1)}$ and $p^{(i+1)}$
	    \If {$\left\|\mathbf{v}^{(i+1)}-\mathbf{v}^{(i)}\right\|$ and 
	    $\left\|p^{(i+1)}-p^{(i)}\right\| < \epsilon_{\mathrm{TOL}}$}
	    	\State break
		\Comment{solution has converged}
	    \Else
	    	\State $\mathbf{v}^{(i)} \leftarrow \mathbf{v}^{(i+1)}$  and $p^{(i)} \leftarrow p^{(i+1)}$
	    	\State $(i) \leftarrow (i+1)$ 
	    \EndIf
    \EndWhile
  \end{algorithmic}
  \caption{Pseudocode for the nonlinear FEA. \label{Algo:MDF_Pseudocode}}
\end{algorithm}

%% file: Sections/Ch4_Numerical_benchmark.tex
\section{NUMERICAL BENCHMARK TESTS}
\label{Ch:benchmark}

\subsection{Dimensionless form of equations}
Numerical studies for subsurface flows like enhanced 
oil recovery can be displayed in dimensionless form, 
thus allowing scaling to real flow conditions. The 
governing equations are non-dimensionalized by 
choosing primary variables that seem appropriate. 
This non-dimensional procedure is different from 
the standard non-dimensionalization procedure for 
incompressible Navier-Stokes in the choice of primary 
variables (in the standard non-dimensionalization of 
Navier-Stokes equations, one employs characteristic 
velocity $V$, characteristic length $L$ and density 
of the fluid $\rho$ as primary variables). Also, the 
present non-dimensionalization is different and seems 
more appropriate than the one employed in reference 
\cite{Nakshatrala_Rajagopal_IJNMF_2011_v67_p342} for 
the chosen applications. 

All non-dimensional quantities are denoted using 
a superposed bar. Let $L$ (reference length in 
the problem), $g$ (acceleration due to gravity) 
and $p_{\mathrm{atm}}$ (atmospheric pressure) be the 
reference quantities. The following non-dimensional 
quantities are then defined: 
\begin{align}
   &\bar{\mathbf{x}} = \frac{\mathbf{x}}{L}, \; 
  \bar{\mathbf{v}} = \frac{\mathbf{v}}{\sqrt{gL}}, \; 
  \bar{\mathbf{b}} = \frac{\mathbf{b}}{g}, \;
  \bar{p} = \frac{p}{p_{\mathrm{atm}}}, \;
  \bar{\rho} = \frac{\rho g L}{p_{\mathrm{atm}}}, \;
  \bar{k} = \frac{k}{L^{2}}\nonumber\\ 
  &\bar{\beta}_{\mathrm{B}} = \beta_{\mathrm{B}}p_{\mathrm{atm}}, \; 
  \bar{\beta}_{\mathrm{F}} = \frac{\beta_{\mathrm{F}}gL^{2}}{p_{\mathrm{atm}}}, \;
  \bar{\alpha} = \alpha \frac{\sqrt{gL^{3}}}{p_{\mathrm{atm}}}, \;
  \bar{\mu}_0 = \frac{\mu_0 \sqrt{g/L}}{p_{\mathrm{atm}}}
\end{align}
The scaled domain $\Omega_{\mathrm{scaled}}$ is defined 
as follows: a point in space with position vector $\bar{\mathbf{x}} \in 
\Omega_{\mathrm{scaled}}$ corresponds to the same point with position vector 
given by $\mathbf{x} = \bar{\mathbf{x}} L \in \Omega$. Similarly, one can 
define the scaled boundaries for $\Gamma^{v}_{\mathrm{scaled}}$ and $\Gamma^{p}_{\mathrm{scaled}}$. 
Using the above non-dimensionalization 
procedure, the governing equations 
\eqref{Eqn:GE_darcy}--\eqref{Eqn:GE_p0} 
can be written as follows:
\begin{subequations}
  \label{Eqn:DF_EXP_Governing}
  \begin{align}
    &\bar{\alpha}(\bar{\mathbf{v}},\bar{p},\bar{\mathbf{x}}) \bar{\mathbf{v}} + \overline{\mathrm{grad}} 
    [\bar{p}(\bar{\mathbf{x}})] = \bar{\rho} \; \bar{\mathbf{b}}(\bar{\mathbf{x}}) \quad 
    \; \mathrm{in} \; {\Omega}_{\mathrm{scaled}} \\
    &\overline{\mathrm{div}}[{\mathbf{\bar{v}}}(\bar{\mathbf{x}})]  = 0 \quad \; 
    \mathrm{in} \; {\Omega}_{\mathrm{scaled}} \\
    &\bar{\mathbf{v}}(\bar{\mathbf{x}})\cdot\mathbf{\hat{n}}(\bar{\mathbf{x}}) = 
    \bar{v}_{0}(\bar{\mathbf{x}})  \quad \; 
        \mathrm{on} \; {\Gamma}^{v}_{\mathrm{scaled}} \\
    &\bar{p}(\bar{\mathbf{x}}) = 
        \bar{p}_{0} (\bar{\mathbf{x}}) 
    \quad \; \mathrm{on} \; {\Gamma}^{t}_{\mathrm{scaled}}
  \end{align}
\end{subequations}

\subsection{Numerical $h$-convergence}
A finite element formulation is said to be convergent if the numerical solutions tend to the exact 
solution with mesh refinement. This section will perform an $h$-convergence analysis on all Darcy 
models where $h$ is taken to be the edge length for quadrilateral elements and the short-edge length 
for triangular elements.
\begin{figure}[t!]
\centering
  \includegraphics[scale=1]{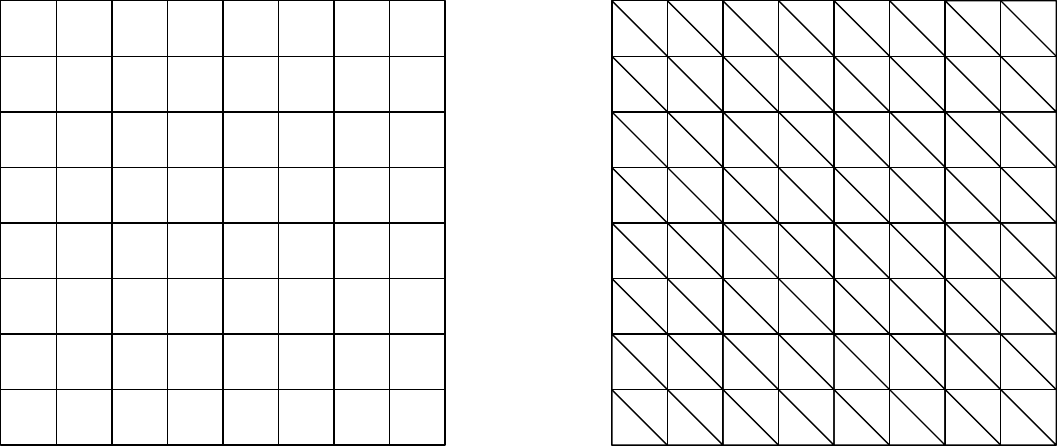}
  \caption{Typical structured meshes using 
    quadrilateral (left) and triangular (right) 
    finite elements, which are employed in the 
    $h$-numerical convergence studies. 
    \label{Fig:meshings}}
\end{figure}
\begin{figure}[h!]
  \centering
  \subfigure[x-velocity]{
  	\includegraphics[scale=0.46]
    {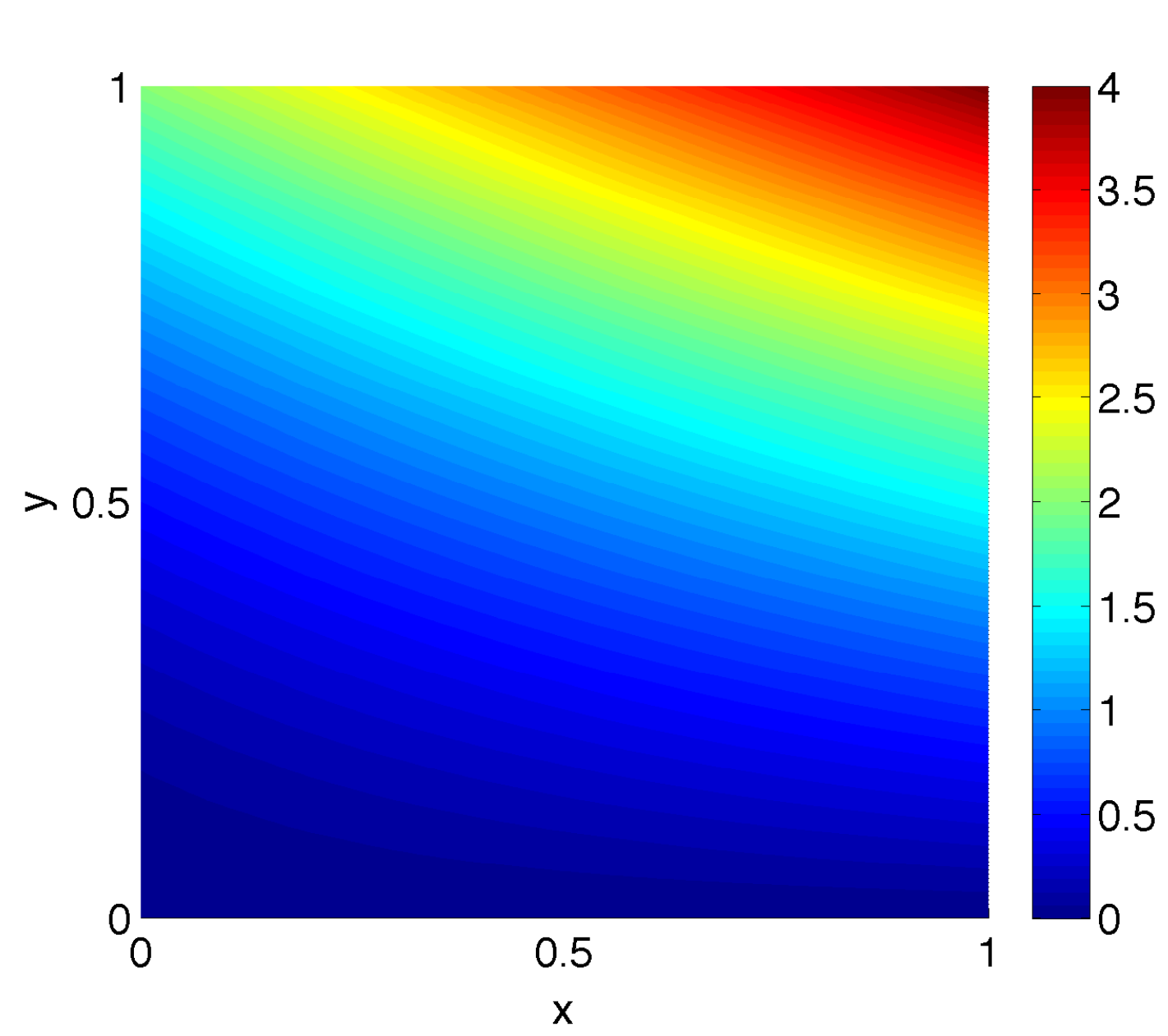}}
  \subfigure[y-velocity]{
  	\includegraphics[scale=0.46]
    {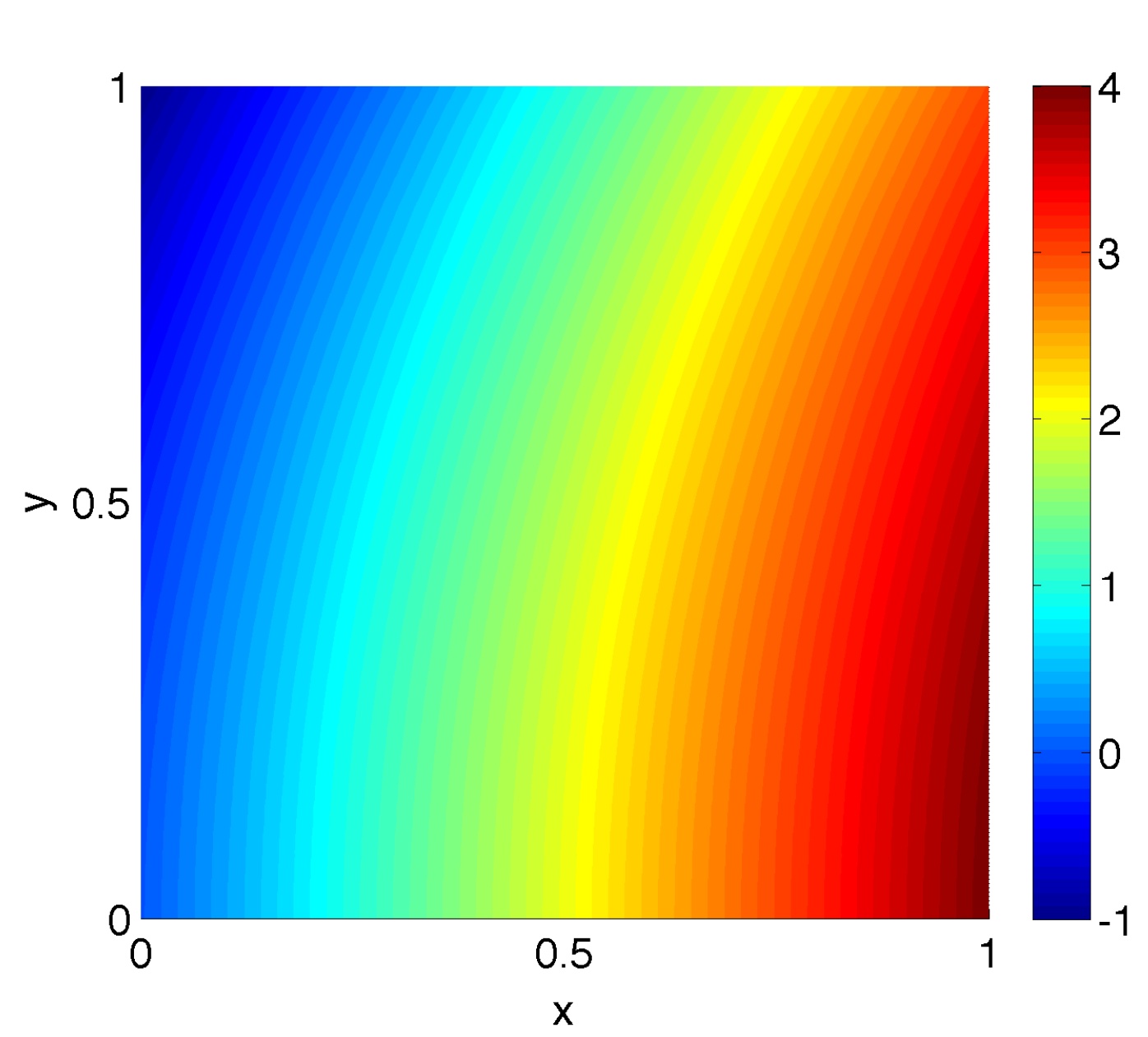}}
  \subfigure[velocity field]{
  	\includegraphics[scale=0.46]
    {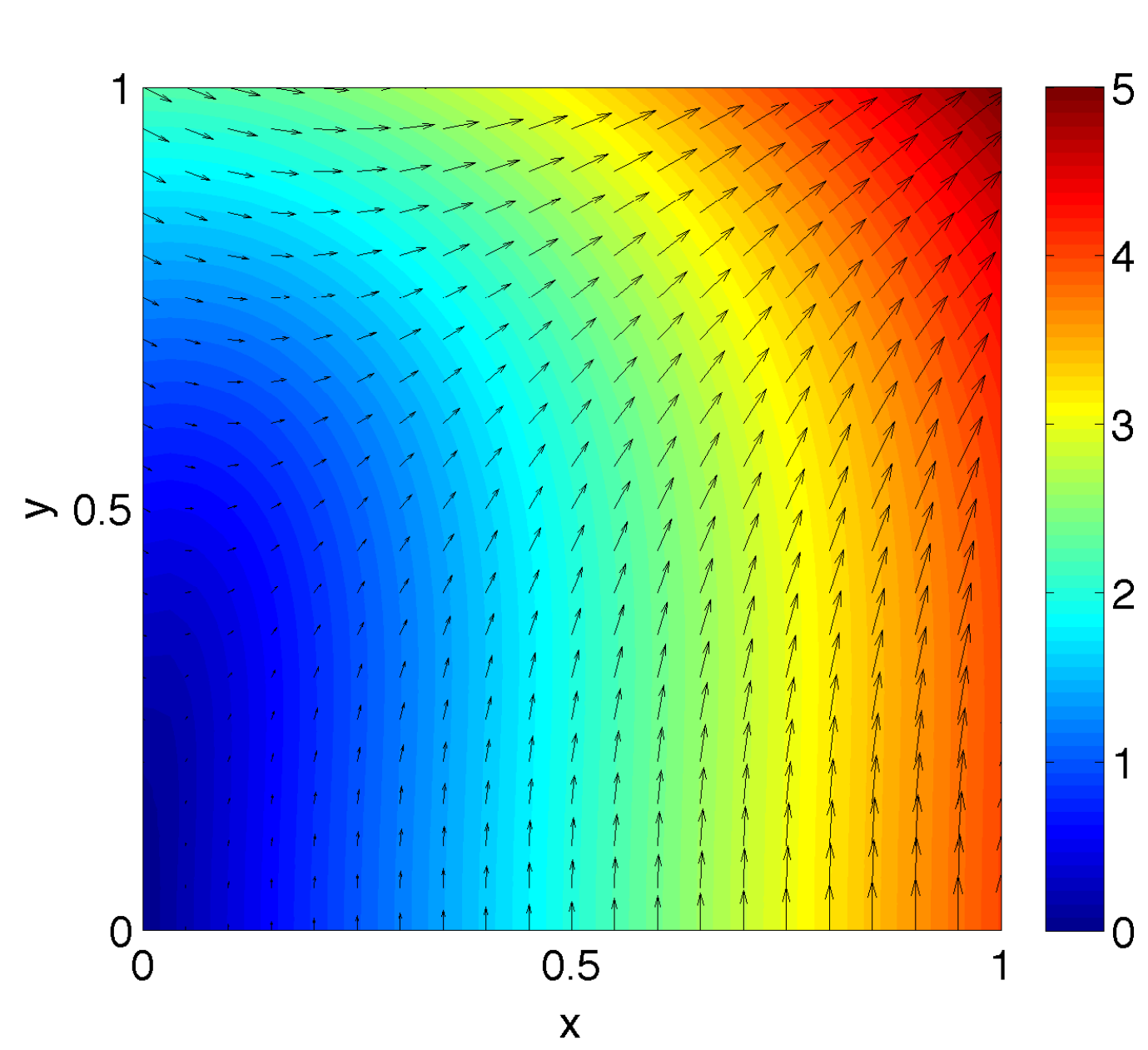}}
  \subfigure[pressure contour]{
  	\includegraphics[scale=0.46]
    {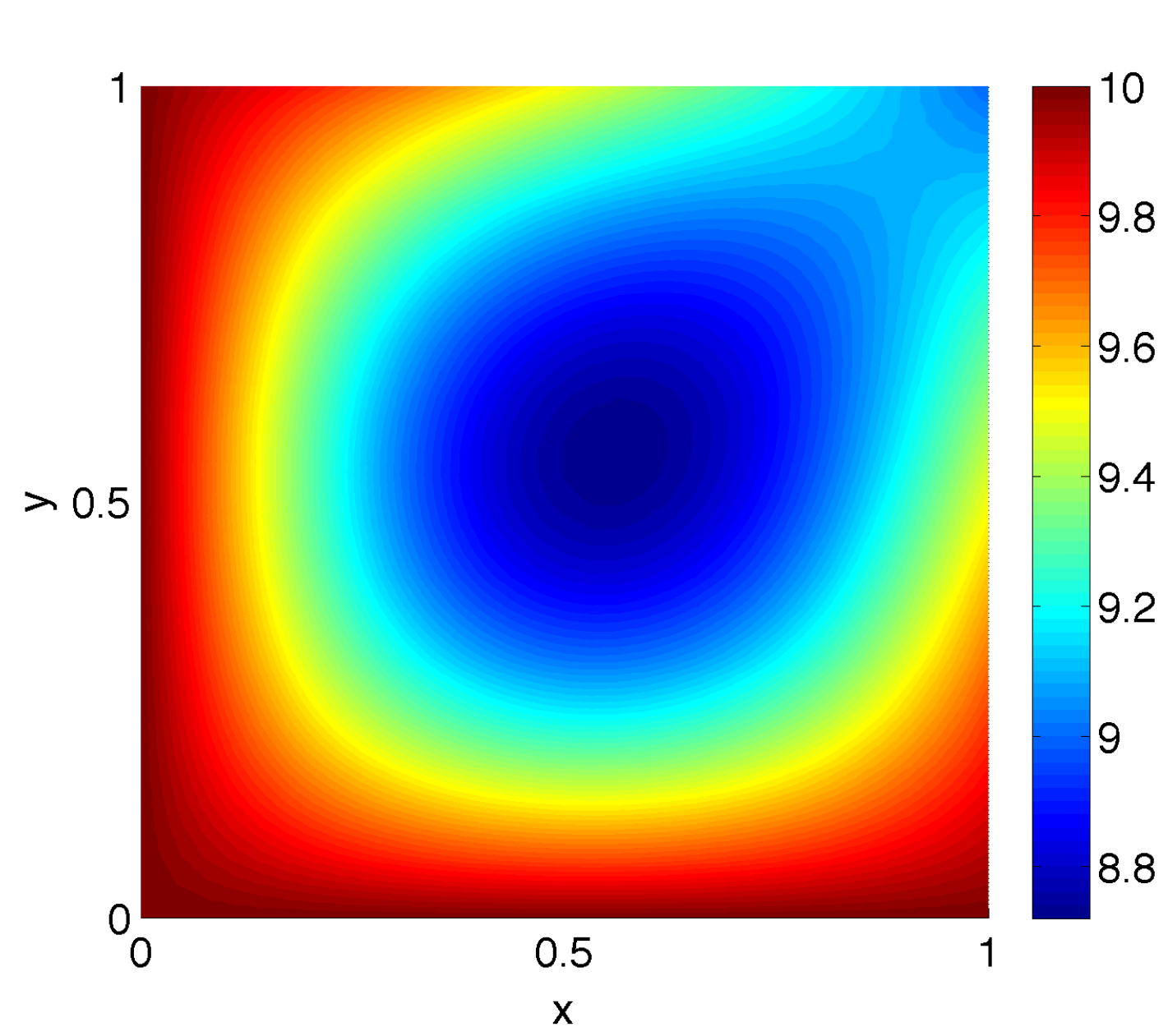}}
  \caption{Numerical $h$-convergence: contours of analytical solution.}
  \label{Fig:Numerical_analytical_solution}
\end{figure}

Consider a unit square as the computational domain. For this and all subsequent 
numerical studies, the FEM utilizes structured meshes as depicted in 
Figure \ref{Fig:meshings}. The velocity and pressure functions for this problem are:
\begin{subequations}
\begin{align}
  & \mathbf{\bar{v}}(x,y) = 
  \begin{Bmatrix}
	2y(x+y), \\
	4x-y^{2}.
  \end{Bmatrix} \\
  & \bar{p}(x,y) = 10-xy-\mathrm{sin}(\pi x )\mathrm{sin}(\pi y)
\end{align}
\label{Eqn:numerical_h}
\end{subequations}
Inserting the velocity and pressure functions back into the Darcy equation results in the following specific body force function:
\begin{align}
  & \mathbf{\bar{b}}(x,y) = \frac{1}{\bar{\rho}}
  \begin{Bmatrix}
    \bar{\alpha} 2y(x+y)-\pi\mathrm{cos}(\pi x )\mathrm{sin}(\pi y)-y\\
    \bar{\alpha}(4x-y^{2})-\pi\mathrm{cos}(\pi y )\mathrm{sin}(\pi x)-x
  \end{Bmatrix}
\end{align}
\begin{table}[h!]
  \centering
  \caption{User-defined inputs for the numerical $h$-convergence problem.}
  \begin{tabular}{cc}
    \hline
    Parameters & Value \\
    \hline
    $\bar{\beta}_{\mathrm{B}}$ & 0.1 \\
    $\bar{\beta}_{\mathrm{F}}$ & 0.5 \\
    $\bar{\alpha}$ & 1 \\
    $\bar{\rho}$ & 1 \\
    $\vartheta$ & 1 \\
    $h$-sizes & 1/4, 1/8, 1/16, 1/32, 1/64 \\
    \hline
  \end{tabular}
  \label{Tab:h_converge}
\end{table}
\begin{figure}[t!]
  \centering
  \subfigure[LS: Q4 elements]{
  	\includegraphics[scale=0.45]
    {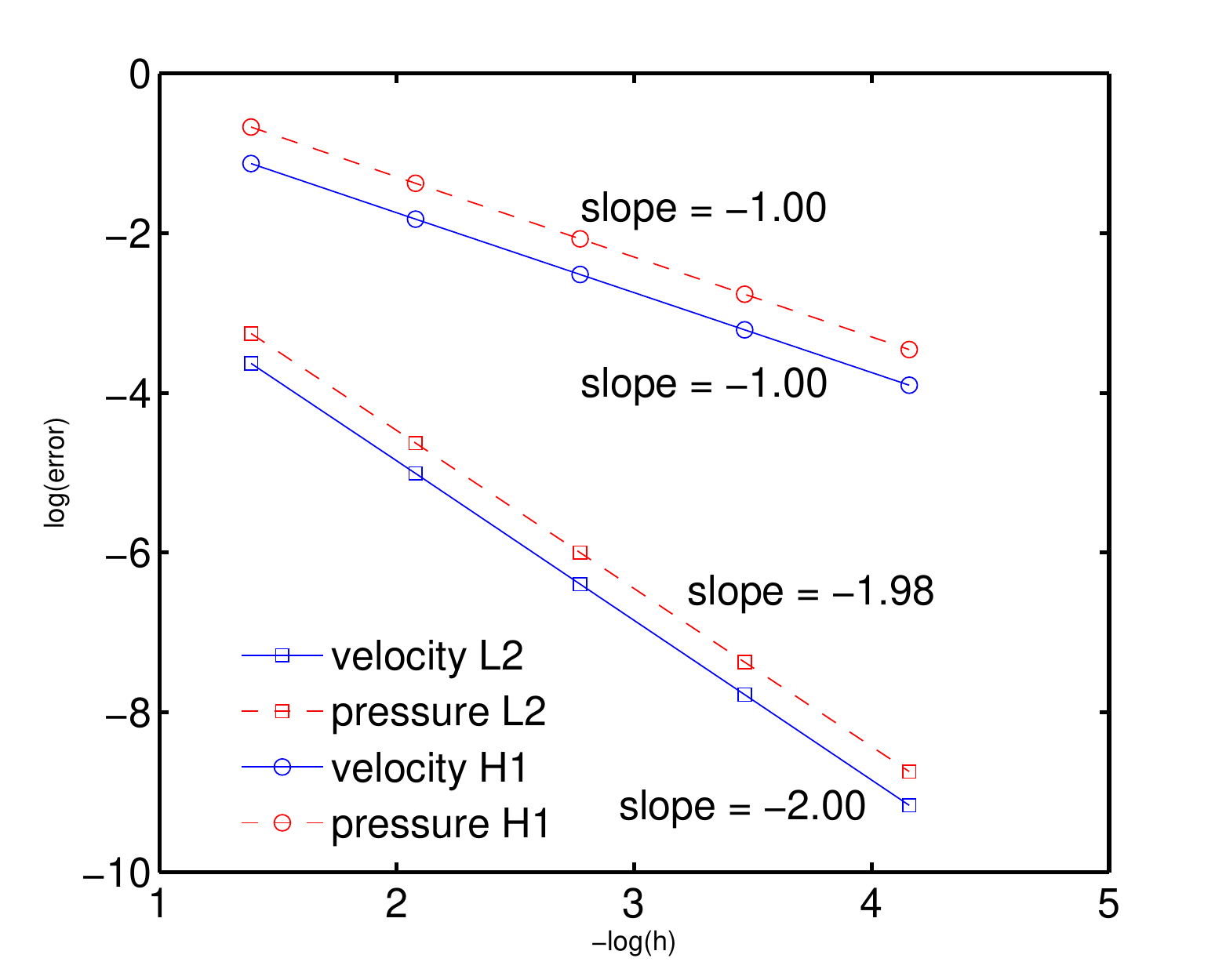}}
  \subfigure[LS: T3 elements]{
  	\includegraphics[scale=0.45]
    {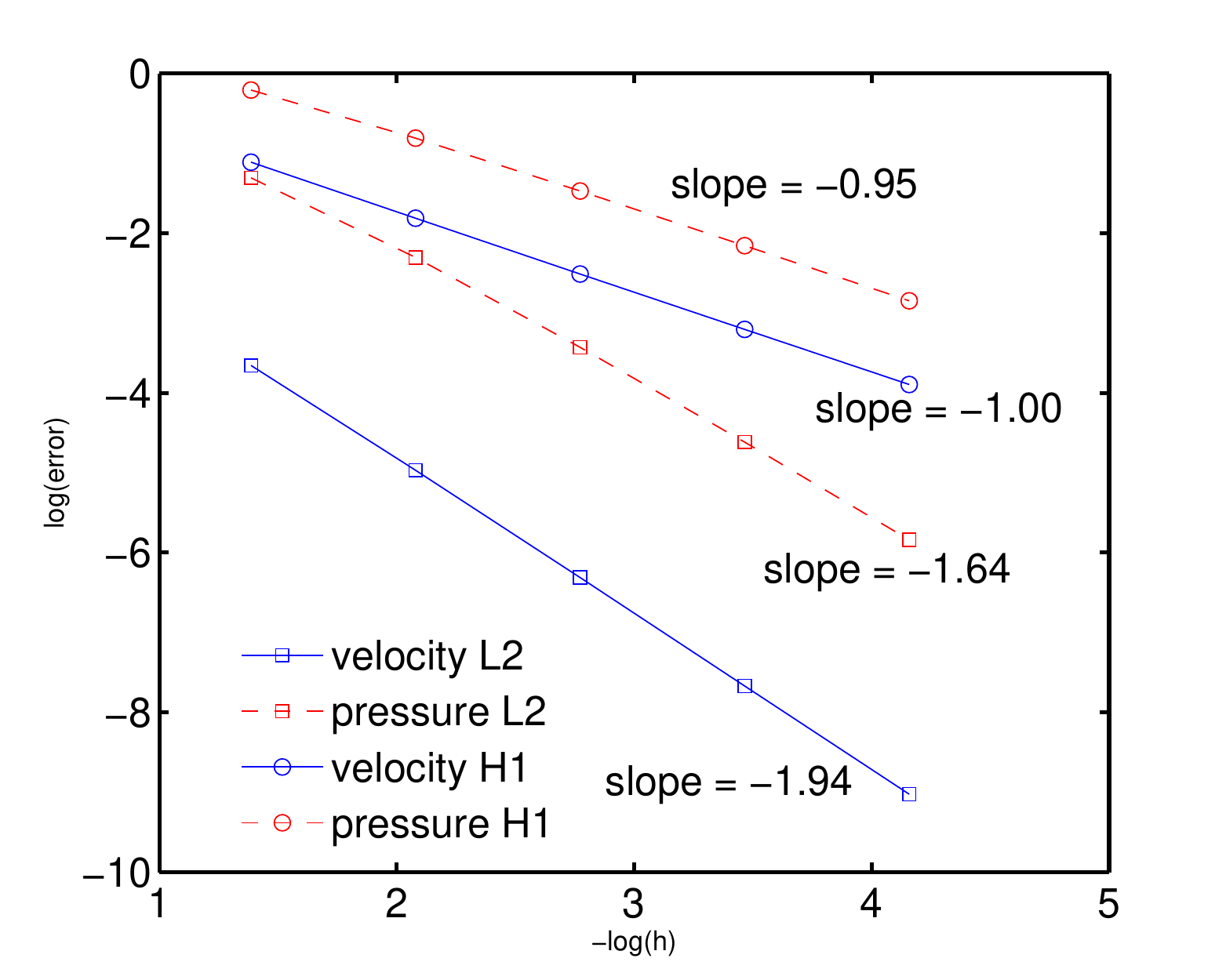}}
  \subfigure[VMS: Q4 elements]{
  	\includegraphics[scale=0.45]
    {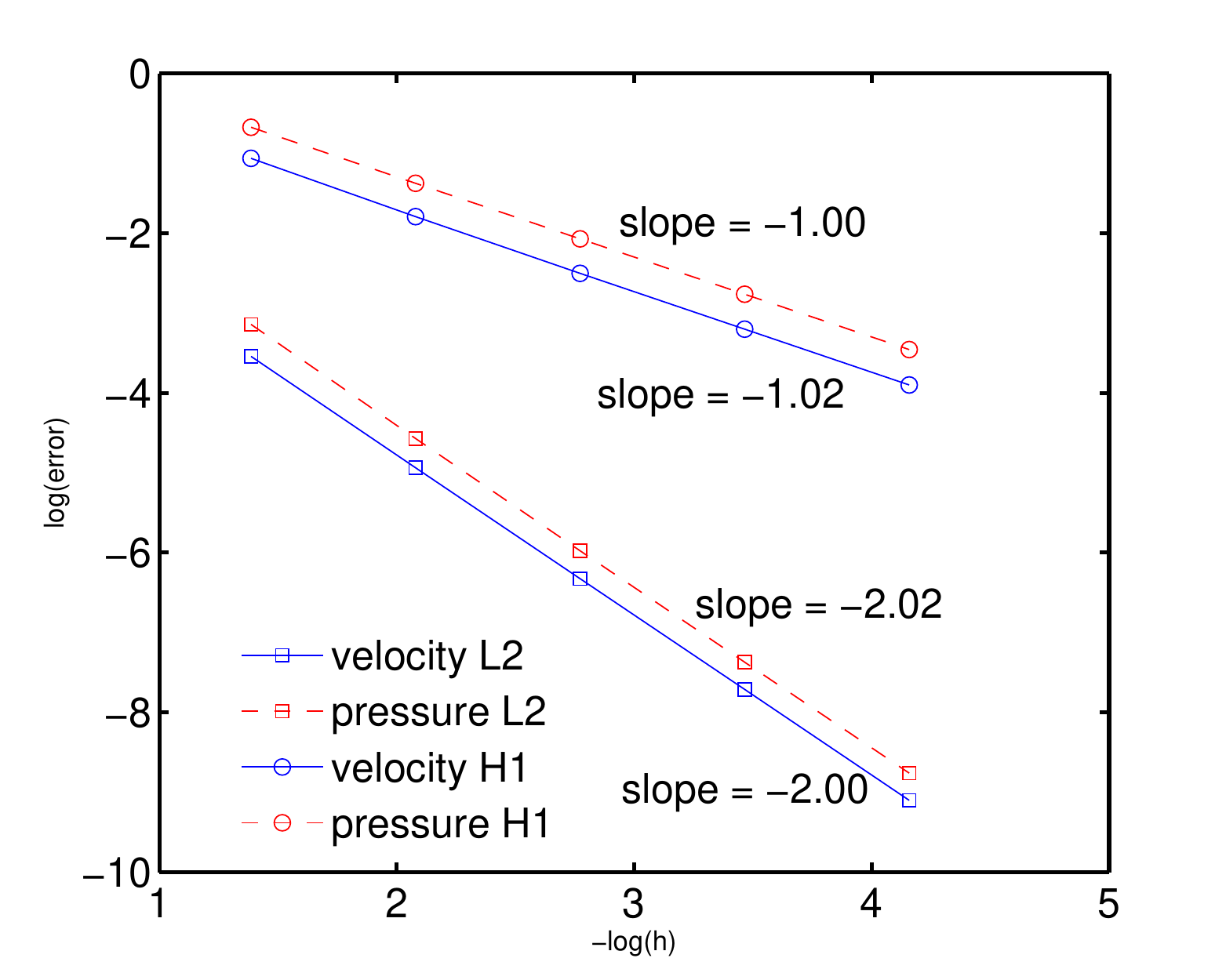}}
  \subfigure[VMS: T3 elements]{
  	\includegraphics[scale=0.45]
    {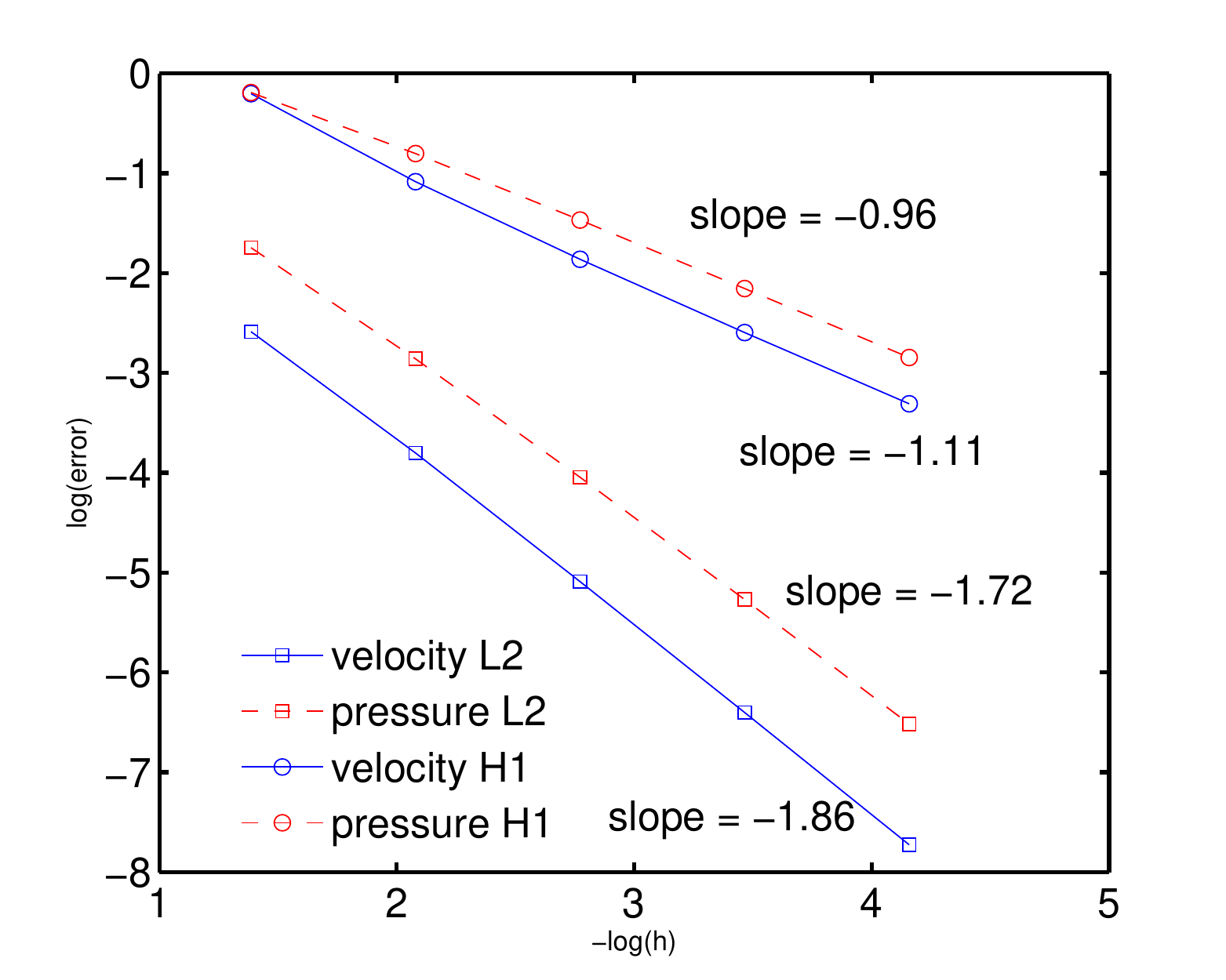}}
  \caption{Numerical $h$-convergence: error slopes for the MBF model}
  \label{Fig:Numerical_slopes_MBF}
\end{figure}
\begin{table}[h!]
  \centering
  \caption{Numerical $h$-convergence slopes for various Darcy models.}
  \begin{tabular}{rcccc|cccc}
      \hline
      $\;$ & \multicolumn{4}{c|}{LS formalism} & \multicolumn{4}{c}{VMS formalism} \\
      $\;$ & D & MB & F & MBF & D & MB & F & MBF \\
      \hline
      Q4 $L_2$ error $v$ & -2.00 & -2.00 & -1.99 & -2.00 & -1.99 & -2.00 & -1.99 & -2.00 \\
      Q4 $H^1$ error $v$ & -1.00 & -1.00 & -1.00 & -1.00 & -1.14 & -1.04 & -1.05 & -1.02 \\
      Q4 $L_2$ error $p$ & -1.95 & -1.96 & -1.99 & -1.98 & -2.01 & -2.03 & -2.02 & -2.02 \\
      Q4 $H^1$ error $p$ & -1.00 & -1.00 & -1.00 & -1.00 & -1.00 & -1.00 & -1.00 & -1.00 \\
      \hline
      T3 $L_2$ error $v$ & - 1.97 & -1.97 & -1.84 & -1.94 & -1.84 & -1.86 & -1.81 & -1.86 \\
      T3 $H^1$ error $v$ & -1.01 & -1.01 & -1.00 & -1.00 & -1.13 & -1.13 & -1.07 & -1.11 \\
      T3 $L_2$ error $p$ & -1.62 &  -1.65 & -1.62 & -1.64 & -1.69 & -1.71 & -1.70 & -1.72 \\
      T3 $H^1$ error $p$ & -0.95 & -0.95 & -0.95 & -0.95 & -0.96 & -0.96 & -0.96 & -0.96 \\
      \hline
  \end{tabular}
  \label{Tab:h_convergence_table}
\end{table}
Normal components of the velocity function in equation \eqref{Eqn:numerical_h} 
are prescribed as the boundary condition. A pressure of 10 is also prescribed
at the bottom left corner to ensure uniqueness in the solution.
Using the parameters listed in Table \ref{Tab:h_converge}, 
Figure \ref{Fig:Numerical_analytical_solution} depicts the analytical velocity 
and pressure solutions to which the finite element solutions shall be compared 
with. Since neither the pressure nor velocity functions depend on the drag 
coefficient, only the specific body force varies with respect to each Darcy model.
The four Darcy models used are: the original Darcy (D), modified Barus (MB), 
Darcy-Forchheimer (F), and modified Darcy-Forchheimer Barus (MBF) models. 
The $L_2$ norm and $H^1$ seminorm error slopes for the modified Darcy-Forchheimer 
Barus models are depicted in Figure \ref{Fig:Numerical_slopes_MBF}. Four-node 
quadrilateral (Q4) and three-node triangular (T3) elements are used to solve 
the problems, and Table \ref{Tab:h_convergence_table} 
lists all the error slopes for the rest of the models.

It can be seen that the numerical solutions perform well; converged solutions should 
have error slopes of approximately -2.00 and -1.00 for $L_2$ norm and $H^1$ seminorm 
respectively. Though not shown, it has been found that 
the error slopes for all six models are similar to one another. Quadrilateral elements 
tend to exhibit faster convergence rates than triangular elements, and one can 
expect even faster rates for higher order
elements like the nine-node quadrilateral Q9 and six-node triangular T6.

\subsection{Quarter five-spot problem}
\begin{figure}[t!]
  \centering
  \includegraphics[scale=1]{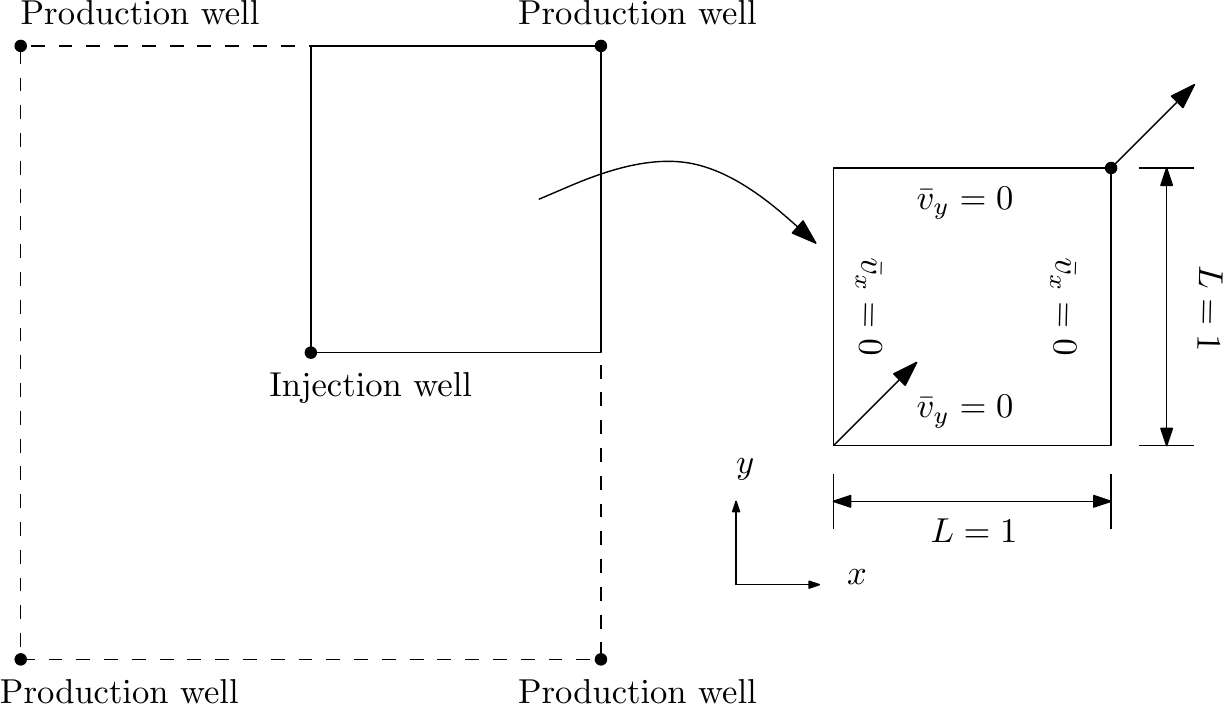}
  \caption{Quarter five-spot problem: A pictorial description.}
  \label{Fig:fivespot}
\end{figure}
This section presents numerical results for a quarter 
spot problem as depicted in Figure \ref{Fig:fivespot}. 
In many enhance oil recovery applications, there is an 
injection well centered around four production wells. 
\begin{table}[b!]
  \centering
  \caption{User-defined inputs for the quarter five-spot problem: Darcy model.}
  \begin{tabular}{cc}
    \hline
    Parameters & Value \\
    \hline
    $\bar{\alpha}$ & 1 \\
    $\bar{\rho}$ & 1 \\
    $\bar{\mathbf{b}}$ & 0 \\
    $Nele$ & 400 \\
    $\bar{p}(1,1)$ & 1 \\
    $\bar{v}_x(0,0)$, $\bar{v}_y(0,0)$ & 1 \\
    $\bar{v}_x(1,1)$, $\bar{v}_y(1,1)$ & 1 \\
    \hline
  \end{tabular}
  \label{Tab:quarter_spot}
\end{table}
When carbon-dioxide is injected into the ground, the pressure build up pushes oil out through the four
injection wells. This schematic forms what is often known as the five spot problem. Numerical results will 
exhibit elliptic singularities near the injection and production wells and provide a good benchmark to test 
the robustness of the finite element formulations. Due to the symmetric nature of the problem, only the top 
right quadrant is considered in the analysis. There is no specific body force or volumetric/sink source, and 
a pressure of $\bar{p}_0 = 1$ is prescribed at the production well or top right node.
Since it has been shown in previous sections that the FEM developed performs well for both Q4 and T3 elements, 
only quadrilateral elements Q4 and Q9 will be used to simulate all proceeding numerical simulations.
\begin{figure}[t!]
  \centering
  \subfigure[LS velocity vector field]{
  	\includegraphics[scale=0.46]
    {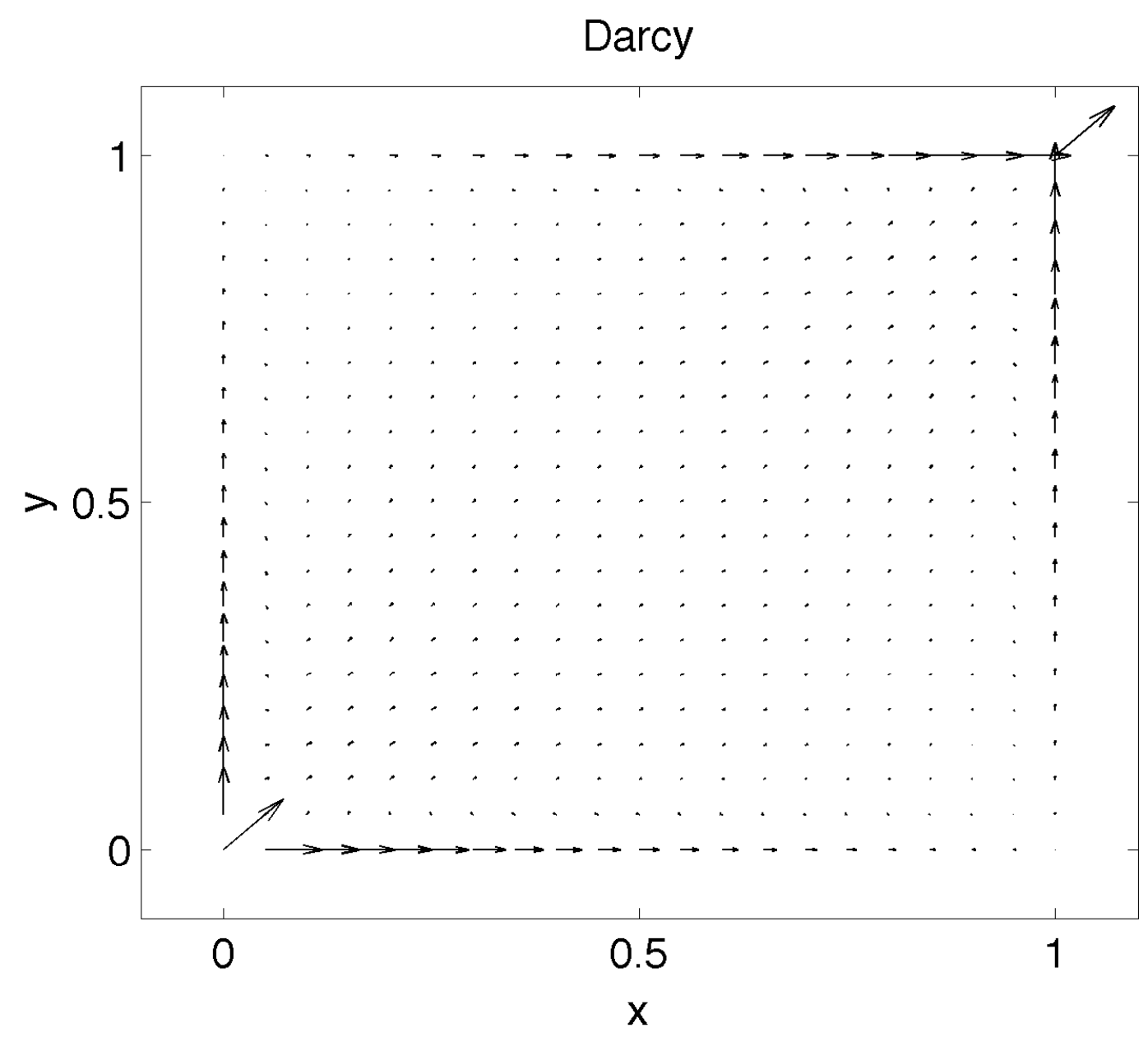}}
  \label{Fig:Quarter_spot_Q4_quiver}
  \subfigure[LS pressure contour]{
  	\includegraphics[scale=0.46]
    {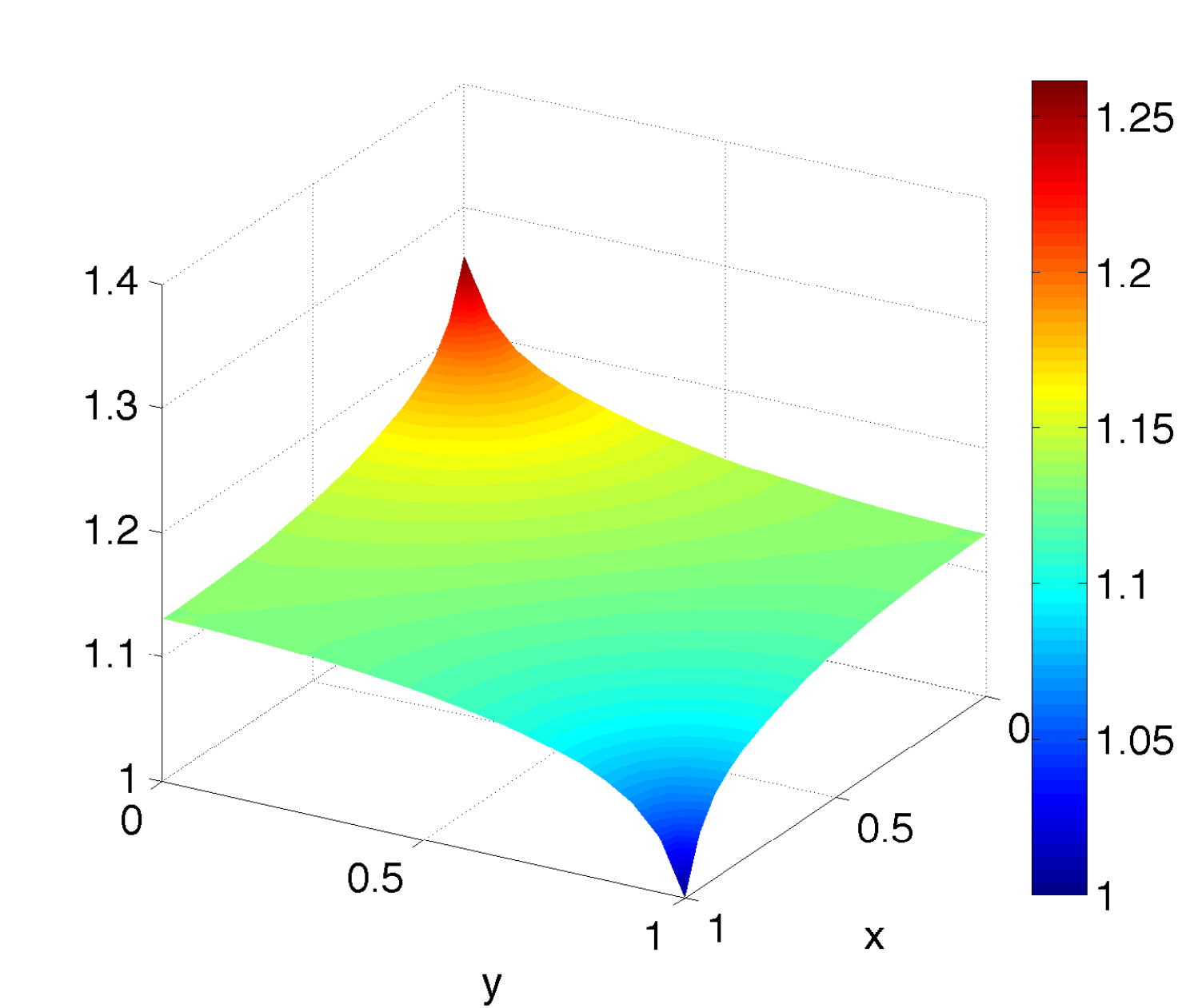}}
  \subfigure[VMS velocity vector field]{
  	\includegraphics[scale=0.46]
    {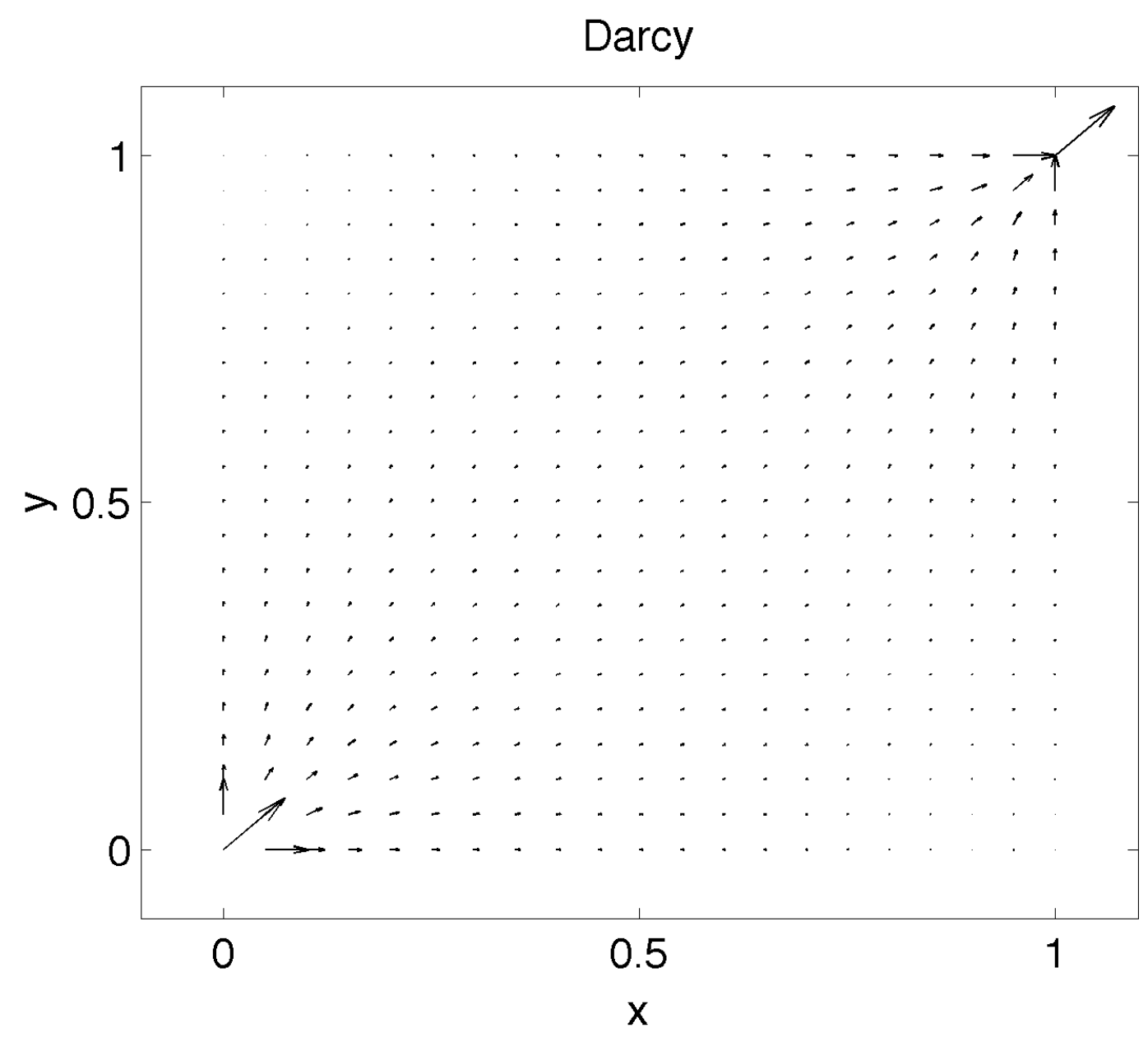}}
  \subfigure[VMS pressure contour]{
  	\includegraphics[scale=0.46]
    {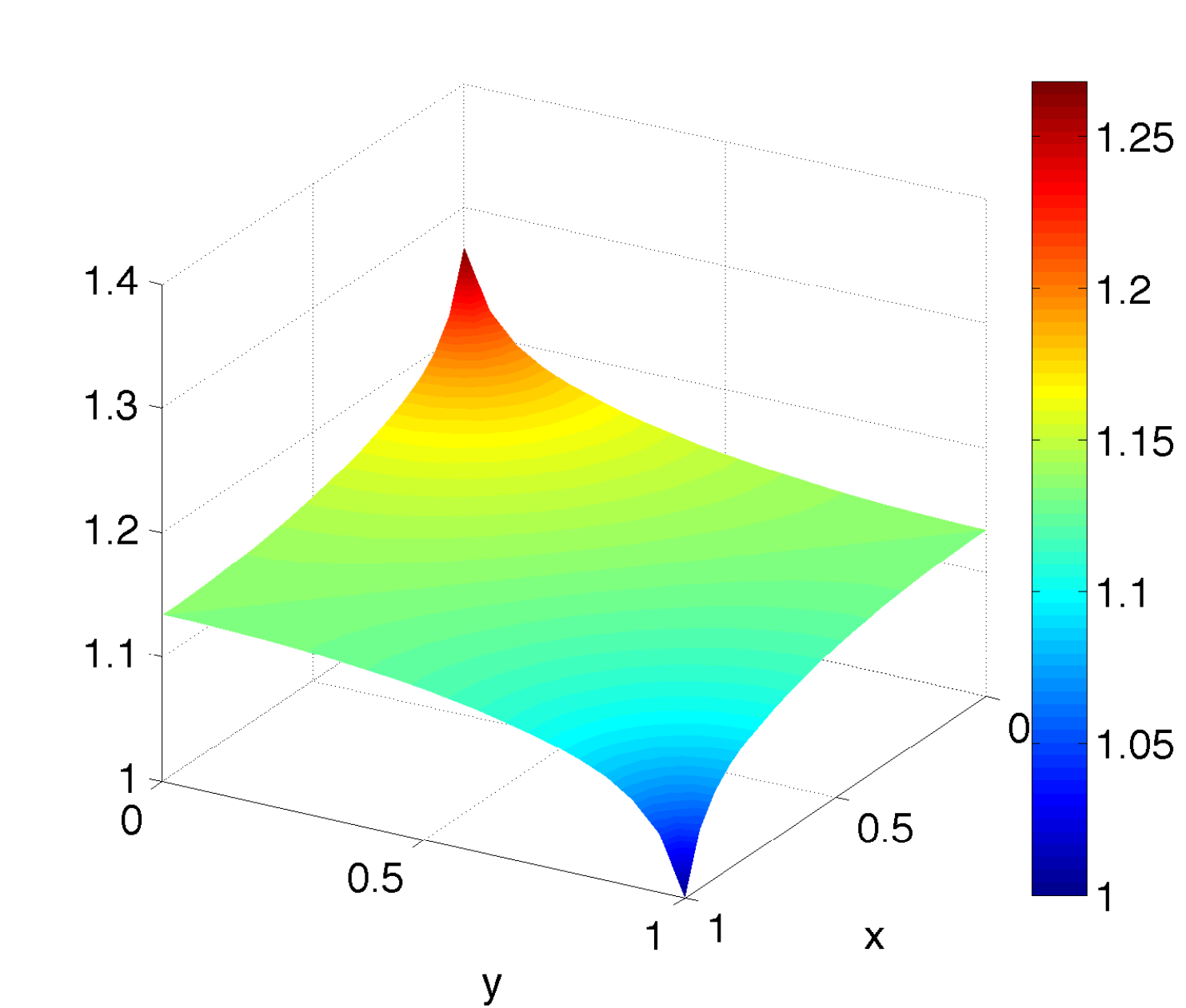}}
  \caption{Quarter five-spot problem: Q4 solutions for Darcy model}
  \label{Fig:Quarter_spot_Q4_D}
\end{figure}
\begin{figure}[t!]
  \centering
  \subfigure[LS velocity vector field]{
  	\includegraphics[scale=0.46]
    {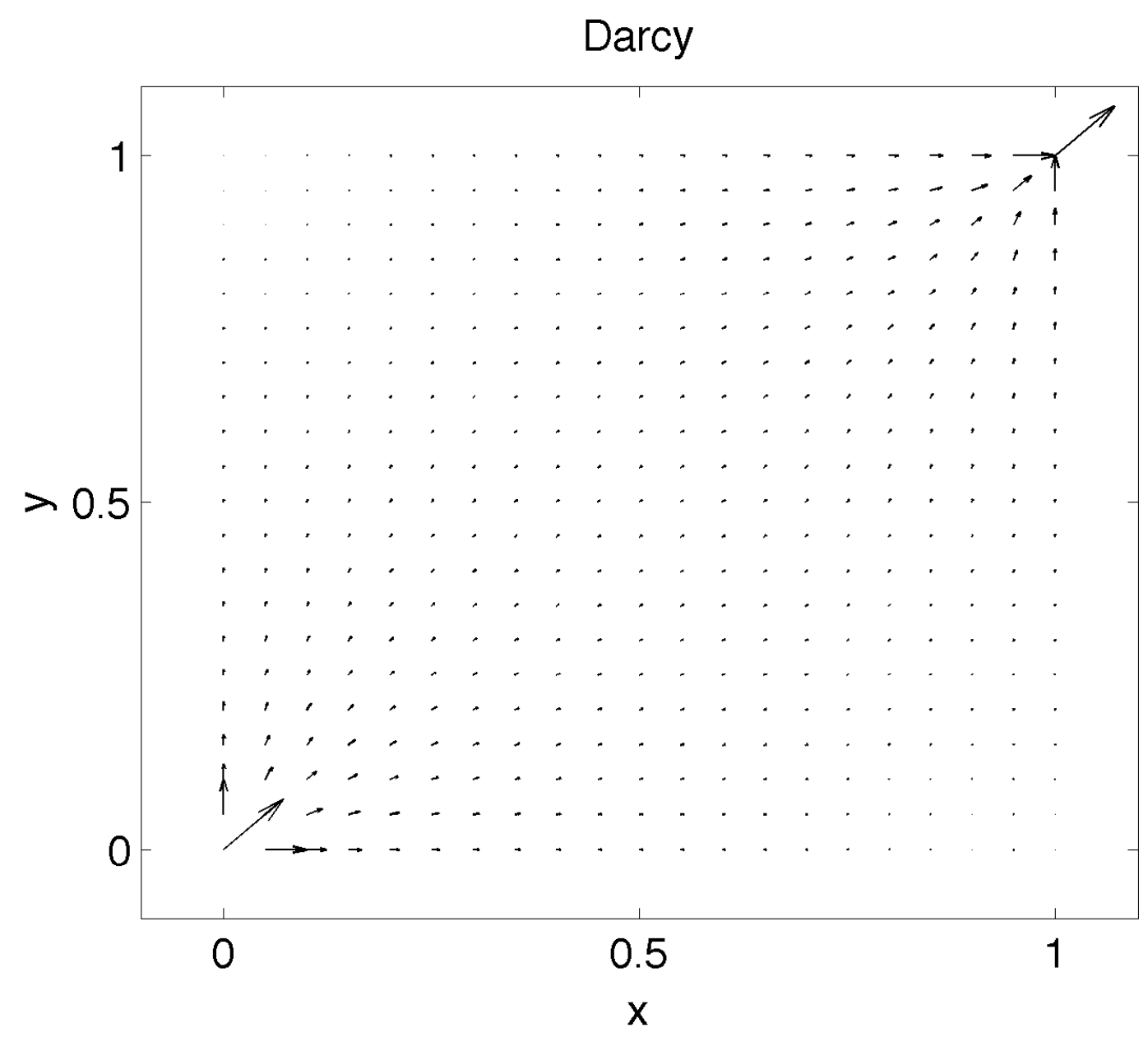}}
  \subfigure[LS pressure contour]{
  	\includegraphics[scale=0.46]
    {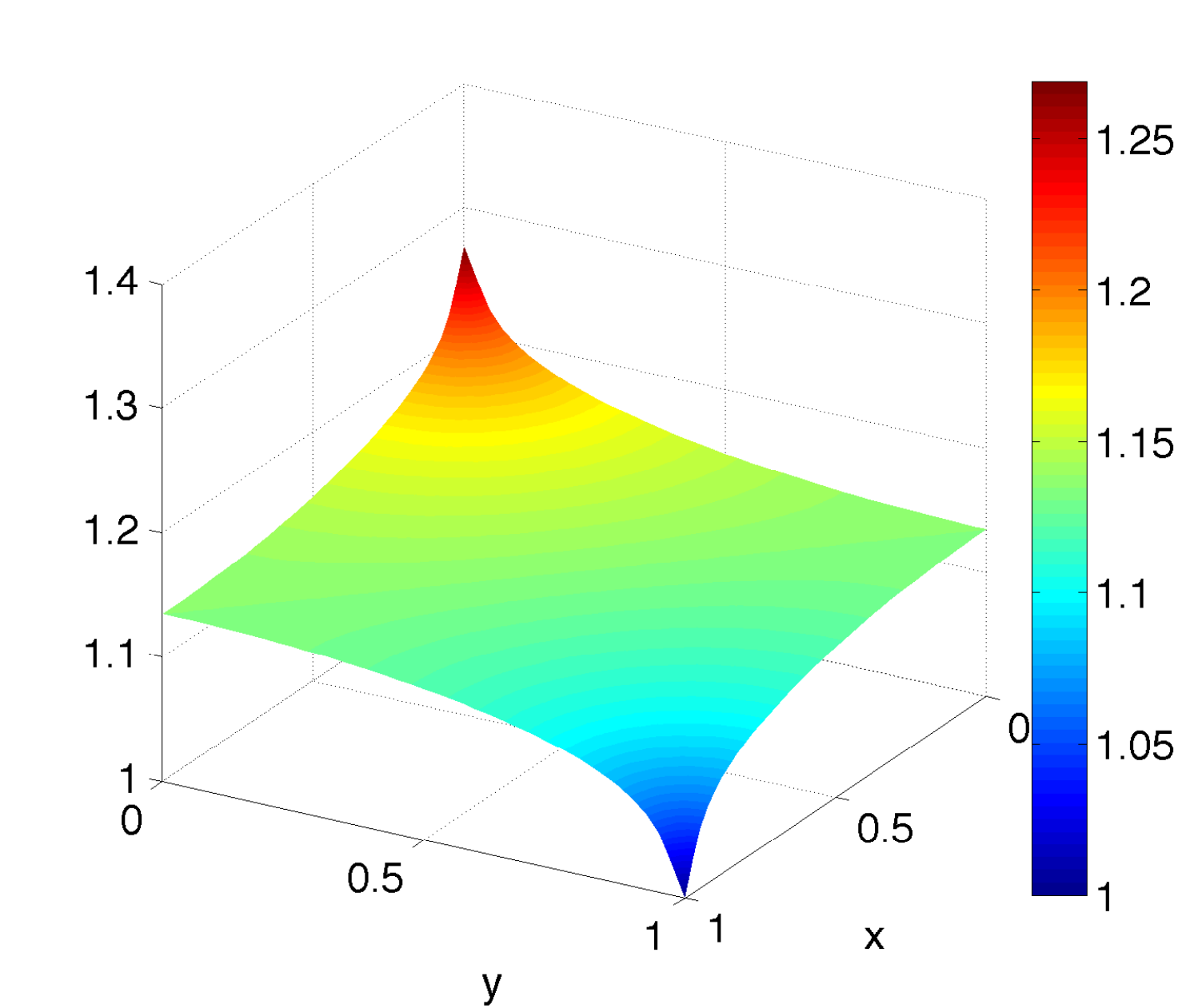}}
  \subfigure[VMS velocity vector field]{
  	\includegraphics[scale=0.46]
    {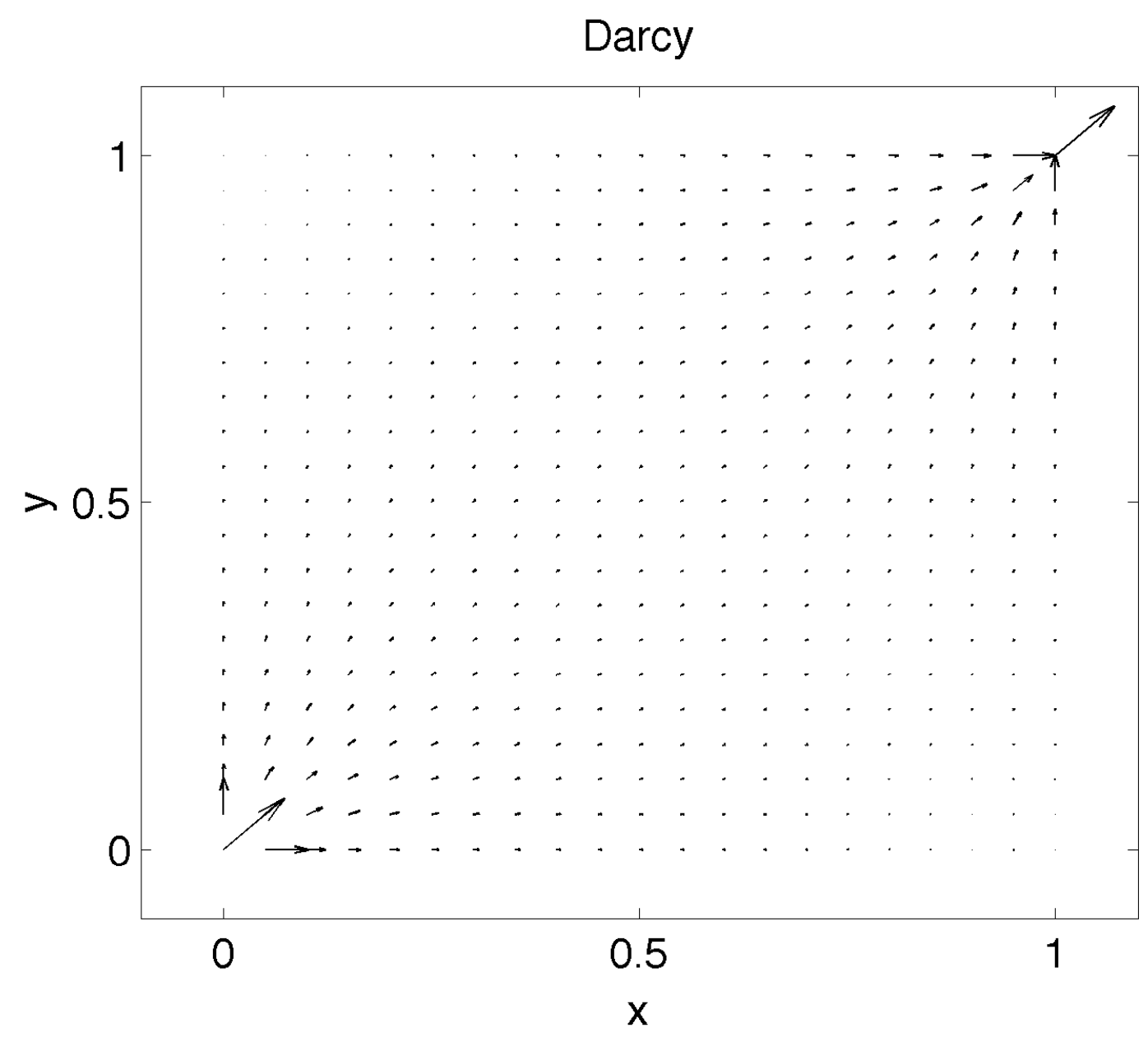}}
  \subfigure[VMS pressure contour]{
  	\includegraphics[scale=0.46]
    {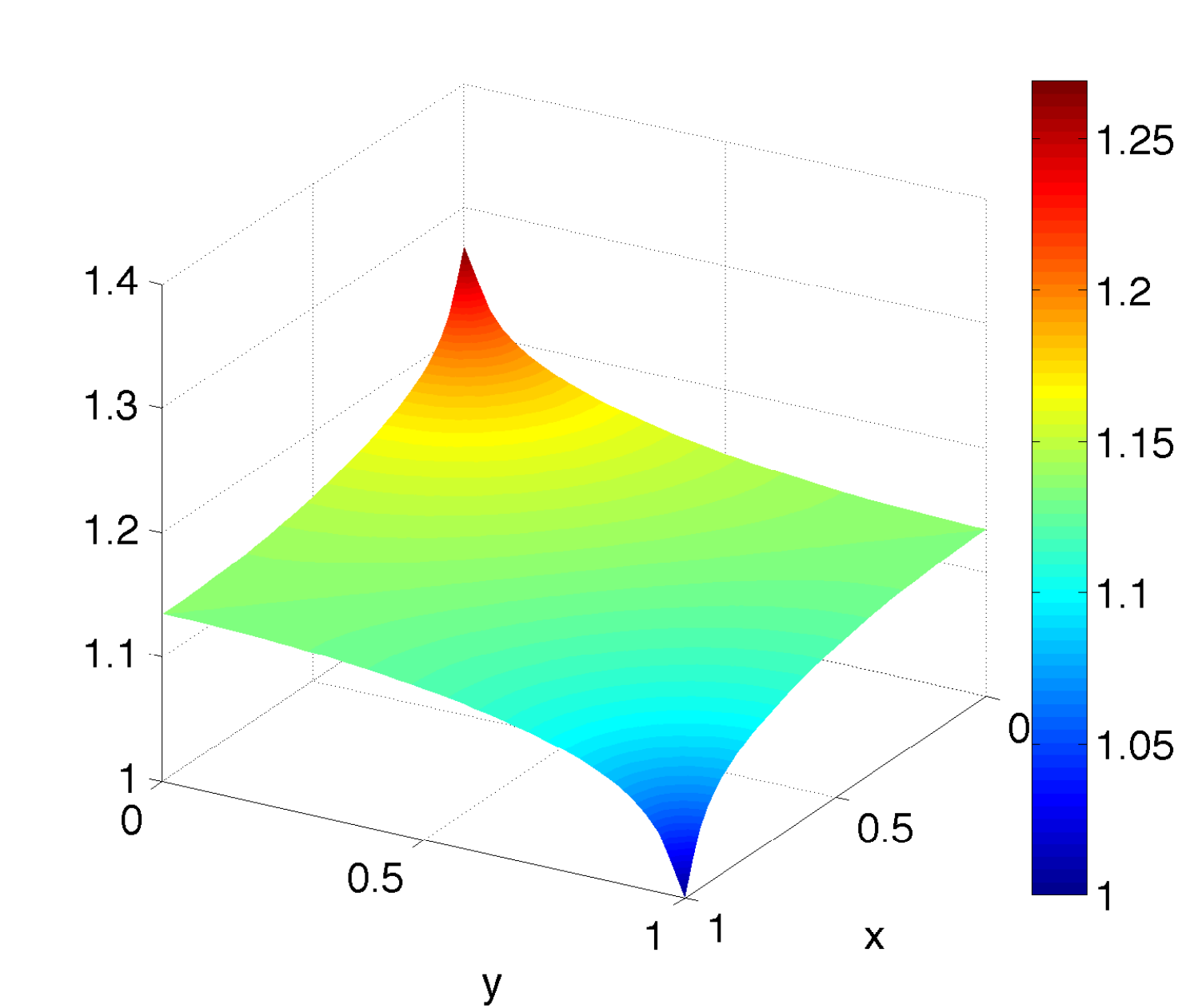}}
  \caption{Quarter five-spot problem: Q9 solutions for Darcy model}
  \label{Fig:Quarter_spot_Q9_D}
\end{figure}

Consider a case where there is $\sqrt{2}$ units of flow through the unit 
square quadrant. To attain this flow rate, one needs to know the
amount of pressure needed at the injection well (i.e., the bottom left node). 
Using Q4 elements and the parameters listed in Table \ref{Tab:quarter_spot}, 
Figure \ref{Fig:Quarter_spot_Q4_D} depicts the qualitative velocity vector field and the 
pressure contour. While both formalisms exhibit similar pressure contours, the velocity 
vector field generated from the LS method exhibits poor dispersion of flow concentration 
at both wells. Intuitively, the profile of Figure \ref{Fig:Quarter_spot_Q4_quiver} makes 
little to no physical sense so when using the LSFEM, neither Q4 nor any other first 
order elements can be used to accurately model velocity contours.

However, when higher order elements are used, the LS velocity vector field resembles 
that of the VMS. Figure \ref{Fig:Quarter_spot_Q9_D} depicts the results
using Q9 elements. It should be noted that the pressure contours remain the same 
regardless of the element order used. Realistic pressure profiles can be 
obtained using either formalism or element type, but obtaining velocity and flow 
solutions with the LSFEM necessitates the use of Q9 or higher ordered elements.

\subsubsection{Least-squares weighting}
For all original Darcy model problems up to this point, the non-dimensionalized drag 
coefficient equals one, so the two possible LS weightings $\mathbf{A}$ in equation 
\eqref{Eqn:GEL_weights} would be the same. When $\bar{\alpha}$ 
no longer equals one, weighting number 2 begins to have a significant impact on the 
numerical solutions. Herein, the injection pressure shall be obtained using various 
drag coefficients (all other user-defined parameters are as stated in Table 
\ref{Tab:quarter_spot}). The VMS formalism serves as a benchmark for the two LS 
weightings.
\begin{table}[t!]
  \centering
  \caption{Quarter five-spot problem:~Injection pressure comparison for different LS weightings.}
  \begin{tabular}{rccccccc}
      \hline
      $\bar{\alpha}:$ & 1 & 20 & 50 & 100 & 250 & 500 & 1000\\
      \hline
      LS weight 1 Q4: & 1.26 & 6.03 & 12.56 & 21.57 & 47.29 & 91.38 & 180.57 \\
      LS weight 1 Q9: & 1.27 & 6.38 & 14.44 & 27.76 & 67.17 & 132.63 & 263.77 \\
      LS weight 2 Q4: & 1.26 & 6.19 & 13.90 & 26.67 & 64.43 & 126.00 & 244.97 \\
      LS weight 2 Q9: & 1.27 & 6.38 & 14.46 & 27.92 & 68.31 & 135.60 & 269.37 \\
      VMS Q4: & 1.27 & 6.37 & 14.42 & 27.84 & 68.09 & 135.18 & 269.37 \\
      VMS Q9: & 1.27 & 6.38 & 14.46 & 27.93 & 68.32 & 135.63 & 270.27 \\
      \hline
  \end{tabular}
  \label{Tab:five_spot_weighting_compare}
\end{table}
From Table \ref{Tab:five_spot_weighting_compare}, it is seen that a divergence in the 
solutions occurs as the drag coefficient increases. LS formalism using Q9 elements has 
comparable stiffness to that of VMS formalism using Q4 elements, but as the drag 
increases, the VMS formalism using Q9 elements requires larger and larger pressures. 
Nonetheless, all the solutions show a linear relationship between drag and injection 
pressure. For highly viscous or lowly permeable reservoirs, one has to apply more pressure 
in order to attain or expect a certain flow. If drag is a function of pressure and/or 
velocity, one can expect even greater injection pressures.

\subsubsection{Comparison of beta coefficients, pressure profiles, and linearization types}
This next study shall illustrate the effect the Barus and Forchheimer coefficients have on the 
pressure profile and convergence of residuals. For pressure dependent viscosities, the Barus 
coefficient for most oils range between 15 to 35 GPa$^{-1}$ (see reference \cite{G_Stachowiak}) 
which translates to a non-dimensionalized coefficient of roughly 0.001 to 0.004. 
\begin{figure}[t!]
  \centering
  \subfigure[Modified Barus - LS Q9]{
  	\includegraphics[scale=0.46]
    {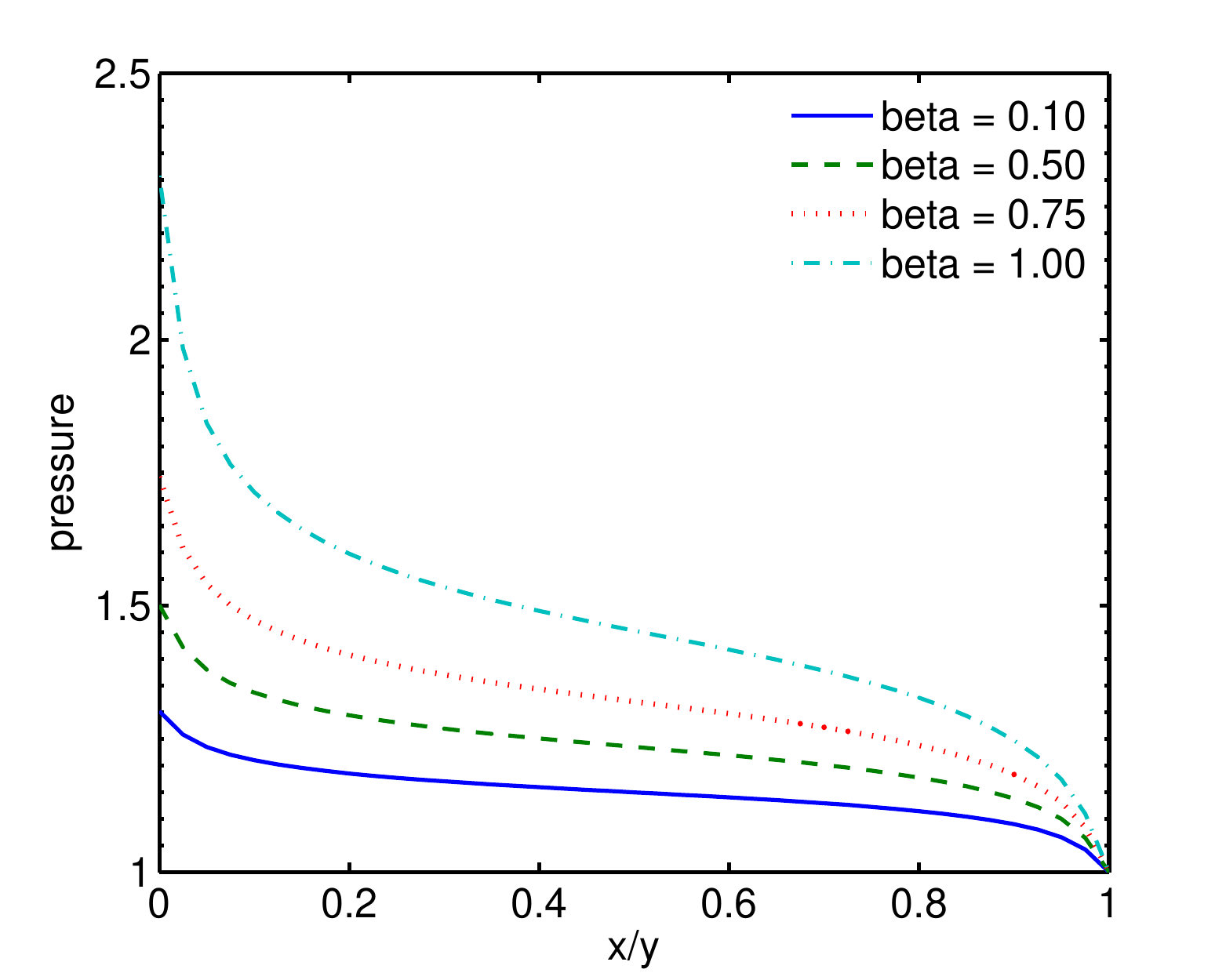}}
  \subfigure[Modified Barus - VMS Q9]{
  	\includegraphics[scale=0.46]
    {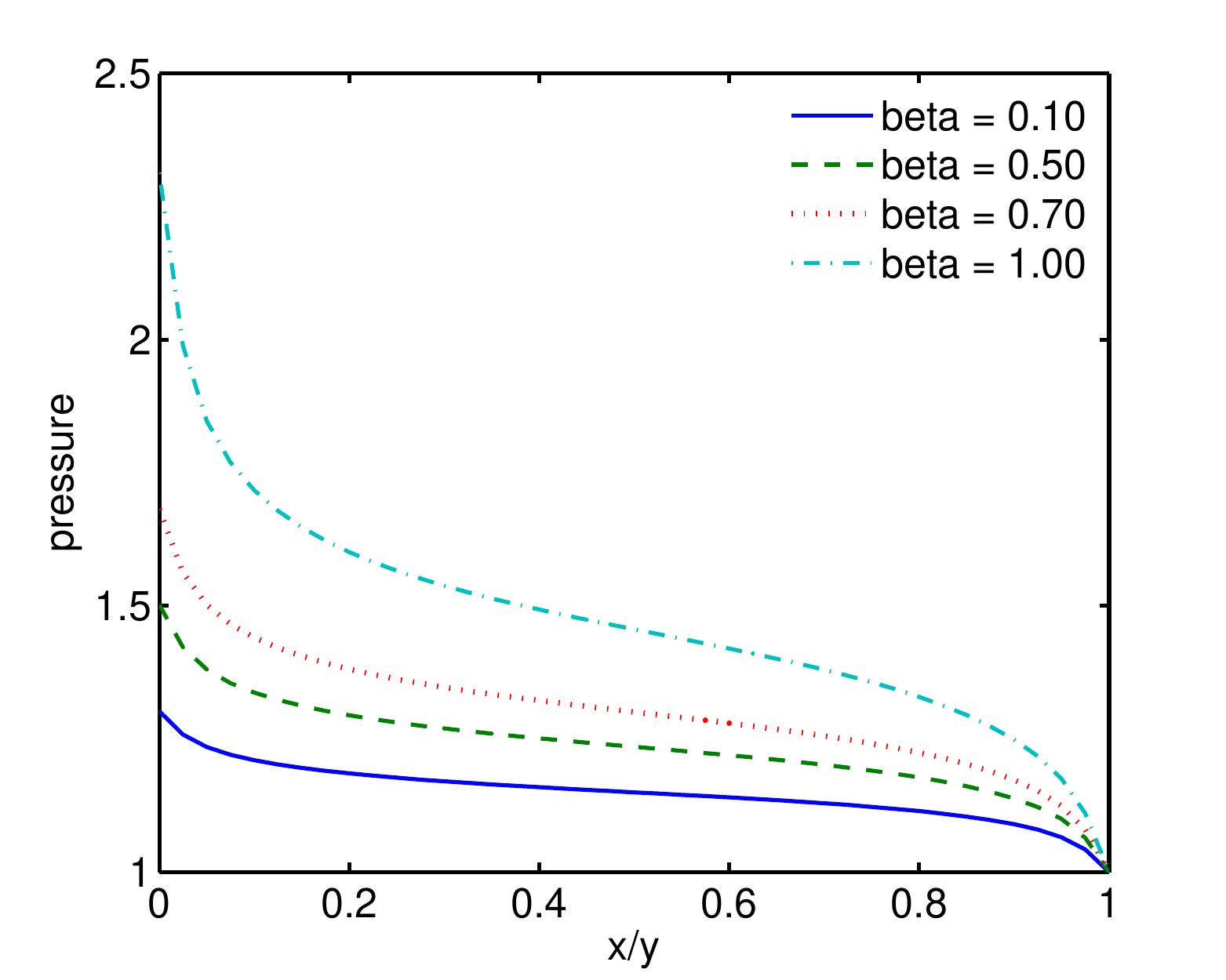}}
  \subfigure[Darcy-Forchheimer - LS Q9]{
  	\includegraphics[scale=0.46]
    {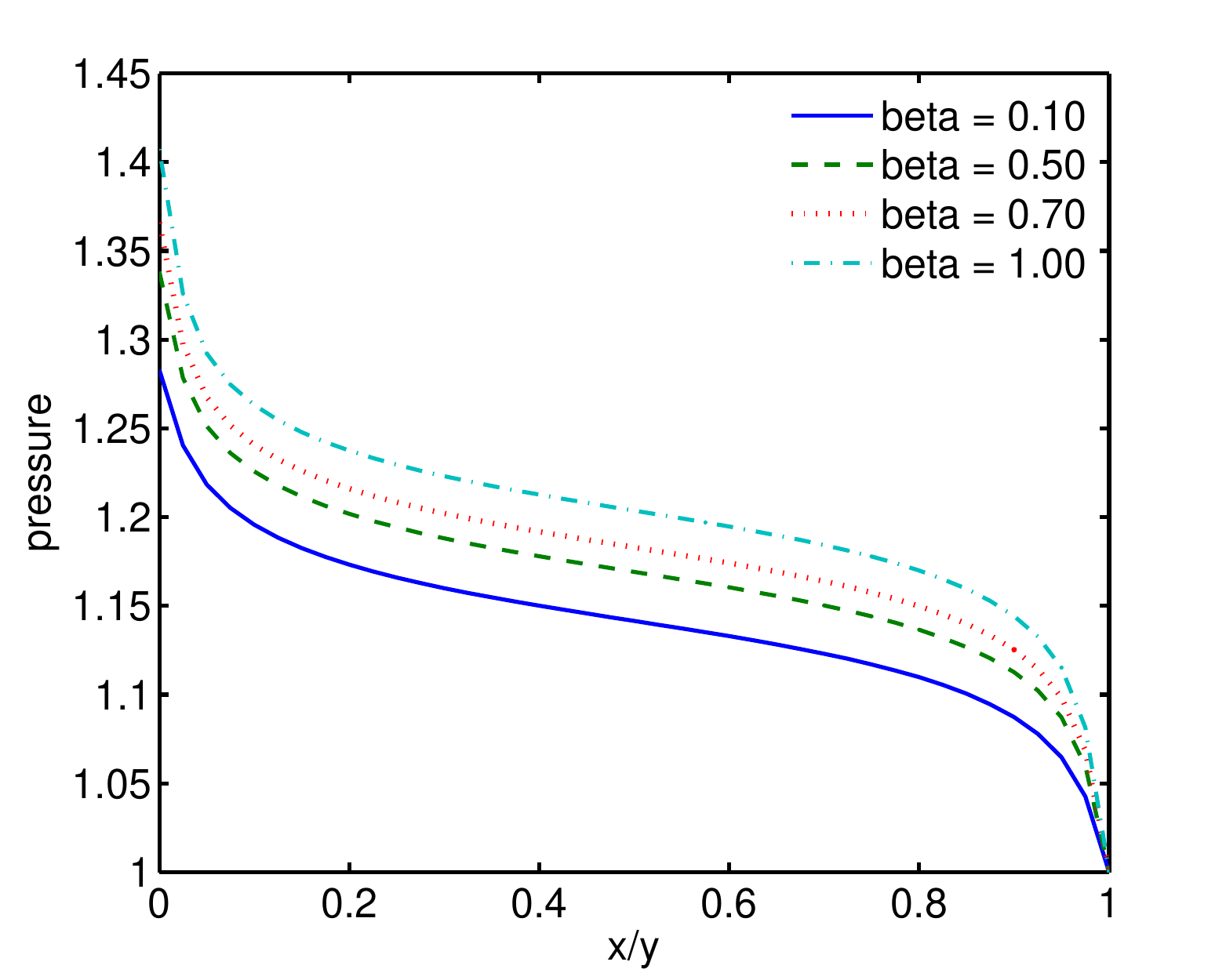}}
  \subfigure[Darcy-Forchheimer - VMS Q9]{
  	\includegraphics[scale=0.46]
    {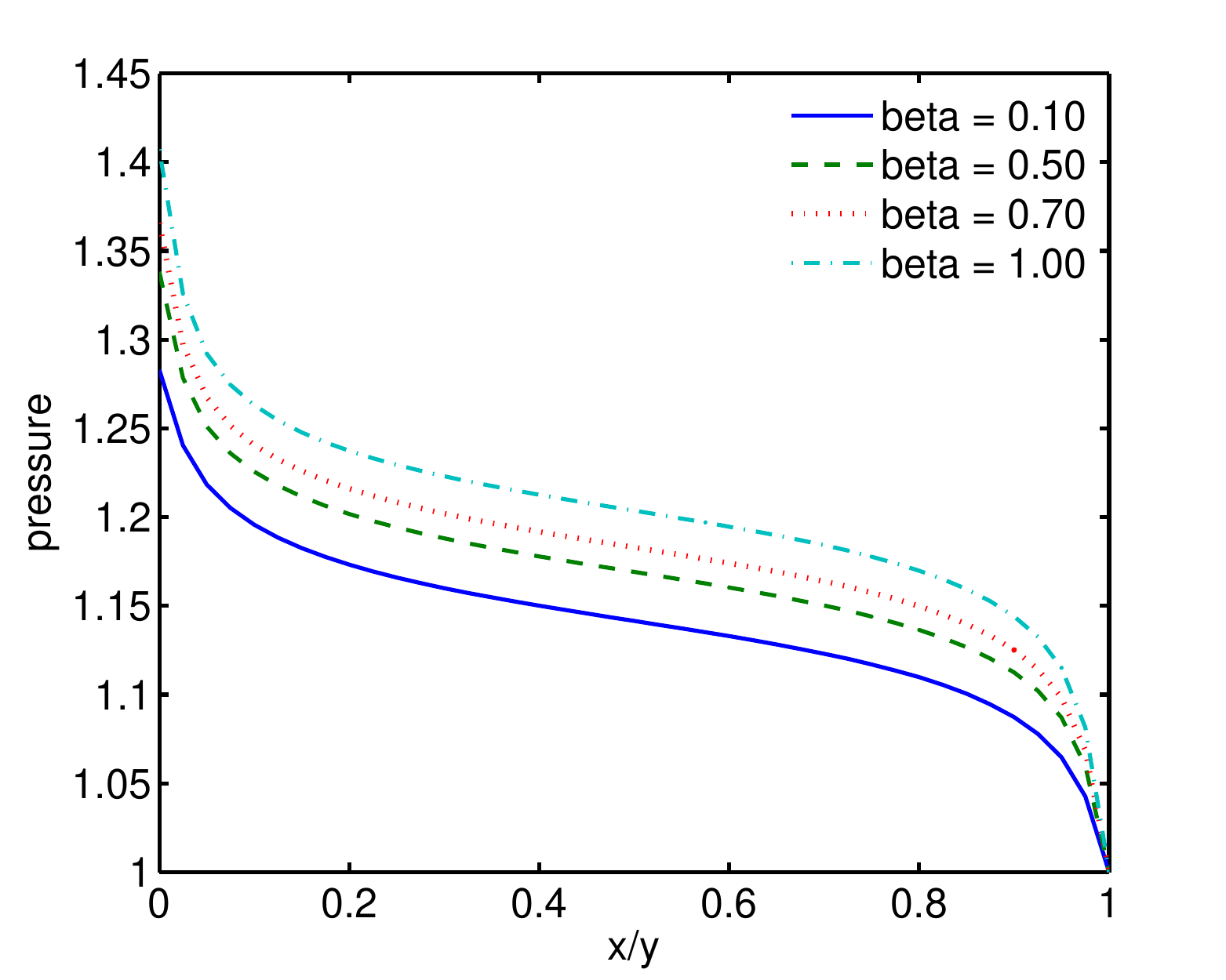}}
  \caption{Quarter five-spot problem: pressure profile vs various $\beta_{\mathrm{B}}$ and $\beta_{\mathrm{F}}$}
  \label{Fig:Quarter_spot_pressure_profile}
\end{figure}
However, for the purpose of this experiment, much higher Barus coefficients shall be used. 
The same Barus coefficients used will also be used for the Forchheimer coefficients. The 
relationship between the coefficients and the number of iterations needed to 
converge the residuals will also be shown for both linearization types.

Figure \ref{Fig:Quarter_spot_pressure_profile} depicts the pressure profile diagonally across 
the quarter region. Various beta values were used for the modified Barus and Darcy-Forchheimer 
models (assume both $\bar{\beta}_{\mathrm{B}}$ and $\bar{\beta}_{\mathrm{F}}$ to be denoted by 
the same $\bar{\beta}$). Overall the LS and VMS formalisms generate similar results.
\begin{table}[t!]
  \centering
  \caption{Quarter five-spot problem: Picard vs. consistent linearization iteration counts for modified Barus model with $\bar{\beta}_{\mathrm{B}} = 0.6$. The top
  table corresponds with LS formalism, and the bottom table corresponds with VMS formalism. Q9 elements are used}
  \begin{tabular}{lcc|cc}
      \hline
      LS formalism:&\multicolumn{2}{c|}{Picard's linearization}&\multicolumn{2}{c}{consistent linearization}\\
      Iteration no. $(i)$&$\mathbf{\bar{v}}$ residual&$\bar{p}$ residual&$\mathbf{\bar{v}}$ residual&$\bar{p}$ residual\\   
      \hline
      1 & 1.637285e+01 & 2.793981e-01 & 1.637323e+01 & 2.853694e-01 \\
      2 & 6.735667e-04 & 6.821094e-03 & 1.578203e-02 & 6.444641e-02 \\
      3 & 8.721274e-05 & 1.096770e-03 & 7.696835e-04 & 1.025242e-02 \\
      4 & 8.150045e-06 & 1.225905e-04 & 7.976516e-07 & 2.339612e-05 \\
      5 & 5.881603e-07 & 1.019954e-05 & 2.200057e-10 & 8.554086e-09 \\
      6 & 3.455235e-08 & 6.755041e-07 & 5.434345e-14 & 3.075479e-12 \\
      7 & 1.712434e-09 & 3.720906e-08 & & \\
      8 & 7.340079e-11 & 1.755375e-09 & & \\
      9 & 2.770668e-12 & 7.240625e-11 & & \\
      \\
       \hline
      VMS formalism:&\multicolumn{2}{c|}{Picard's linearization}&\multicolumn{2}{c}{consistent linearization}\\
      Iteration no. $(i)$&$\mathbf{\bar{v}}$ residual&$\bar{p}$ residual&$\mathbf{\bar{v}}$ residual&$\bar{p}$ residual\\   
      \hline
      1 & 6.474444e-02 & 1.376963e-01 & 2.086226e-01 & 1.376963e-01 \\
      2 & 5.513505e-03 & 3.454112e-03 & 1.506223e-01 & 3.675604e-02 \\
      3 & 4.187925e-04 & 5.555865e-04 & 4.681837e-03 & 1.500411e-02 \\
      4 & 2.726750e-05 & 6.208117e-05 & 6.326269e-06 & 9.377515e-05 \\
      5 & 1.551030e-06 & 5.160684e-06 & 9.996752e-10 & 5.528937e-08 \\
      6 & 7.728631e-08 & 3.412260e-07 & 1.985333e-13 & 1.116834e-11 \\
      7 & 3.389569e-09 & 1.874543e-08 & & \\
      8 & 1.318682e-10 & 8.807571e-10 & & \\
  \end{tabular}
  \label{Tab:five_spot_MB_iterations}
\end{table}
As the coefficient $\bar{\beta}$ increase, the pressure gradients at the two wells 
steepen. The modified Barus model exhibits the steepest gradients which is expected.
The Darcy-Forchheimer solutions also exhibit increases in the injection pressure, but the 
qualitative nature of the pressure gradients near the wells are slightly different. Since
the modified and Darcy-Forchheimer models rely on separate non-Darcy coefficients and 
different dependent variables, no true comparisons can be drawn. In the next 
Section however, distinction of results from pressure dependent and velocity dependent 
drag coefficients will become more evident.

It should be noted that the pressure profiles in Figure 
\ref{Fig:Quarter_spot_pressure_profile} were generated using Picard's 
linearization (i.e. $\vartheta = 0$). While Picard's 
and consistent linearization theoretically yield the same results, the 
residual convergence schemes differ. Table \ref{Tab:five_spot_MB_iterations} 
contains the iteration count and residual norms for the modified Barus model 
evaluated at $\bar{\beta}_{\mathrm{B}} = 0.6$. Consistent linearization exhibits 
terminal quadratic convergence whereas Picard's linearization exhibits terminal
linear convergence. As the betas and/or applied pressure increases, the number
of iterations needed increases.

\subsubsection{Modified Darcy-Forchheimer numerical results}
So far it has been established in this section that quadratic elements and LS weight 2 
are preferred for the LS formalism. Numerical simulations have also 
shown that high Barus and Forchheimer coefficients yield results that differ 
quantitatively from the original Darcy model. This next example shall study
the effects of combining the Darcy models and employs a finer mesh.

\begin{table}[h!]
  \centering
  \caption{Quarter five-spot problem: expected injection pressures for various Darcy 
  models and mesh sizes}
  \begin{tabular}{cc|cccc}
      \hline
      $Nele$ & formalism & D & MB & F & MBF \\
      \hline
      400 & LS & 1.2692 & 1.5017 & 1.3382 & 1.5806 \\
      400 & VMS & 1.2693 & 1.5020 & 1.3382 & 1.5809 \\
      900 & LS & 1.1967 & 1.3538 & 1.2430 & 1.4045 \\
      900 & VMS & 1.1967 & 1.3539 & 1.2430 & 1.4047 \\
      \hline
  \end{tabular}
  \label{Tab:quarter_spot_pvm}
\end{table}
Using the same boundary conditions as before, a non-dimensionalized Barus and 
Forchheimer coefficient of 0.5 and 0.5 shall be used. Since the qualitative nature of 
the pressure contours are similar no matter which model is used, only the required 
injection pressure is listed. It can be seen from the results in Table 
\ref{Tab:quarter_spot_pvm} that refining the mesh lowers the pressure. For problems where 
flow quantities are fixed, coarse meshes over predict the required pressure needed. 
The Barus, linear, and Forchheimer models all predict pressures greater than that of the 
original Darcy model, and when the modified Darcy-Forchheimer models are employed, 
we get even higher pressures. The original Darcy model under predicts the amount of 
pressure needed so it is important to use the modified Darcy-Forchheimer models to 
attain an accurate visualization of the pressure contours.

\subsubsection{Minimum dissipation}
\begin{figure}[t!]
  \centering
   \subfigure[LS formalism]{
  	\includegraphics[scale=0.46]
  	{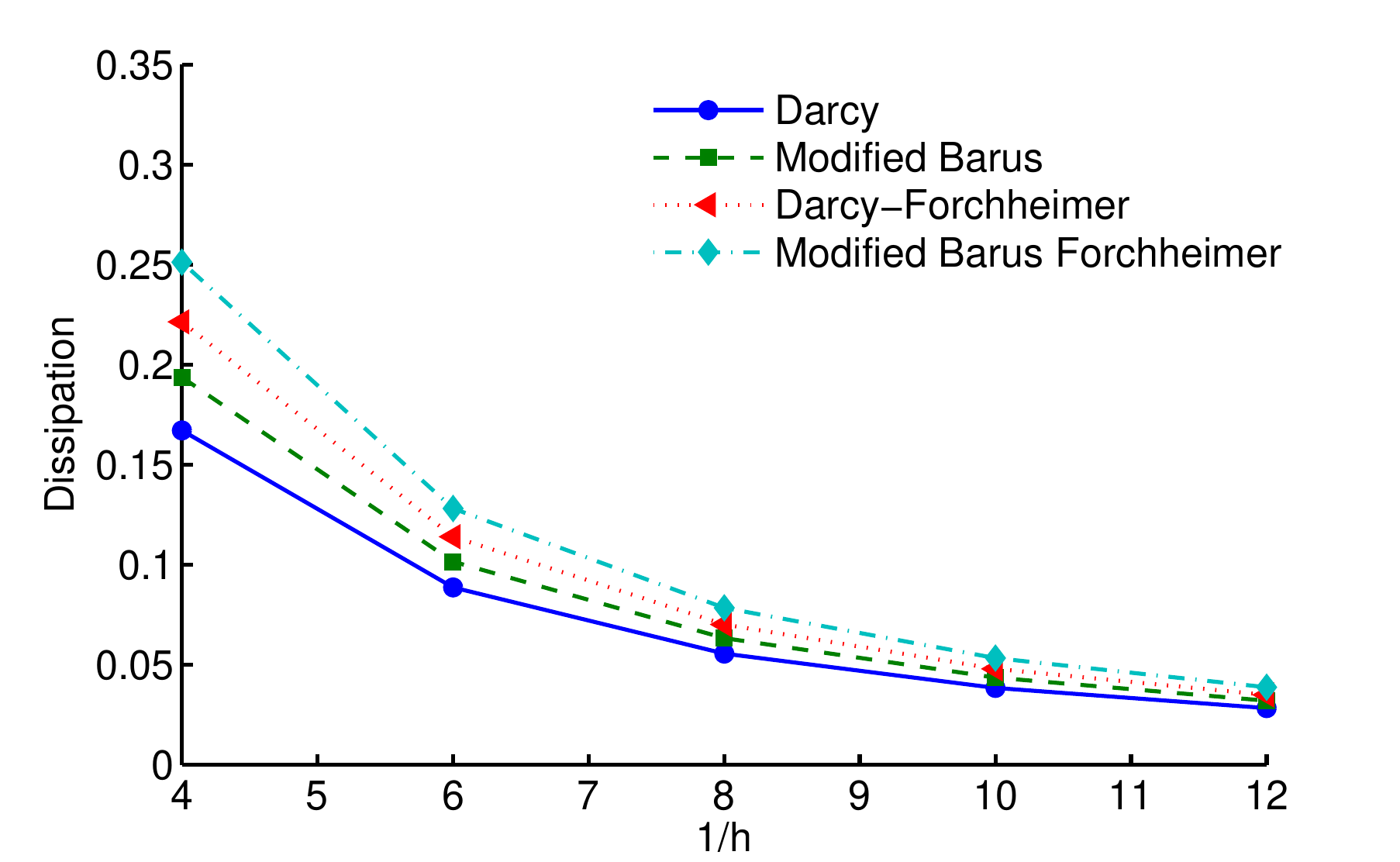}}
  \subfigure[VMS formalism]{
  	\includegraphics[scale=0.46]
  	{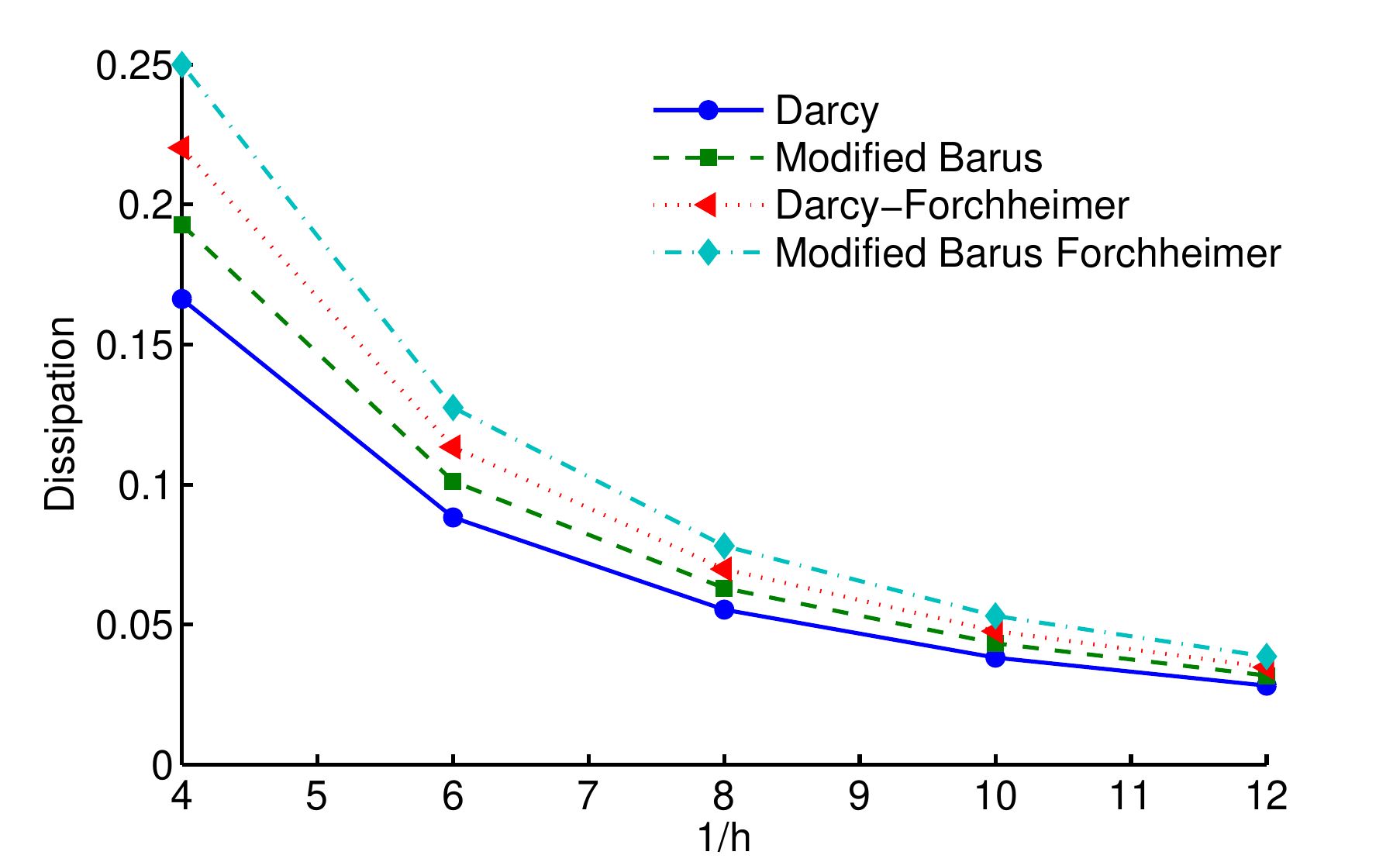}}
  \caption{Quarter five-spot problem: dissipation vs $h$-size}
  \label{Fig:quarter_spot_dissipation}
\end{figure}
Dissipation shall now be measured for this problem using various mesh sizes (Barus and
Forchheimer coefficients of 0.1 and 0.5 shall be used respectively). The 
$h$-sizes used for this problem are 1/4, 1/6, 1/8, 1/10, and 1/12. All four Darcy models are
evaluated using both formalisms, and the results can be found in Figure 
\ref{Fig:quarter_spot_dissipation}. It is seen that as the mesh gets finer the dissipation
becomes smaller hence satisfying the minimum dissipation inequality. Though the LS
formalism yields slightly higher dissipations, the results
of both formalisms are very similar.

\subsection{Three-dimensional constant flow}
It has been claimed in references 
\cite{Nakshatrala_Turner_Hjelmstad_Masud_CMAME_2006_v195_p4036,
Hughes_Masud_Wan_CMAME_2006_v195_p3347} that the VMS mixed 
formulation for Darcy model is the only known mixed 
formulation that satisfies constant flow patch test 
in three dimensions on non-constant Jacobian finite 
elements. 
We show that the mixed formulation based on least-squares 
formalism also satisfy the constant patch test in three 
dimensions for Darcy model. We shall also use the test 
problem to show that the proposed mixed formulations 
perform well even for other modifications of the Darcy 
model. (It should be emphasized that this problem can 
be considered as a patch test only for Darcy model, 
and not for modified Darcy-Forchheimer, as the 
exact solution under the modified Darcy-Forchheimer 
model will not be neither linear nor constant.)  

The computational domain is a unit cube, which is meshed 
using eight-node brick elements. Normal components of the 
velocity are prescribed as unity on the y-z planes at x = 
0 and x = 1. The other four planes have normal components 
of velocity equal to zero, and a pressure value of zero 
is prescribed at (0,0,0) to ensure uniqueness to the 
solution. Using the values defined in Table 
\ref{Tab:three_dimensional}, the LS results for the 
original Darcy, modified Barus, Darcy-Forchheimer, 
and modified Darcy-Forchheimer Barus models are shown 
in Figure \ref{Fig:Three_dimensional_pressure_LS}. 
Clearly, the LS-based mixed formulation satisfies 
the constant patch test. 

\begin{table}[b!]
  \centering
  \caption{User-defined inputs for the three-dimensional problem}
  \begin{tabular}{cc}
    \hline
    Parameter & Value \\ \hline
    $\bar{\beta}_{\mathrm{B}}$ & 0.5\\
    $\bar{\beta}_{\mathrm{F}}$ & 1\\	
    $\bar{k}$ & 1\\
    $\bar{\mu}_0$ & 1\\
    $\vartheta$ & 1 \\
    $\bar{\mathbf{b}}(\mathbf{x})$ & 0 \\
    $Nele$ & 216\\
    \hline
  \end{tabular}
  \label{Tab:three_dimensional}
\end{table}

\begin{figure}[t!]
  \centering
  \subfigure[Darcy model]{
  	\includegraphics[scale=0.46]
    {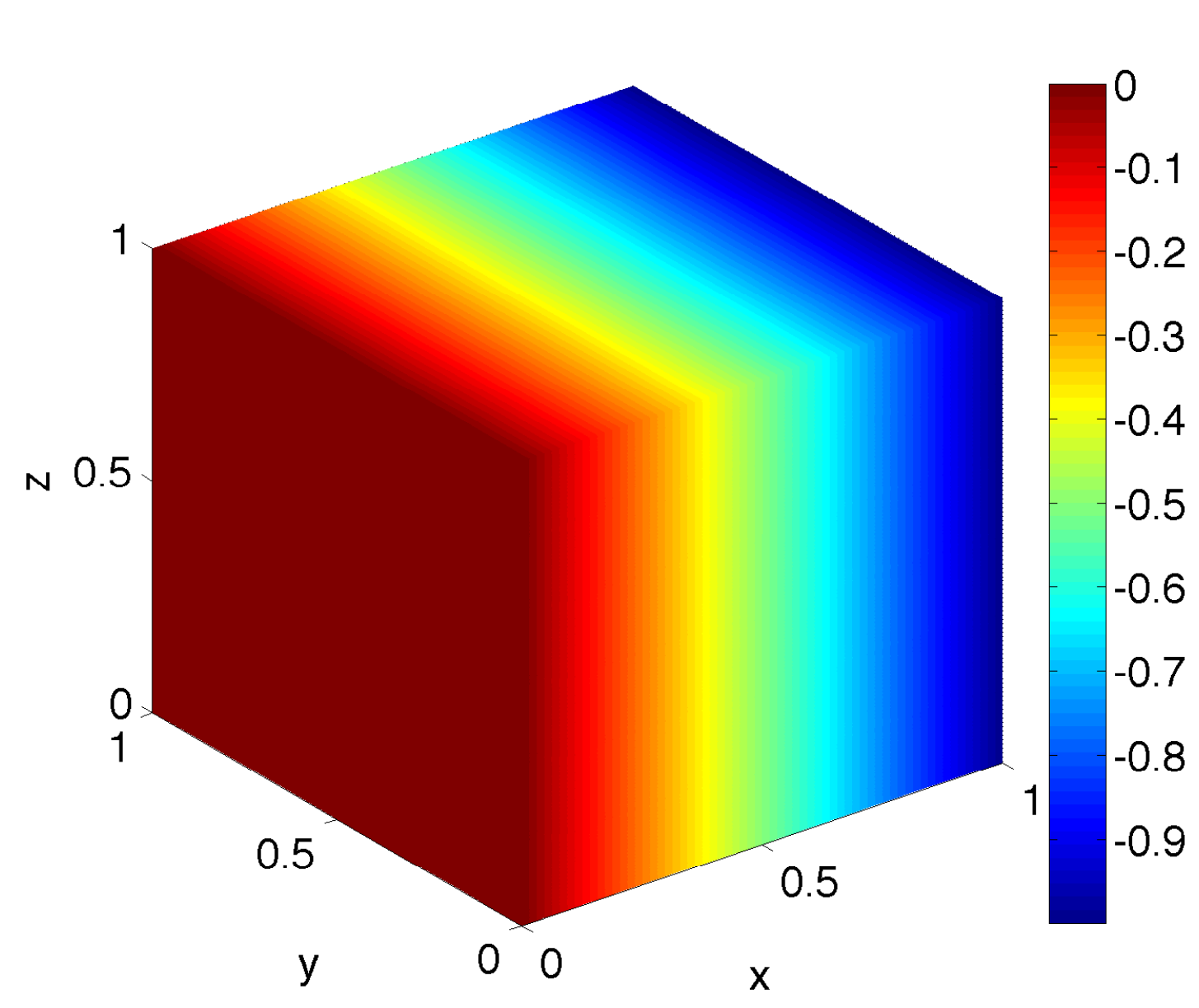}}
  \subfigure[Modified Barus]{
  	\includegraphics[scale=0.46]
    {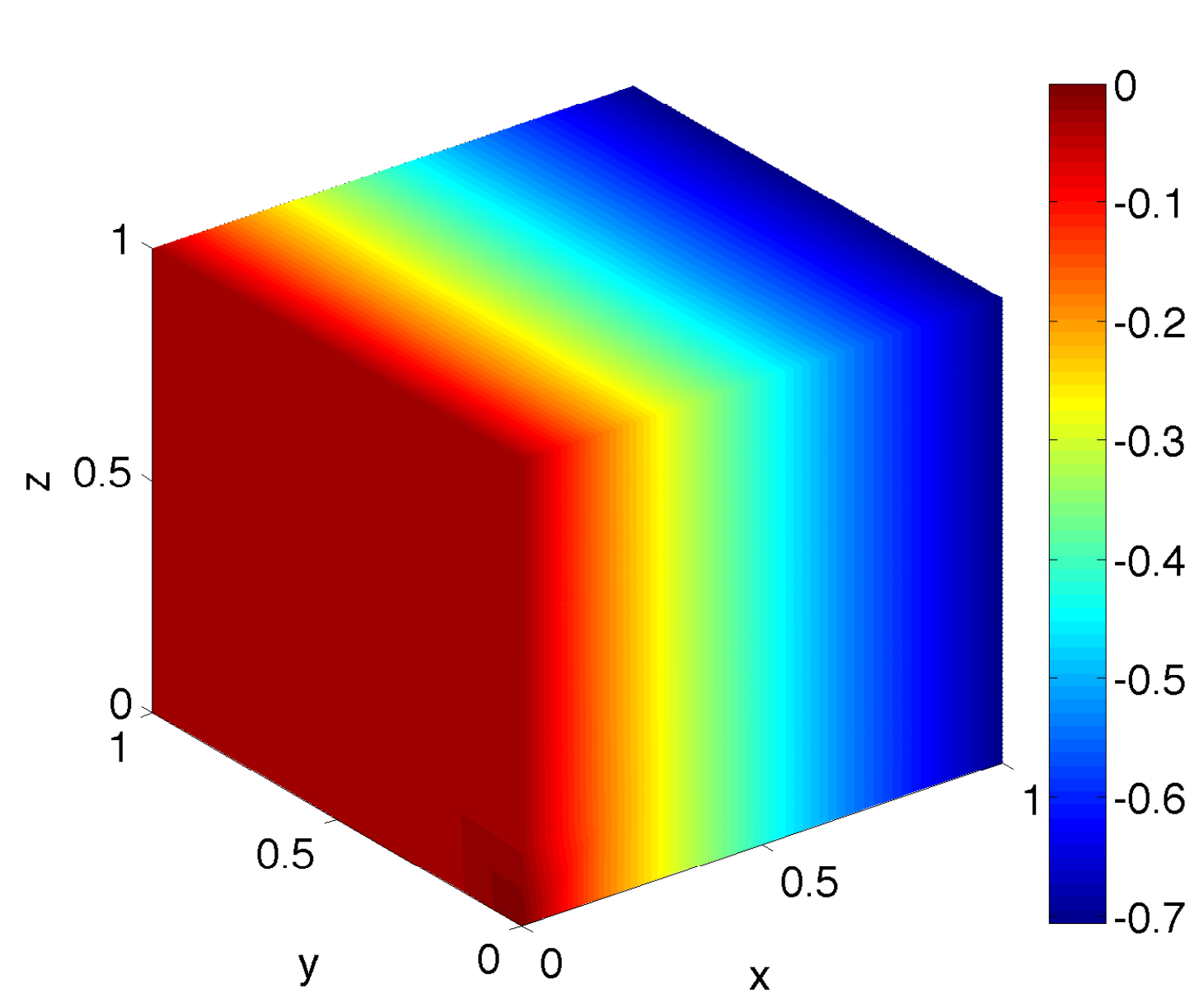}}
  \subfigure[Darcy-Forchheimer]{
  	\includegraphics[scale=0.46]
    {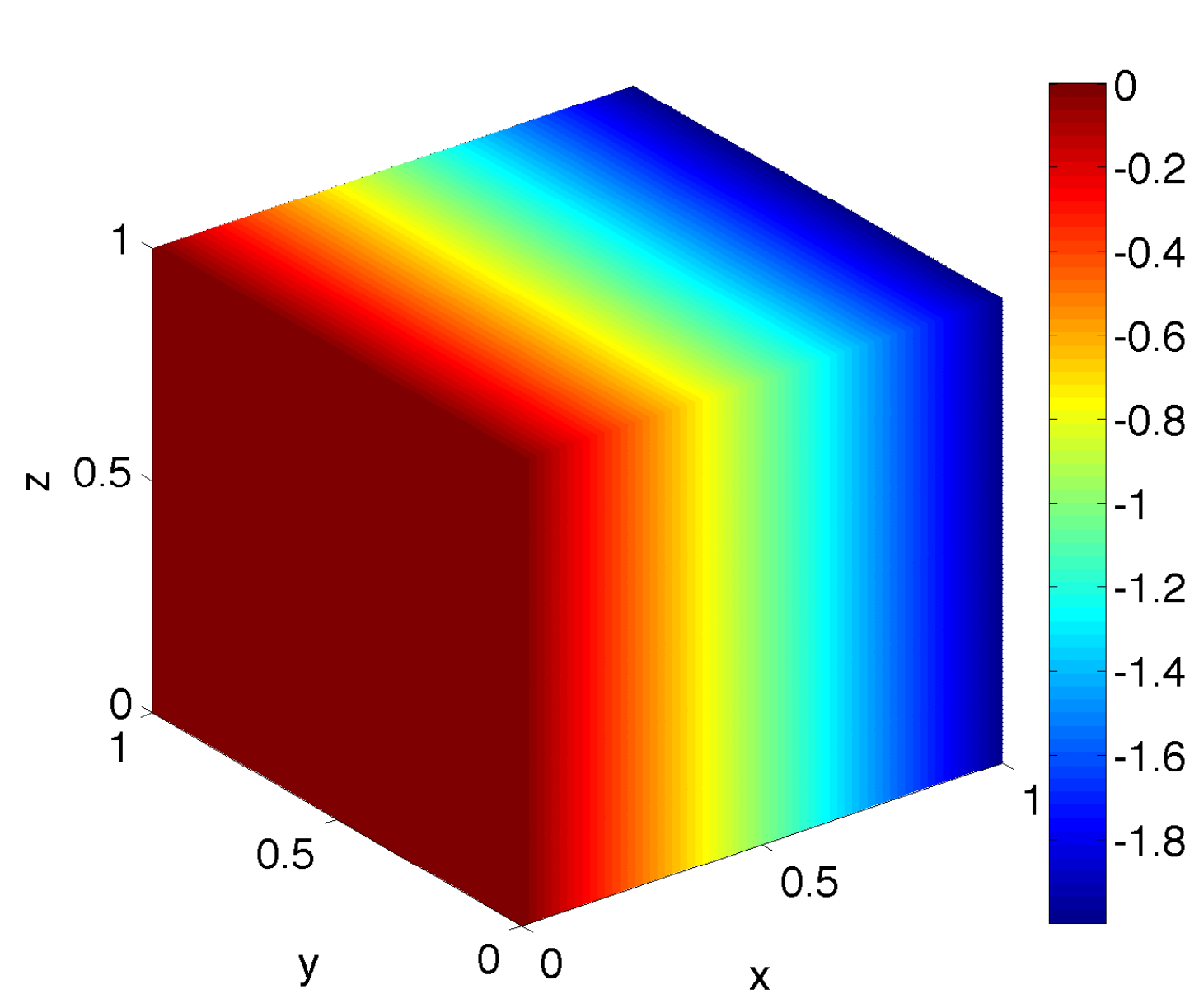}}
  \subfigure[Modified Darcy-Forchheimer Barus]{
  	\includegraphics[scale=0.46]
    {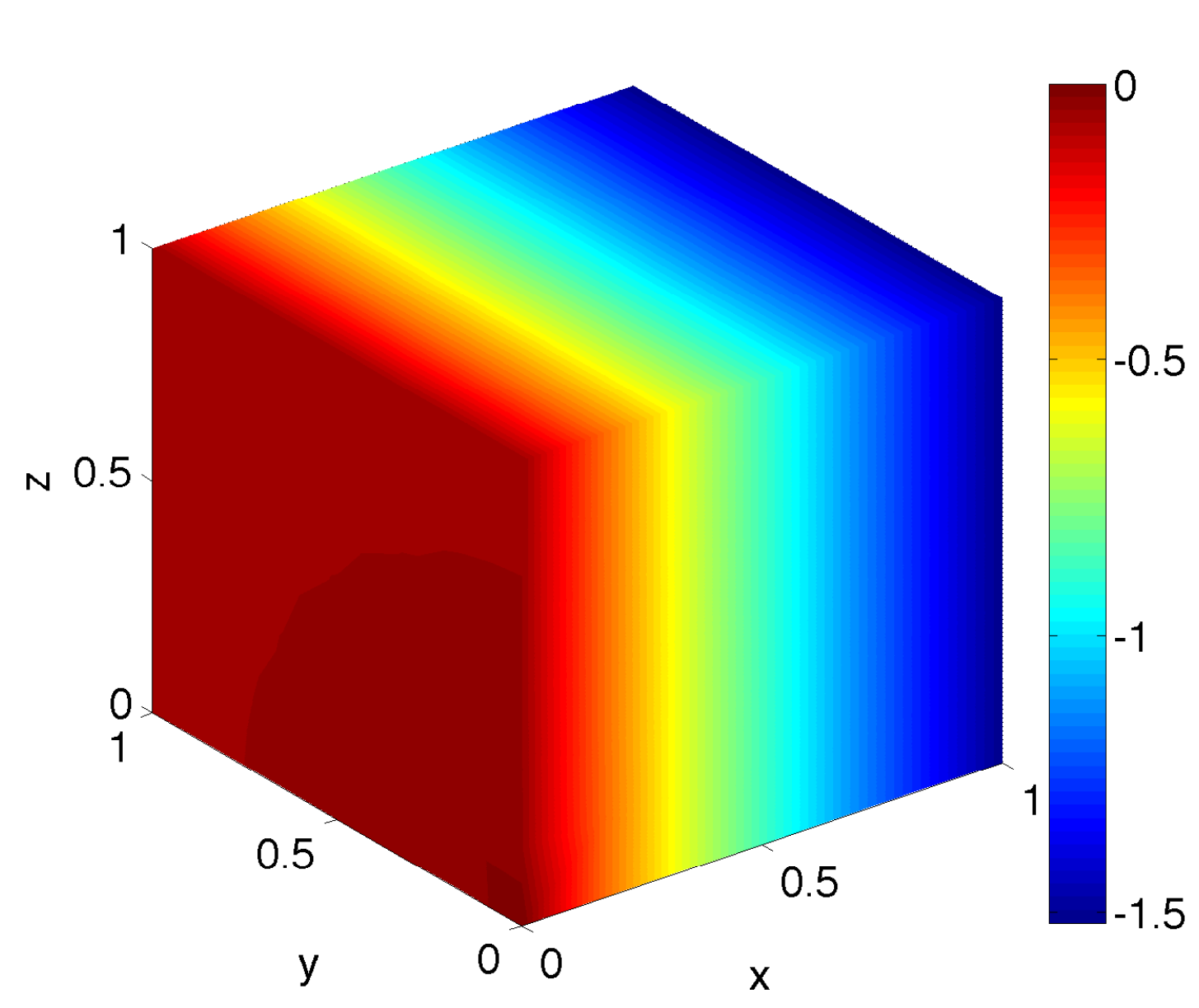}}
  \caption{Three dimensional problem: Pressure 
    contours using LS formalism.}
  \label{Fig:Three_dimensional_pressure_LS}
\end{figure}

%% file: Sections/Ch5_Enhanced_oil_recovery.tex
\section{ENHANCED OIL RECOVERY APPLICATIONS}
\label{Ch:EOR}

It has been shown in the previous section that the 
proposed mixed formulations perform well for the 
benchmark tests and that various modifications to 
the Darcy model have a significant impact on the 
results. This section focuses on relevant enhanced 
oil recovery applications, which are more complex 
by nature. Pressure contours, flow rates, and 
errors in the local/element-wise mass balance. 
Q9 elements shall be used for each problem.
\subsection{Oil reservoir problem}
For high pressure applications like enhanced oil 
recovery, one is interested in the quantitative 
and qualitative nature of the pressure contours 
and velocities within the oil reservoir. The 
pictorial description of a typical oil reservoir 
is depicted in Figure \ref{Fig:reservoir}. Injection 
wells are located on either side the production well, 
and carbon-dioxide is pumped into the reservoir to 
ease the extraction of raw oil through the production 
well. The parameters used for this study are listed 
in Table \ref{Tab:oil_reservoir}. All Darcy models 
and finite element formulations are expected to yield 
differing flow patterns, but the general qualitative 
velocity vector can be depicted in Figure 
\ref{Fig:oil_reservoir_quiver}. As the oil fluid 
nears the production well, the Darcy velocities 
increases.
\begin{table}[h!]
  \centering
  \caption{User-defined inputs for the oil reservoir problem}
  \begin{tabular}{cc}
    \hline
    Parameter & Value \\ \hline
    $\bar{\beta}_{\mathrm{B}}$ & 0.005\\
    $\bar{\beta}_{\mathrm{F}}$ & 0.01\\	
    $\bar{k}$ & 1\\
    $\bar{\mu}_0$ & 1\\
    $\bar{\rho}$ & 1 \\
    $\vartheta$ & 1 \\
    $\bar{\mathbf{b}}(\mathbf{x})$ & $\begin{Bmatrix}0;-1\\ \end{Bmatrix} $\\
    $Nele$ & 1600\\
    $\bar{p}_{\mathrm{enh}}$ & 1000\\
    \hline
  \end{tabular}
  \label{Tab:oil_reservoir}
\end{table}
Pressure contours within the oil reservoirs are important to know because high pressures can result 
in cracking of the solid. Figures \ref{Fig:Oil_reservoir_pressure_LS} and 
\ref{Fig:Oil_reservoir_pressure_VMS} contain the pressure contours using the LS and VMS formalisms respectively.

It can be seen from each model that the pressure contours within the reservoirs vary both qualitatively 
and quantitatively. For the Barus model, there are steep pressure gradients near the 
injection well, and the pressures within the reservoir are generally smaller than that of the Darcy 
model. However, the Darcy-Forchheimer models exhibits steep pressure gradients near the production 
well, thus predicting higher pressures throughout the reservoir. While pressure dependent 
viscosity may yield favorable pressure contours, one has to account for increases in 
pressure due to inertial effects, so combining the Barus and Forchheimer models should 
yield the most accurate results. Figure \ref{Fig:oil_reservoir_surface} depicts the
pressure profiles of all models and formalisms at the top most interface of the reservoir.

It should be noted that there are some minor differences in the pressure profiles 
between the LS and VMS formalisms. While both formalisms have strongly prescribed velocity 
boundary conditions, the VMS boundary condition for pressures are weakly prescribed and 
consequently exhibit some oscillations. The oscillations diminish with mesh refinement, 
but one must recognize the potential ramifications oscillatory boundary conditions 
may have on the solutions, especially for more complex prescribed pressures.

In reservoir simulations, another quantity of interest is the outflow of raw oil. The flow 
rate or total flux at the production well is calculated using
\begin{align}
	&\int_{\Gamma^{p}}\mathbf{\bar{v}}\cdot\mathbf{\hat{n}}\; \mathrm{d}\Gamma,
\end{align}
where $\Gamma^{p}$ corresponds with the prescribed atmospheric pressure boundary. 
In Figure \ref{Fig:oil_reservoir_flowrate}, a comparison of flow rates versus prescribed
pressures is shown for both formalisms. The original Darcy models predict a linear relationship between 
prescribed pressures and flow rates but the non-linear Darcy models exhibit ceiling 
fluxes. As the pressure increases, the original Darcy models becomes increasingly 
unreliable as it over predicts the amount of oil production one can expect.  It is 
interesting to note that for both the Darcy and Barus models, the LS formalism predicts 
higher flows for a fixed injection pressure whereas the VMS formalism predicts higher Forchheimer 
flow rates. Nevertheless, the ceiling fluxes for the Barus and Forchheimer models differ for various 
betas, but combining the two models will always yield smaller flow rates.

As stated in Section \ref{Ch:Intro}, neither the LS nor VMS formalisms have local mass conservation. 
The ratios of local mass balance errors over the total predicted flux for the modified Darcy-Forchheimer 
Barus models are shown in Figure \ref{Fig:Oil_reservoir_mass_error_LS} and 
\ref{Fig:Oil_reservoir_mass_error_VMS}. When one encounters high velocity 
contours, one can also expect higher local mass balancing errors. The calculations 
show that all models exhibit the greatest errors near the production wells. It is 
interesting to note that while both formalisms predict roughly the same velocity flow rates, 
the VMS formalism shows greater local mass balancing error. Ratios of 0.25-0.35 are 
considered quite large, but for lower pressure and velocity applications, 
the ratios should be much smaller.

\subsection{Multilayer reservoir problem}
One may not always encounter constant permeability within the subsurface. Some layers 
within the oil reservoir may consist of coarse sands while others may consist of less permeable 
material. This numerical experiment shall study the effect varying permeability regions
has on the pressure contours, flow rates, and local mass balance errors. Consider the domain 
depicted in Figure \ref{Fig:layers} with the same boundary conditions
as that in Figure \ref{Fig:reservoir}. Regions with higher permeability have larger velocities as 
depicted in Figure \ref{Fig:layered_reservoir_quiver}. The parameters used for this problem are 
listed in Table \ref{Tab:layered_reservoir}, and the pressure contours for LS and VMS formalisms are depicted in Figures 
\ref{Fig:Layered_reservoir_pressure_LS} and \ref{Fig:Layered_reservoir_pressure_VMS} respectively.
\begin{table}[b!]
  \centering
  \caption{User-defined inputs for the layered reservoir problem}
  \begin{tabular}{cc}
    \hline
    Parameter & Value \\ \hline
    $\bar{\beta}_{\mathrm{B}}$ & 0.005\\
    $\bar{\beta}_{\mathrm{F}}$ & 0.01\\	
    $\bar{k}$ & varies\\
    $\bar{\mu}_0$ & 1\\
    $\bar{\rho}$ & 1 \\
    $\vartheta$ & 1 \\
    $\bar{\mathbf{b}}(\mathbf{x})$ & $\begin{Bmatrix}0;-1\\ \end{Bmatrix} $\\
    $Nele$ & 3200\\
    $\bar{p}_{\mathrm{enh}}$ & 1000\\
    \hline
  \end{tabular}
  \label{Tab:layered_reservoir}
\end{table}

Results show that the layers with higher permeability contain higher pressures and that 
steep gradients occur at the interfaces between the layers. The LS formalism predicts higher pressures 
and larger flow rates for the Darcy and modified Barus models as seen from Table 
\ref{Tab:layered_reservoir_flowrates}. Like with the previous oil reservoir problem, 
the VMS formalism predicts higher flow rates for the Darcy-Forchheimer model. 
The ratio of local mass balance errors and total predicted fluxes are depicted in 
Figures \ref{Fig:Layered_reservoir_mass_error_LS} and \ref{Fig:Layered_reservoir_mass_error_VMS}. 
While the VMS formalisms still have slightly higher errors, the overall error ratios for this problem 
are smaller despite having larger flow rates.
\begin{table}[t!]
  \centering
  \caption{Layered reservoir problem: flow rates for LS and VMS formalism at $\bar{p}_{\mathrm{enh}}$ = 1000}
  \begin{tabular}{l|cccc}
    \hline
    Darcy models: & D & MB & F & MBF \\ \hline
    LS & 1038 & 210 & 133 & 75 \\
    VMS & 1025 & 204 & 137 & 77 \\
    \hline
  \end{tabular}
  \label{Tab:layered_reservoir_flowrates}
\end{table}
\subsection{Flow in a porous media with staggered impervious zones}
Consider flow through a region with staggered impervious zones in Figure 
\ref{Fig:staggered}. In any heterogeneous 
flow through porous media applications, one may encounter domains where oil must flow 
through a complex domain with many impervious regions. The qualitative velocity vector field in Figure 
\ref{Fig:staggered_reservoir_quiver} indicates that higher flows occur around the sharp 
bends.
\begin{table}[b!]
  \centering
  \caption{User-defined inputs for the staggered impervious zones problem}
  \begin{tabular}{cc}
    \hline
    Parameter & Value \\ \hline
    $\bar{\beta}_{\mathrm{B}}$ & 0.005\\
    $\bar{\beta}_{\mathrm{F}}$ & 0.01\\	
    $\bar{k}$ & 1\\
    $\bar{\mu}_0$ & 1\\
    $\bar{\rho}$ & 1 \\
    $\vartheta$ & 1 \\
    $\bar{\mathbf{b}}$ & 0 \\
    Element type & Q9\\
    $Nele$ & 1696\\
    $\bar{p}_{\mathrm{enh}}$ & 1000\\
    \hline
  \end{tabular}
  \label{Tab:staggered_zones}
\end{table}
The pressure contours are depicted in Figures \ref{Fig:Staggered_reservoir_pressure_LS} and
\ref{Fig:Staggered_reservoir_pressure_VMS}. The same non-dimensionalized injection pressure has been prescribed for 
this problem (see Table \ref{Tab:staggered_zones} for key parameters used in this 
problem), and it can still be seen that the different Darcy models make an impact on the 
qualitative nature of the pressure contours. Again, the LS formalism yields higher pressures throughout the domain 
and predicts larger flow rates as seen in Table \ref{Tab:staggered_reservoir_flowrates}. Errors in the local 
mass balance tend to be greatest in regions with high velocities (i.e., the sharp bends around the impervious layers). 
Local mass balancing errors are shown in Figures \ref{Fig:Staggered_reservoir_mass_error_LS} and 
\ref{Fig:Staggered_reservoir_mass_error_VMS}.
\begin{table}[t!]
  \centering
  \caption{Staggered impervious zones problem: flow rates for LS and VMS formalism at $\bar{p}_{\mathrm{enh}}$ = 1000}
  \begin{tabular}{l|cccc}
    \hline
    Darcy models: & D & MB & F & MBF \\ \hline
    LS & 150.6 & 31.9 & 50.4 & 24.3 \\
    VMS & 131.1 & 26.5 & 47.9 & 20.3 \\
    \hline
  \end{tabular}
  \label{Tab:staggered_reservoir_flowrates}
\end{table}

%% file: Sections/Ch6_Conclusions.tex
\section{CONCLUDING REMARKS}
\label{Ch:Conclusions}

The work in this thesis proposes a modification to the 
standard Darcy model that takes into account both the 
dependence of the viscosity on the pressure and the 
inertial effects, which have been observed in many 
physical experiments. The current models in the 
literature consider either of the effects but not 
both. The proposed model will be particularly important 
for predictive simulations of applications involving 
high pressures and high velocities (e.g., enhanced oil 
recovery). This modification has been referred to as the 
\emph{modified Darcy-Forchheimer model}. 
It has been shown numerically that the results obtained by 
taking into account the dependence of drag coefficient 
on the pressure and on the velocity are both qualitatively 
and quantitatively different from that the results obtained 
using the standard Darcy model, Darcy-Forchheimer equation 
(which neglects the dependence of drag coefficient and 
viscosity on the pressure) or modified Darcy model 
\cite{Nakshatrala_Rajagopal_IJNMF_2011_v67_p342,
Nakshatrala_Turner_2013_arXiv} (which neglects the 
dependence of the drag coefficient on the velocity). 

This thesis has also developed stable mixed finite element formulations 
for the resulting governing equations using two different 
approaches: VMS formalism and LS formalism. Using numerical experiments, 
we have compared their merits and demerits.

The LS formulation has more terms to evaluate than the VMS formulation, 
and hence the LS formulation is slightly more computationally expensive 
than the VMS formulation. However, it should be emphasized that this is 
not significant in a parallel setting as element-level calculations are 
embarrassingly parallel. It is also observed that the LS formulation 
with $p$-refinement produces accurate results. Another point that is 
worth mentioning is that the VMS formalism weakly prescribes pressure 
boundary conditions, and it has been shown that minor oscillations occur 
when meshes are not adequately refined. The error in element-wise / local 
mass balance for various Darcy-type models is also quantified, and the error 
becomes significant when there are large pressures and velocities. 

There are several ways one can extend the research work 
presented in this thesis. \emph{On the modeling front}, 
a good but difficult research problem is to develop 
mathematical models that couple deformation and damage 
of the porous solid with the flow aspects and reactive 
transport across several spatial and temporal scales. 
The following are some possible future works 
\emph{on the numerical front}: 
\begin{enumerate}[(a)]
\item Develop mixed finite element formulations with 
  better local mass balance property under equal-order 
  interpolation for the pressure and the velocity.
\item Develop multi-scale models by coupling continuum / 
  macro-scale flow models with meso-scale models (e.g., 
  lattice Boltzmann models). The advantage is that 
  meso-scale models can easily handle complex pore 
  structure, which may not computationally feasible 
  if one uses only a macro-scale model.  
\item Another important but difficult problem is 
  to develop numerical upscaling techniques for 
  heterogeneous porous media. In layman terms, 
  numerical upscaling captures fine-scale features 
  on coarse computational grids. 
\item Develop stable and accurate coupling algorithms 
  for coupling flow, deformation and transport aspects. 
\end{enumerate}  
\emph{On the computer implementation front}, a possible 
work is to implement the mixed formulations taking the 
advantage of GPU processors, and implementing on 
heterogeneous parallel computing environment.

%% file: Sections/App_Figures.tex
\clearpage
\begin{figure}
  \centering
  \includegraphics[scale=0.5]{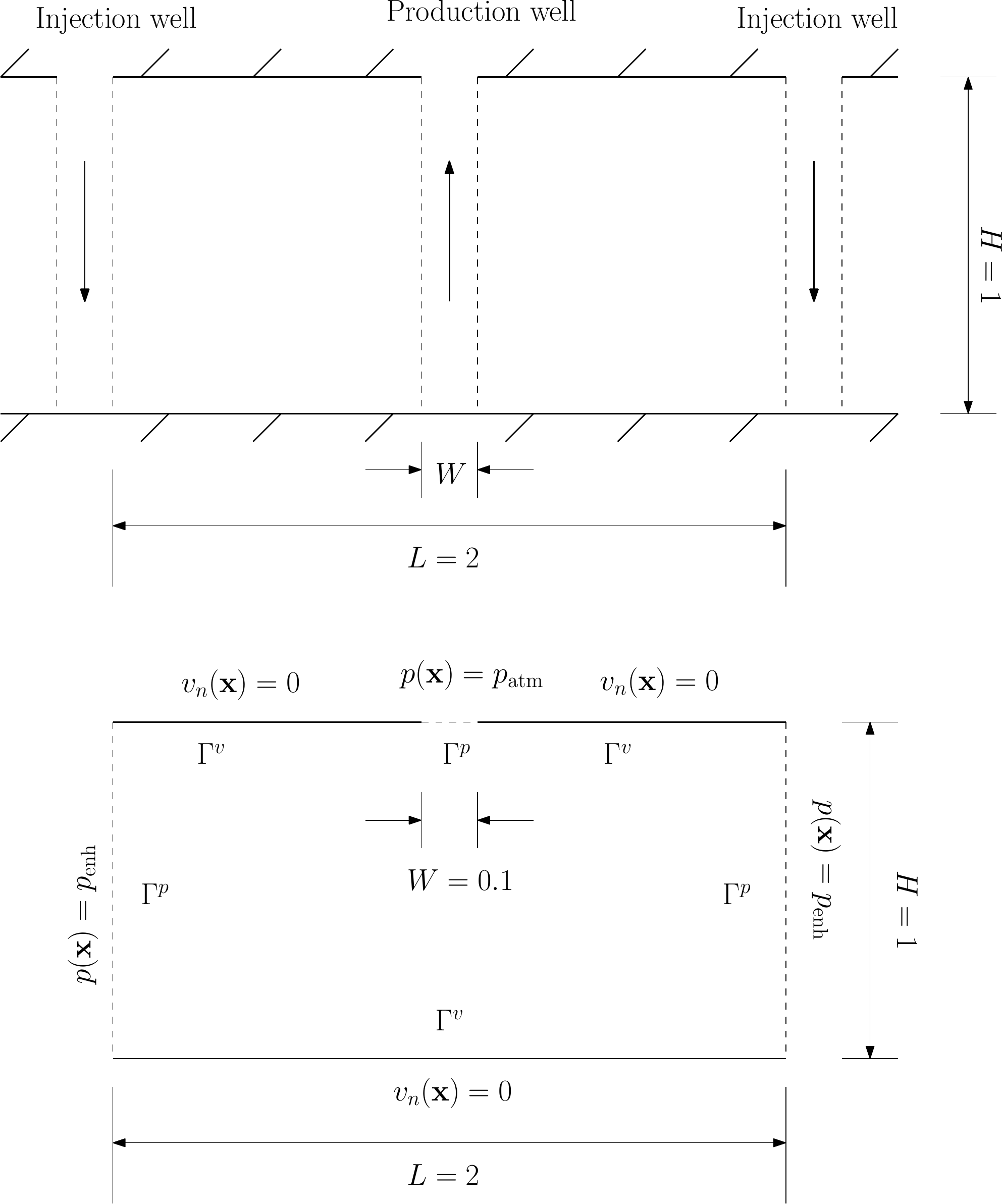}
  \caption{Oil reservoir problem. Top figure 
    is the pictorial description of enhanced 
    oil recovery, and the bottom figure is 
    the idealized computational domain with 
    appropriate boundary conditions. 
    \label{Fig:reservoir}}
\end{figure}
\begin{figure}
  \centering
  \includegraphics[scale=0.46]
  	{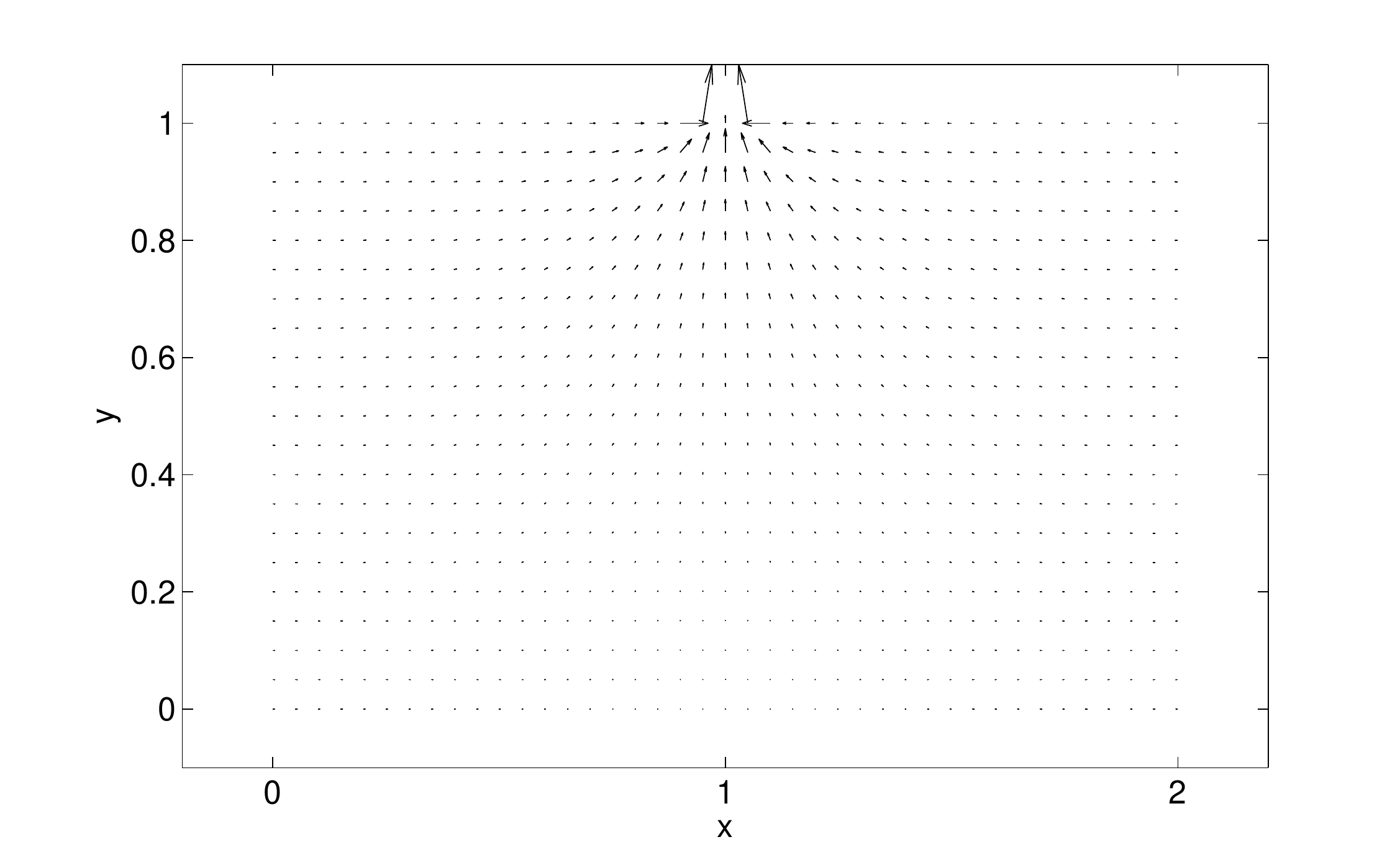}
  \caption{Oil reservoir problem: qualitative velocity vector field}
  \label{Fig:oil_reservoir_quiver}
\end{figure}
\begin{figure}
  \centering
  \subfigure[Darcy model]{
  	\includegraphics[scale=0.46]
    {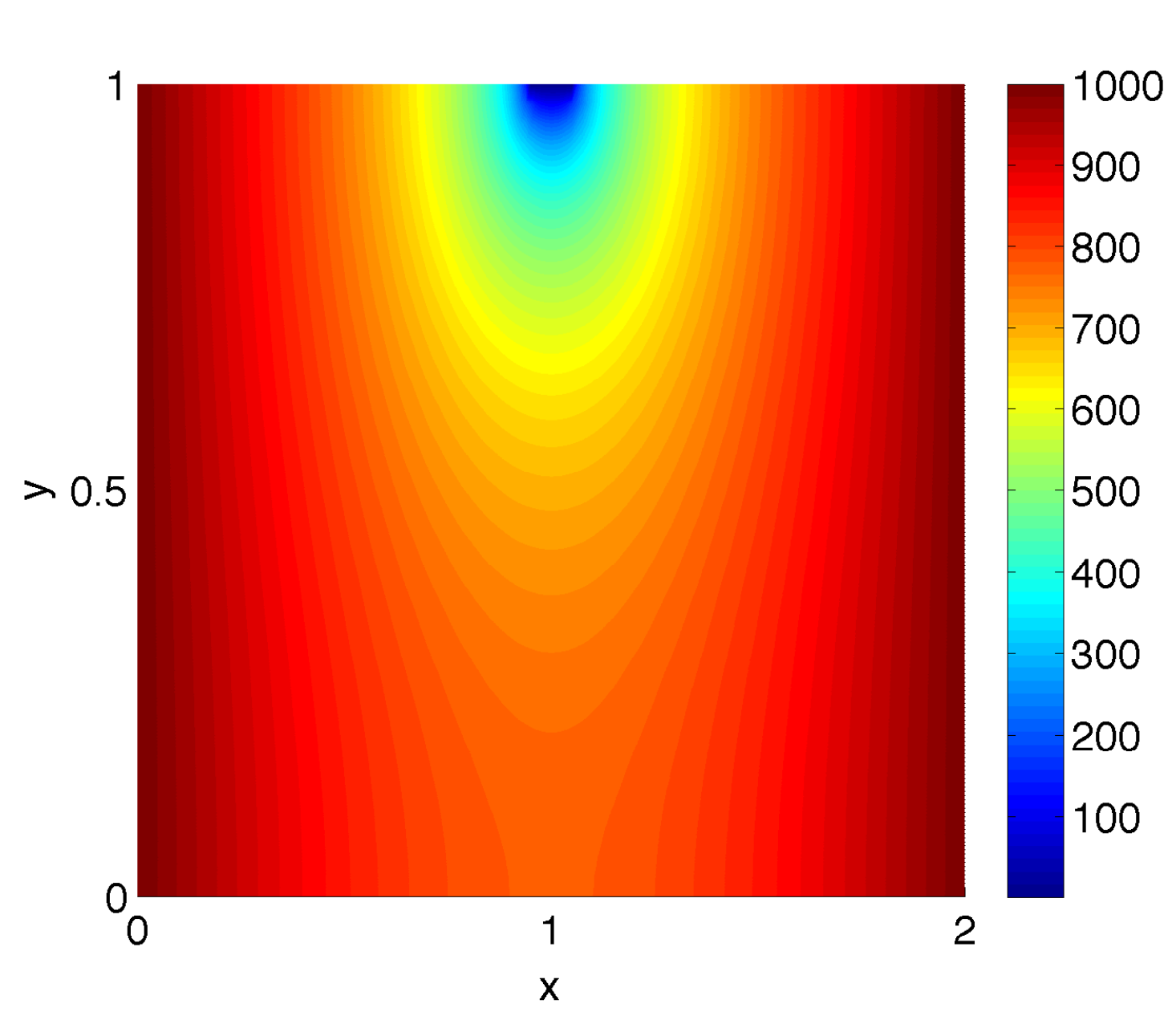}}
  \subfigure[Modified Barus]{
  	\includegraphics[scale=0.46]
    {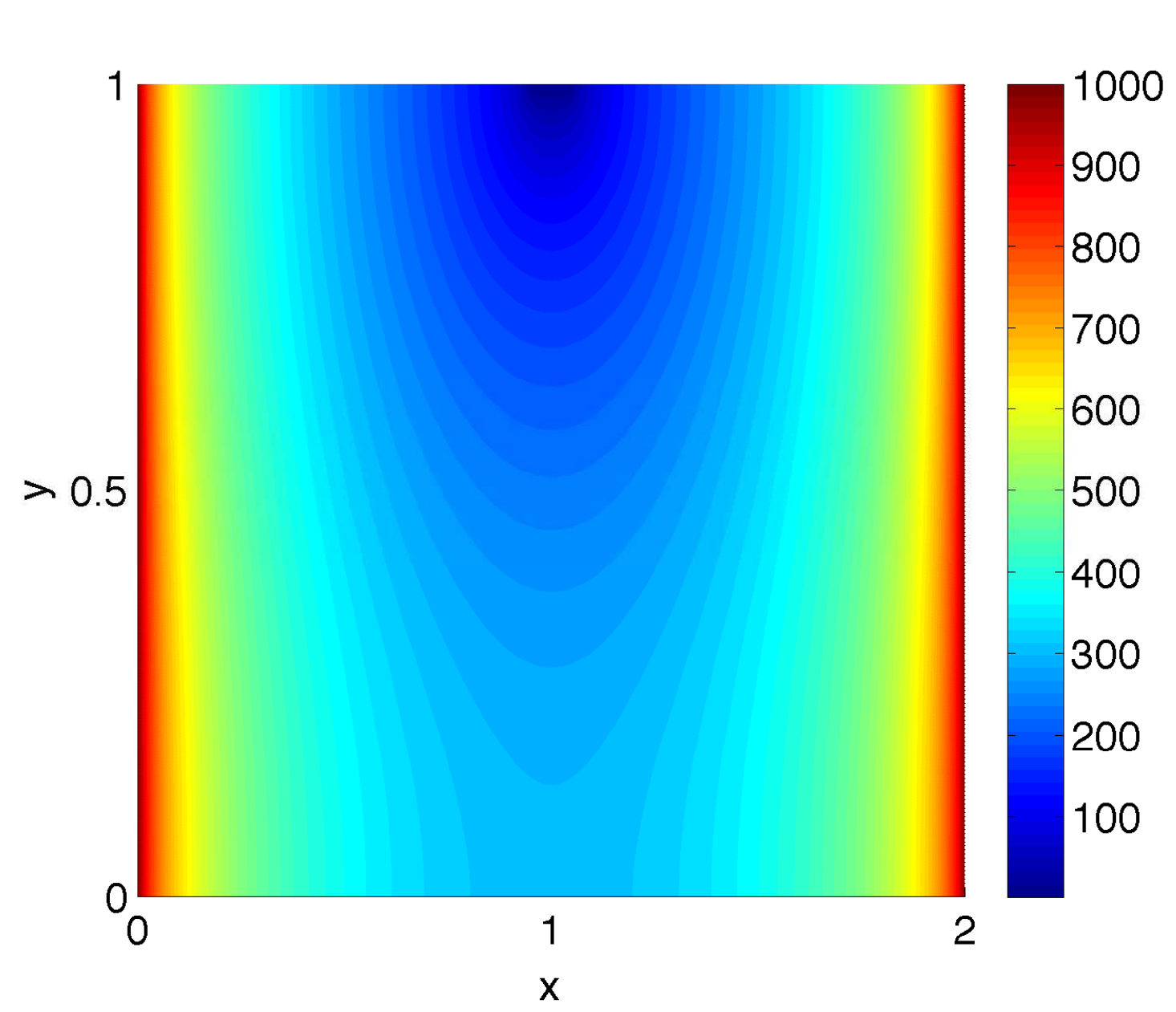}}
  \subfigure[Darcy-Forchheimer]{
  	\includegraphics[scale=0.46]
    {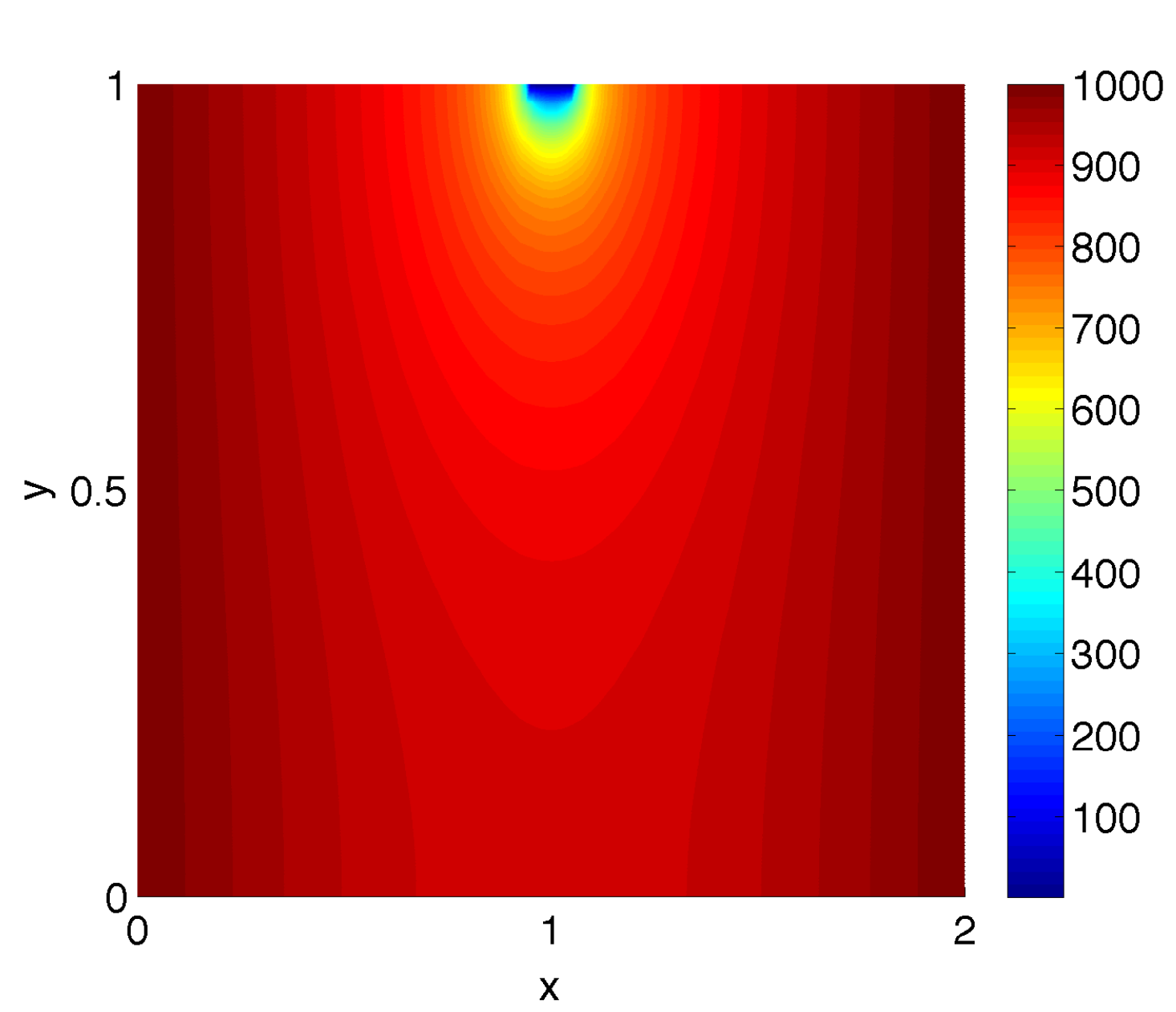}}
  \subfigure[Modified Darcy-Forchheimer Barus]{
  	\includegraphics[scale=0.46]
    {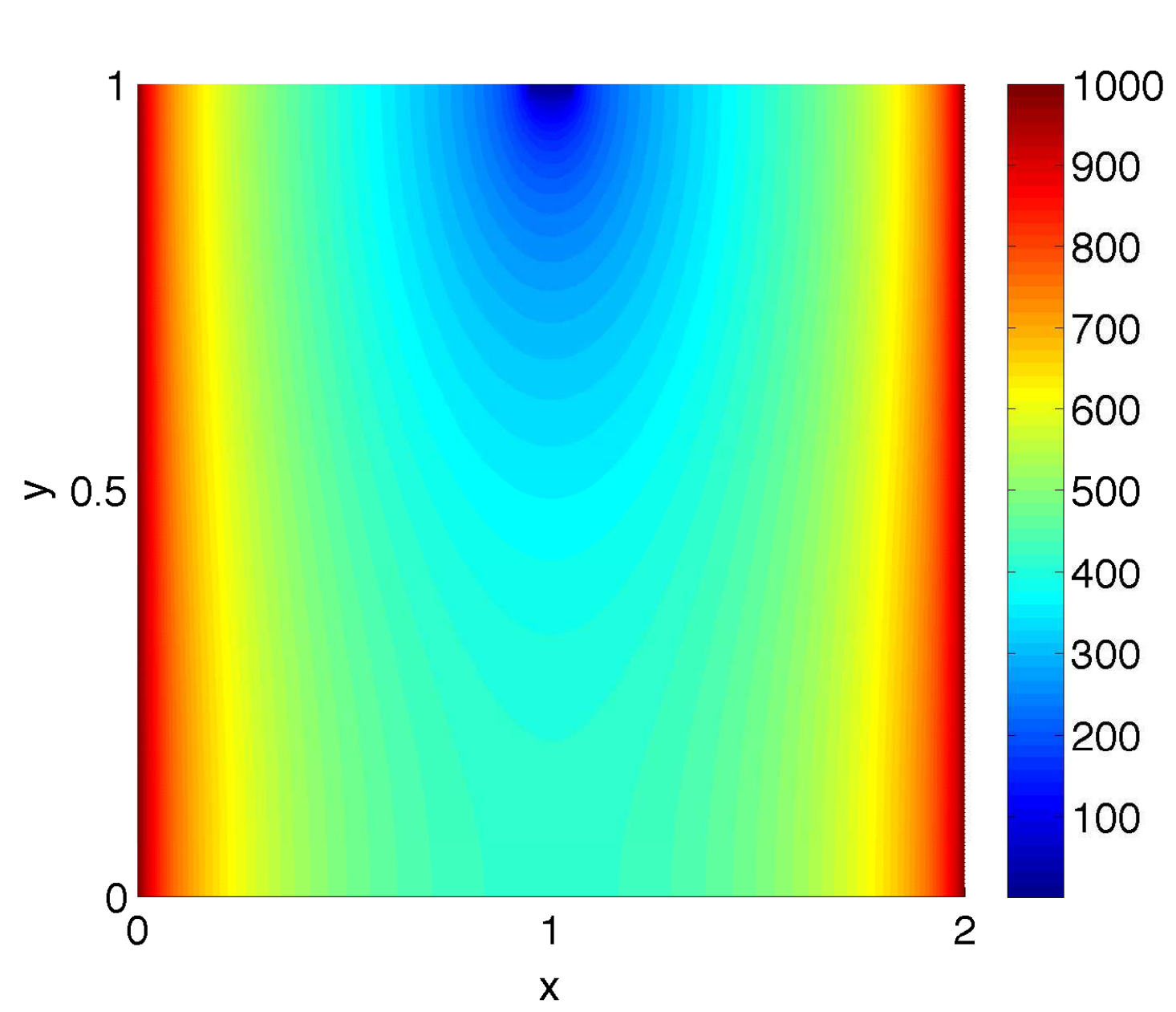}}
  \caption{Oil reservoir problem: pressure contours using LS formalism}
  \label{Fig:Oil_reservoir_pressure_LS}
\end{figure}
\begin{figure}
  \centering
  \subfigure[Darcy model]{
  	\includegraphics[scale=0.46]
    {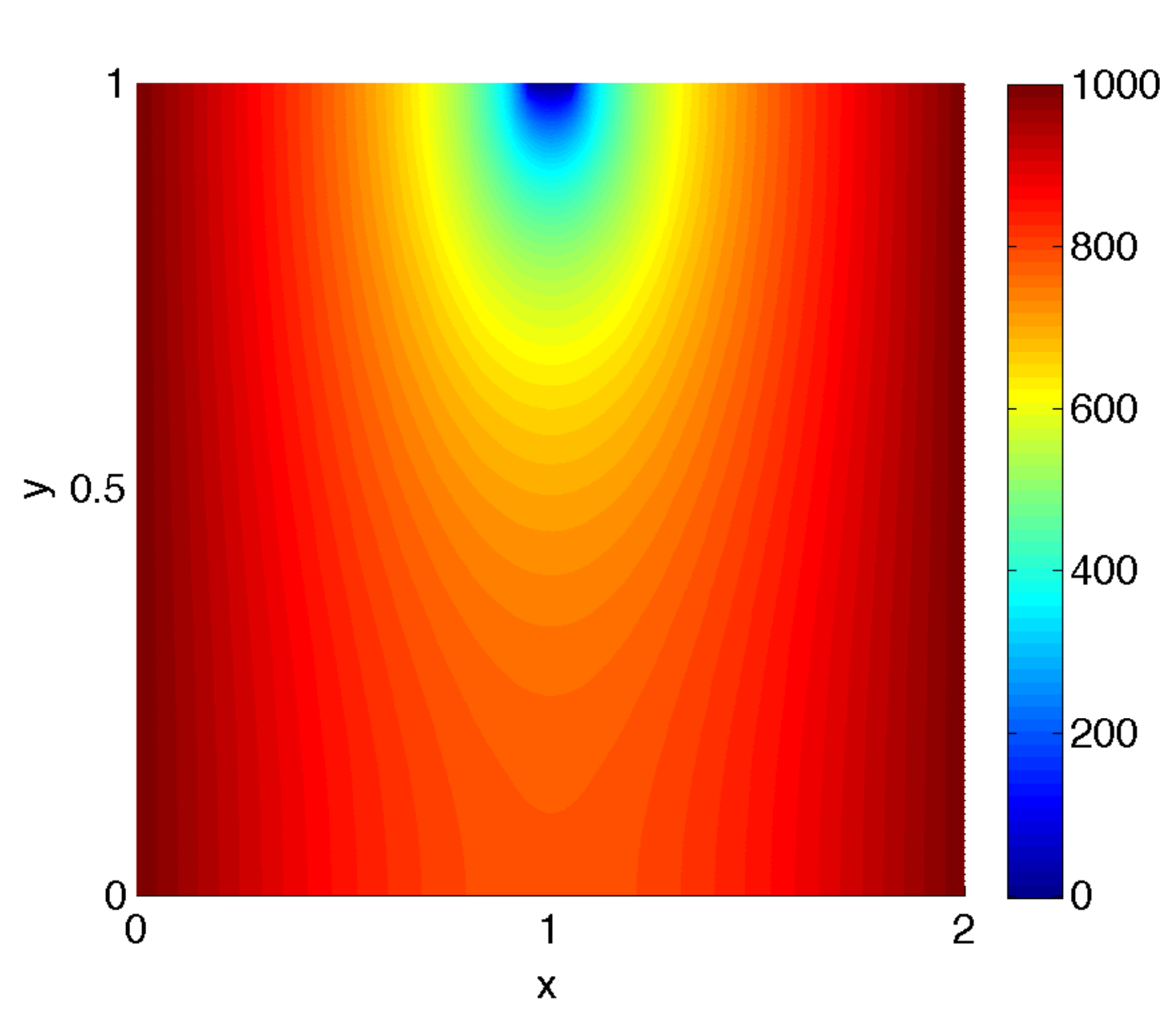}}
  \subfigure[Modified Barus]{
  	\includegraphics[scale=0.46]
    {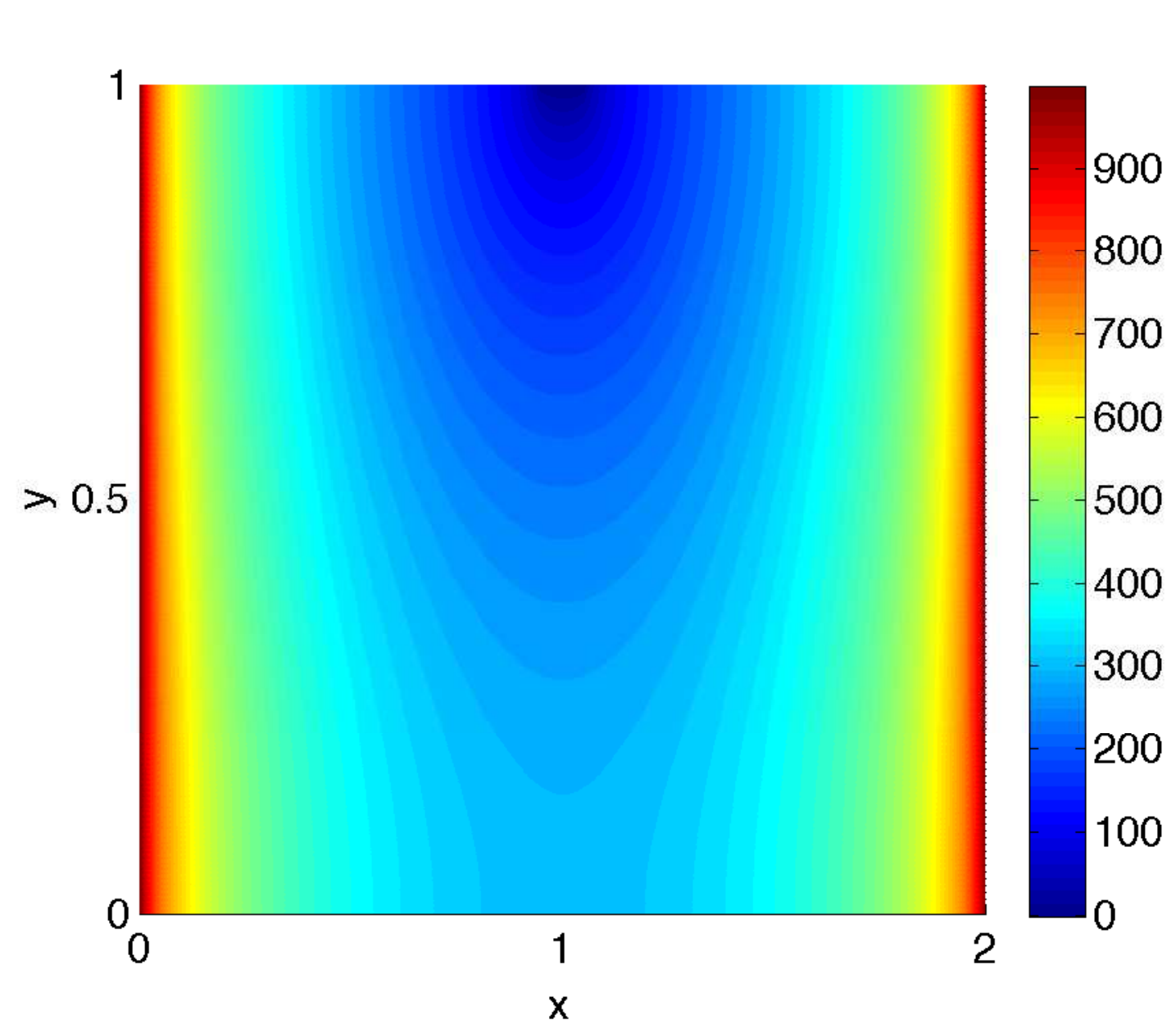}}
  \subfigure[Darcy-Forchheimer]{
  	\includegraphics[scale=0.46]
    {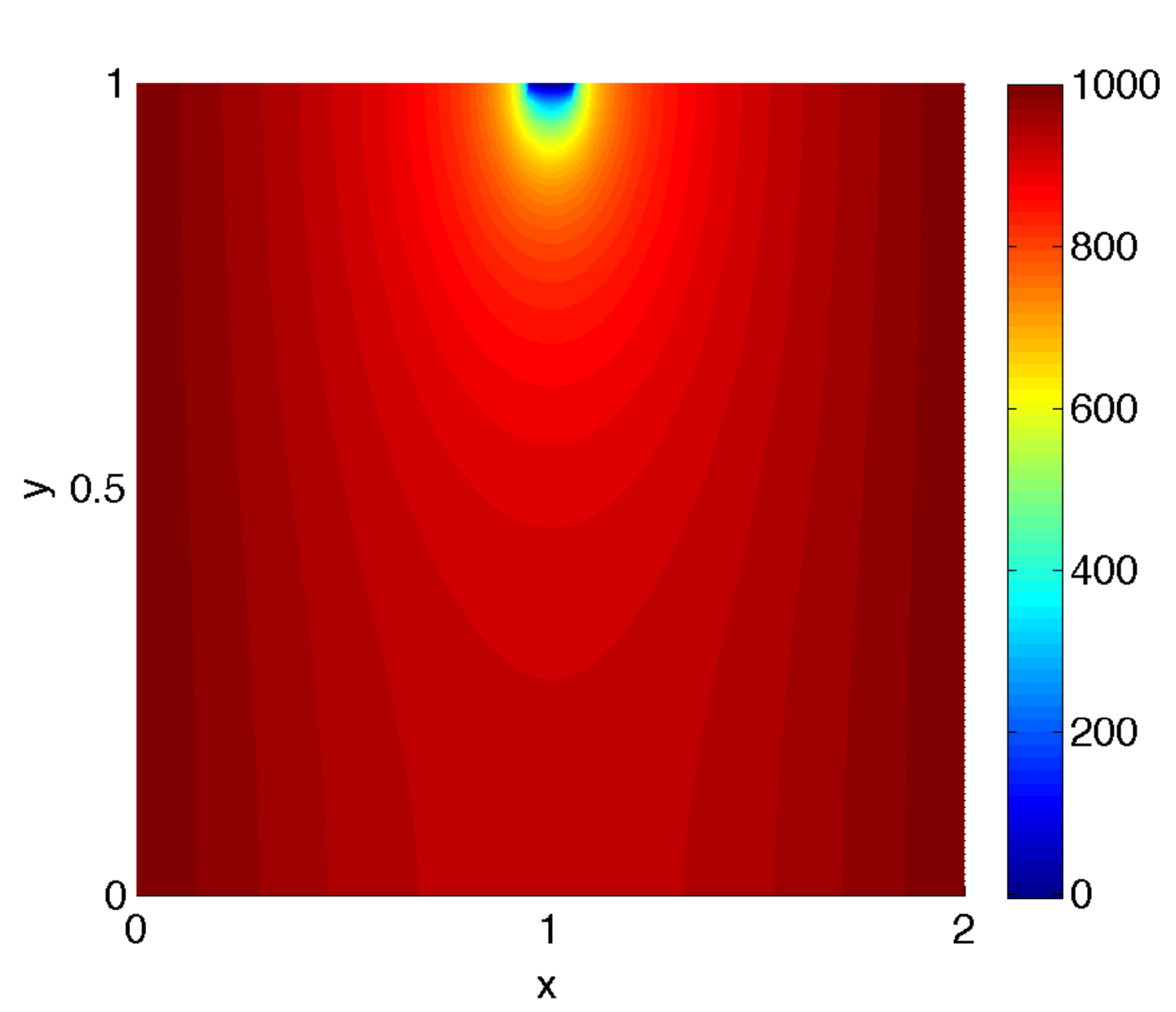}}
  \subfigure[Modified Darcy-Forchheimer Barus]{
  	\includegraphics[scale=0.46]
    {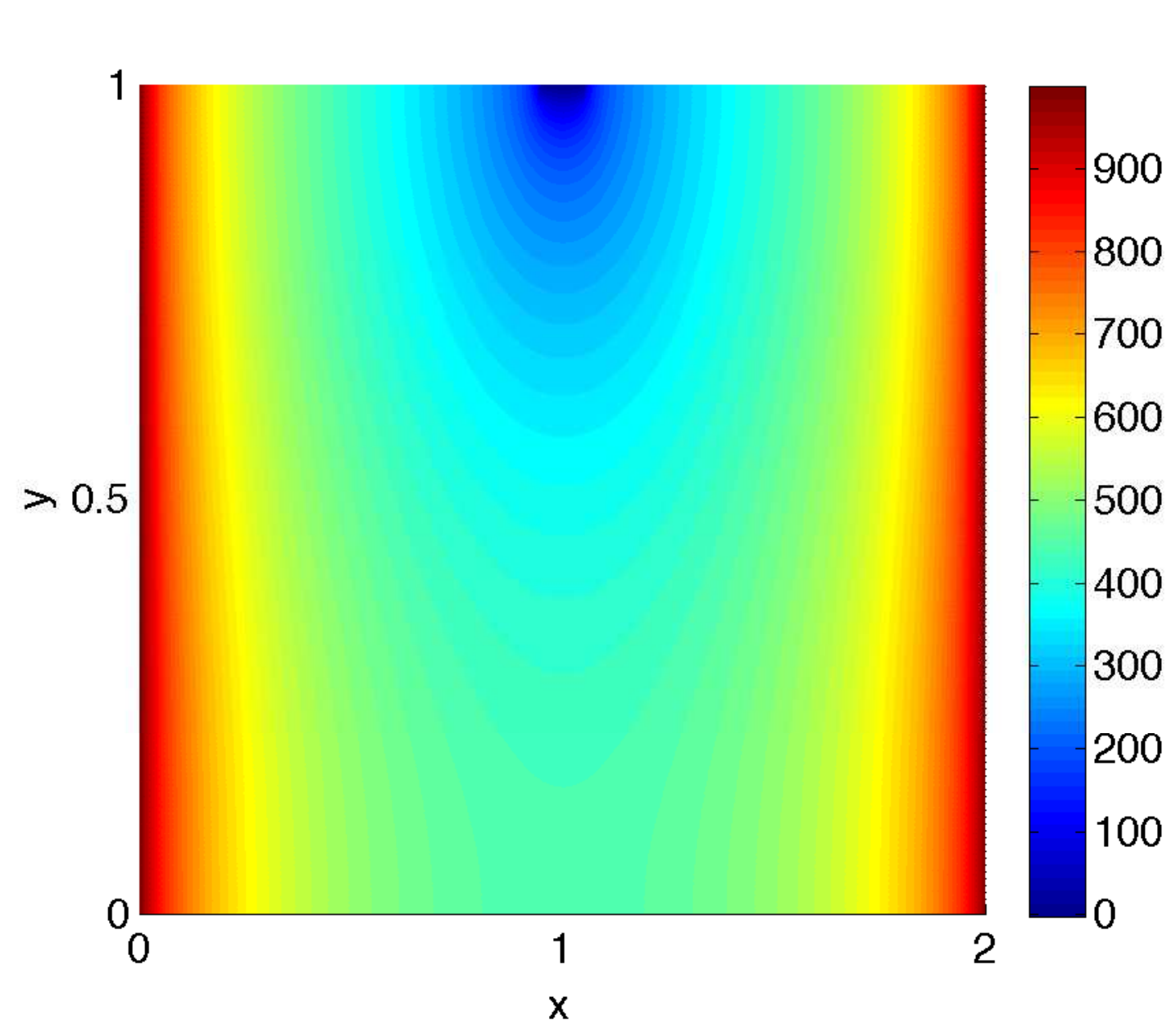}}
  \caption{Oil reservoir problem: pressure contours using VMS formalism}
  \label{Fig:Oil_reservoir_pressure_VMS}
\end{figure}
\begin{figure}
  \centering
   \subfigure[LS formalism]{
     	\includegraphics[scale=0.46]
  	{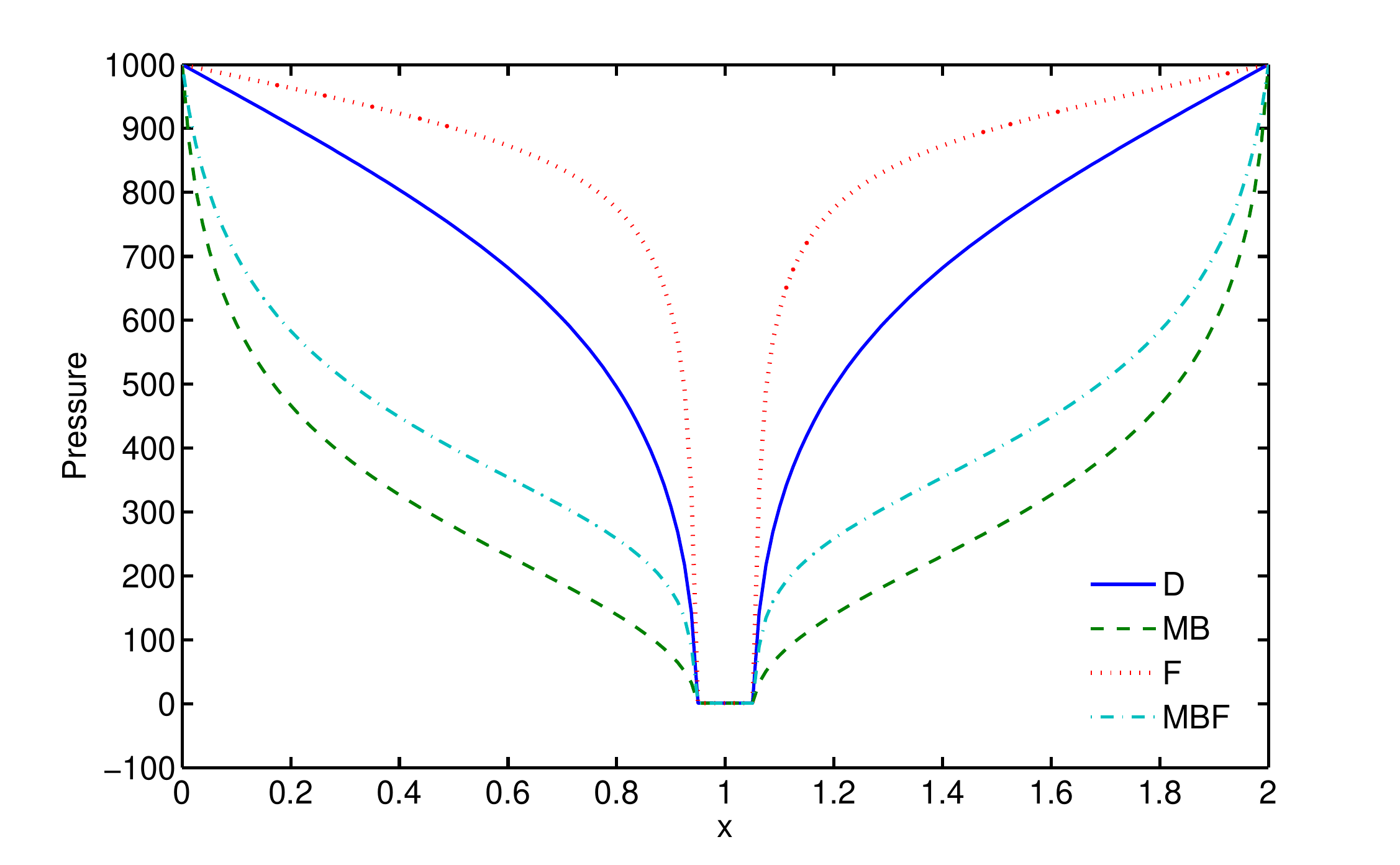}}
  \subfigure[VMS formalism]{
  	\includegraphics[scale=0.46]
  	{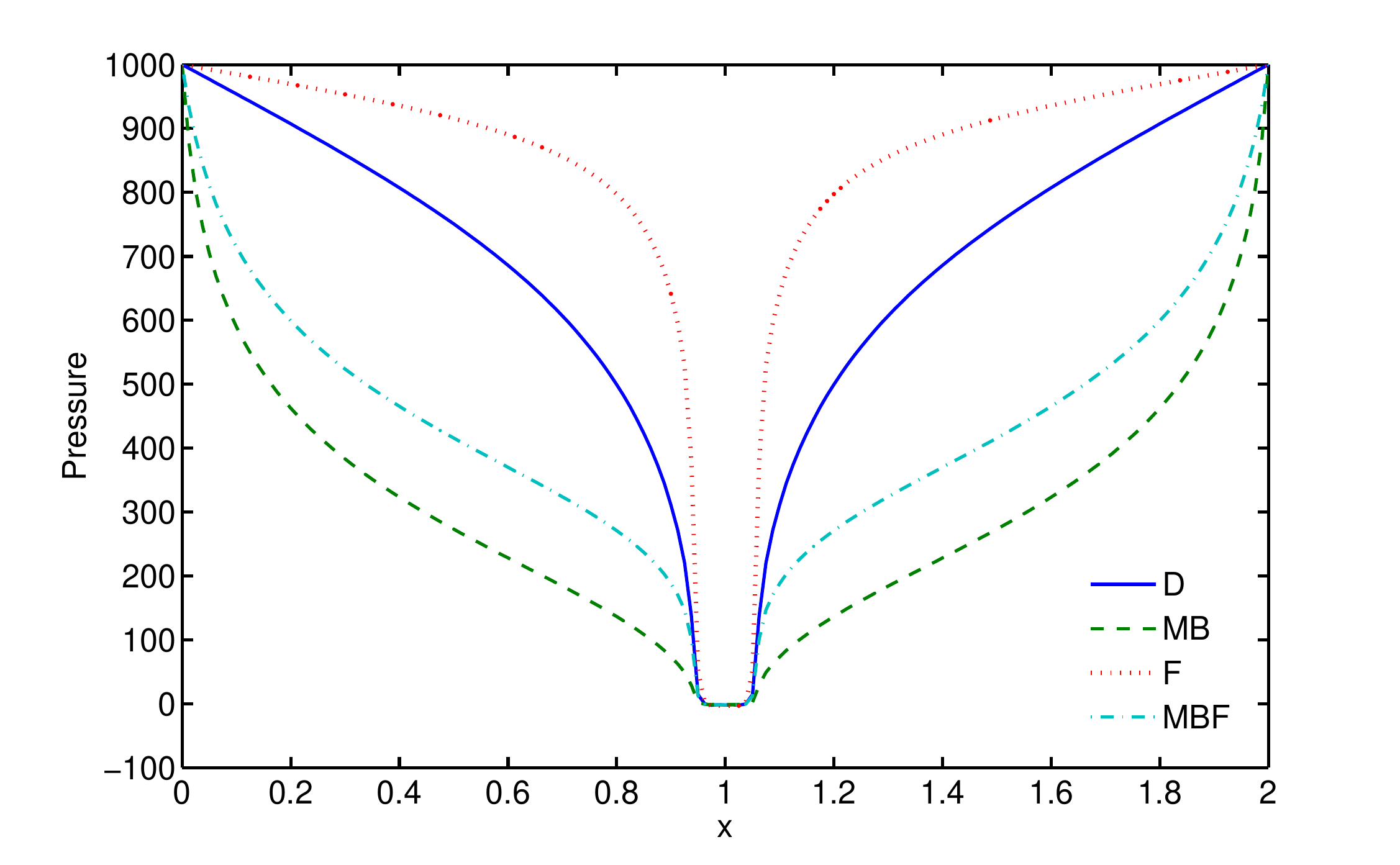}}
  \caption{Oil reservoir problem: comparison of pressure profiles at $y=1$}
  \label{Fig:oil_reservoir_surface}
\end{figure}
\begin{figure}
  \centering
   \subfigure[LS formalism]{
  	\includegraphics[scale=0.46]
  	{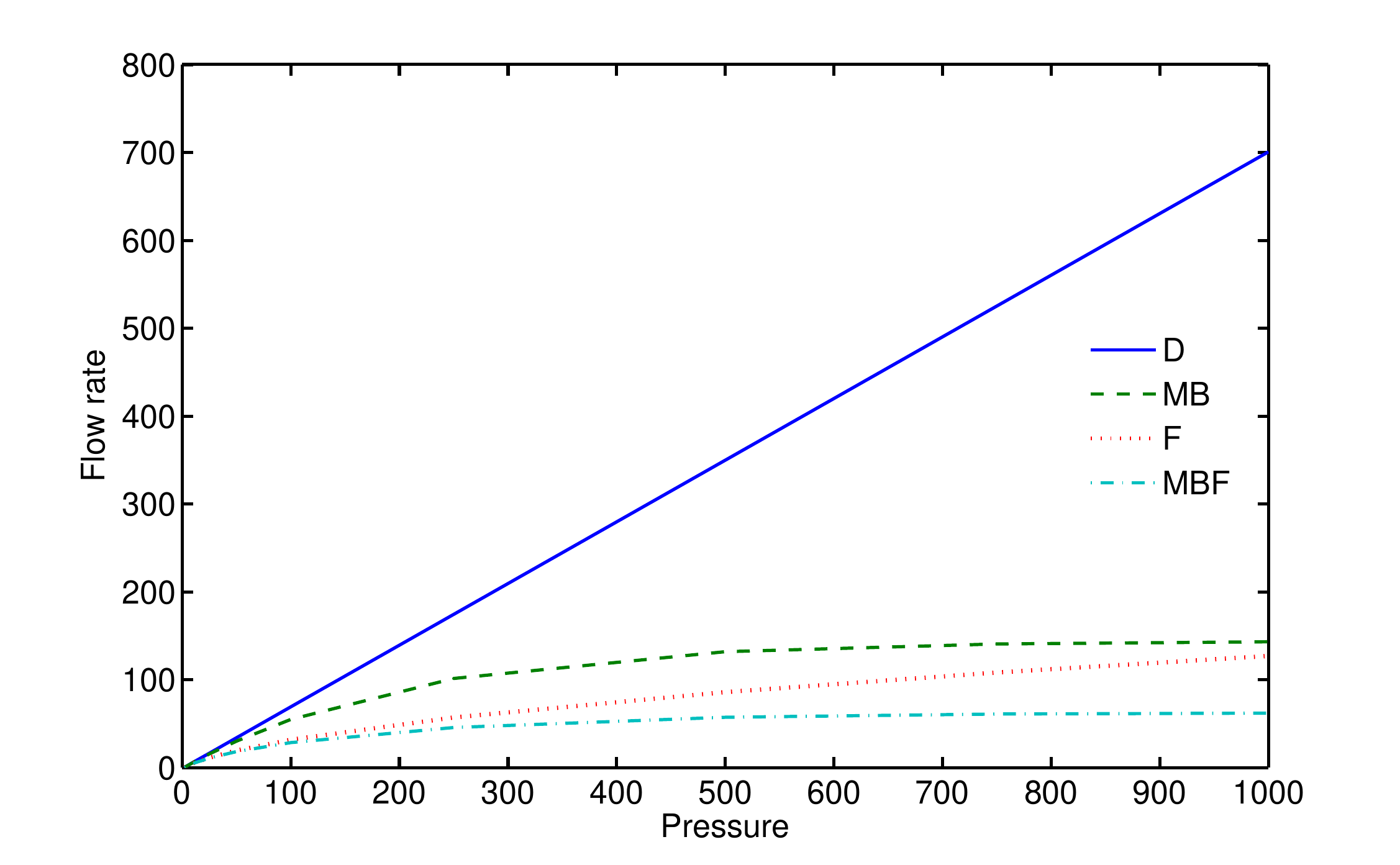}}
  \subfigure[VMS formalism]{
  	\includegraphics[scale=0.46]
  	{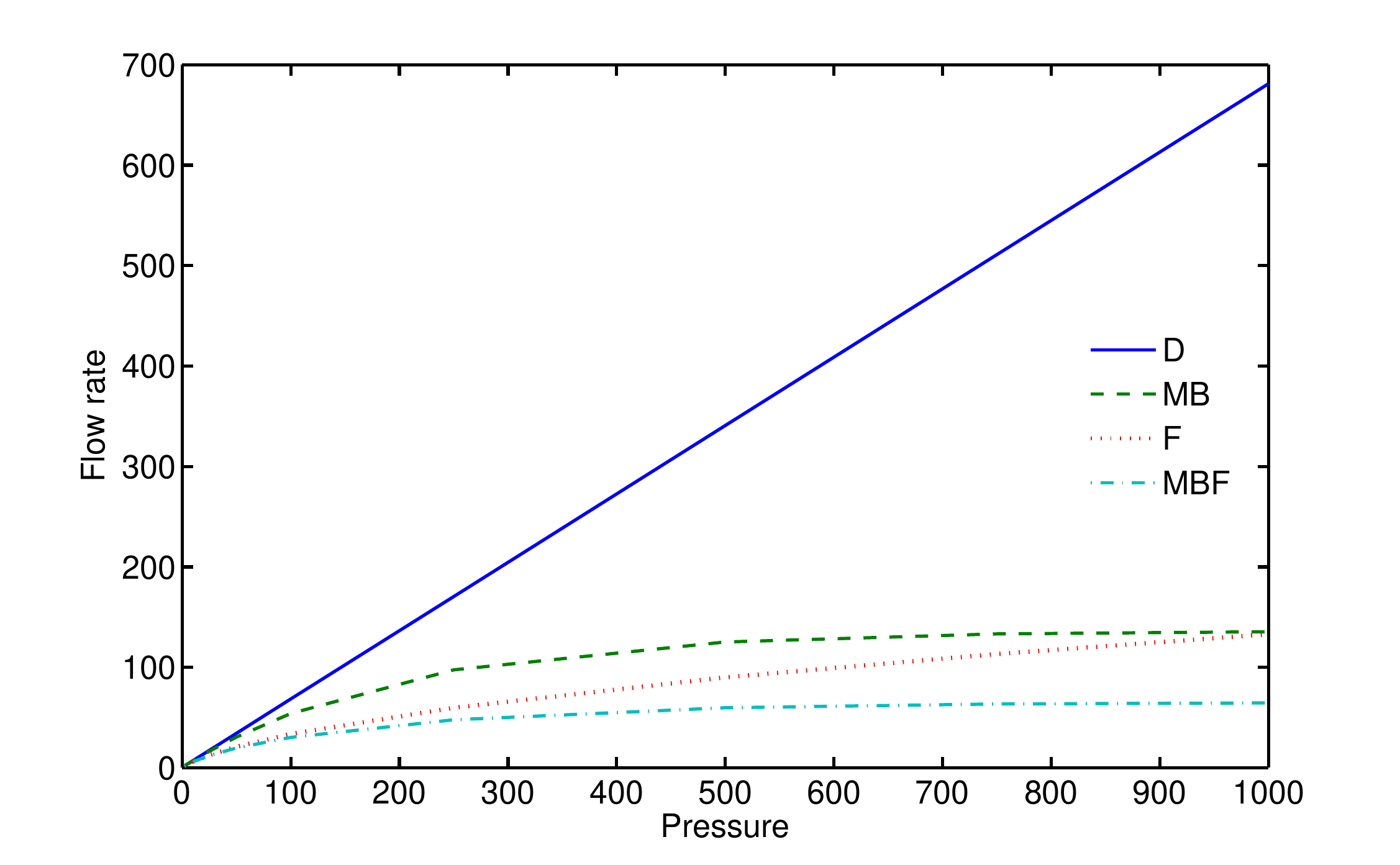}}
  \caption{Oil reservoir problem: comparison of injection pressures vs. flow rates}
  \label{Fig:oil_reservoir_flowrate}
\end{figure}
\begin{figure}
  \centering
  \subfigure[Darcy model]{
  	\includegraphics[scale=0.46]
    {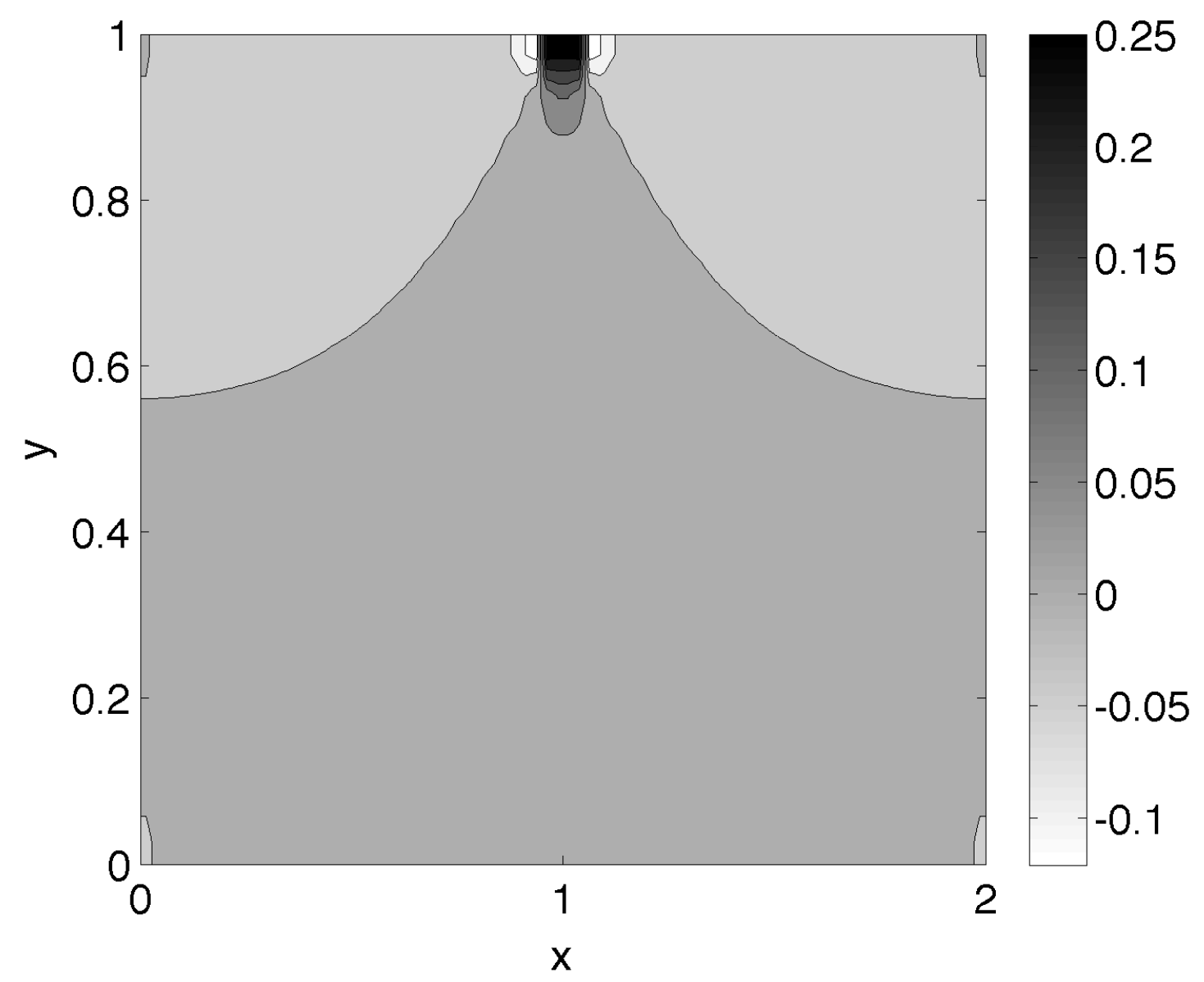}}
  \subfigure[Modified Barus]{
  	\includegraphics[scale=0.46]
    {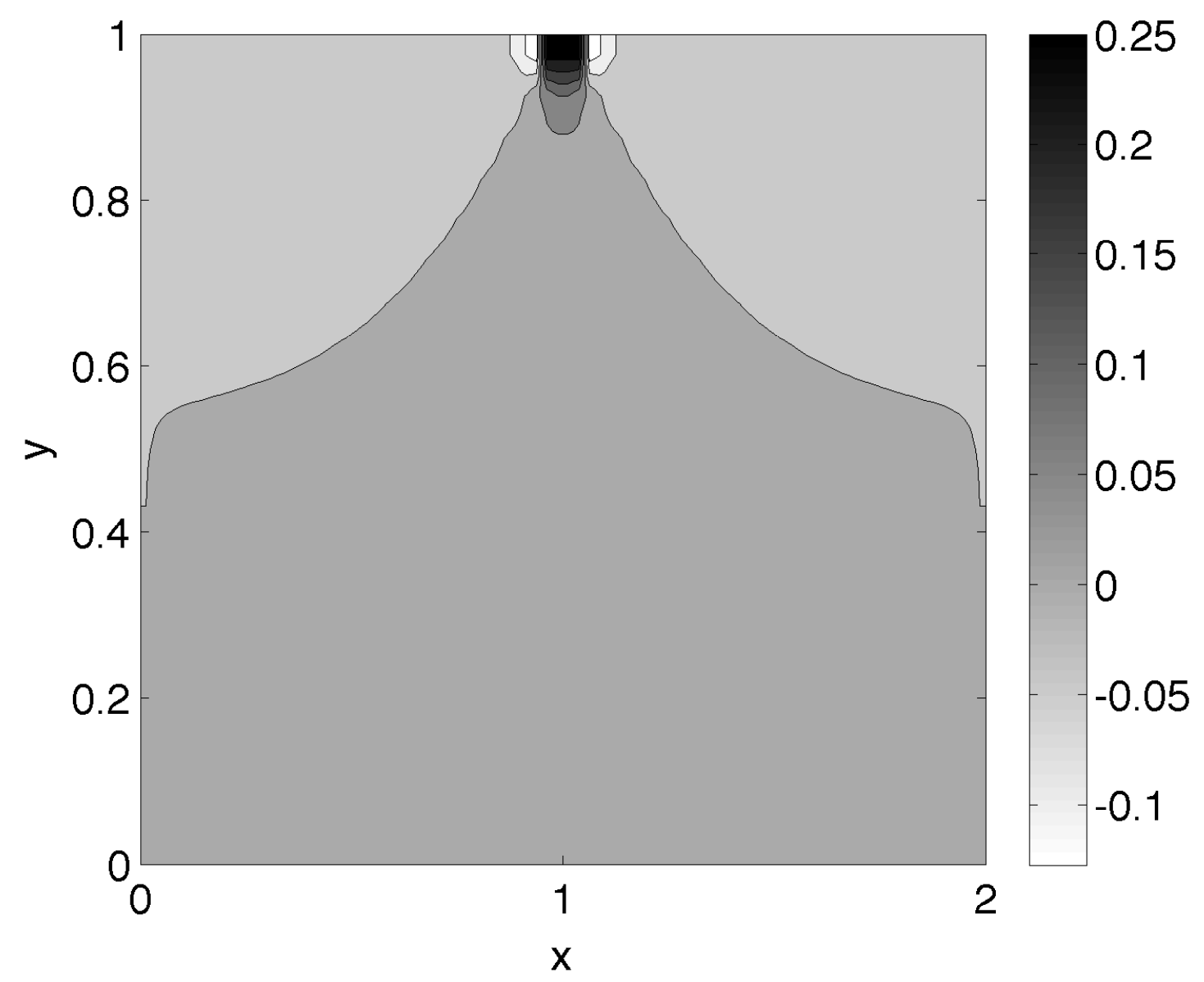}}
  \subfigure[Darcy-Forchheimer]{
  	\includegraphics[scale=0.46]
    {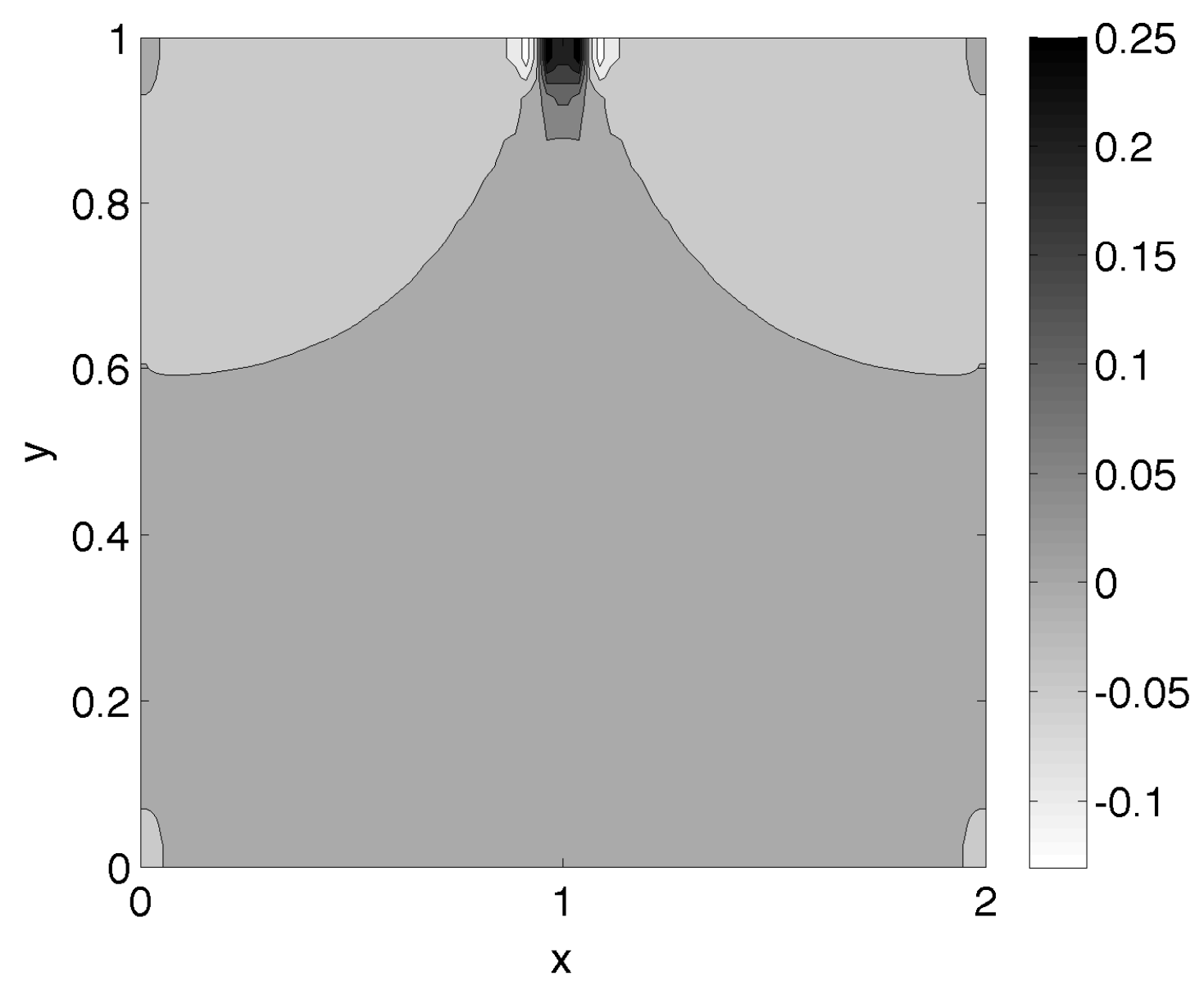}}
  \subfigure[Modified Darcy-Forchheimer Barus]{
  	\includegraphics[scale=0.46]
    {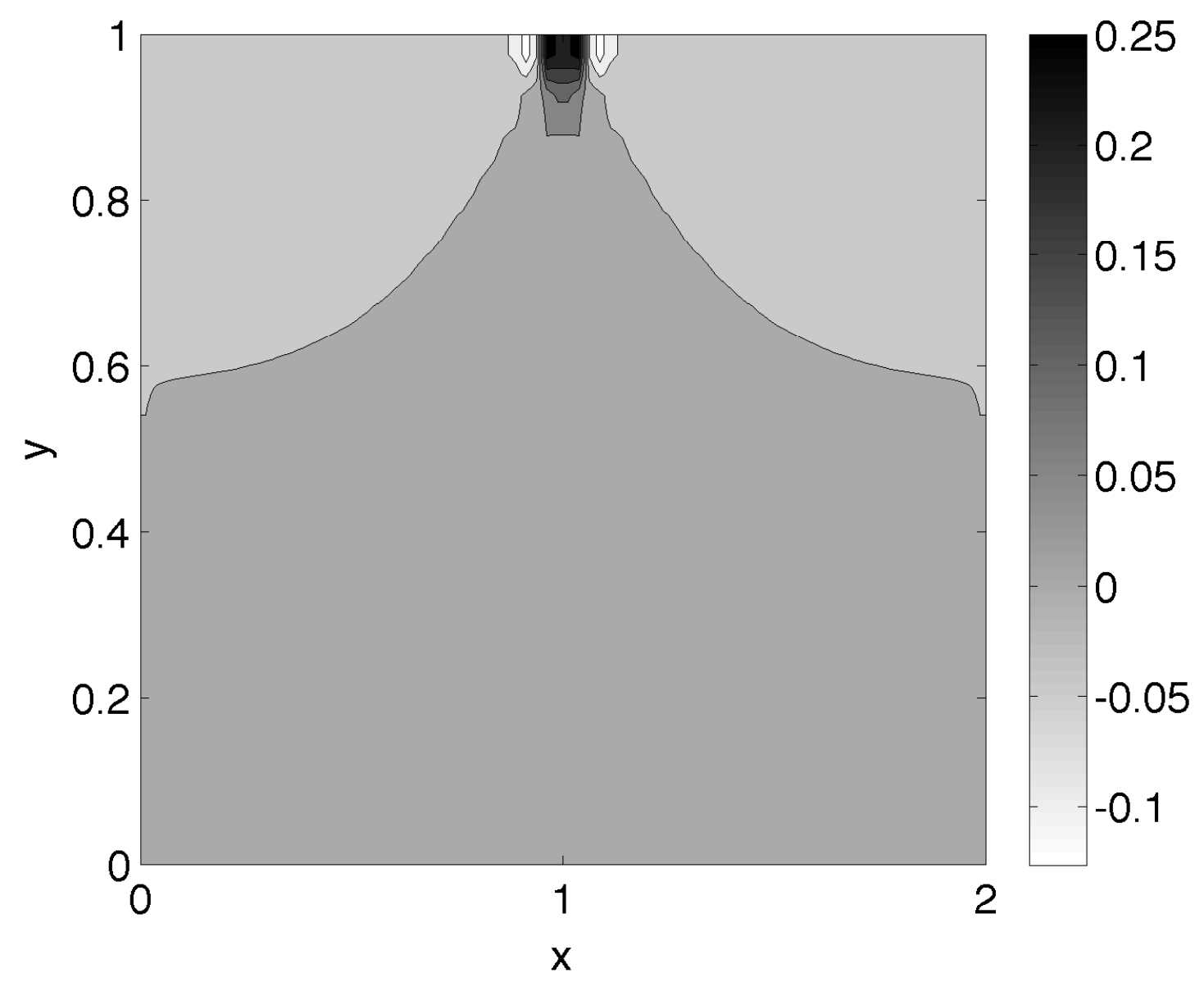}}
  \caption{Oil reservoir problem: ratios of local mass balance error over total predicted flux using LS formalism}
  \label{Fig:Oil_reservoir_mass_error_LS}
\end{figure}
\begin{figure}
  \centering
  \subfigure[Darcy model]{
  	\includegraphics[scale=0.46]
    {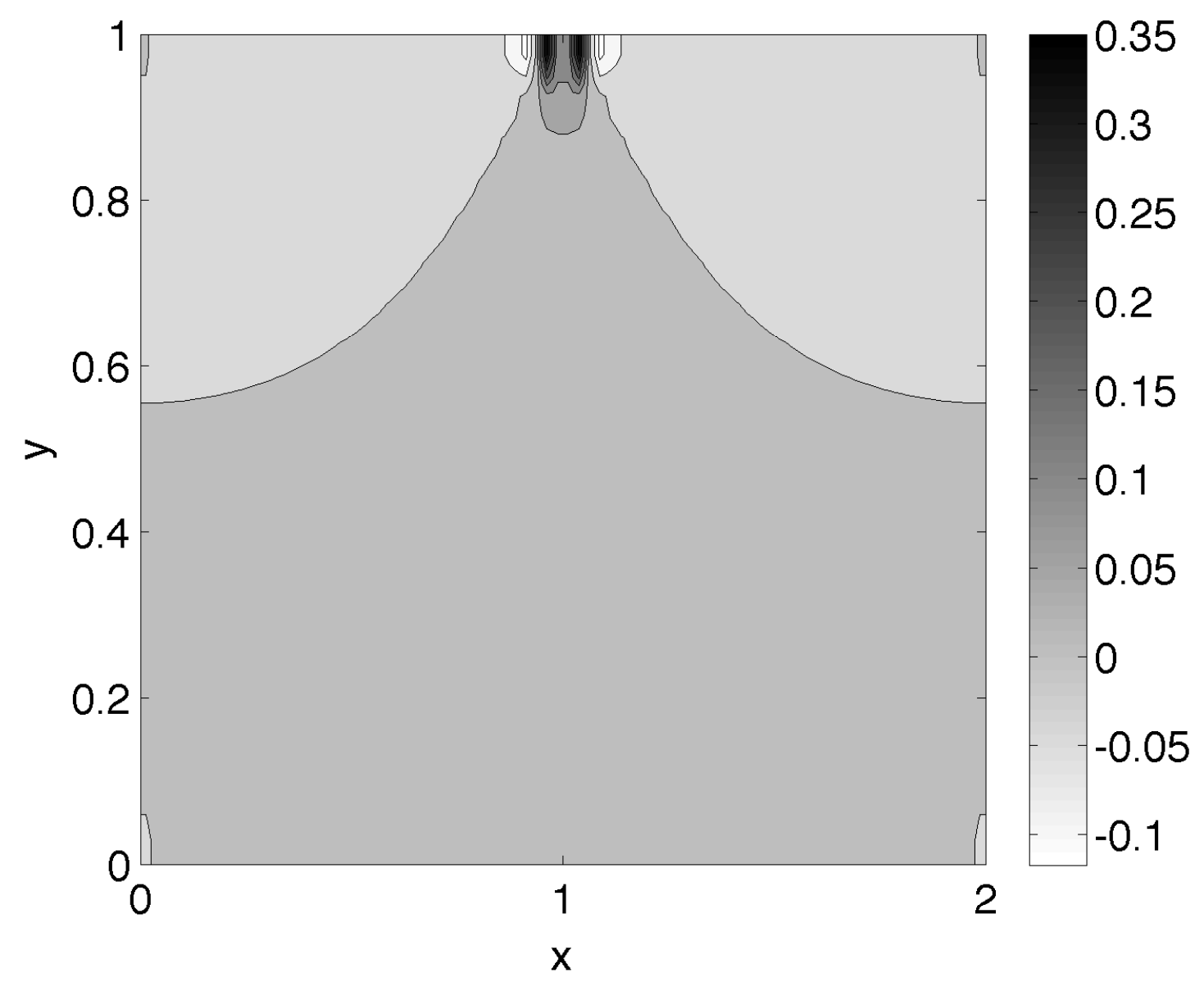}}
  \subfigure[Modified Barus]{
  	\includegraphics[scale=0.46]
    {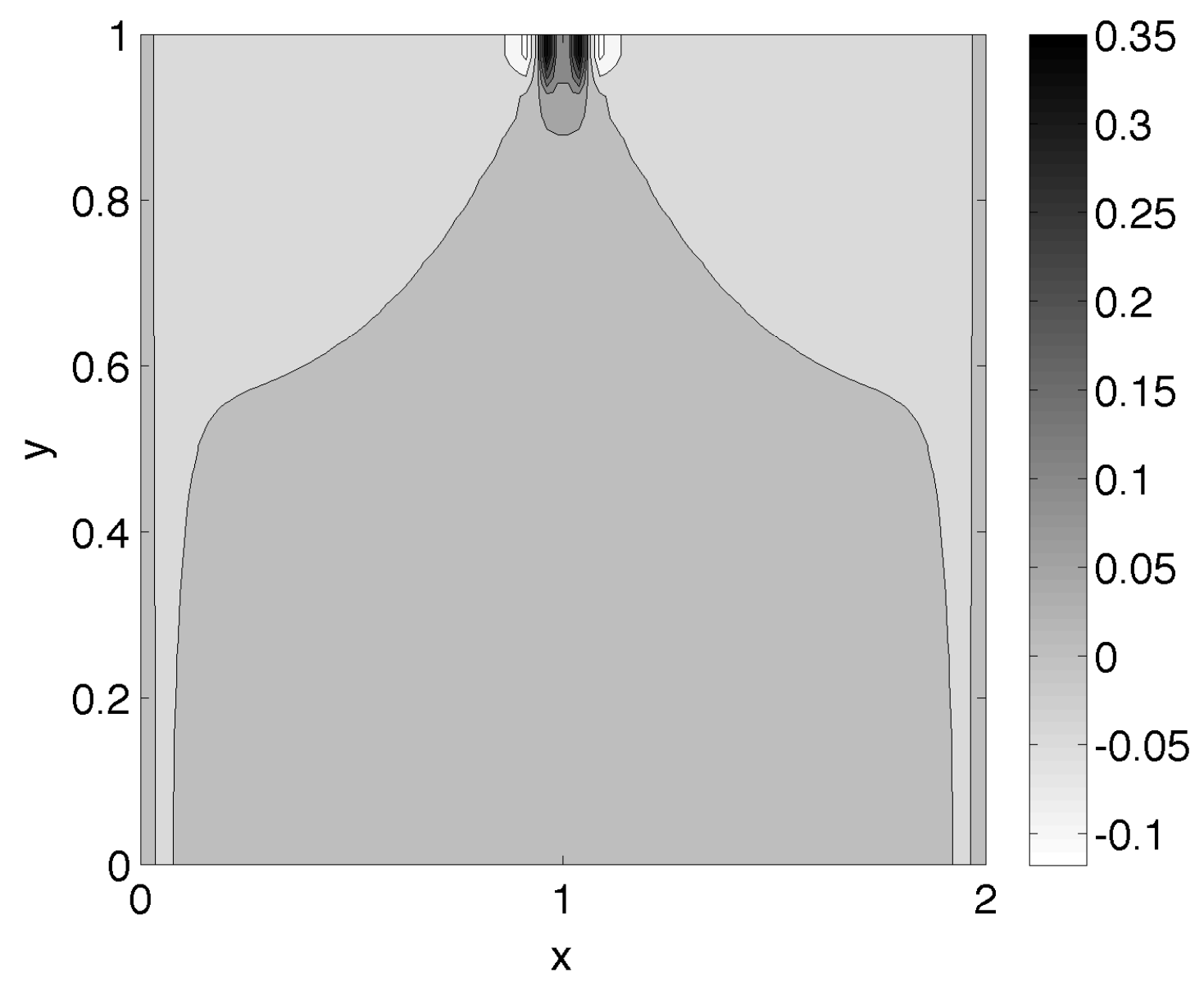}}
  \subfigure[Darcy-Forchheimer]{
  	\includegraphics[scale=0.46]
    {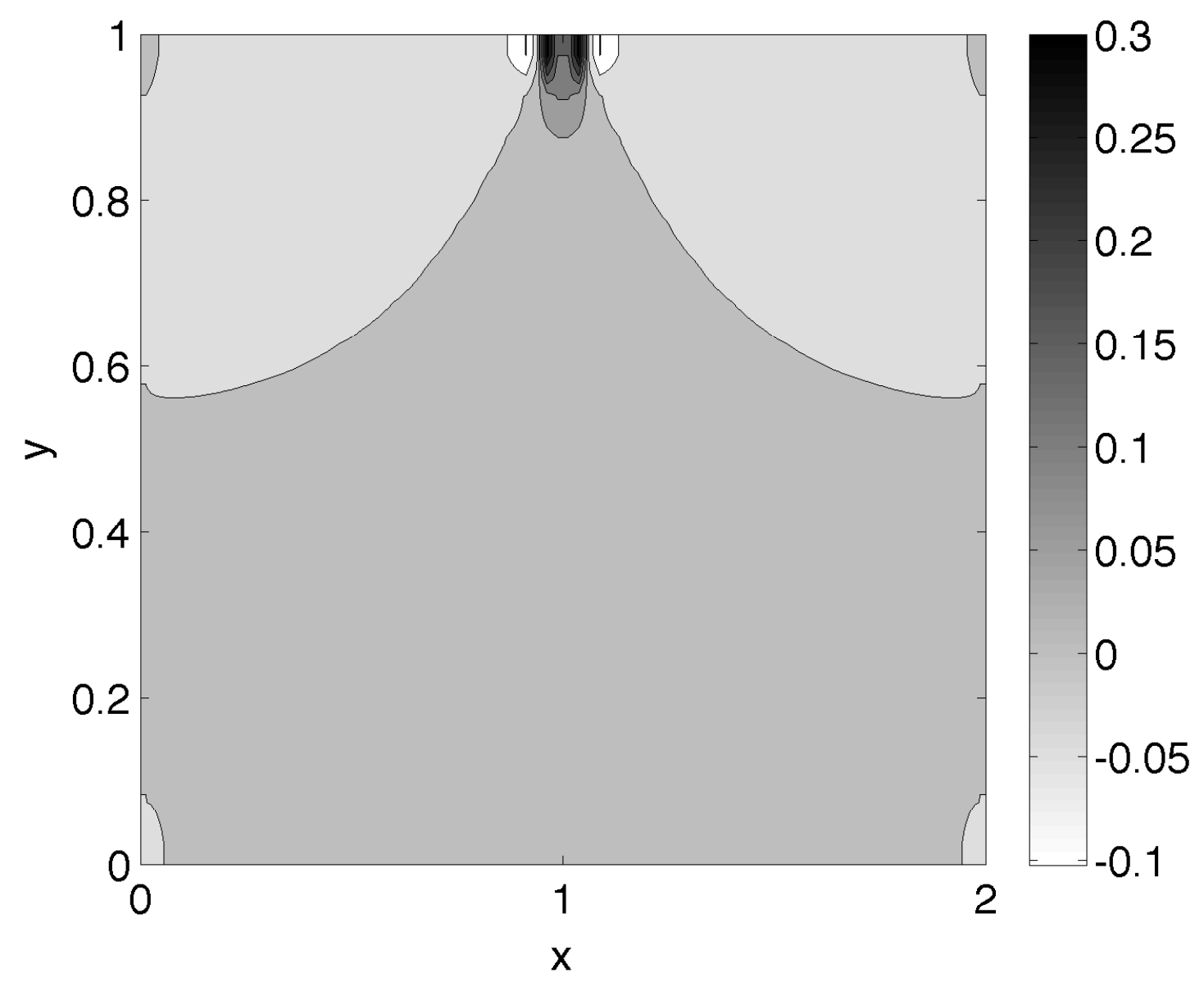}}
  \subfigure[Modified Darcy-Forchheimer Barus]{
  	\includegraphics[scale=0.46]
    {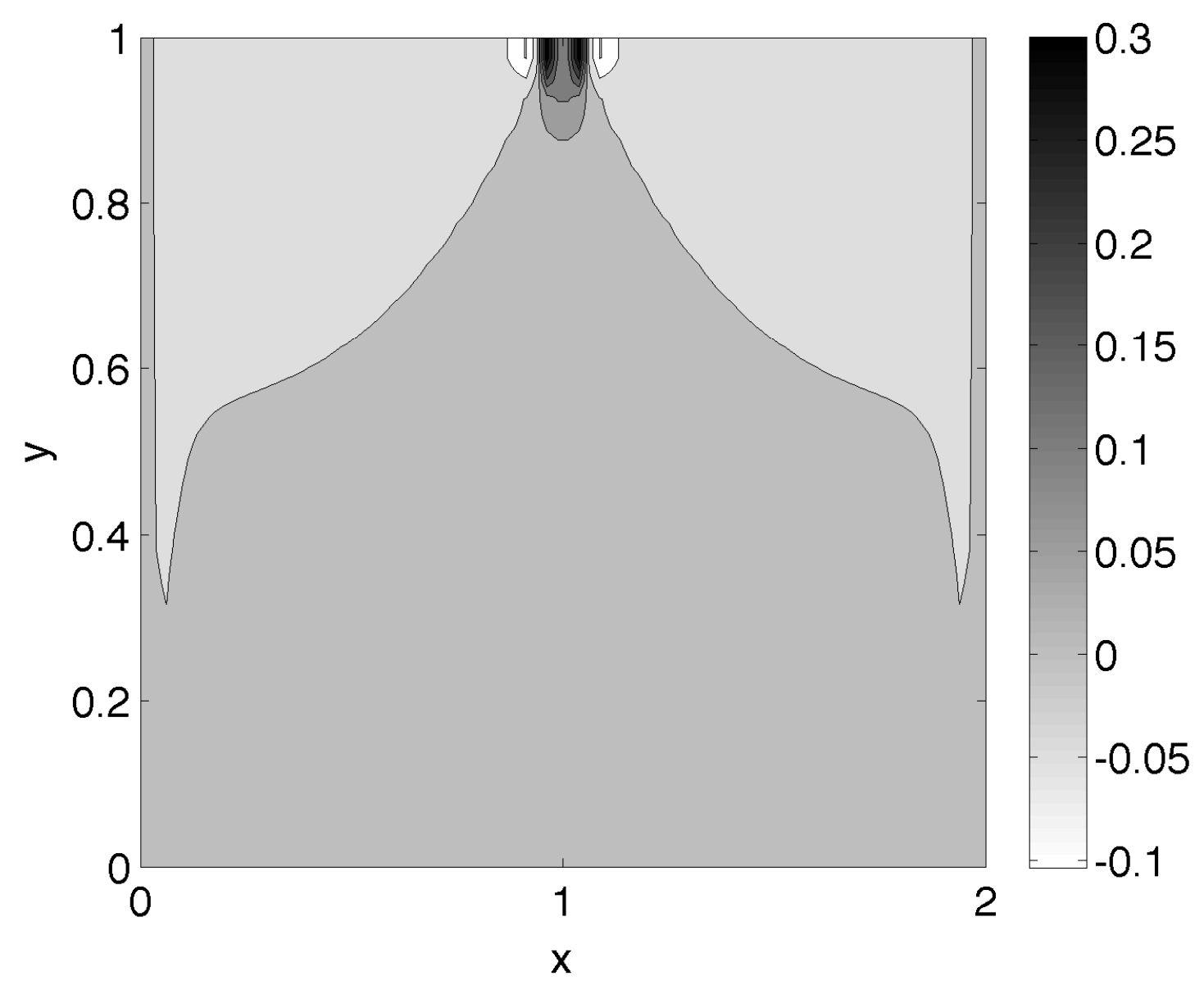}}
  \caption{Oil reservoir problem: ratios of local mass balance error over total predicted flux using VMS formalism}
  \label{Fig:Oil_reservoir_mass_error_VMS}
\end{figure}
\clearpage
\begin{figure}
  \centering
  \includegraphics[scale=0.6]{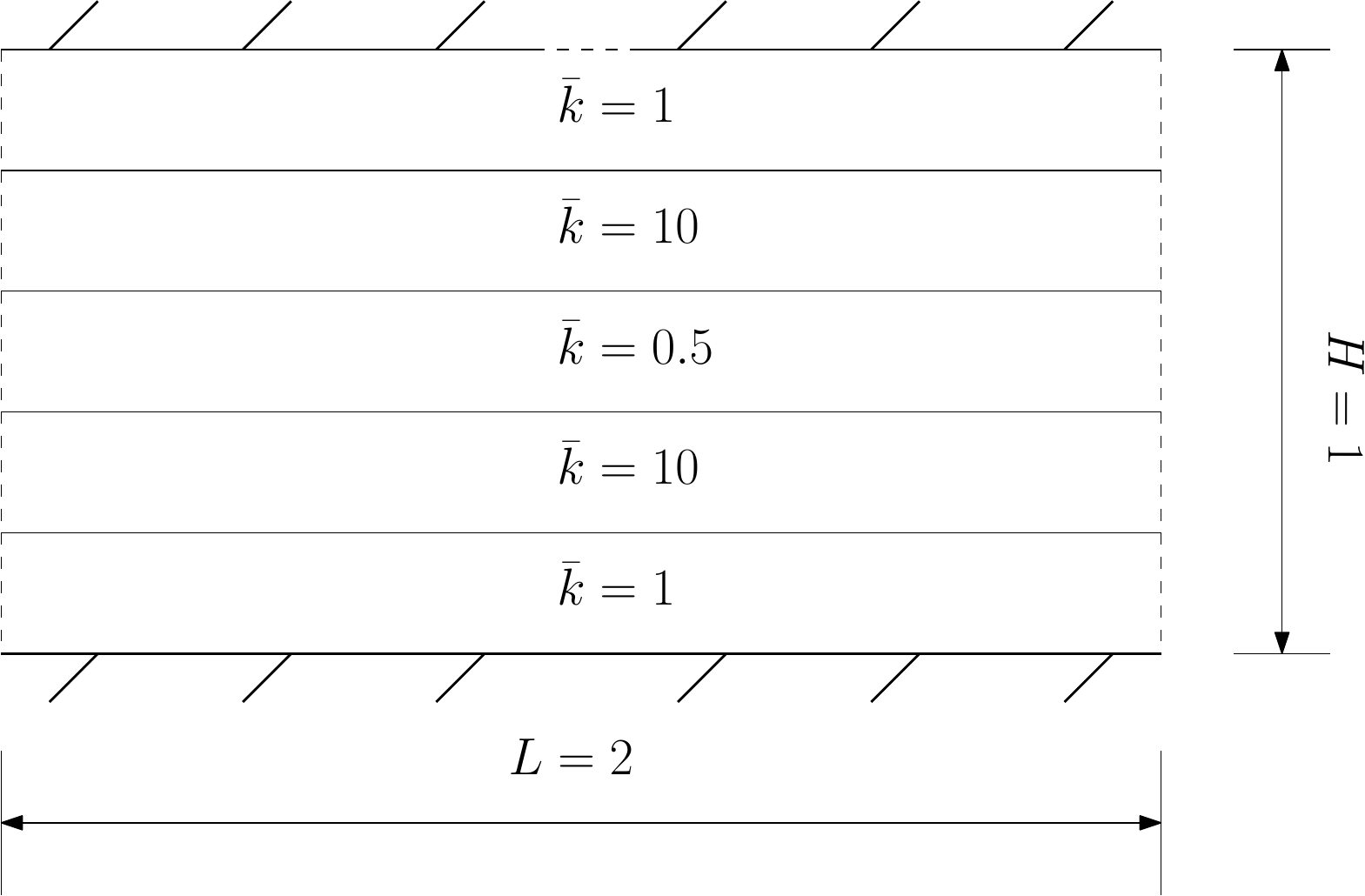}
  \caption{Layered reservoir problem: A pictorial description.}
   \label{Fig:layers}
\end{figure}
\begin{figure}
  \centering
  \includegraphics[scale=0.45]
  	{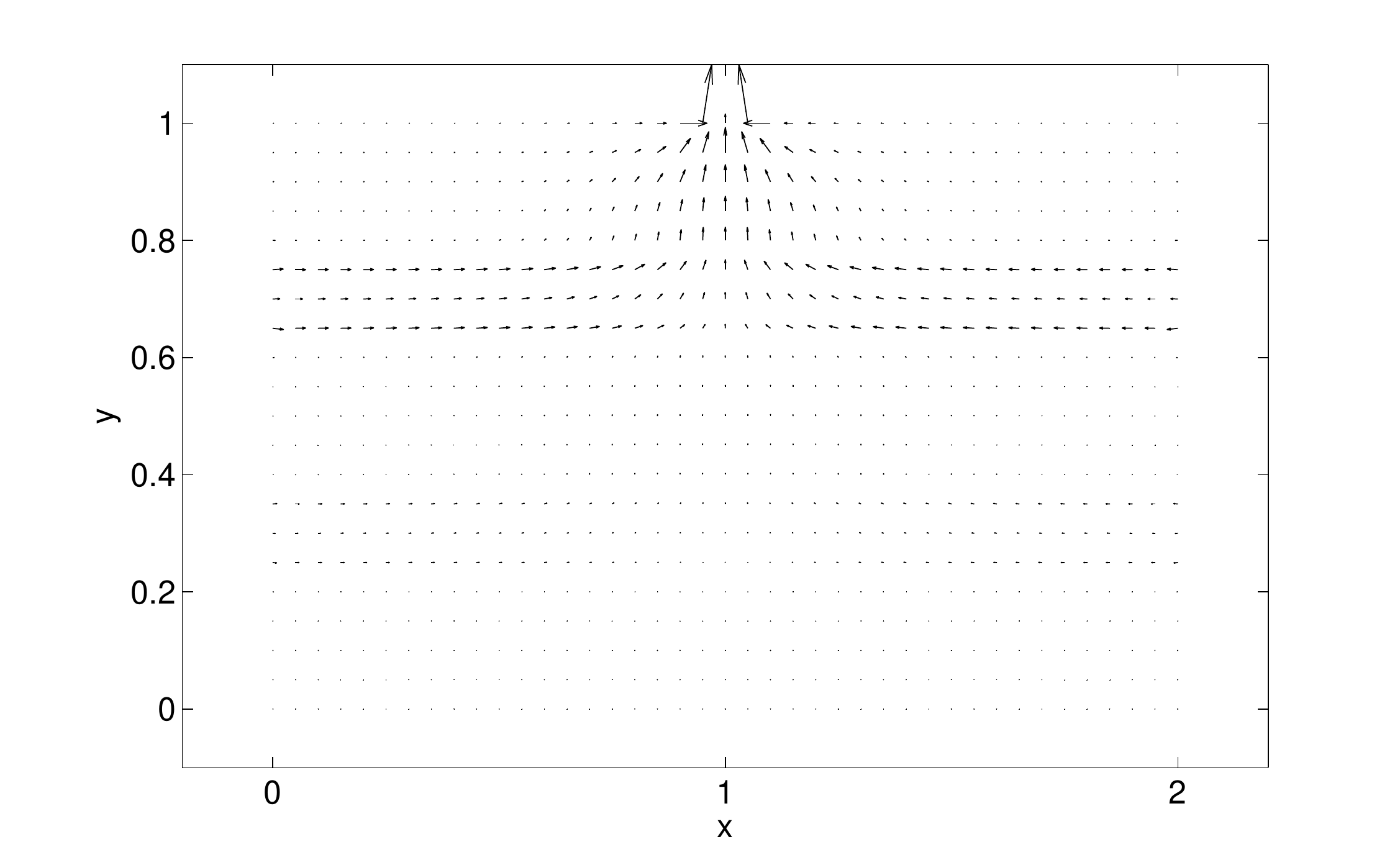}
  \caption{Layered reservoir problem: qualitative velocity vector field}
  \label{Fig:layered_reservoir_quiver}
\end{figure}
\begin{figure}
  \centering
  \subfigure[Darcy model]{
  	\includegraphics[scale=0.46]
    {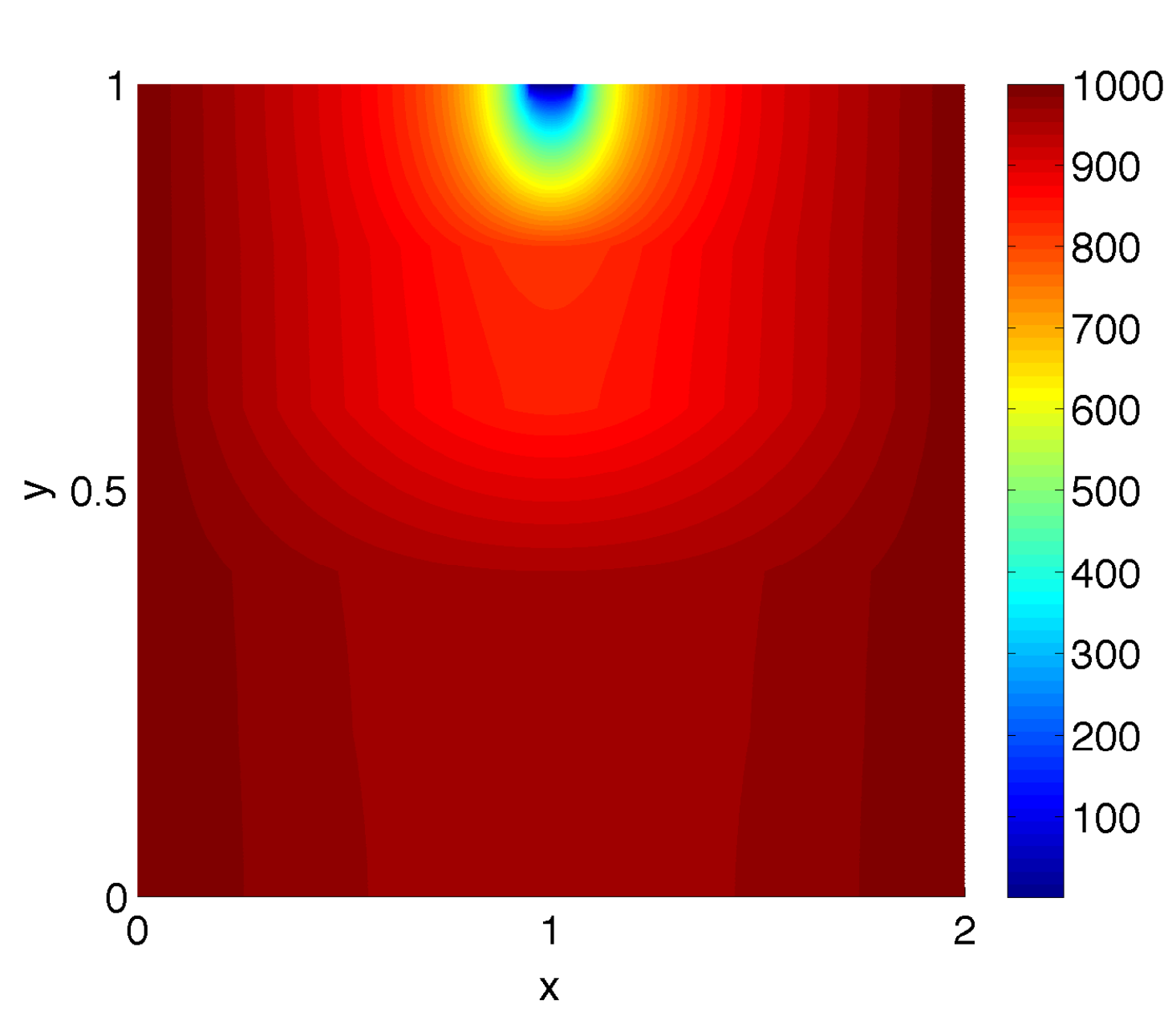}}
  \subfigure[Modified Barus]{
  	\includegraphics[scale=0.46]
    {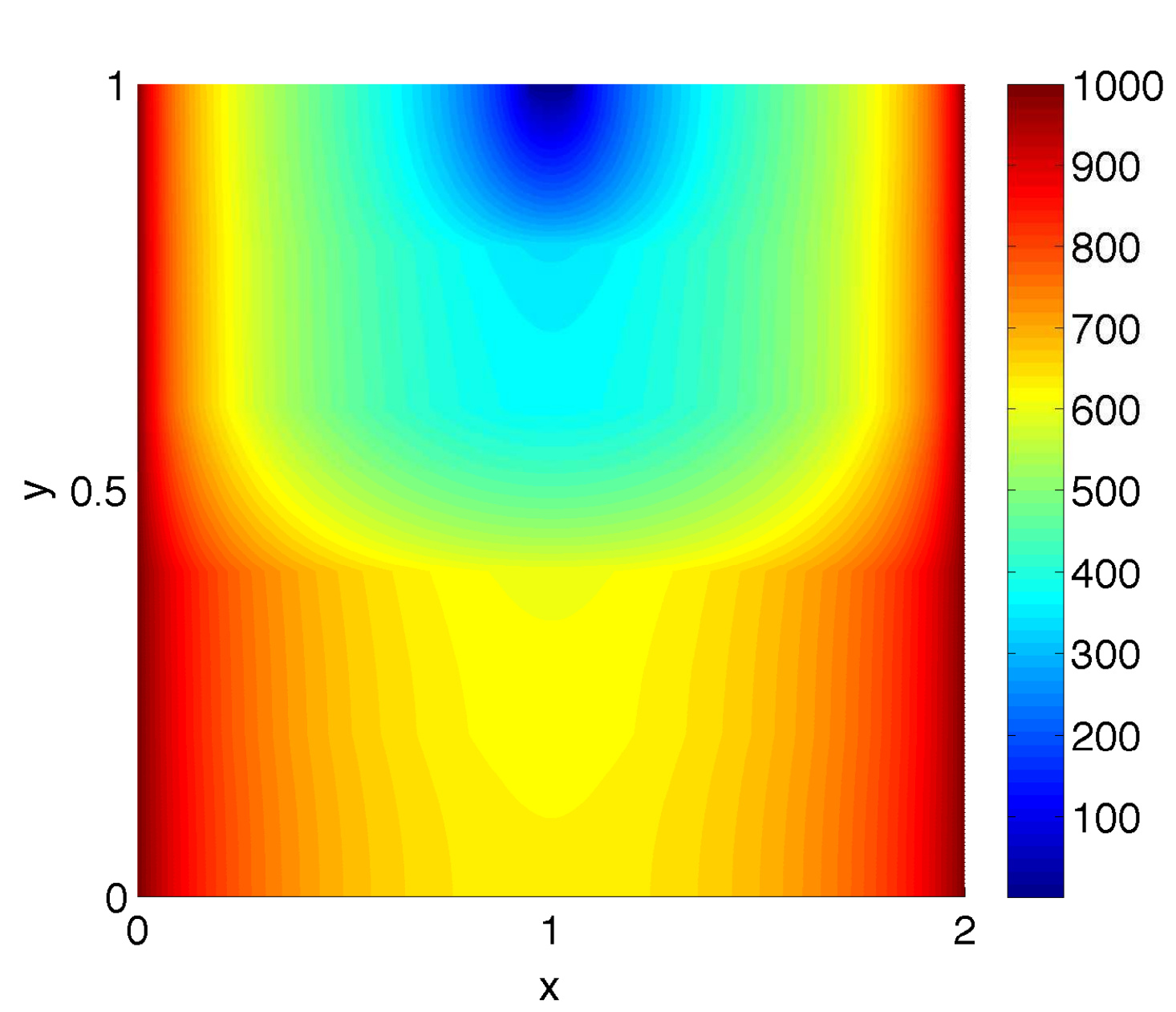}}
  \subfigure[Darcy-Forchheimer]{
  	\includegraphics[scale=0.46]
    {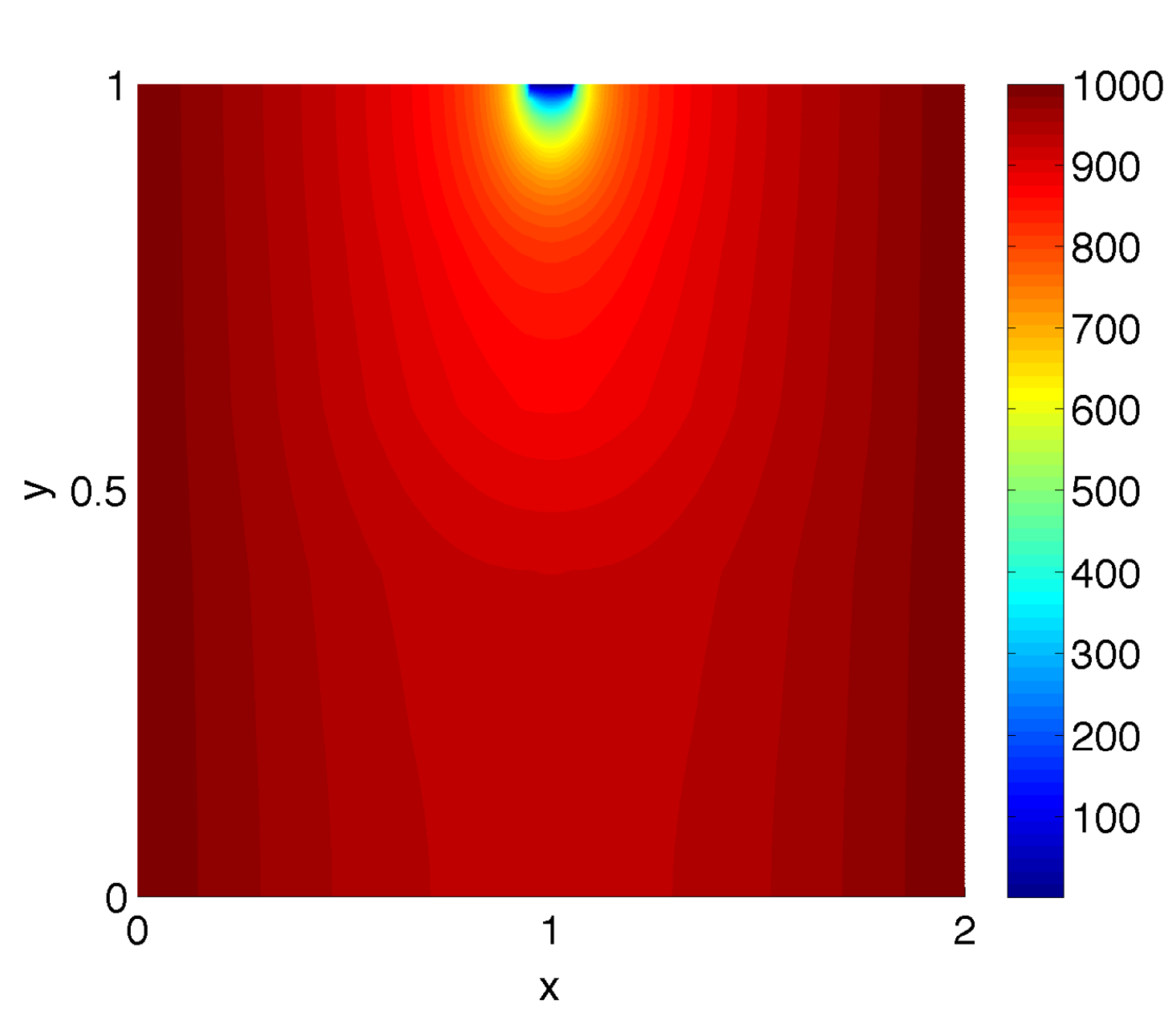}}
  \subfigure[Modified Darcy-Forchheimer Barus]{
  	\includegraphics[scale=0.46]
    {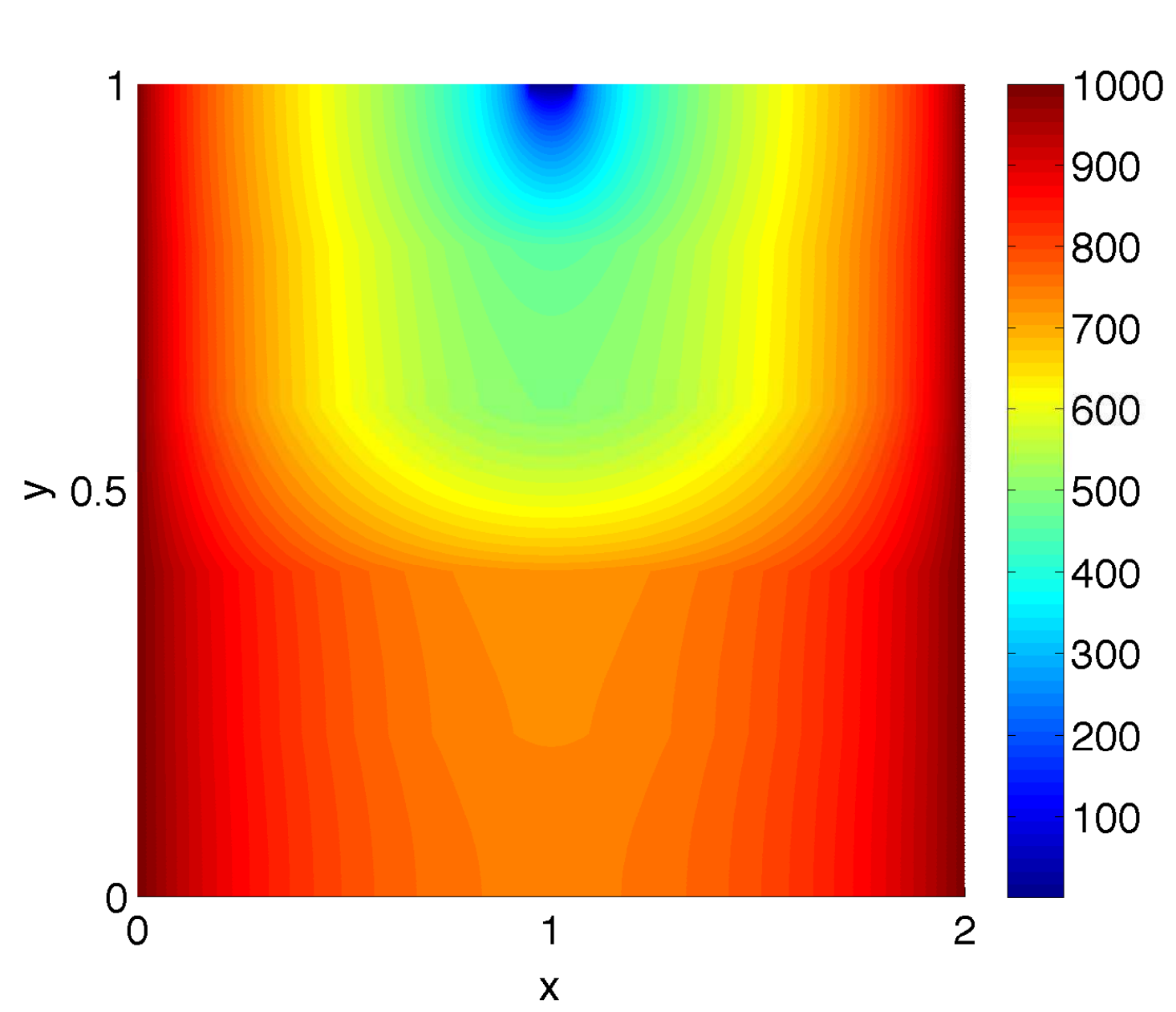}}
  \caption{Layered reservoir problem: pressure contours using LS formalism}
  \label{Fig:Layered_reservoir_pressure_LS}
\end{figure}
\begin{figure}
  \centering
  \subfigure[Darcy model]{
  	\includegraphics[scale=0.46]
    {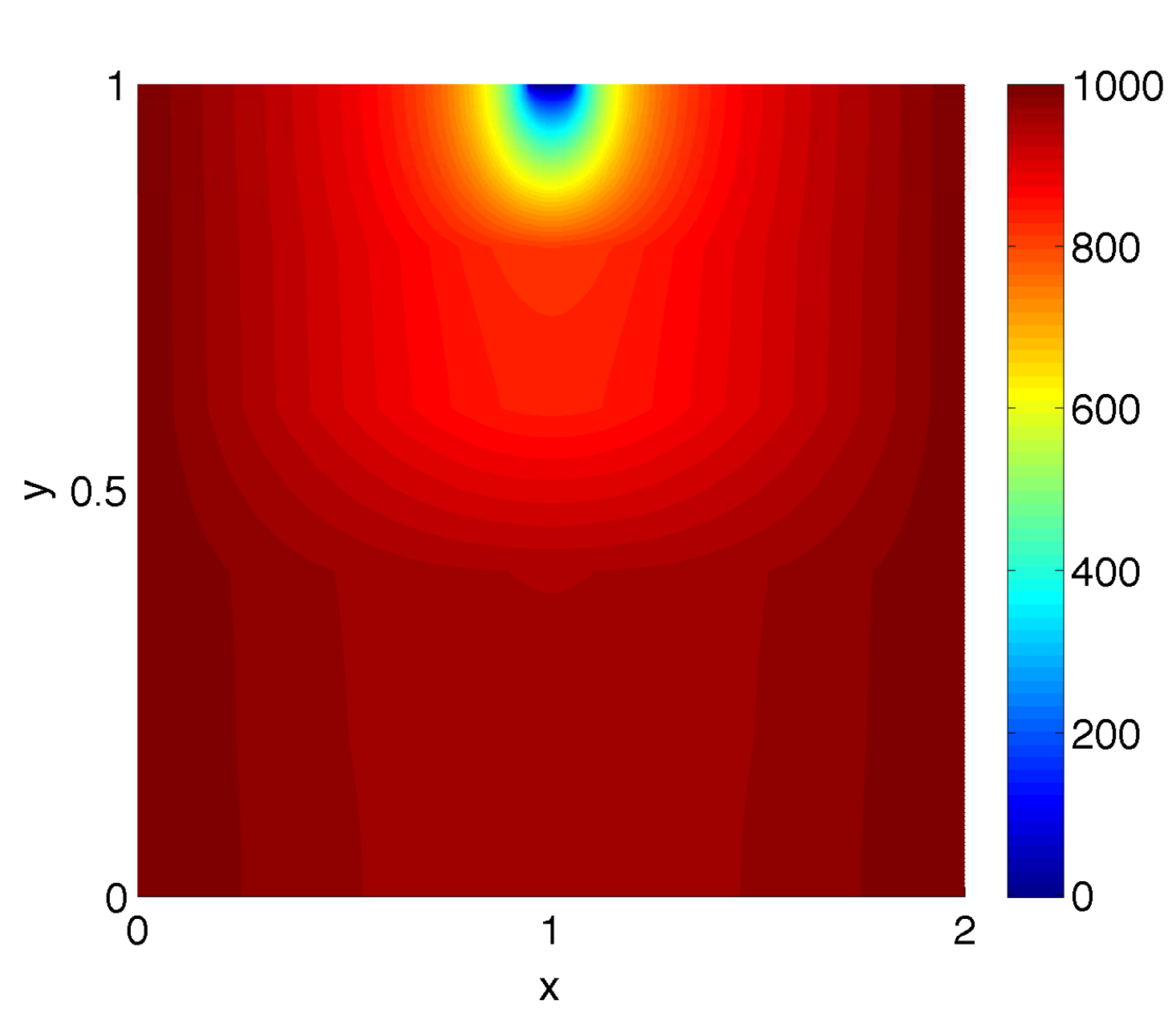}}
  \subfigure[Modified Barus]{
  	\includegraphics[scale=0.46]
    {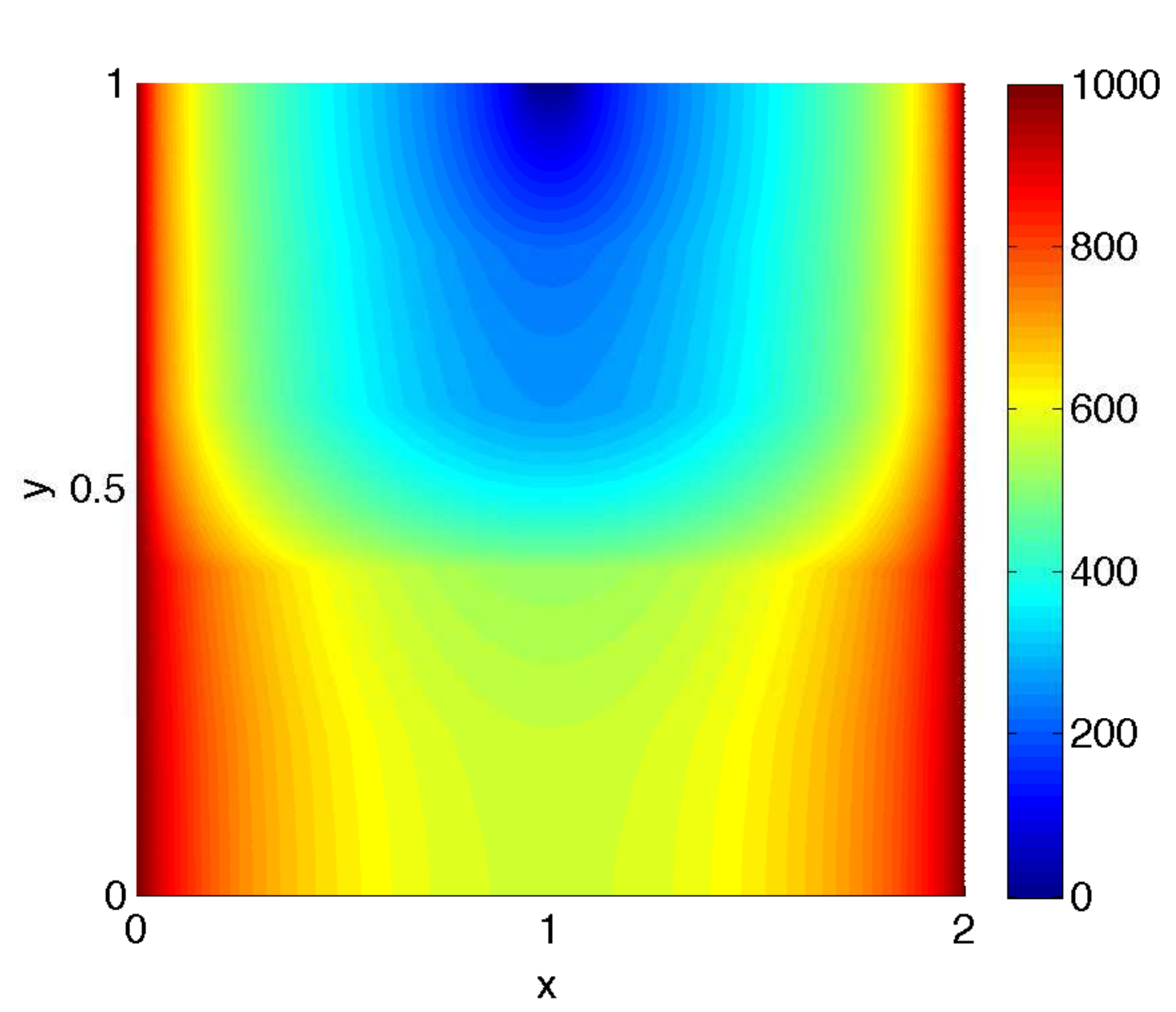}}
  \subfigure[Darcy-Forchheimer]{
  	\includegraphics[scale=0.46]
    {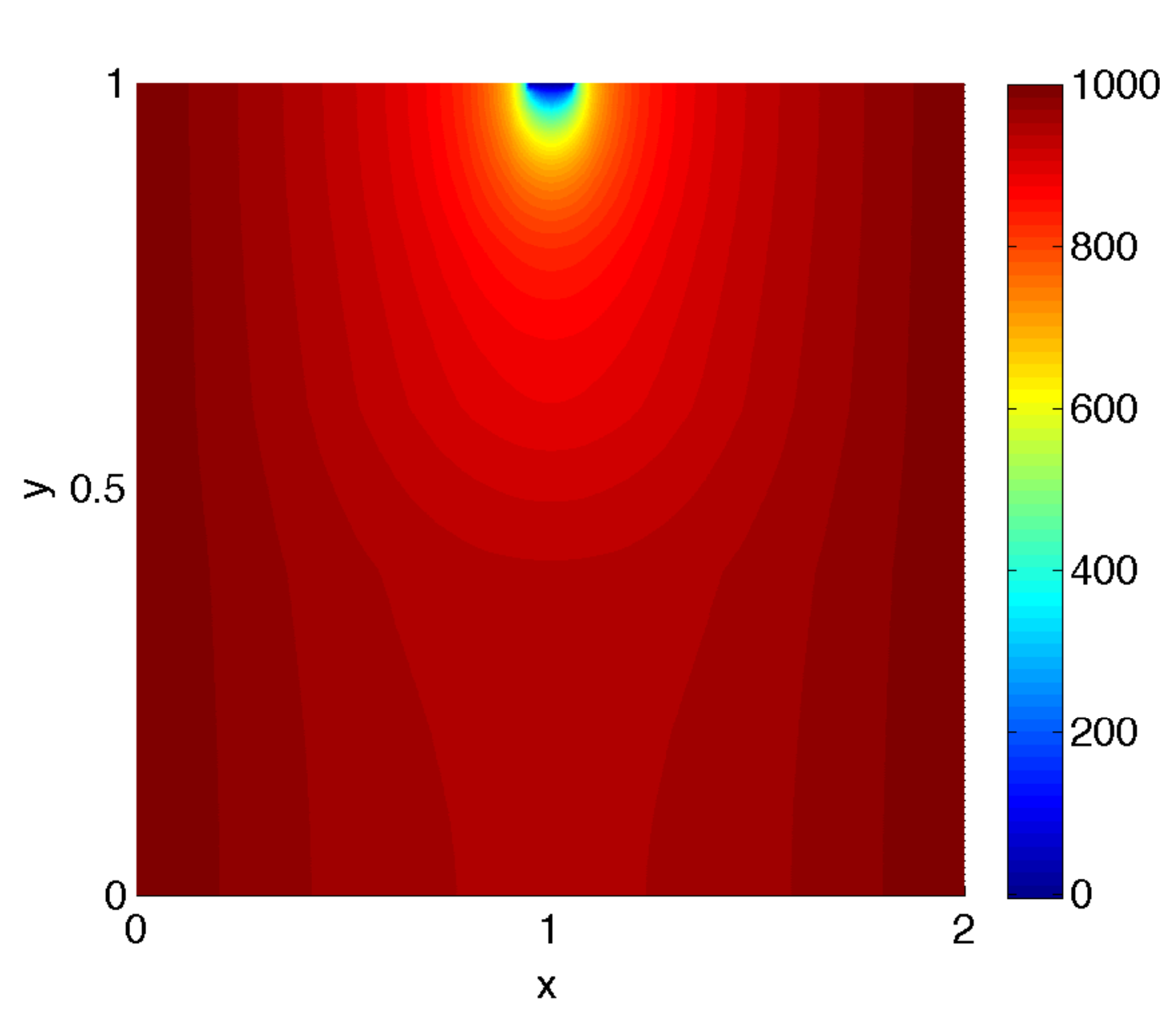}}
  \subfigure[Modified Darcy-Forchheimer Barus]{
  	\includegraphics[scale=0.46]
    {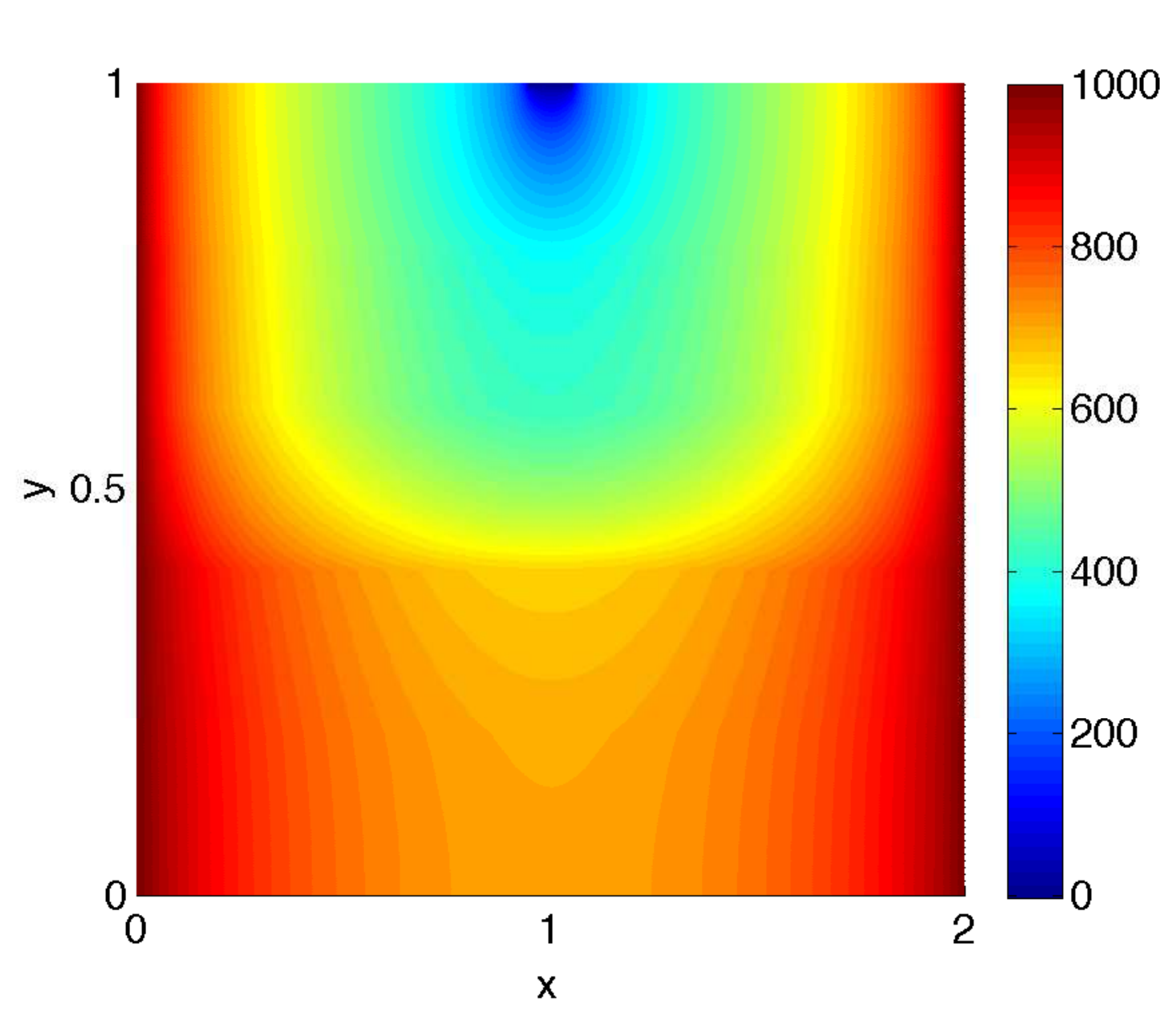}}
  \caption{Layered reservoir problem: pressure contours using VMS formalism}
  \label{Fig:Layered_reservoir_pressure_VMS}
\end{figure}
\begin{figure}
  \centering
  \subfigure[Darcy model]{
  	\includegraphics[scale=0.46]
    {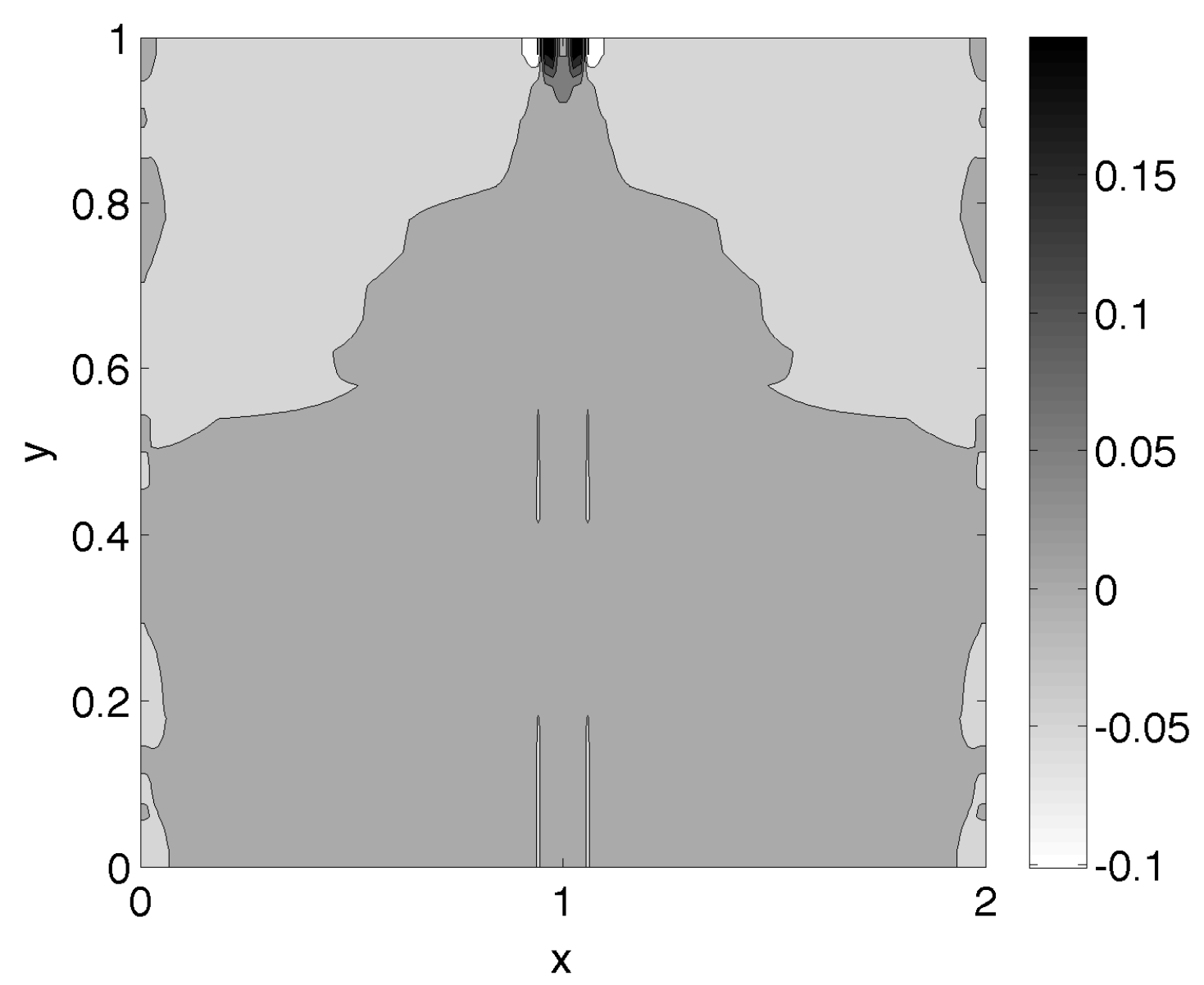}}
  \subfigure[Modified Barus]{
  	\includegraphics[scale=0.46]
    {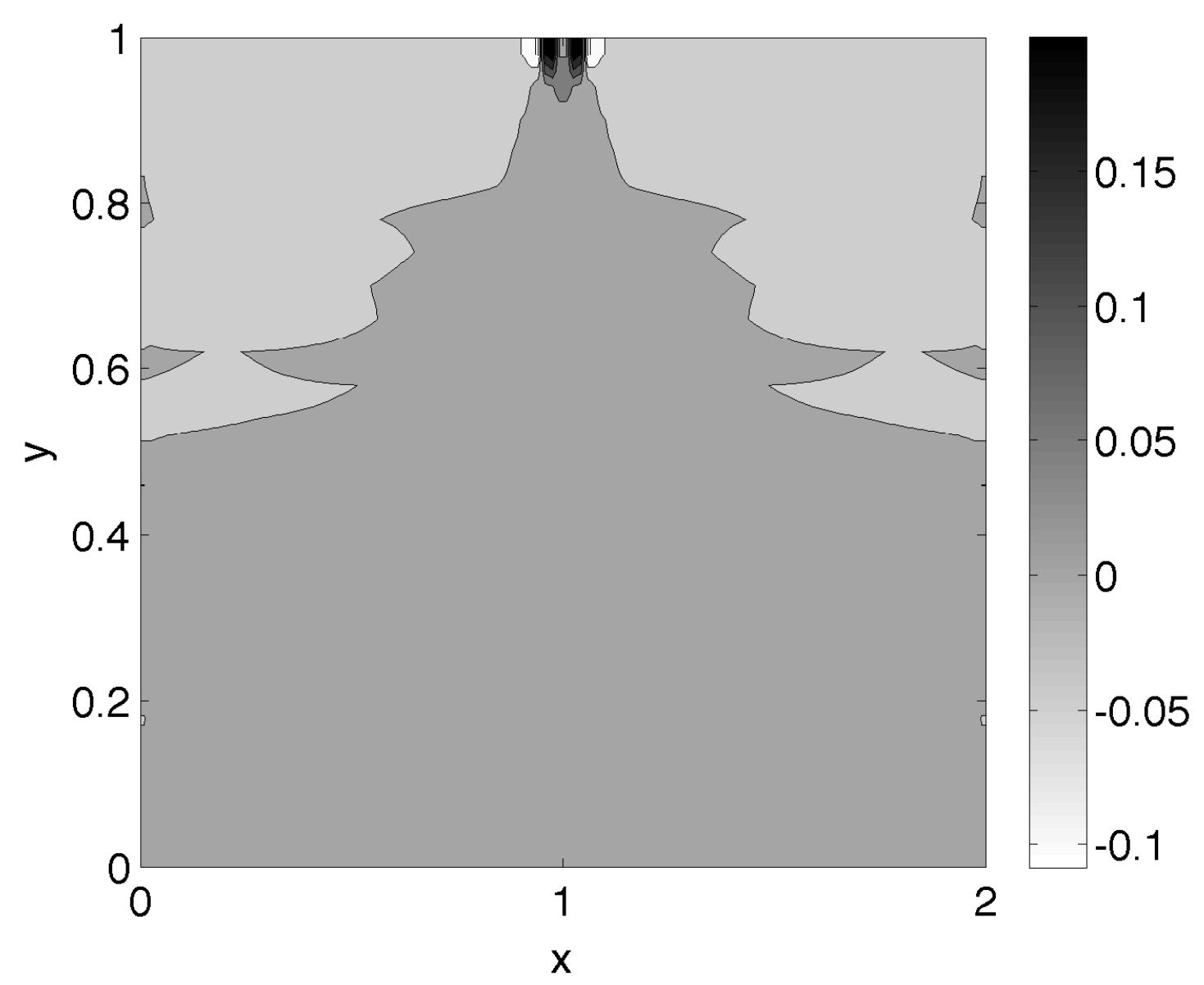}}
  \subfigure[Darcy-Forchheimer]{
  	\includegraphics[scale=0.46]
    {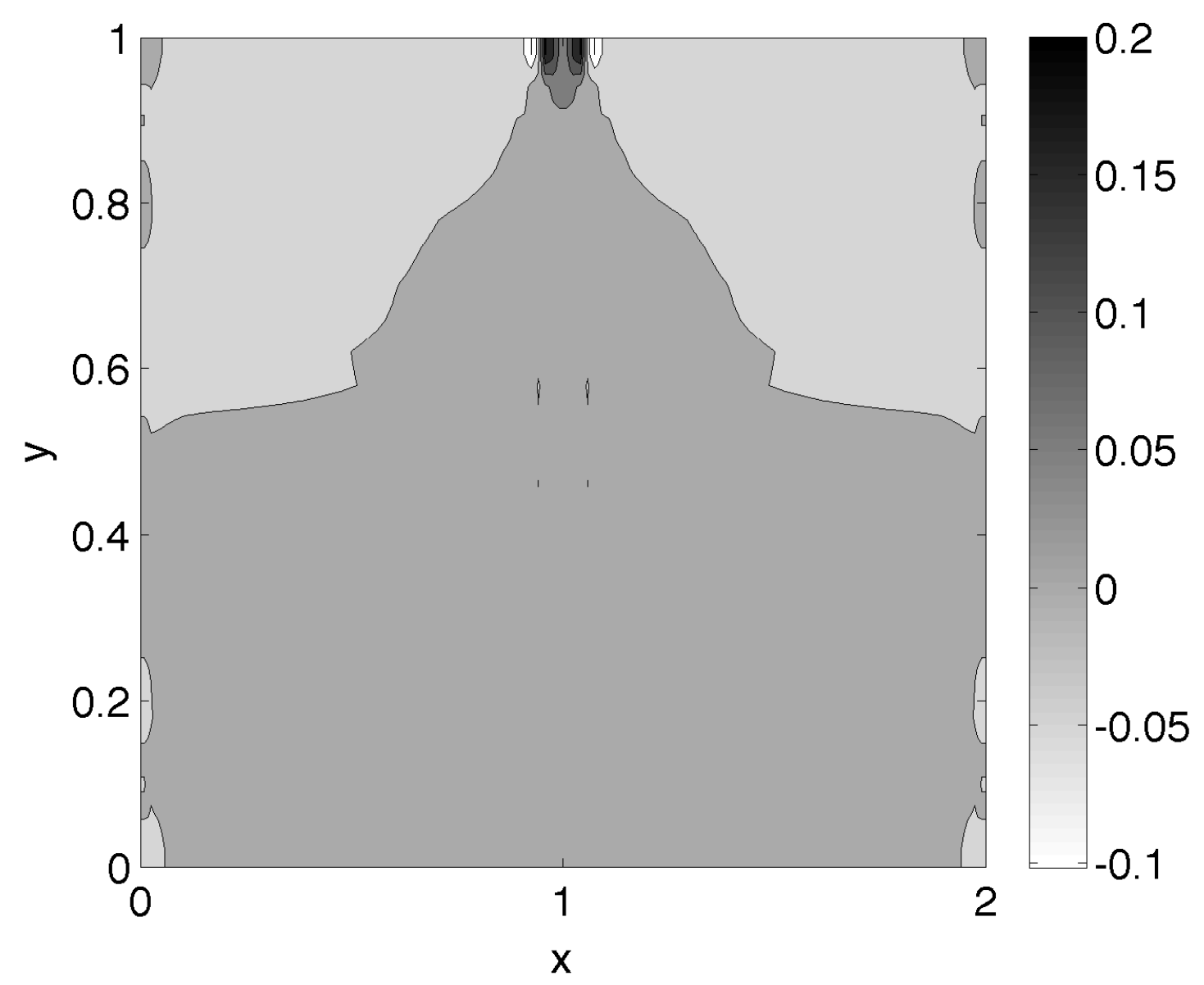}}
  \subfigure[Modified Darcy-Forchheimer Barus]{
  	\includegraphics[scale=0.46]
    {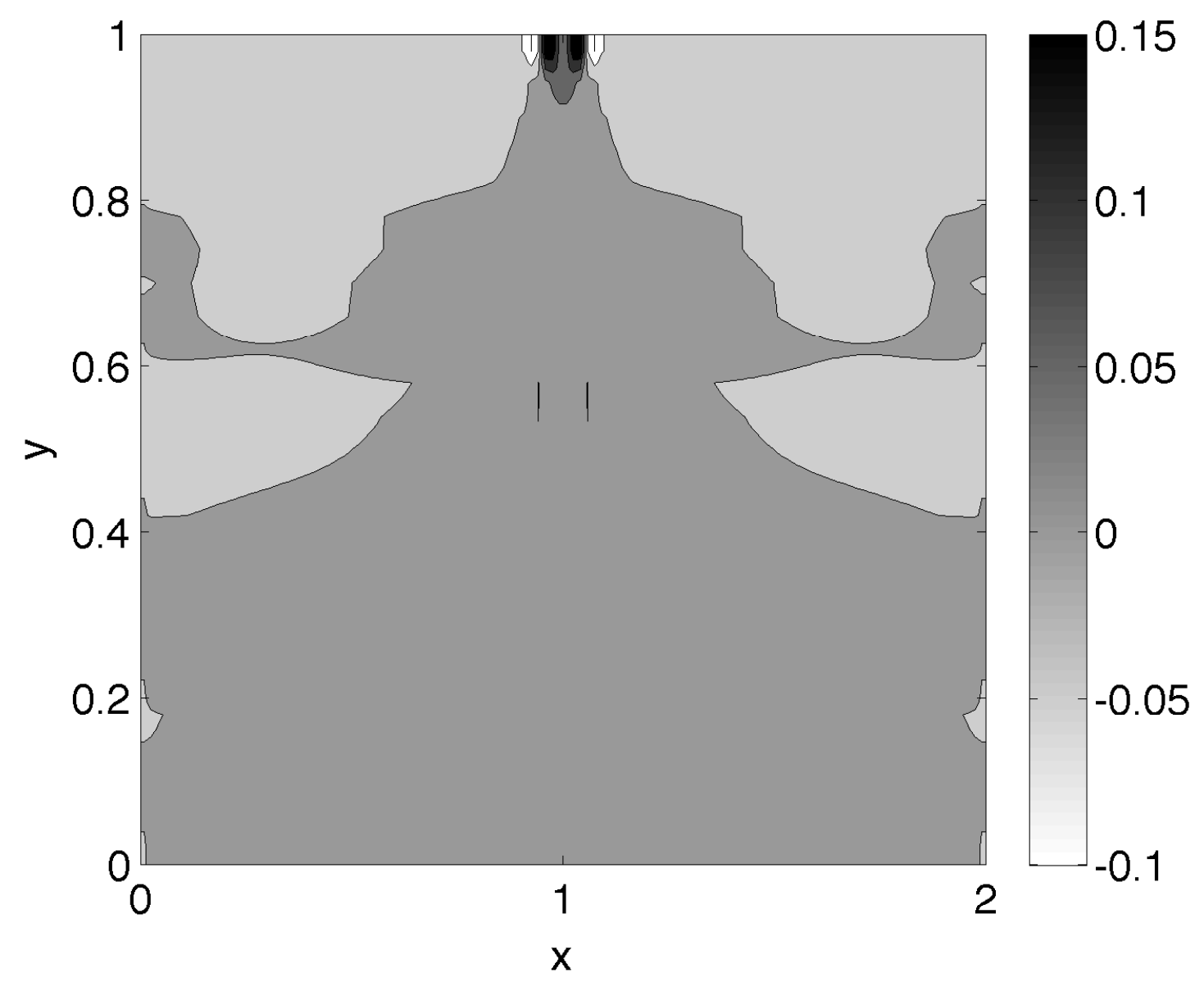}}
  \caption{Layered reservoir problem: ratios of local mass balance error over total predicted flux using LS formalism}
  \label{Fig:Layered_reservoir_mass_error_LS}
\end{figure}
\begin{figure}
  \centering
  \subfigure[Darcy model]{
  	\includegraphics[scale=0.46]
    {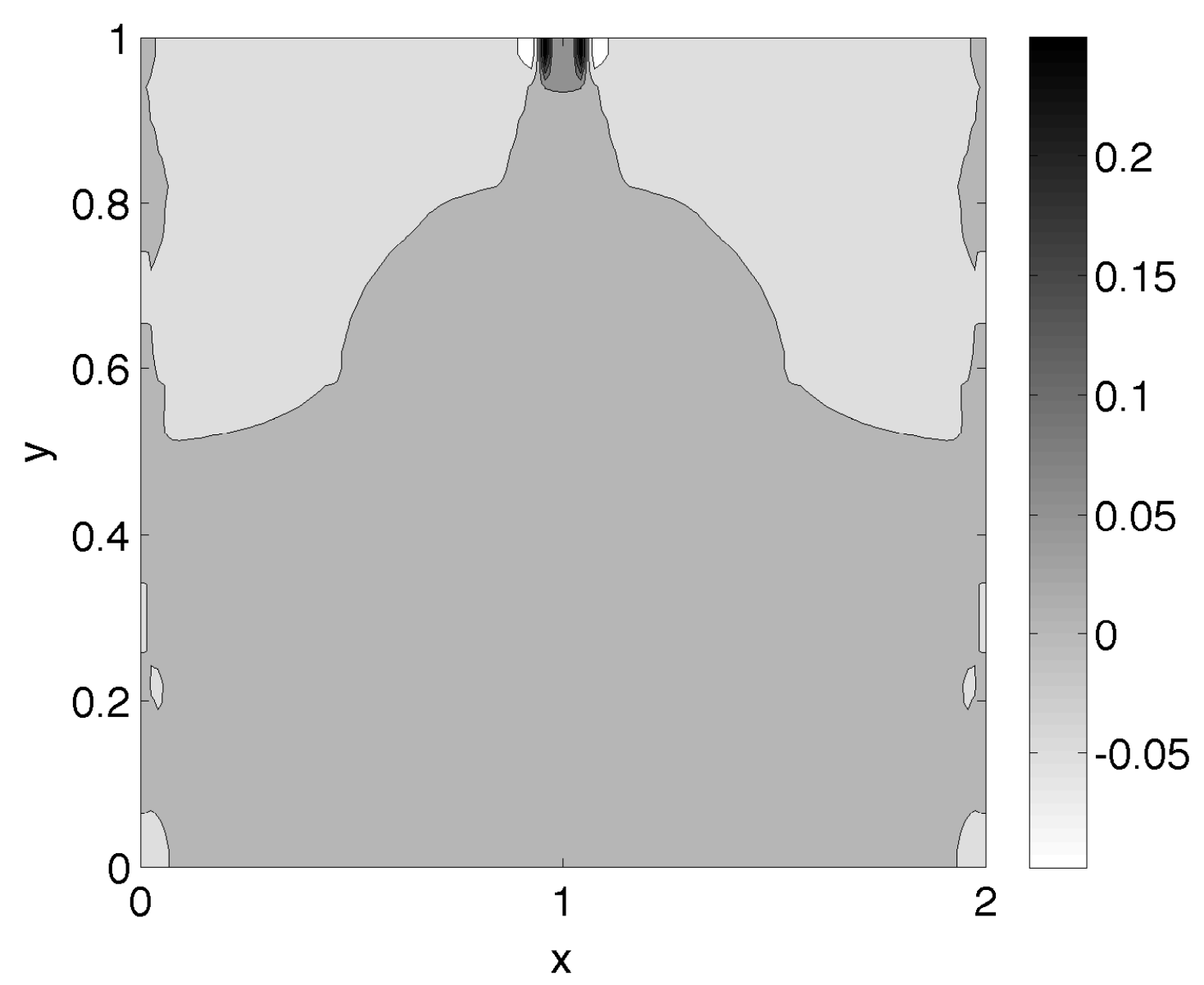}}
  \subfigure[Modified Barus]{
  	\includegraphics[scale=0.46]
    {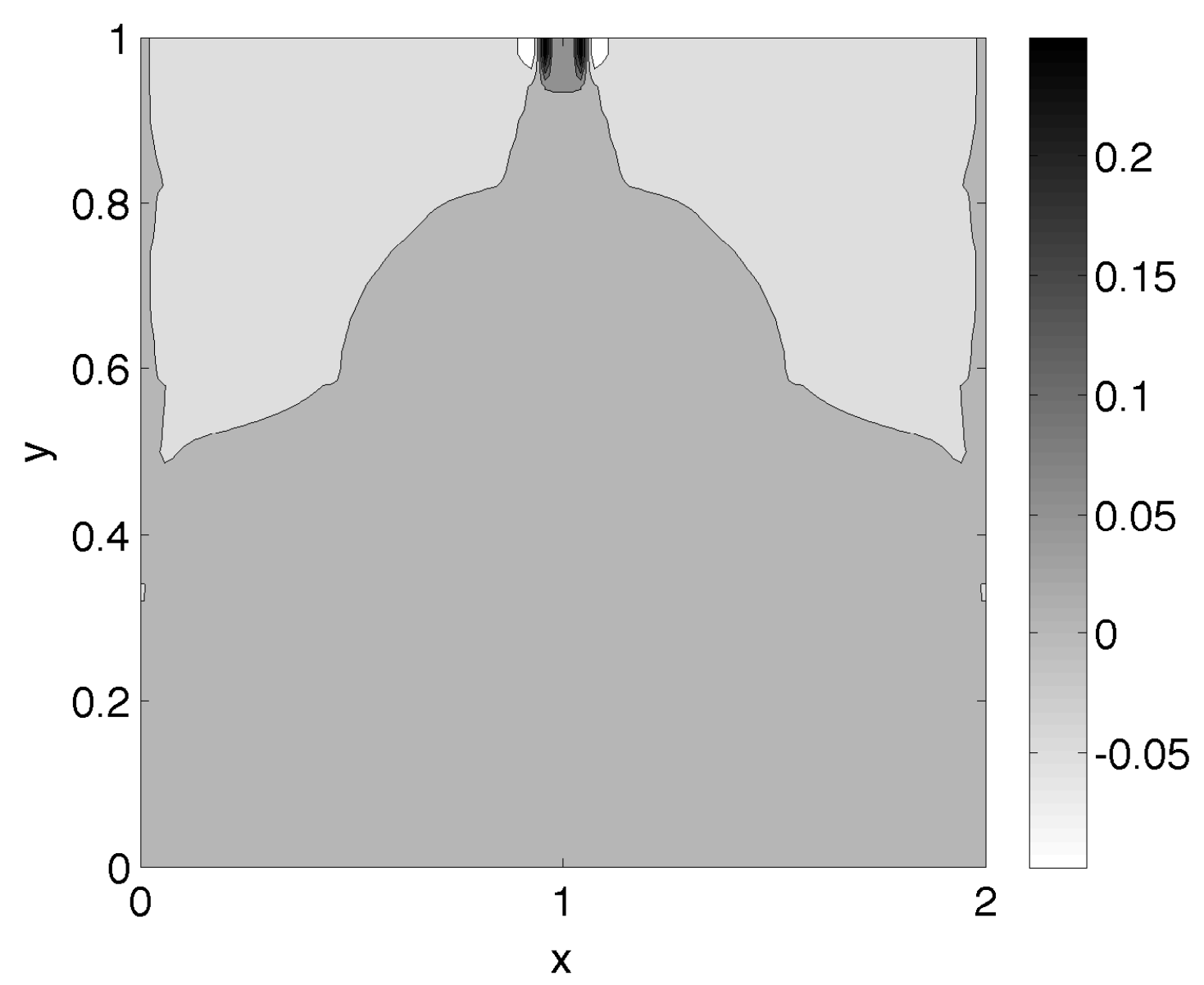}}
  \subfigure[Darcy-Forchheimer]{
  	\includegraphics[scale=0.46]
    {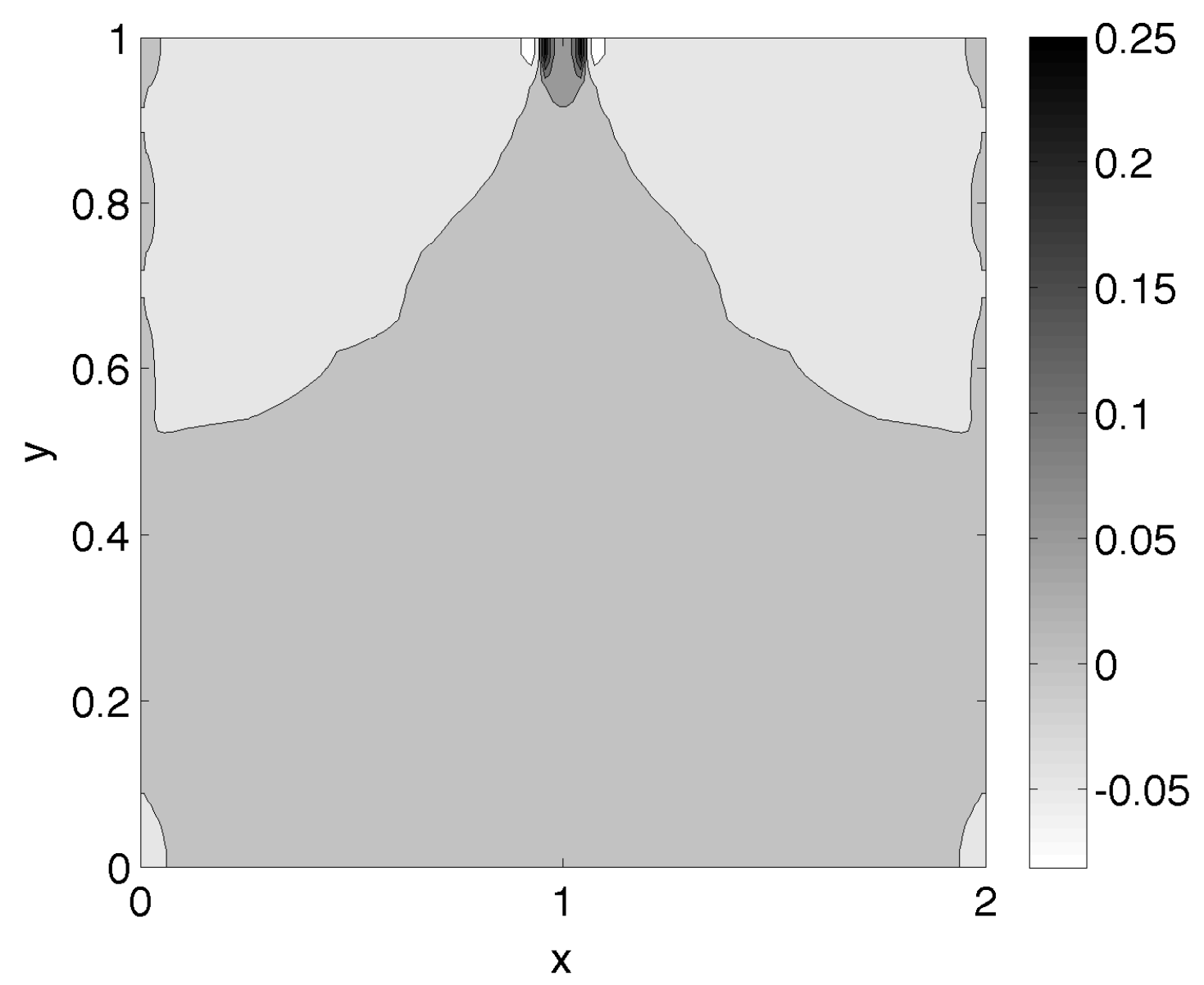}}
  \subfigure[Modified Darcy-Forchheimer Barus]{
  	\includegraphics[scale=0.46]
    {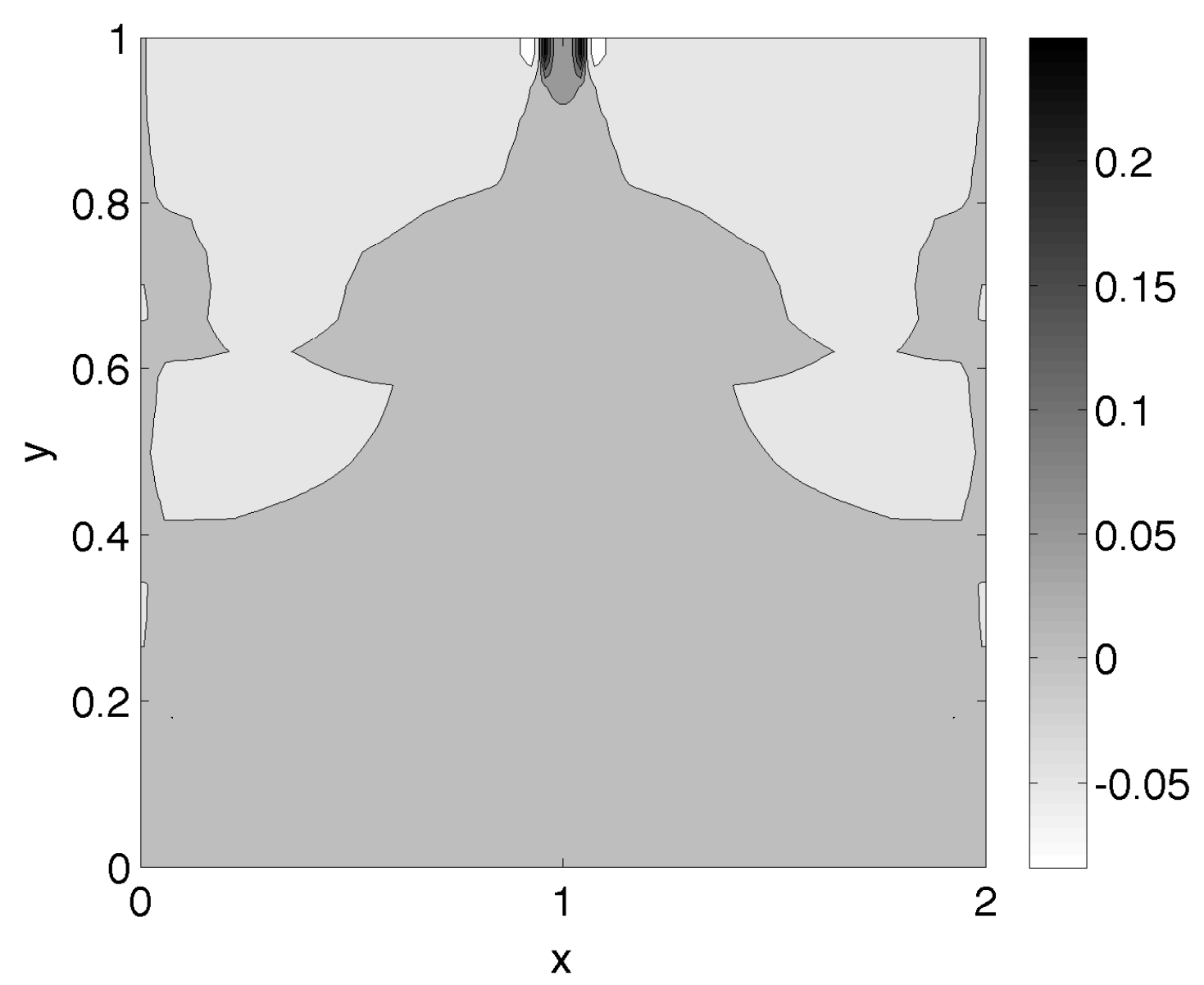}}
  \caption{Layered reservoir problem: ratios of local mass balance error over total predicted flux using VMS formalism}
  \label{Fig:Layered_reservoir_mass_error_VMS}
\end{figure}
\clearpage
\begin{figure}
  \centering
  \includegraphics[scale=1]{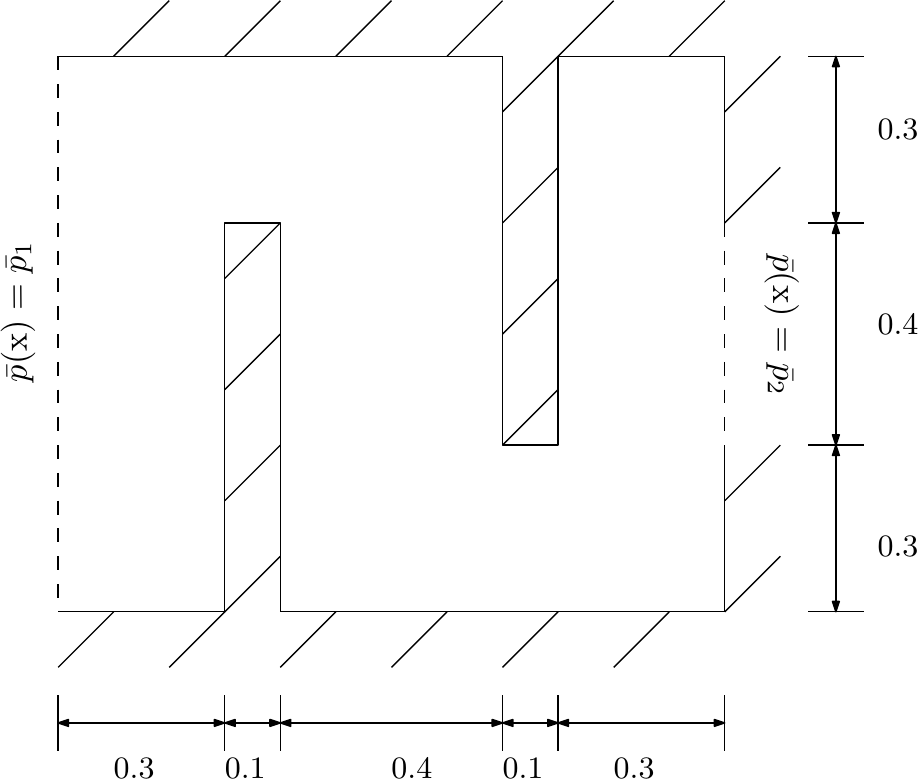}
  \caption{Staggered impervious zones problem: 
    A pictorial description. \label{Fig:staggered}}
\end{figure}
\begin{figure}
  \centering
  \includegraphics[scale=0.46]
  	{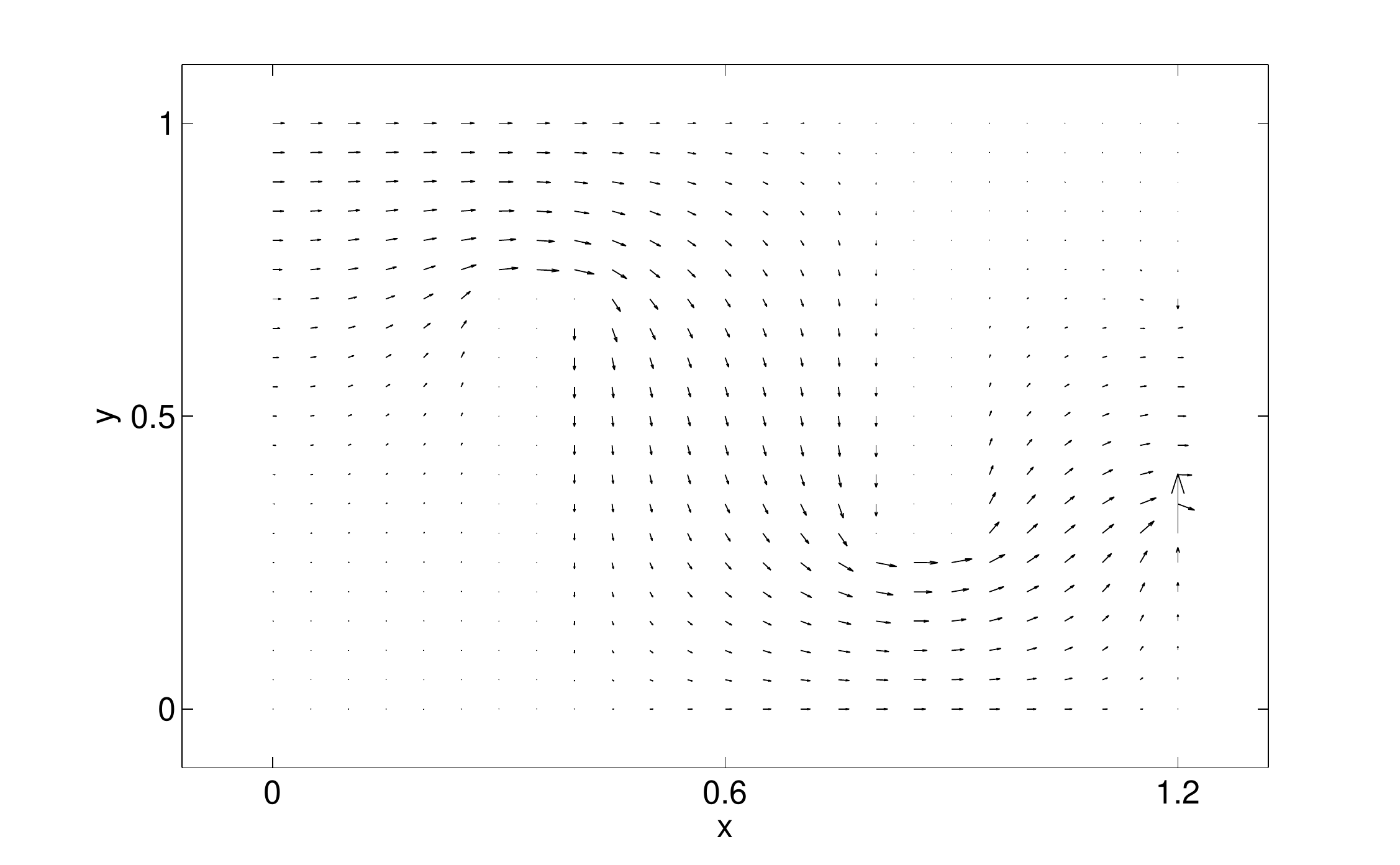}
  \caption{Staggered impervious zones problem: qualitative velocity vector field}
  \label{Fig:staggered_reservoir_quiver}
\end{figure}
\begin{figure}
  \centering
  \subfigure[Darcy model]{
  	\includegraphics[scale=0.46]
    {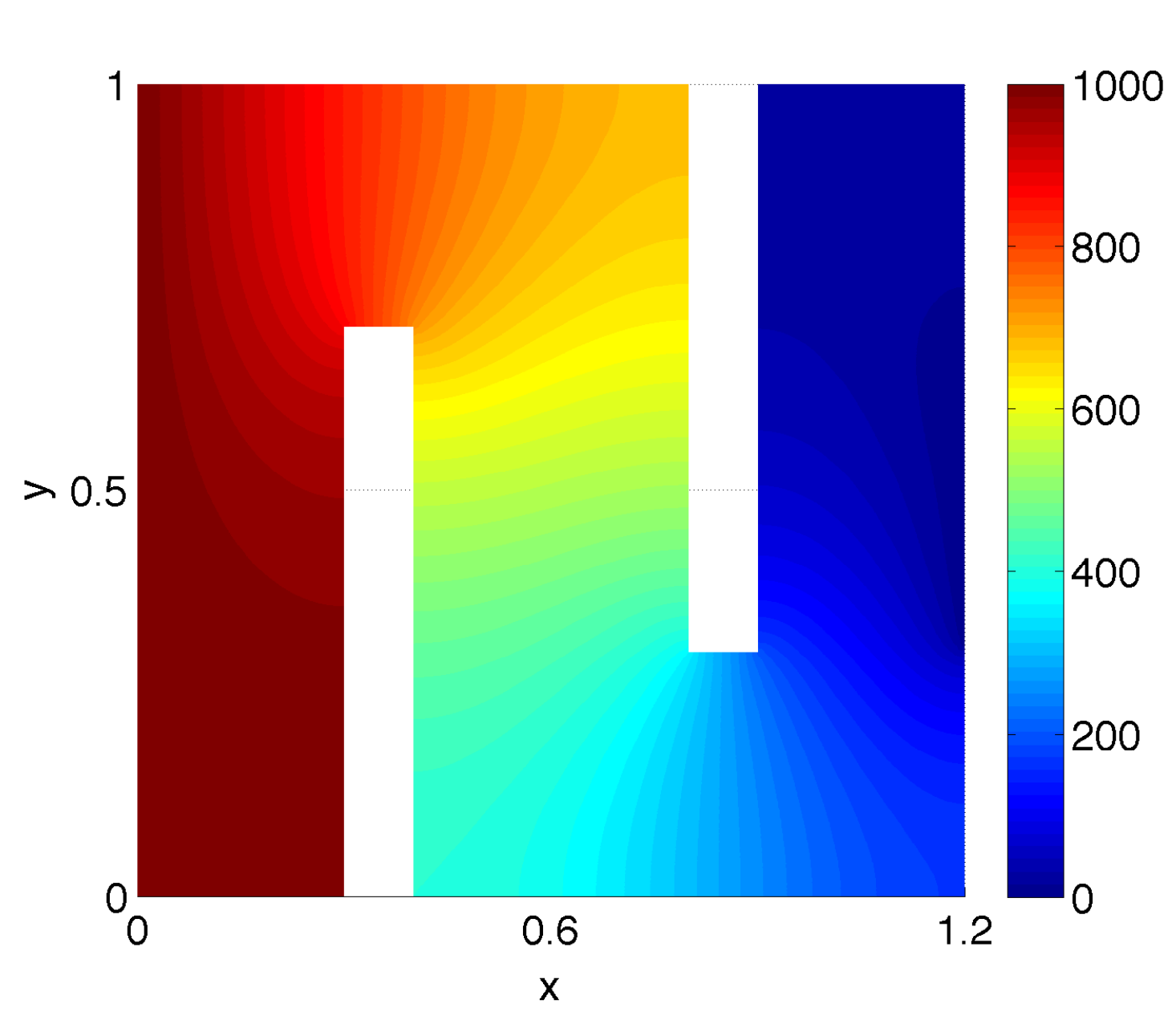}}
  \subfigure[Modified Barus]{
  	\includegraphics[scale=0.46]
    {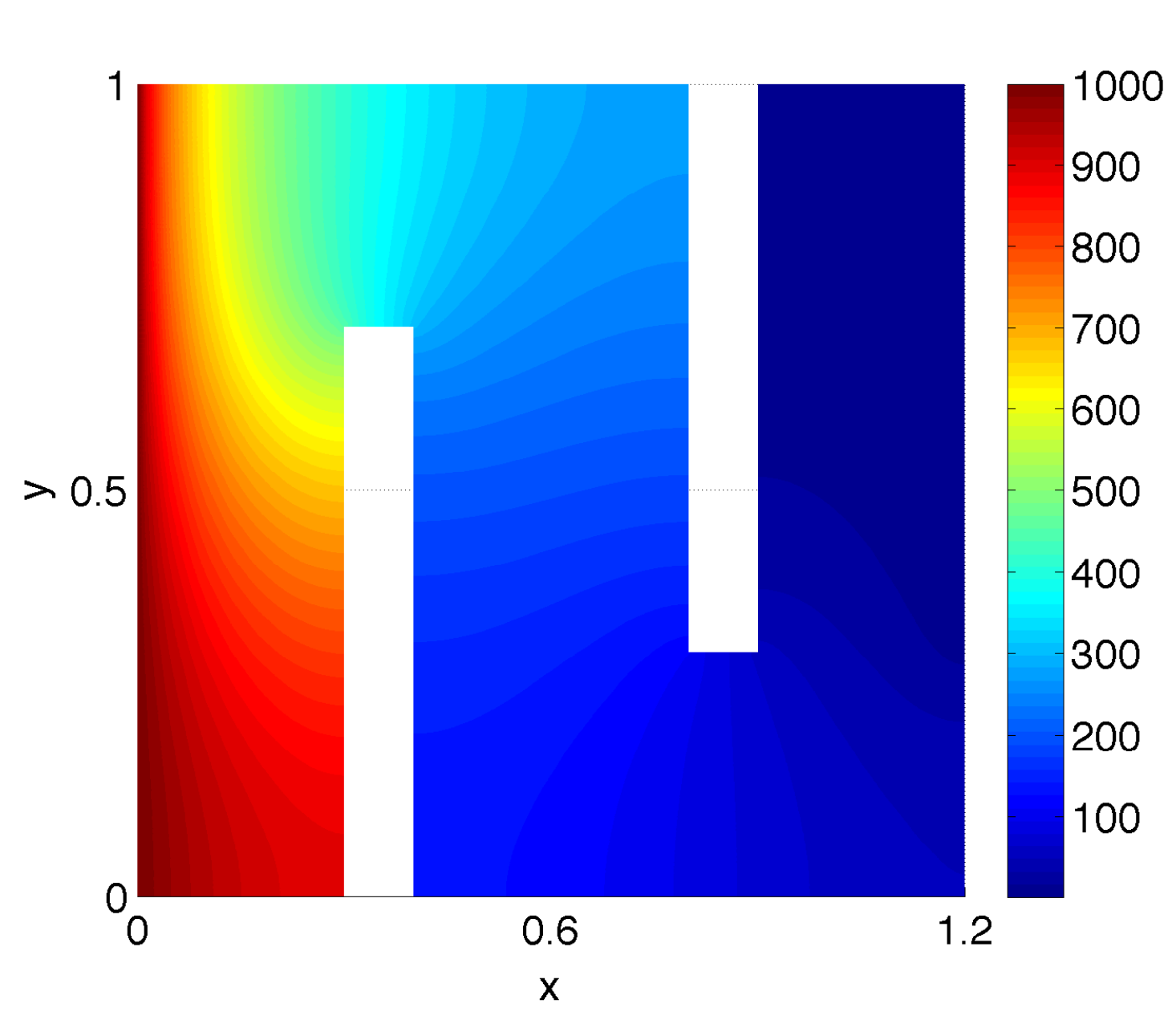}}
  \subfigure[Darcy-Forchheimer]{
  	\includegraphics[scale=0.46]
    {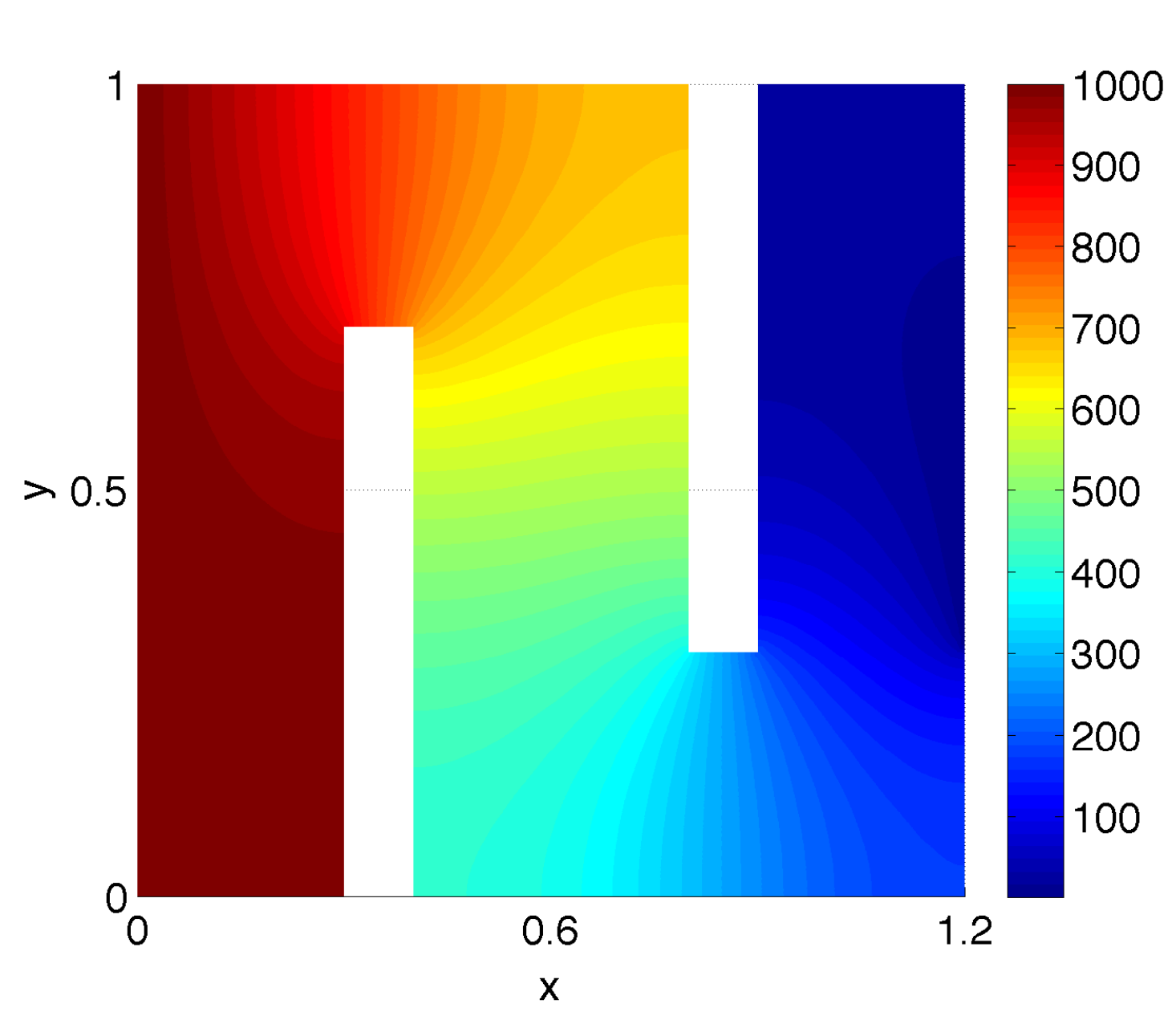}}
  \subfigure[Modified Darcy-Forchheimer Barus]{
  	\includegraphics[scale=0.46]
    {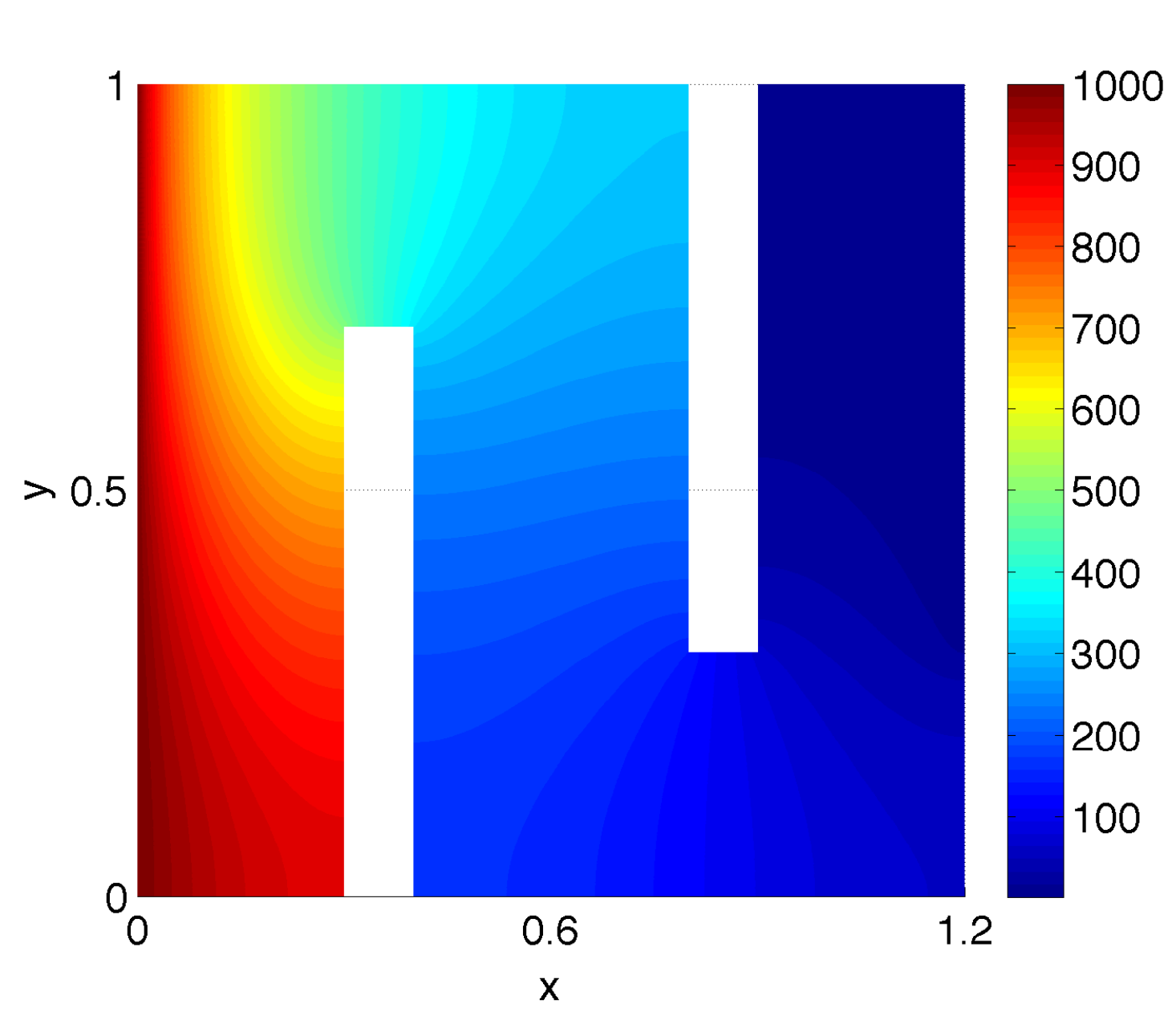}}
  \caption{Staggered impervious zones problem: pressure contours using LS formalism}
  \label{Fig:Staggered_reservoir_pressure_LS}
\end{figure}
\begin{figure}
  \centering
  \subfigure[Darcy model]{
  	\includegraphics[scale=0.46]
    {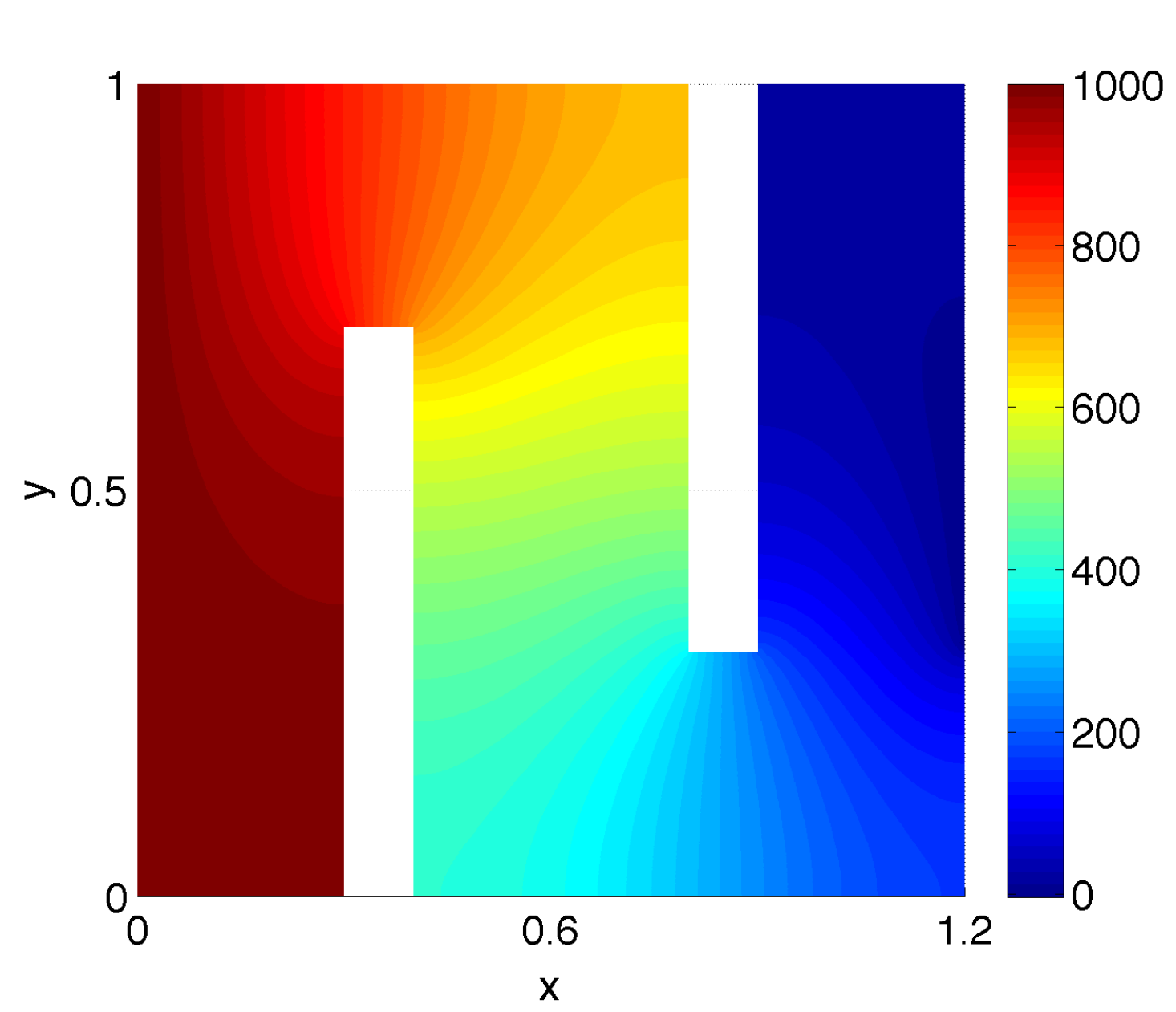}}
  \subfigure[Modified Barus]{
  	\includegraphics[scale=0.46]
    {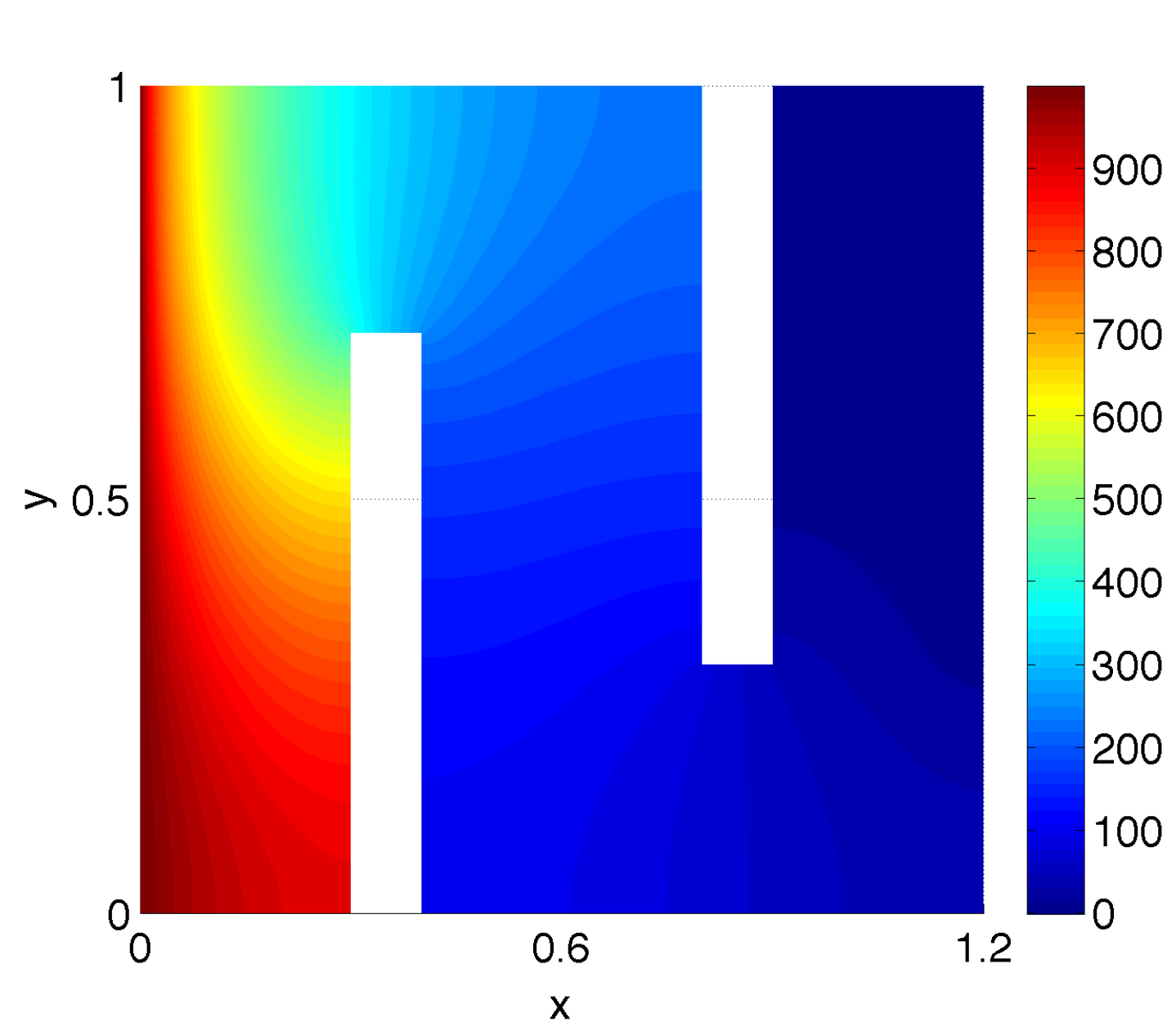}}
  \subfigure[Darcy-Forchheimer]{
  	\includegraphics[scale=0.46]
    {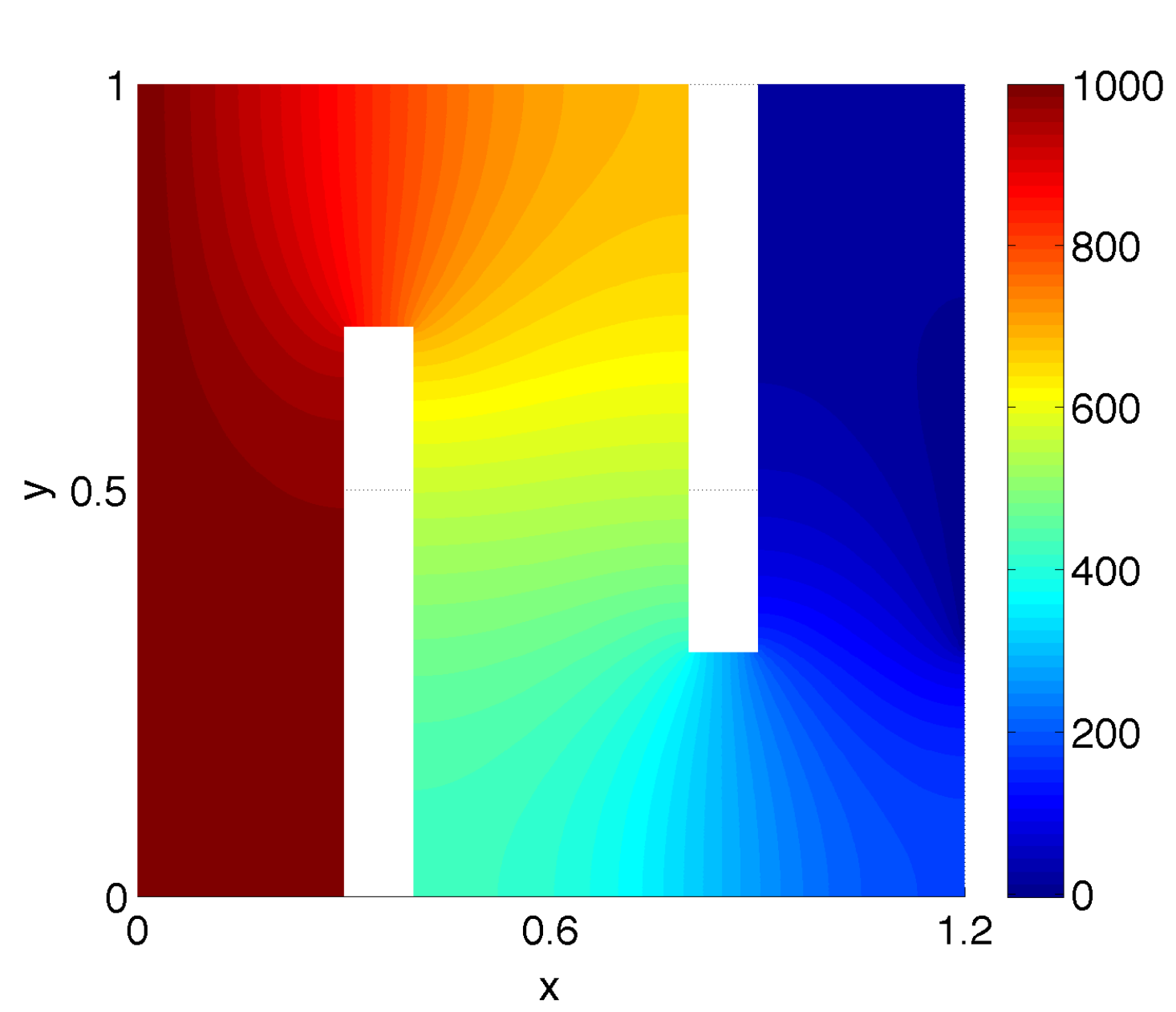}}
  \subfigure[Modified Darcy-Forchheimer Barus]{
  	\includegraphics[scale=0.46]
    {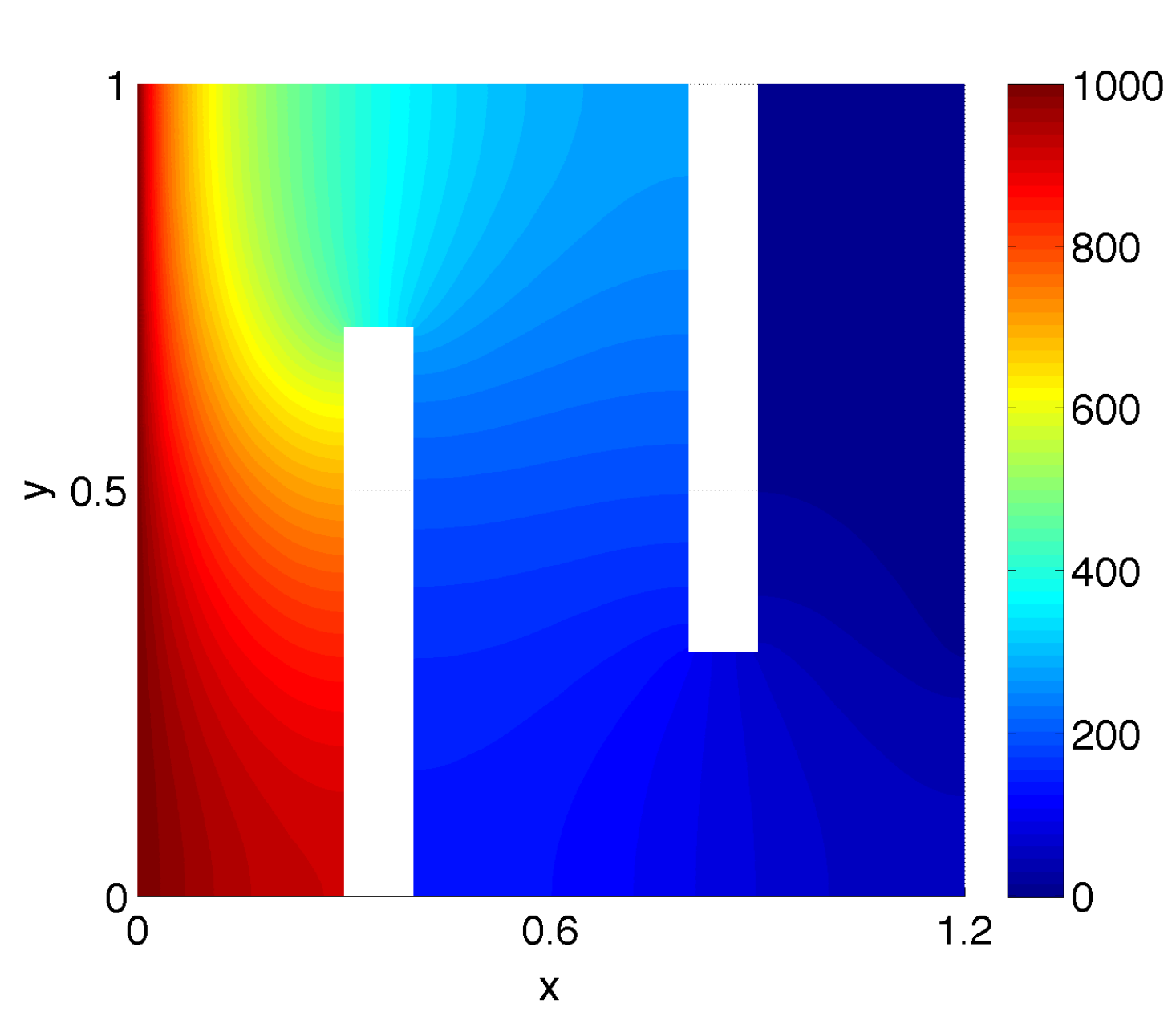}}
  \caption{Staggered impervious zones problem: pressure contours using VMS formalism}
  \label{Fig:Staggered_reservoir_pressure_VMS}
\end{figure}
\begin{figure}
  \centering
  \subfigure[Darcy model]{
  	\includegraphics[scale=0.46]
    {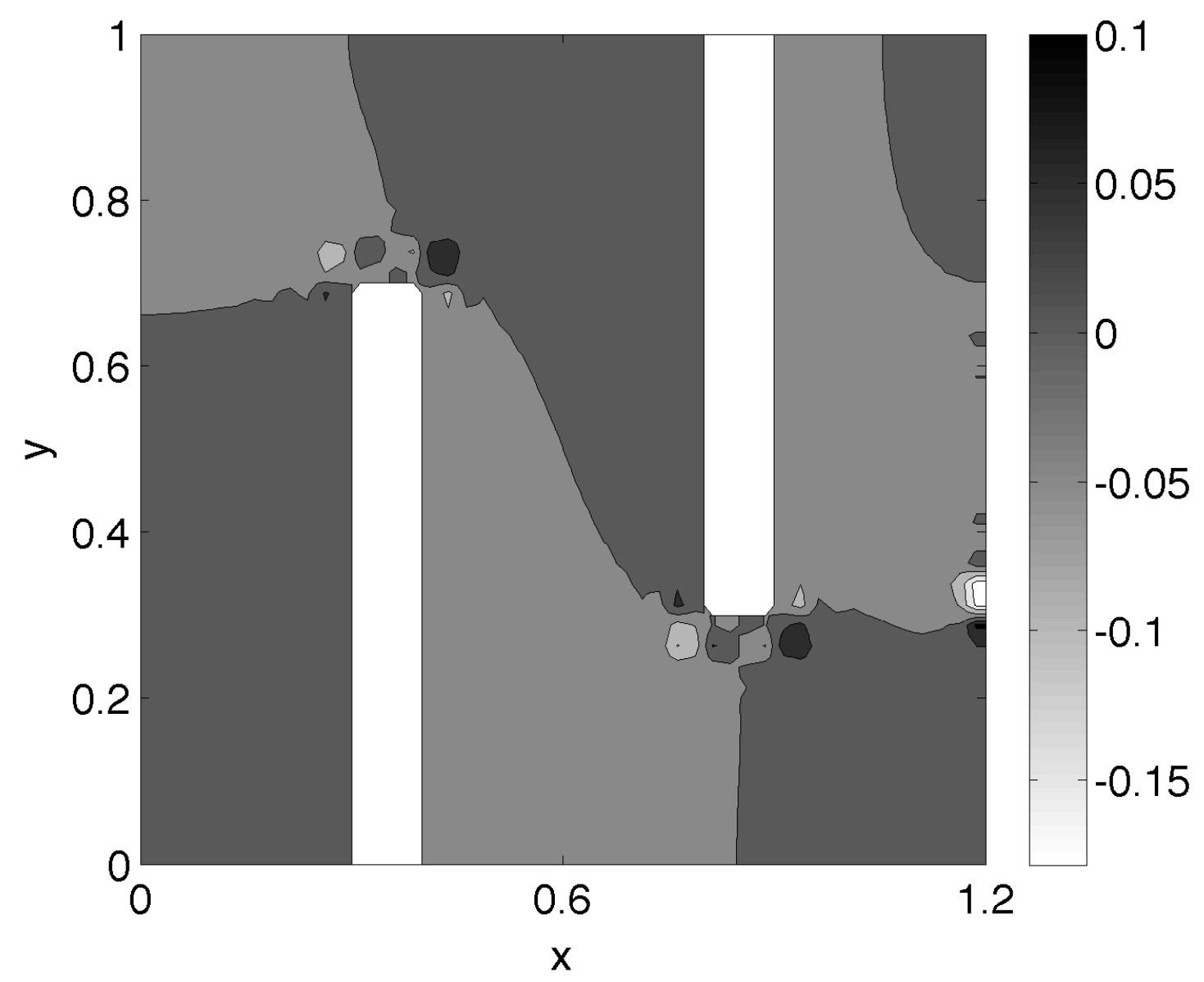}}
  \subfigure[Modified Barus]{
  	\includegraphics[scale=0.46]
    {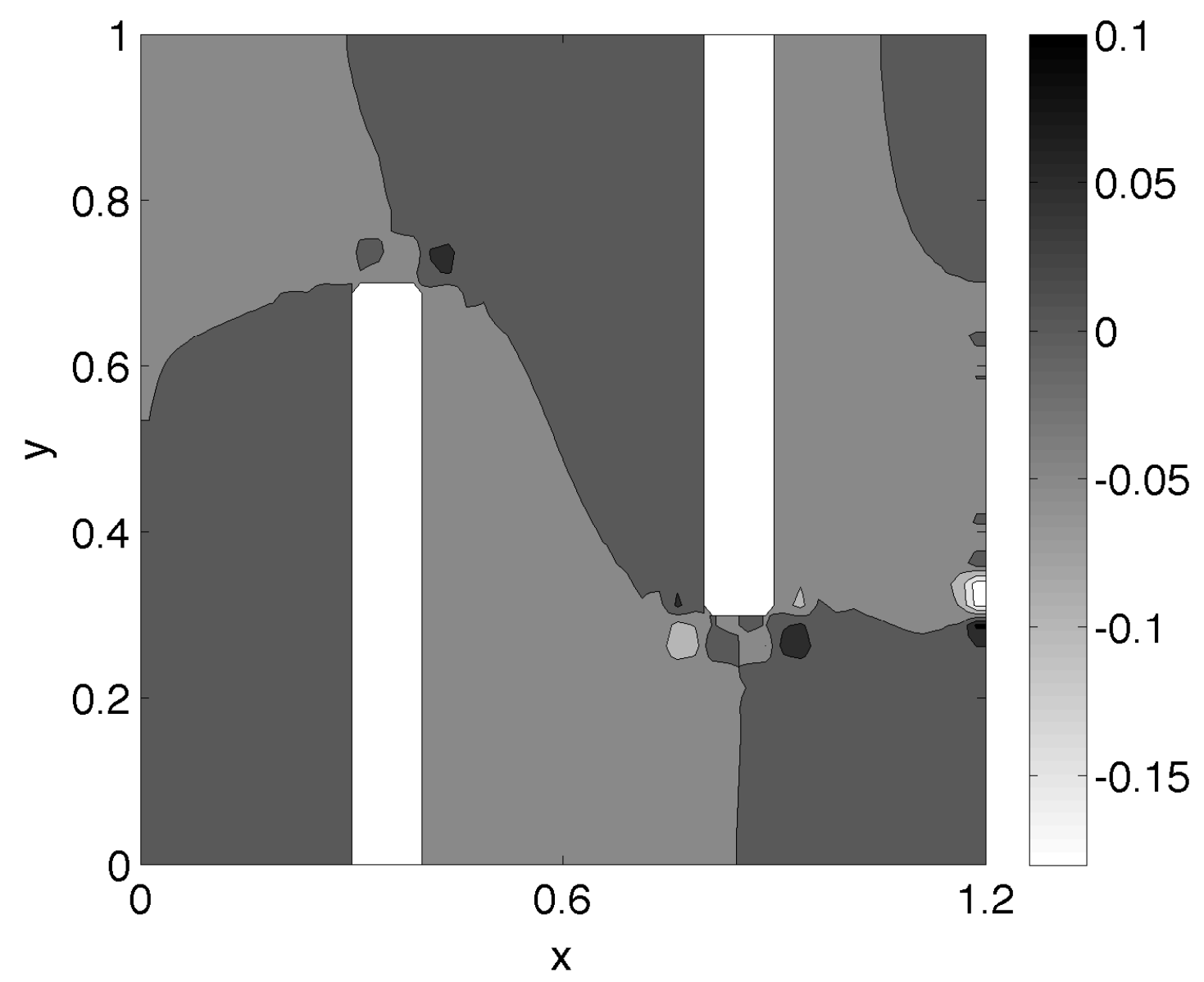}}
  \subfigure[Darcy-Forchheimer]{
  	\includegraphics[scale=0.46]
    {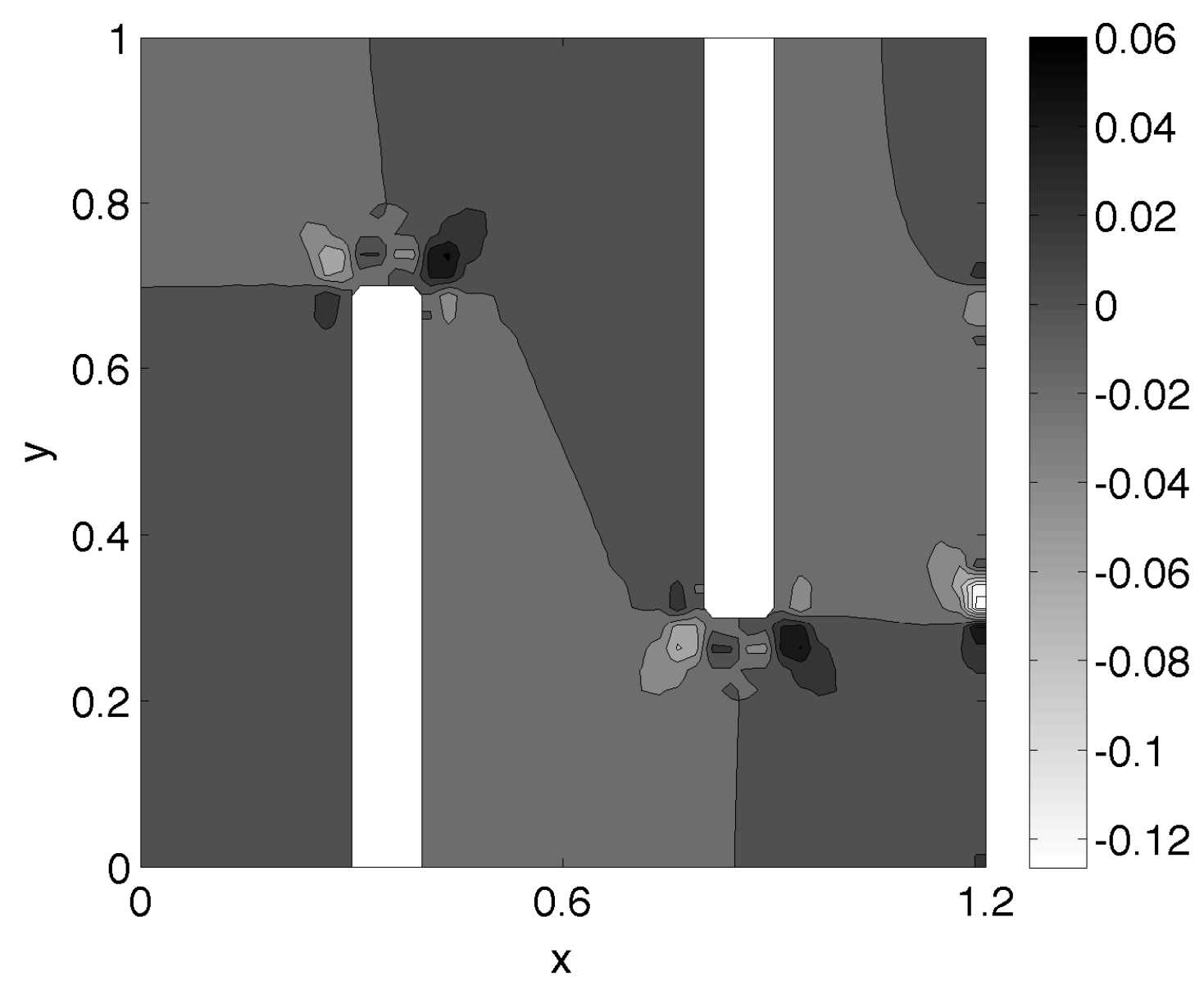}}
  \subfigure[Modified Darcy-Forchheimer Barus]{
  	\includegraphics[scale=0.46]
    {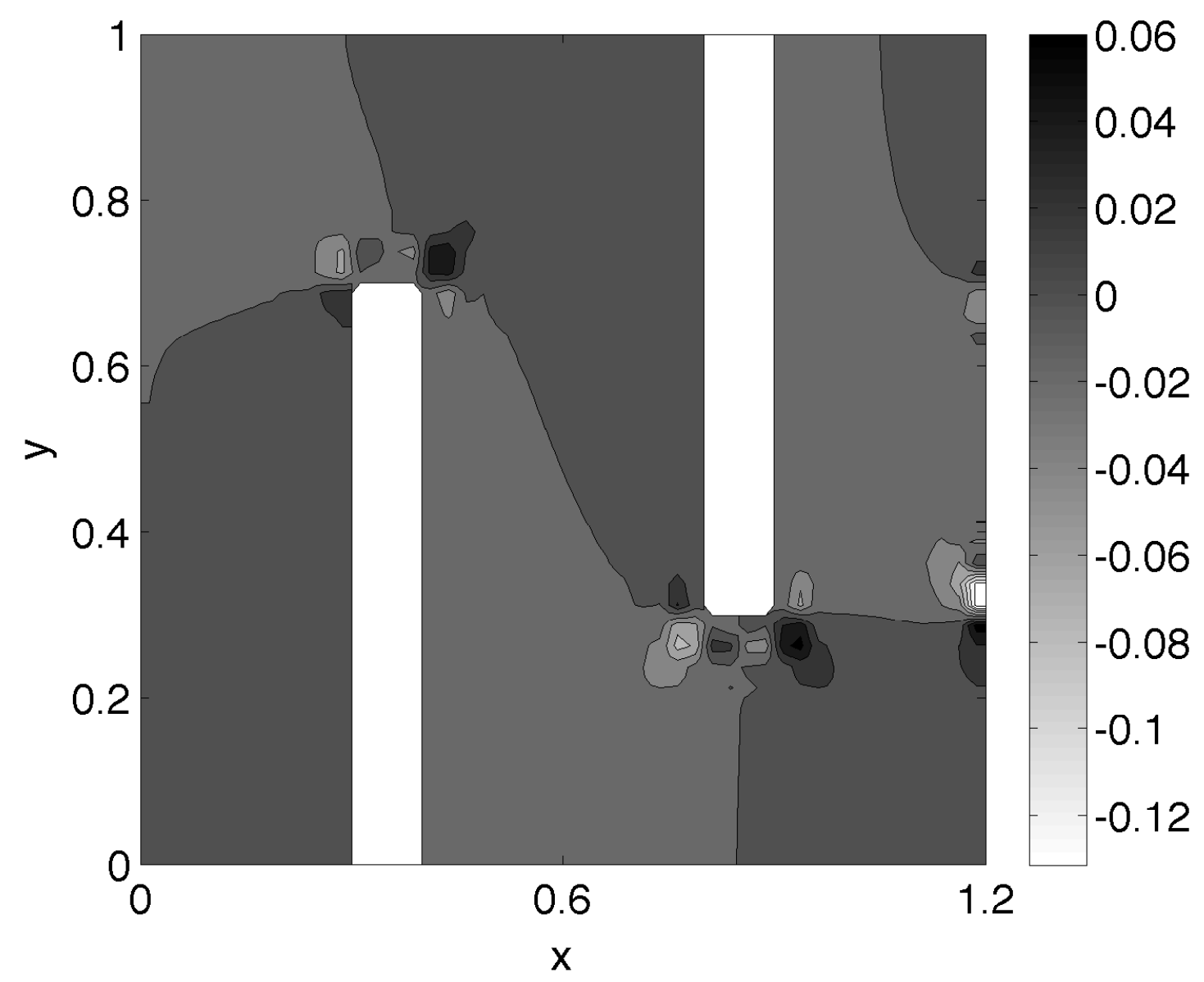}}
  \caption{Staggered impervious zones problem: ratios of local mass balance error over total predicted flux using LS formalism}
  \label{Fig:Staggered_reservoir_mass_error_LS}
\end{figure}
\begin{figure}
  \centering
  \subfigure[Darcy model]{
  	\includegraphics[scale=0.46]
    {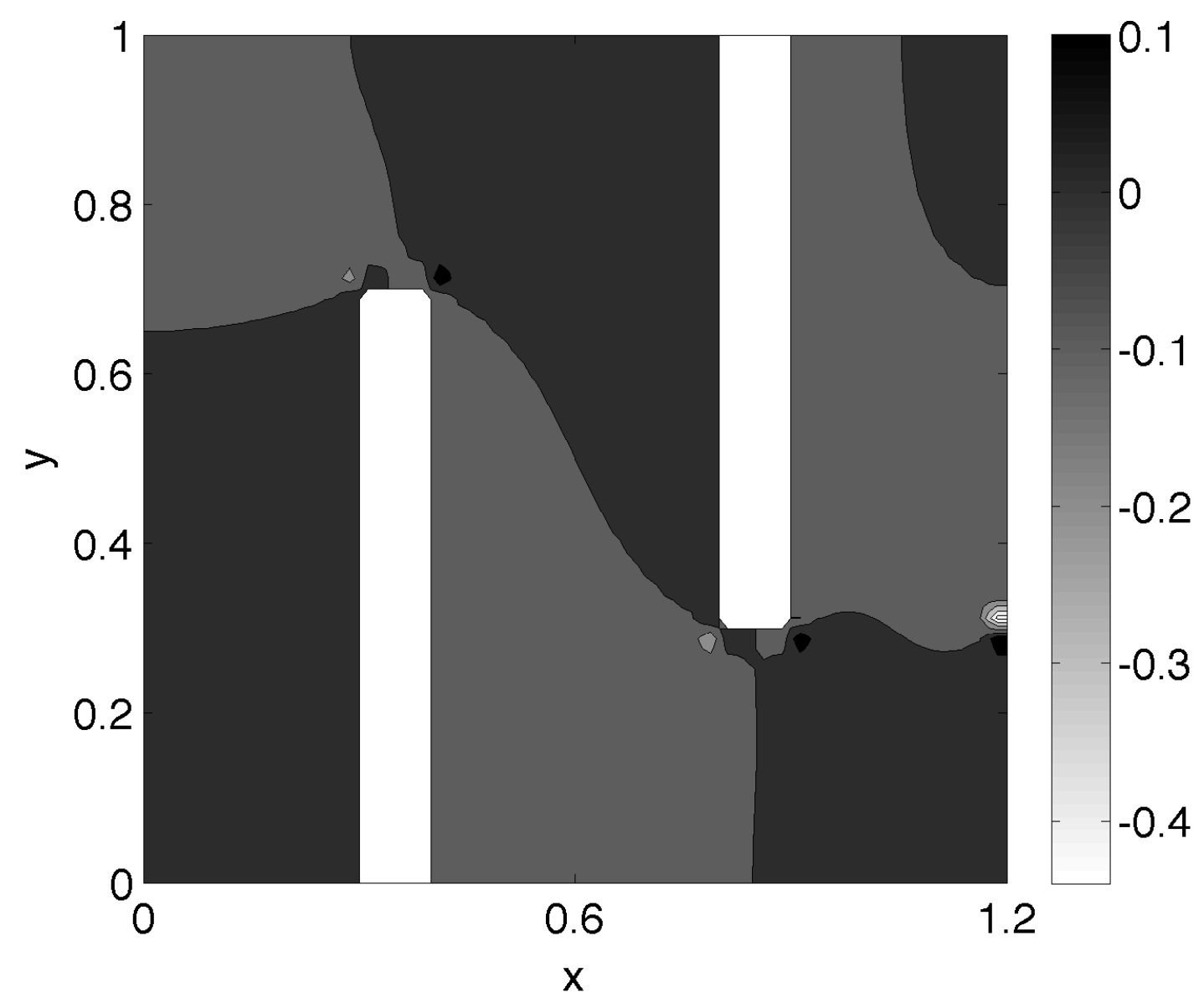}}
  \subfigure[Modified Barus]{
  	\includegraphics[scale=0.46]
    {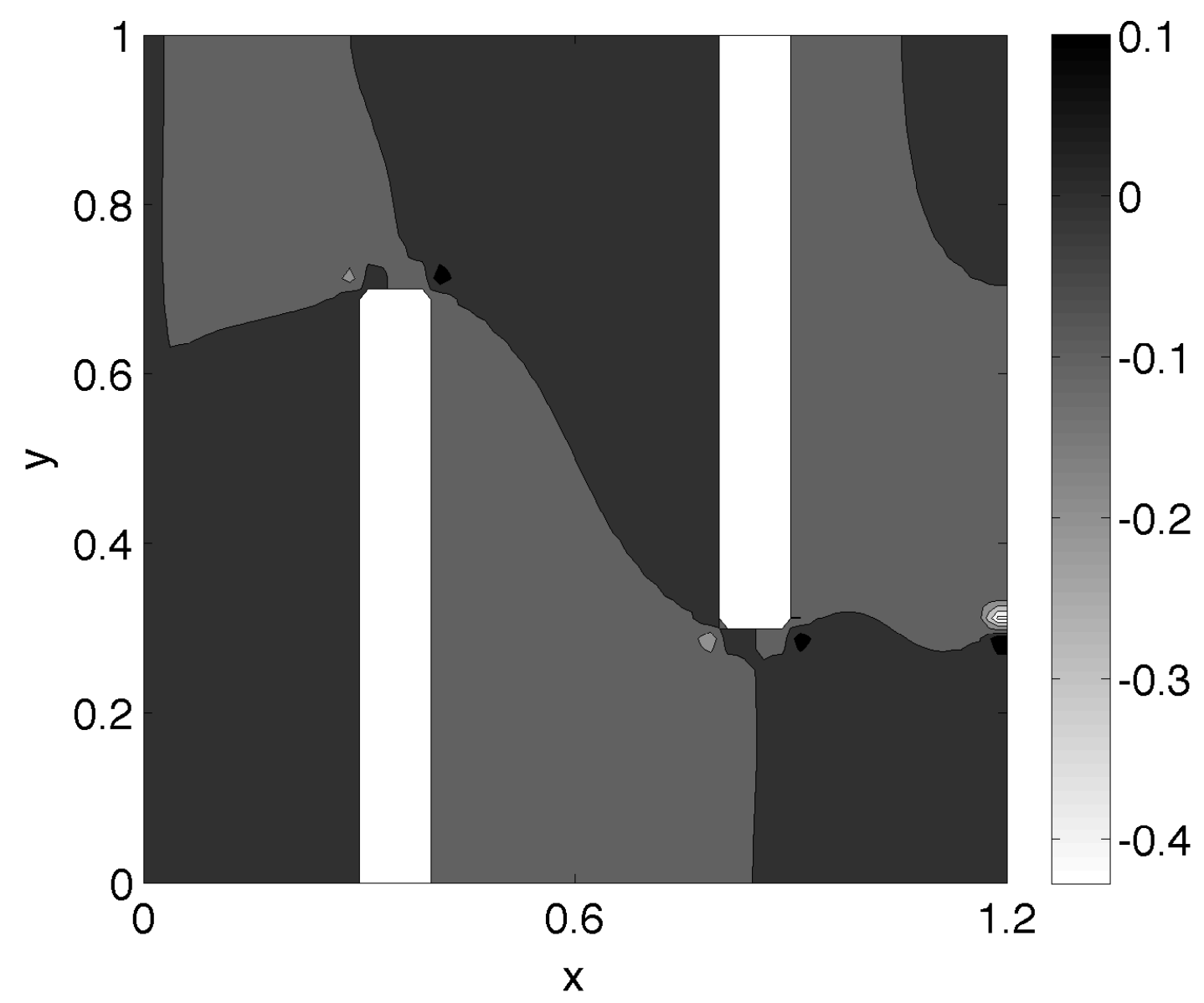}}
  \subfigure[Darcy-Forchheimer]{
  	\includegraphics[scale=0.46]
    {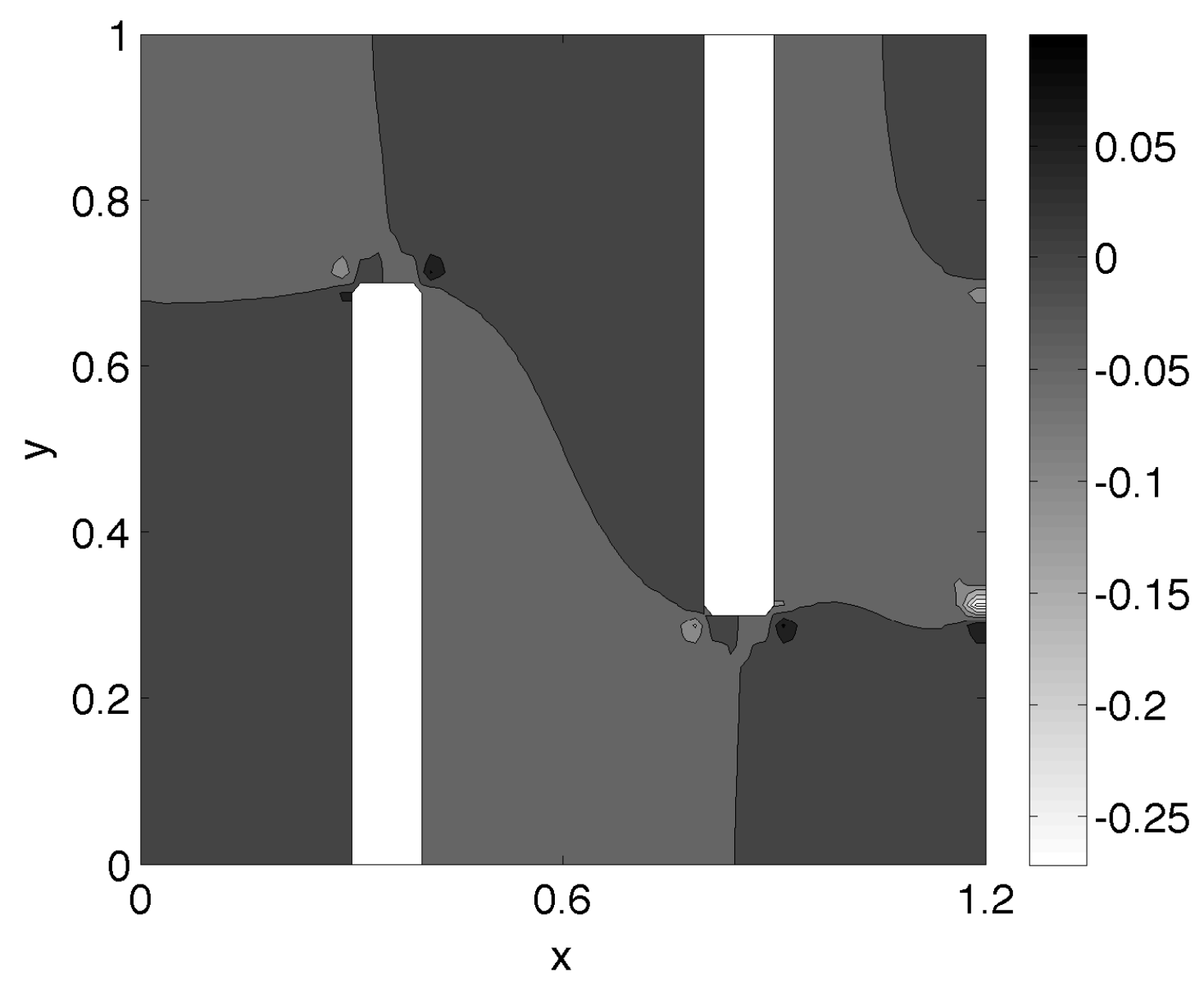}}
  \subfigure[Modified Darcy-Forchheimer Barus]{
  	\includegraphics[scale=0.46]
    {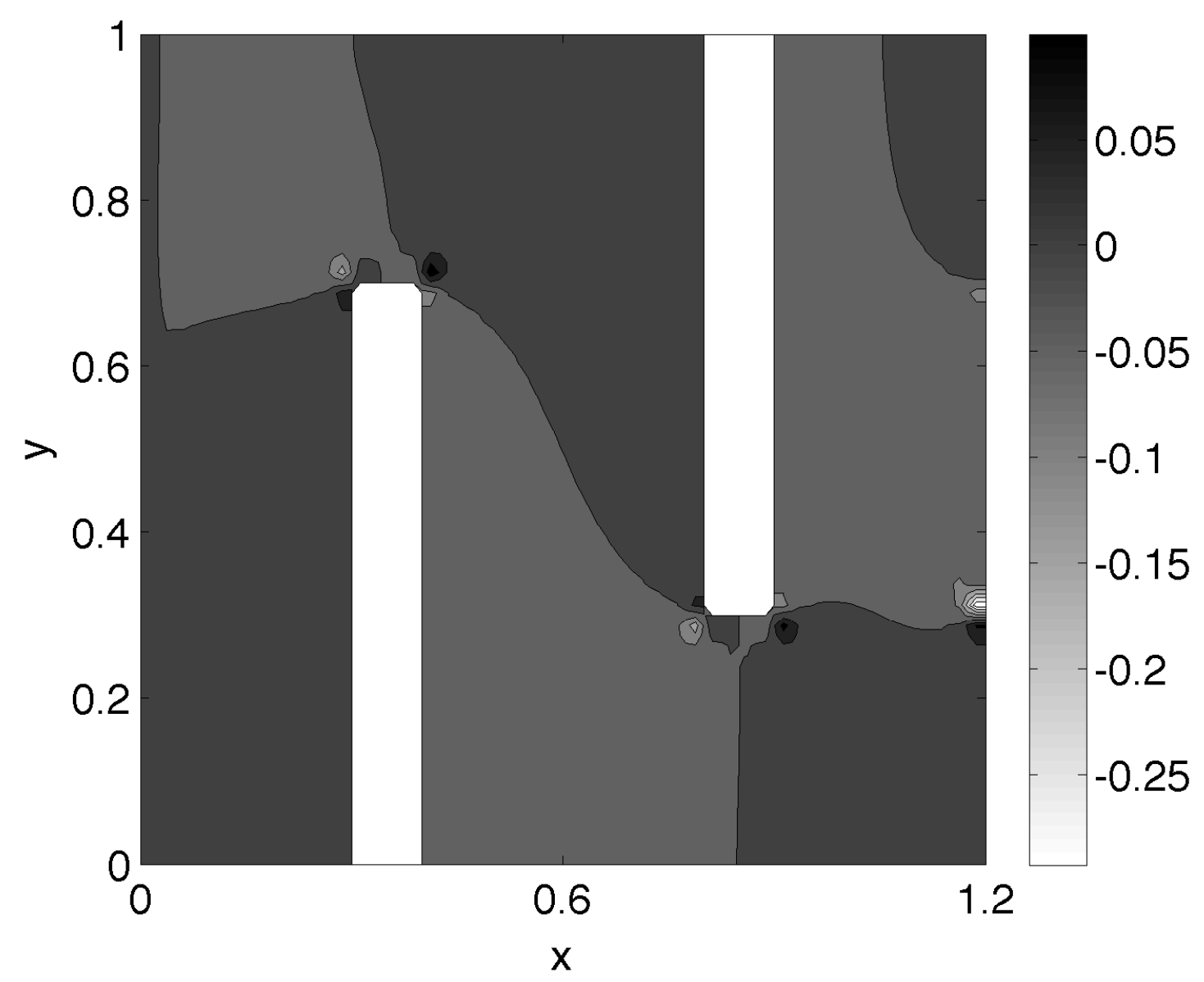}}
  \caption{Staggered impervious zones problem: ratios of local mass balance error over total predicted flux using VMS formalism}
  \label{Fig:Staggered_reservoir_mass_error_VMS}
\end{figure}